%% file: Arxiv_Version.tex
\documentclass[authoryear, review]{elsarticle}

\usepackage[table,xcdraw,dvipsnames]{xcolor}

\usepackage[english]{babel}
\usepackage[utf8]{inputenc}

\usepackage{amsmath}
\usepackage{amssymb}
\usepackage{amsfonts}
\usepackage{amsthm}
\usepackage{mathtools}
\usepackage{bm}
\usepackage{mathrsfs}
\usepackage{dsfont}
\allowdisplaybreaks[4]

\usepackage{algorithm}
\usepackage{algorithmic}
\usepackage{textcomp}

\usepackage{graphicx} 
\usepackage{rotating}
\usepackage{pdflscape}
\usepackage{pdfpages}
\usepackage{subcaption} 

\usepackage{array}
\usepackage{tabularx}
\usepackage{multirow}
\usepackage{booktabs}
\usepackage{longtable}
\usepackage{cellspace}
\makeatletter
\let\@startpbox@action\@startpbox
\makeatother
\usepackage{makecell}
\usepackage{hhline}
\usepackage{colortbl}
\usepackage{arydshln}
\usepackage{soul}
\usepackage{enumitem}
\usepackage[title]{appendix}
\usepackage{comment}
\usepackage{catchfile}
\usepackage[a4paper,
top=2.3cm,
bottom=2.3cm,
left=3cm,
right=3cm]{geometry}
\usepackage{eurosym}

\usepackage[table,xcdraw,dvipsnames]{xcolor}
\usepackage{pifont}

\usepackage{tikz}
\usetikzlibrary{automata,arrows,positioning,calc,patterns,shapes.geometric}
\usepackage{setspace}
\usepackage{hyperref}
\usepackage{cleveref}
\hypersetup{
	colorlinks=true,
	linkcolor=blue!66!black,
	citecolor=blue!66!black,
	urlcolor=blue!66!black
}

\theoremstyle{definition}
\newtheorem{mydef}{Definition}
\theoremstyle{remark}

\theoremstyle{proposition}
\newtheorem{myprop}{Proposition}
\theoremstyle{example}
\newtheorem{myex}{Example}
\newtheorem{remark}{Remark}
\newtheorem{lemma}{Lemma}

\newcommand{\VV}[1]{\textcolor{black}{#1}}

\newcommand{\V}[1]{\textcolor{black}{#1}}
\newcommand{\VVP}[1]{\textcolor{black}{#1}}

\begin{document}

\begin{frontmatter}

\title{\textbf{Bounds for multi-horizon stochastic optimization with application to power generation and transmission expansion planning}}

\author[1]{Giovanni Micheli}
\author[2]{V Varagapriya}
\author[1]{Francesca Maggioni}
\author[3]{G{\"u}zin Bayraksan}

\affiliation[1]{organization={Department of Management, Information and Production Engineering, University of Bergamo},
            city={Bergamo},
            country={Italy}}
\affiliation[2]{organization={Decision Sciences, Indian Institute of Management Visakhapatnam},
            city={Visakhapatnam},
            country={India}}
\affiliation[3]{organization={Department of Integrated Systems Engineering, the Ohio State University},
            city={Columbus},
            country={OH, USA}}       

\begin{abstract}
This paper investigates computationally efficient methods for deriving bounds on the optimal value of multi-horizon stochastic optimization problems, with a particular focus on applications in power generation and transmission expansion planning. Multi-horizon stochastic programs capture sequential decision-making under uncertainty across multiple time scales---e.g., strategic (long-term) and operational (short-term)---jointly. Due to their inherent complexity, especially when uncertainties span several time horizons, solving these problems directly becomes computationally prohibitive. To address this, the paper develops and analyzes various novel bounding techniques, based on the dissection of scenario trees.
We investigate systematically dissecting (i) only the operational, (ii) only the strategic, or (iii) both  scenario trees   simultaneously, and we  devise two ways to   recombine them to obtain valid bounds.
Each method leads to a monotonic chain of inequalities that approximate the optimal value of the original problem from below. 
 One of these methods results in a significantly smaller number of subgroups to recombine in  the operational and simultaneous dissections, leading to substantial computational savings. 
Numerical results on a  multi-horizon mixed-integer generation and transmission  expansion planning problem show the efficiency of the proposed approach through different dissection and recombination strategies.
\end{abstract}


\end{frontmatter}

\section{Introduction}
\vspace*{-0.05in}

In many real-world decision-making problems, such as infrastructure development, network design, and supply chain management, decisions unfold across multiple time scales. These problems typically involve long-term strategic decisions, followed by a sequence of short-term operational decisions within each strategic period. Uncertainty affects both levels, and the interaction between strategic and operational decisions plays a crucial role in determining overall solution quality.

A representative example can be found in electric power systems, where long-term capacity expansion decisions must be coordinated with short-term operational choices such as generation dispatch and unit commitment  \citep[see][]{parpas2014stochastic}. While strategic decisions are often made annually or monthly, operational decisions occur at finer temporal resolutions, daily or even hourly. Moreover, these systems face significant uncertainties, particularly from variable renewable generation and demand-side variability. Therefore, in order to make better decisions, it is essential to integrate uncertainty and decisions across multiple time scales (or, multiple horizons).

Incorporating multi-time-scale uncertainties leads to complex stochastic optimization models that suffer from severe dimensionality growth. To address this challenge, numerous approaches have been proposed, including the separation of strategic and operational decisions and the use of simplified operational models \citep[see][]{myklebust2010techno,schutz2009supply}. However, such simplifications often fail to capture the critical interdependencies between long-term investment choices and short-term operational flexibility and therefore lead to suboptimal decisions.

To overcome these limitations, \cite{kaut2014multi} introduced the Multi-Horizon Stochastic Program (MHSP), a modeling framework that explicitly represents strategic and operational uncertainties through a nested scenario tree structure. In this formulation, operational scenarios are embedded within their corresponding strategic scenarios, enabling a more accurate representation of multi-time-scale uncertainty. The MHSP framework has since been applied to a wide range of real-world problems, including generation and transmission planning under renewable uncertainty \citep[see][]{skar2016multi}, multi-fuel market analysis \citep[see][]{su2015multi}, and pumped-storage hydropower scheduling \citep[see][]{abgottspon2016multi}. It has also been used to study policy impacts on energy system development \citep[see][]{seljom2017impact}  and to integrate long- and short-term uncertainties in capacity expansion models \citep[see][]{zhang2025integrated}. From a theoretical standpoint, MHSP has been employed to investigate time consistency \citep[see][]{werner2013risk} and to incorporate stochastic dominance constraints \citep[see][]{escudero2018capacity}.

Although MHSP significantly reduces the dimensionality compared to fully integrated formulations, solving realistic instances remains computationally challenging, especially when both the strategic and operational dimensions are represented with high granularity. This motivates the development of  decomposition and approximation techniques tailored to MHSP structures. 
Classical decomposition methods such as Benders \citep[e.g.,][]{vojvodic2023experimentation,jacobson2024computationally,zhang2024stabilised}, Lagrangian \citep{zhang2024decomposition}, time-block \citep{rigaut2024decomposition}, and Dantzig-Wolfe \citep{downward2020judge}  decompositions have been extensively studied for this purpose.

Bounding techniques offer another important approach to managing complexity and assessing the quality of approximate solutions in MHSPs. 
While there exists a rich body of literature on bounding techniques for classical multistage stochastic programs (see, e.g.,  \cite{narum2023safe} and \cite{maggioni2024stochastic} for a review on the topic), there is very little work on MHSPs.
The first study of bounding techniques for MHSPs was presented by \cite{maggioni2020bounds}, who derived classical bounds, such as the expected value and the wait-and-see bounds, within the MHSP framework, establishing chains of inequalities and formal relationships between them. These results laid the foundation for rigorous bound analysis in multi-horizon settings, where both strategic and operational uncertainties are explicitly considered.

Unlike classical bounds, scenario tree decomposition based bounds, which are closely related to this work, were introduced in \cite{maggioni2014bounds,maggioni2016monotonic} and \cite{maggioni2016bounds} for multistage stochastic programs. 
In particular, \cite{maggioni2016monotonic} devised  bounds for stochastic multistage mixed-integer linear programs by solving a sequence of group subproblems, showed the monotonicity of the chain of lower bounds, and established their algorithmic use.  
\cite{maggioni2016bounds} considered multistage convex problems with a concave risk functional, proposing new refinement chains of lower bounds that require far fewer subproblems than in \cite{maggioni2016monotonic}, while maintaining monotonicity. 
A scalable bounding framework for multistage stochastic programs is further developed in \cite{sandikcci2017scalable}, extending \cite{sandikcci2013hierarchy}.

This paper builds upon this line of research and, to the best of our knowledge, introduces the first scenario tree decomposition based bounding techniques for MHSPs. The distinctive structure of the MHSP scenario tree necessitates devising new techniques.  
Specifically, we dissect the strategic component, operational component, or both, of the MHSP scenario tree into smaller subgroups and subsequently recombine them in two alternative ways---expectation-based and node-based---to construct valid lower bounds. Due to the unique structure of the multi-horizon scenario tree, the node-based scheme leads to a significantly smaller number of subgroups in the operational and joint dissections, while the two approaches are equivalent in the strategic dissection. We establish the monotonicity of the resulting chain of lower bounds in each scheme.

In summary, the main contributions of this paper are as follows:\vspace*{-0.05in}
\begin{enumerate}
\item We propose a new methodology for deriving lower bounds by dissecting the whole multi-horizon scenario tree 
into smaller subgroups.
 Three distinct dissection strategies are introduced:
(i) dissecting only the operational subtree;
(ii) dissecting only the strategic tree; and 
(iii) dissecting both simultaneously.
We then present two aggregation schemes, expectation-based and node-based, to recombine subproblem solutions into valid lower bounds, forming monotonic chains of inequalities that tighten as subgroup size increases.
\vspace*{-0.05in}

\item To demonstrate the effectiveness and scalability of the proposed bounding techniques, we conduct a comprehensive numerical study on a practically important application of MHSP---namely, an investment planning problem in electricity generation and transmission---using small, medium, and large instances.\vspace*{-0.05in}
\end{enumerate}

The remainder of the paper is organized as follows.
Section \ref{sec:P_B} introduces the MHSP formulation.
Section \ref{sec:LBmeasure_diss} presents various proposed lower-bounding methodologies based on MHSP scenario tree dissection.
Numerical experiments are reported in Section \ref{sec:numres}, and concluding remarks are given in Section \ref{sec:conclusion}.
The appendices include supplementary material such as detailed model formulations, formal proofs, illustrative examples, and the heuristic procedures used for the large-scale instances.
\vspace*{-0.23in}

\section{Preliminaries and Background}\label{sec:P_B}
\vspace*{-0.05in}

This section recalls the notation and formulation to describe an MHSP. We consider a finite-horizon sequential decision making problem under uncertainty, where a decision maker must jointly take strategic decisions for long-term planning and operational decisions for short-term planning to minimize expected overall costs. The following assumptions are made:
(i) strategic uncertainty does not depend on the operational uncertainty of all time periods;
(ii) strategic decisions are independent of earlier operational decisions; and
(iii) there is no connection between operational uncertainty of two consecutive strategic nodes, meaning that the first operational decision in a strategic node does not depend on the last operational decision of the previous period.

We now introduce the relevant notation for the problem.
Let ${\cal H}:=\{0,\ldots,H\}$ be the set of strategic stages, with $H$ denoting the strategic planning horizon. The strategic uncertainty is described by the uncertain process $\mbox{\boldmath $\xi$}_{H}:=(\xi_0, \dots, \xi_{H}) \in \mathbb{R}^{d_0} \times \ldots \times \mathbb{R}^{d_{H}}$, defined on a measurable space $(\Xi,\mathcal{A})$, where $\Xi= {\sf X}_{t=0}^{H} \Xi_t$ and $\mathcal{A} = (\mathcal{A}_0, \dots, \mathcal{A}_{H})$ is the filtration generated by projections of $\Xi$ onto  ${\sf X}_{i=0}^t \Xi_i $ for each $t$. 
We assume $\xi_0$ is a constant. Let us denote by $\Pi$ the probability distribution on $\mathcal{A}$.
With $\mbox{\boldmath $\xi$}_{t}:=(\xi_0,\dots,\xi_t)$, $t\in \mathcal{H}$, we denote the history of the process up to time $t$, while with $\mathbb{E}_{\xi_t}$, we denote the expectation with respect to the distribution of $\xi_t$.
The strategic decision process begins with initial strategic decision $\mathrm{x}_0 \in {\mathbb R}_+^{n_0} \times \mathbb{Z}_+^{n'_0}$ at stage $t=0$, called the \textit{first-stage} or \textit{here-and-now} decision, followed by sequential decisions $\mathrm{x}_t \in \mathbb{R}_+^{n_t}\times \mathbb{Z}_+^{n'_t}$ at stages $t \in \mathcal{H}\setminus\{0\}$. 
In the following, for a simpler presentation, the feasibility condition on $\mathrm{x}_t$ will be
omitted even if assumed to be satisfied.
The history of the strategic decision process  at a given point in time $t \in {\cal H}$ is denoted by $\mathbf{x}_{t} \vcentcolon =(\mathrm{x}_0, \mathrm{x}_1, \dots, \mathrm{x}_t)$.
The possibly nonlinear strategic cost functions are given by $c_0: \mathbb{R}_+^{n_0}\times \mathbb{Z}_+^{n'_0} \rightarrow \mathbb{R}$ in the first stage and by $c_t: \mathbb{R}_+^{n_t}\times \mathbb{Z}_+^{n'_t} \times \mathbb{R}^{d_t} \rightarrow \mathbb{R}$ in stages $t \in \mathcal{H} \setminus \{0\}$, which are  assumed to be $\mathcal{A}_t$-measurable. 
Additionally, let us denote the vectors and matrices related to uncertain strategic parameters at strategic time $t\in\mathcal{H}\backslash\left\{0\right\}$ as 
 $h_{t}\in\mathbb{R}^{{m}_t}$,  $T_{t} \in\mathbb{R}^{{m}_{t}\times ({n}_{t-1} + n_{t-1}^\prime )}$ and $W_{t}\in\mathbb{R}^{{m}_t\times ({n}_{t} + n_t^\prime)}$, respectively. In the first strategic stage, the strategic parameters $h_0$ and $T_0=A$ are known, while $W_0=0$ (i.e., the null matrix).
 
For each strategic stage $t \in {\cal H}$, let $\mathcal{T}_t := \left\{1,\ldots,O_t\right\}$ be the corresponding set of operational stages.
On the operational side, after a strategic decision at time $t\in\mathcal{H}$ is taken, we describe the operational uncertainty by the random process ${\bm{\eta}_t^{O_t}} \vcentcolon = (\eta_t^1,\dots,\eta_t^{O_t}) \in \mathbb{R}^{d_t^1} \times \ldots \times \mathbb{R}^{d_t^{O_t}}$ 
defined on a measurable space $(\Omega_n,\mathcal{F}_t)$. Let us denote with $P_t$ the probability distribution of the operational uncertainty on $\mathcal{F}_t$.
With $\mbox{\boldmath $\eta$}^{\tau}_t:=(\eta^1_t,\dots,\eta^{\tau}_t)$, $\tau \in \mathcal{T}_t,$ we denote the history of the process up to time $\tau$ at strategic time $t\in\mathcal{H}$. Furthermore, we denote with $\mathbb{E}_{\eta_t^\tau}$, the expectation with respect to the distribution of $\eta^\tau_t$. Notice that 
$\mbox{\boldmath $\eta$}^{\tau}_{t+1}$ is independent of  $\mbox{\boldmath $\eta$}^{\tau}_t$, $\tau\in\mathcal{T}_t$, $t\in\mathcal{H}$.
The operational decisions consist of sequential decisions $\mathrm{y}_t^1,\ldots,\mathrm{y}_t^{O_t}$, $t\in\mathcal{H}$, with $\mathrm{y}_t^\tau \in {\mathbb R}_+^{n^\tau_t} \times \mathbb{Z}_+^{n^{\prime \tau}_t}, \tau \in {\cal T}_t, t \in {\cal H}$. 
In the following, for a simpler presentation, the feasibility condition on $\mathrm{y}_t^\tau$ will be omitted even if assumed to be satisfied.
The history of the operational decision process, at a given point in time $\tau \in {\cal T}_t, t \in {\cal H}$, is denoted by $\mathbf{y}_t^{\tau} \vcentcolon =(\mathrm{y}_t^1, \mathrm{y}_t^2, \dots, \mathrm{y}_t^\tau)$. 
The possibly nonlinear cost functions in operational stages $\tau \in {\cal T}_t$
are given by $q_t^\tau: \mathbb{R}_+^{n_t^\tau}\times \mathbb{Z}_+^{n_t^{\prime\tau}} \times \mathbb{R}^{d_t^\tau} \rightarrow \mathbb{R}$. 
Let us denote 
the vectors and matrices of uncertain operational parameters at 
operational stage $\tau\in\mathcal{T}_t$ of strategic stage $t\in \mathcal{H}$ as $h_{t}^{\tau}\in\mathbb{R}^{m_{t}^{\tau}}$, 
$T_{t}^{1}\in\mathbb{R}^{m_t^{1}\times ({n}_{t} + n_t^\prime) }$, $T_{t}^{\tau}\in\mathbb{R}^{{m}_{t}^{\tau}\times ( {n}_{t}^{\tau-1} +{n_t^{\prime\tau-1}})},$  and $W_{t}^{\tau}\in\mathbb{R}^{{m}_{t}^{\tau}\times ( {n}_{t}^{\tau} + {n_t^{\prime\tau}}) }$.

With this notation, a \textit{Multi-Horizon Stochastic Program} is formulated as follows:\vspace*{-0.15in} 
\begin{subequations}
\label{MHSP}
\begin{align}
     & \min\limits_{\bm{\mathrm{x}},\bm{\mathrm{y}}} 
 c_0(\mathrm{x}_0,{{\xi}}_0) 
  + 
 \mathbb{E}_{\eta^1_0}\Big[
\min\limits_{\mathrm{y}_0^1} q_0^1(\mathrm{y}_0^1,\bm{\eta}_0^1) \ +   
  \ldots + \mathbb{E}_{\eta^{O_0}_0} \big[ \min\limits_{\mathrm{y}_0^{O_0}} q_0^{O_0}(\mathrm{y}_0^{O_0},\bm{\eta}_0^{O_0}) 
  \big] \Big] +  \nonumber
  \\ \nonumber
    &\quad \quad +
\mathbb{E}_{{{\xi}_1}}\Bigg[\min\limits_{\mathrm{x}_1} c_1(\mathrm{x}_1,{\bm{\xi}}_1) 
+\mathbb{E}_{\eta^1_1}\Big[  
  \min\limits_{\mathrm{y}_1^1} q_1^1(\mathrm{y}_1^1,\bm{\eta}_1^1)  \! + \!  \ldots  \! 
 +  \! \mathbb{E}_{\eta^{O_1}_1} \big[ \min\limits_{\mathrm{y}_1^{O_1}} q_1^{O_1}(\mathrm{y}_1^{O_1},\bm{\eta}_1^{O_1}) 
  \big] \Big]  
  +  \\  
  &\quad \quad + \ldots +\mathbb{E}_{\xi_{H}}\bigg[\min\limits_{\mathrm{x}_H} c_H(\mathrm{x}_H,{\bm{\xi}}_H) +
  \mathbb{E}_{\eta^1_H}\Big[  
\min\limits_{\mathrm{y}_H^{1}} q_H^{1}(\mathrm{y}_H^{1},\bm{\eta}_H^{1}) + \ldots  
 + \nonumber \\
 & \quad \quad + \mathbb{E}_{\eta^{O_H}_H} \big[ \min\limits_{\mathrm{y}_H^{O_H}} q_H^{O_H}(\mathrm{y}_H^{O_H},\bm{\eta}_H^{O_H}) 
  \big] \Big]  \bigg]  \Bigg]  \label{eq:1a} \\ 
 & \qquad \textrm{ s.t. } A \mathrm{x_0} = h_0,\label{eq:1b}\\  
 & \qquad\qquad\, T_t(\bm{\xi_{t-1}})\mathrm{x}_{t-1} + W_t(\bm{\xi_{t-1}}) \mathrm{x}_{t} = h_t (\bm{\xi_{t-1}}), \ t \in \mathcal{H}\setminus \{0\}, \label{eq:1d}\\  
& \qquad\qquad\, T_{0}^1 (\bm{\eta}_0^{1})\mathrm{x_0} + W_0^1 (\bm{\eta}_0^{1})\mathrm{y_0^1} = h_0^1 (\bm{\eta}_0^{1}), \label{eq:1e}\\  
& \qquad \qquad \, T_{t}^1 (\bm{\eta}_t^{1})\mathrm{x}_t + W_t^1 (\bm{\eta}_t^{1})\mathrm{y}_t^1 = h_t^1 (\bm{\eta}_t^{1}), \  t \in \mathcal{H}\setminus \{0\}, \label{eq:1f}\\ 
& \qquad\qquad\, T_{t}^\tau (\bm{\eta}_t^{\tau-1})\mathrm{y}_t^{\tau-1} + W_t^1 (\bm{\eta}_t^{\tau})\mathrm{y}_t^\tau = h_t^\tau (\bm{\eta}_t^{\tau}), \  \tau \in \mathcal{T}_t \setminus \{1\}, t \in \mathcal{H}. \label{eq:1g}
 \end{align}  
\end{subequations}
The objective function \eqref{eq:1a} minimizes the expected total cost of investments and operations over the entire planning horizon, accounting for both strategic and operational decisions. 
Equation \eqref{eq:1b} defines a deterministic constraint on the first-stage strategic decisions.
Equation \eqref{eq:1d} links strategic decisions across consecutive strategic stages $t-1$ and $t$, with $t \in \mathcal{H} \setminus \{0\}$.
Constraint \eqref{eq:1e} establishes the connection between strategic decisions at the root node and the operational decisions in the first operational period, while constraint \eqref{eq:1f} extends this linkage to subsequent strategic stages, i.e., for $t \in \mathcal{H} \setminus \{0\}$.
Finally, constraint \eqref{eq:1g} connects operational decisions across consecutive operational periods within the same strategic stage.\vspace*{-0.05in}

\subsection{Multi-Horizon Scenario Tree Approximations} \label{MH Tree}

Typically, MHSPs rely on continuous distributions to characterize both strategic and operational uncertainties. However, directly solving such ``infinite'' problems is often computationally intractable; therefore, scenario tree approximations of the underlying stochastic processes are commonly employed to make the problem tractable.
This is done by considering a finite number of realizations of both the
random strategic and operational processes.

The information structure at both the strategic and operational levels can be described in the form of a \textit{multi-horizon scenario tree} $\mathfrak{T}$. At each strategic stage $t\in\mathcal{H}$, the tree contains a discrete number of strategic nodes, where a specific realization of the uncertain parameters at strategic level takes place.  
Let  $\mathcal{N}_{t}$ be the set of ordered strategic nodes at stage $t \in \mathcal{H}$ and $\mathcal{N} = \bigcup_{t\in \mathcal{H}}\mathcal{N}_{t}$. By assumption, we have a discrete number, $|\mathcal{N}_t|$, of nodes at each stage $t \in \mathcal{H}$. 
Each strategic node at stage $t$, except the root, is connected to a unique node at stage $t-1$, called its ancestor, and to nodes at stage $t+1$, called its successors. 
For each strategic node $n$, we denote its ancestor by $a(n)$. 
The probability measure $\Pi$, defined on $\Xi$, induces a probability distribution on all strategic scenarios of the tree.  A \textit{strategic scenario} $s_{i}, \; i=1, \ldots, |\mathcal{N}_{H}|$ is a path through nodes from the root node at $t=0$ to a strategic leaf node at $t=H$. Let us denote with $\pi_{a(n),n}$, the conditional probability of the random process in node $n$ given its history up to the ancestor node $a(n)$. We indicate with $\pi^{s_i}$, the probability of a strategic scenario $s_i$ passing through strategic nodes $n_0, n_1, \ldots, n_H$ (where $n_t$, $t \in \mathcal{H}$  represents a generic strategic node at stage $t$), defined as $\pi^{s_i} \vcentcolon = \pi_{n_0, n_1} \cdot \pi_{n_1, n_2} \cdot \ldots \cdot \pi_{n_{H-1}, n_H}$.
For the node-based formulation, we also need the following notation.
At each strategic stage $t \in \mathcal{H}$, we denote the marginal probability of the strategic nodes of that stage as $\pi_n$ for  $n\in \mathcal{N}_t$. By definition, $\pi_{n_0}=1$ at stage $t=0$. For other strategic stages $t \in \mathcal{H}\setminus \{0\}$, the marginal probability $\pi_{n}$ of strategic node $n\in \mathcal{N}_t$ is calculated by multiplying the conditional probabilities up to node $n$; that is, $\pi_n:=\pi_{n_0, n_1} \cdot \pi_{n_1, n_2} \cdot \ldots\cdot\pi_{a(n),n}$, where $\{n_o, n_1, \ldots, a(n),n\}$ represent the nodes on the path from the root node to node $n$. Because each strategic scenario $s_i$ has a unique leaf node $n_H \in \mathcal{N}_{H}$, this means $\pi^{s_i}=\pi_{n_H}$.

In order to identify the operational uncertainty, we now consider operational components at each strategic node $n \in \mathcal{N}_{t}$.
Let  $\Omega_n$ denote the set of possible \textit{operational scenarios} at strategic node $n \in \mathcal{N}_t, \ t \in \mathcal{H}$.
We indicate with $\omega_{n,i}$, $i=1,\ldots,|\Omega_n|$, an operational scenario
associated with strategic node $n\in \mathcal{N}_t$, $t\in \mathcal{H}$, and with $\pi_{\omega_{n,i}}$ its probability.
For notational convenience, operational scenarios are explicitly indexed as $\omega_{n,i}$ only when they need to be distinguished. 
Whenever the strategic node is clear from the context, we use the simplified notation $\omega$ and $\pi_{\omega}$.
The distinction between $\pi_n$ and $\pi_{\omega}$ will be clear from the context.
Moreover, $\sum\limits_{\omega \in \Omega_n} \pi_{\omega}=1$, $ n \in \mathcal{N}_t,$ $t\in\mathcal{H}$.
\Cref{Fig:Common_Tree} provides an example of a multi-horizon scenario tree, $\mathfrak{T}(7)$ with 3 strategic stages and 4 strategic scenarios, where each strategic node consists of 2 operational stages and 4 operational scenarios.
\vspace*{-0.1in} 
\input{Common_Tree}
\vspace*{-0.1in}

Additionally, let $c_{n}$, $h_{n}$, $T_{n}$, and $W_{n}$ be the uncertain vectors and matrices at strategic node $n \in\mathcal{N}_{t}$, $t\in\mathcal{H}\setminus\left\{0\right\}$. If $n \in\mathcal{N}_0$ (i.e., strategic root node), we assume $T_{n}=A$, $W_{n}=0$ (i.e., the null matrix), $c_{n}$ and $h_{n}$ to be known vectors.
Operational vectors and matrices at operational stage $\tau$ in operational scenario $\omega$ derived by strategic node $n \in\mathcal{N}_t$, $t\in\mathcal{H}$, are given by $q_{n}^{\omega,\tau}$, $h_{n}^{\omega,\tau}$, $T_{n}^{\omega,\tau}$, and $W_{n}^{\omega,\tau}$. 
The strategic (mixed-integer) decision variable is given by $\mathbf{x}:=\left\{x_{n}\ |\ n \in\mathcal{N}_{t},t\in\mathcal{H}\right\}$,  with $x_{n}\in\mathbb{R}^{n_t}_{+} \times \mathbb{Z}_+^{n'_t}$. 
 The operational (mixed-integer) decision variable is $\mathbf{y}:=\left\{y_{n}^{\omega,\tau}\ |\ n \in \mathcal{N}_t, \tau \in \mathcal{T}_t, \omega \in \Omega_t, t\in\mathcal{H}\right\}$,  with 
 $y_{n}^{\omega,\tau} \in\mathbb{R}^{n^{\tau}_{t}}_{+} \times \mathbb{Z}_+^{{n'}^\tau_t}.$ 
 
Using the node notation and the multi-horizon scenario tree, we can equivalently formulate the MHSP given in \eqref{MHSP} as follows:\vspace*{-0.09in}
\begin{align} 
\label{MHSP1}
\textit{MHSP}  := & \min \limits_{\mathbf{x},\mathbf{y}} \sum\limits_{t \in \mathcal{H}}\sum\limits_{n \in \mathcal{N}_t} \pi_{n} \big( c_n x_n + \sum\limits_{\omega \in \Omega_n} \pi_{\omega} \sum\limits_{\tau \in \mathcal{T}_t}  q_{n}^{\omega,\tau} y_{n}^{\omega,\tau} \big) 
\\
& \textrm{ s.t. }   Ax_{n}=h_{n},\  n \in  \mathcal{N}_0, 
\nonumber 
\\
& \qquad \, T_{n} x_{a({n})} + W_{n} x_{n}=h_{n},\   n \in  \mathcal{N}_t,\ t\in\mathcal{H}\setminus\{0\},
\nonumber 
\\
& \qquad \, T_{n}^{\omega,1} x_{n} + W_{n}^{\omega,1} y_{n}^{\omega,1} = h_{n}^{\omega,1},\   n \in \mathcal{N}_t,\  \omega \in \Omega_n,\ t\in\mathcal{H}, \nonumber
\\
& \qquad \, T_{n}^{\omega,\tau} y_{n}^{\omega,\tau -1} + W_{n}^{\omega,\tau} y_{n}^{\omega,\tau} = h_{n}^{\omega,\tau},\   n \in \mathcal{N}_t, \tau \in \mathcal{T}_t\setminus\left\{1\right\},\omega \in \Omega_n,\ t\in\mathcal{H}. 
\nonumber
\end{align}
\vspace*{-0.25in}

We now introduce a second multi-horizon scenario tree formulation of the problem (\ref{MHSP}) based on strategic and operational scenarios. 
Let $\mathcal{S}$ represent the set of strategic scenarios.  
Denote by ($T_{t}^{\mathbf{s}},W_{t}^{\mathbf{s}})$ and  $x^{\mathbf{s}}_t$ the uncertain strategic matrices and decision variables, respectively, at the strategic stage $t\in \mathcal{H}$ and the strategic scenario $\mathbf{s}\in\mathcal{S}$. Similarly, let 
$\left(q_{\mathbf{s},t}^{\omega,\tau},h_{\mathbf{s},t}^{\omega,\tau},T_{\mathbf{s},t}^{\omega,\tau},W_{\mathbf{s},t}^{\omega,\tau}\right)$ and $y_{\mathbf{s},t}^{\omega,\tau}$ be the operational parameters and decision variables, respectively, at the operational scenario $\omega\in\Omega_t$ at operational stage $\tau\in \mathcal{T}_t$ of strategic scenario $\mathbf{s} \in \mathcal{S}$ and strategic stage $t \in \mathcal{H}$. 
Finally, let 
 $\pi^{\mathbf{s}}$ denote the probability of strategic scenario $\mathbf{s}\in\mathcal{S}$, and similarly let $\pi_{\mathbf{s},t}^{\omega}$ denote the probability of operational scenario $\omega$ at strategic scenario $\mathbf{s} \in \mathcal{S}$ and strategic stage $t \in \mathcal{H}$. 
 Given this scenario-based notation, the MHSP \eqref{MHSP1} can also be expressed as follows:\vspace*{-0.07in}
\begin{align*} 
\textit{MHSP} =  & \min\limits_{\mathbf{x},\mathbf{y}}
\sum\limits_{t \in \mathcal{H}} \sum\limits_{\mathbf{s} \in \mathcal{S} }\pi^{\mathbf{s}}\Big( c^{\mathbf{s}}_t x^{\mathbf{s}}_t + \sum\limits_{\omega \in \Omega_t} \pi^{\omega}_{\mathbf{s},t}\sum\limits_{\tau \in \mathcal{T}_t}  q_{\mathbf{s},t}^{\omega,\tau} y_{\mathbf{s},t}^{\omega,\tau} \Big) \nonumber \\
& \textrm{ s.t. }      A x_{0} =h_{0}, \nonumber \\
& \qquad \,
T_{t}^{\mathbf{s}} x_{t-1}^{\mathbf{s}} + W_{t}^{\mathbf{s}} x_{t}^{\mathbf{s}}=h_{t}^{\mathbf{s}},\ \mathbf{s}\in\mathcal{S},\   t\in\mathcal{H}\setminus\left\{0\right\},\nonumber \\
& \qquad  \,
T_{\mathbf{s},t}^{\omega,1} x_{t}^{\mathbf{s}}+W_{\mathbf{s},t}^{\omega,1} y_{\mathbf{s},t}^{\omega,1} = h_{\mathbf{s},t}^{\omega,1},\ \mathbf{s}\in\mathcal{S}, \  \omega \in \Omega_t,\ t \in \mathcal{H}, \nonumber\\
& \qquad  \,
T_{\mathbf{s},t}^{\omega,\tau} y_{\mathbf{s},t}^{\omega,\tau -1}+W_{\mathbf{s},t}^{\omega,\tau} y_{\mathbf{s},t}^{\omega,\tau} = h_{\mathbf{s},t}^{\omega,\tau},\   \mathbf{s}\in\mathcal{S},\ \omega \in \Omega_t,\ \tau \in \mathcal{T}_t\setminus\left\{1\right\}, t \in \mathcal{H},  \nonumber\\
& \qquad \, x_{t}^{\mathbf{s}^{\prime}}=x_{t}^{\mathbf{s}^{\prime\prime}},\ \forall \mathbf{s}^{\prime},\mathbf{s}^{\prime\prime}\in\mathcal{S},\textrm{ for which } \mathbf{s}^{\prime}=\mathbf{s}^{\prime\prime} \textrm{up to  stage } t \nonumber, \\
& \qquad \,y_{\mathbf{s},t}^{\omega^{\prime},\tau}=y_{\mathbf{s},t}^{\omega^{\prime\prime},\tau},\ \forall {\omega}^{\prime},{\omega}^{\prime\prime}\in\Omega_t,\textrm{ for which } {\omega}^{\prime}={\omega}^{\prime\prime} \textrm{up to  stage } \tau,
\end{align*}
\vspace*{-0.25in}

\noindent
 where the non-anticipativity of strategic and operational decision variables is enforced by the last two constraints, respectively.

In Table~\ref{tb:basicnotation}, we  summarize the notations introduced for constructing multi-horizon scenario tree approximations at strategic and operational levels, based on both scenario and nodal formulations. The column ``Example'' 
refers to the multi-horizon scenario tree $\mathfrak{T}(7)$ presented in \Cref{Fig:Common_Tree}.

\begingroup
\centering
\footnotesize
\begin{longtable}{p{1.6cm}p{1cm} p{7cm}Sl}
\hline
     Horizon Type &   Notation  & \multicolumn{1}{c}{Explanation}  & Example \\
        \hline 
        \endfirsthead
        \hline
        Horizon type &   Notation  & \multicolumn{1}{c}{Explanation} & Example \\
        \hline
        \endhead
\midrule
    \multicolumn{4}{r}{\footnotesize\itshape Continued on the next page}
\endfoot
\endlastfoot 
Strategic 
       & $\mathcal{N}_{t}$ & Set of ordered strategic nodes at strategic stage $t\in {\cal H}.$ & $\begin{aligned}[t]
       &
       \mathcal{N}_{0} = \{1\}, 
       \mathcal{N}_{1} = \{2,3\}, \\
       & \mathcal{N}_{2} = \{4,5,6,7\} \end{aligned}$ \\
       & $\mathcal{N}$ & Set of all strategic nodes, i.e., $\mathcal{N} = \bigcup\limits_{t\in {\cal H}} \mathcal{N}_{t}$; an arbitrary element of $\mathcal{N}$ is denoted by $n$. & $\mathcal{N}=\{1,2,3,4,5,6,7\}$ \\
       & $\pi_n$ & Marginal probability of strategic node $n \in \mathcal{N}_t$ at strategic stage $t \in \mathcal{H}$. & $\pi_2 = \frac{1}{3}$, $\pi_3 = \frac{2}{3}$ at $t=1$\\
       & $ \mathcal{S}$ & {Set of  strategic scenarios; 
        an arbitrary element of $\mathcal{S}$ is denoted by $\mathbf{s}$.} & $ \begin{aligned}[t] 
 \mathcal{S} = \{\{1,2,4\}, \{1,2,5\}, \\ \{1,3,6\}, \{1,3,7\}\} 
 \end{aligned}$ \\
 & $\pi^{\mathbf{s}}$ & Probability of strategic scenario $\mathbf{s} \in \mathcal{S}$. & $
       \pi^{\{1,2,4\}}= \frac{2}{9}
       $ \\
       \midrule
    Operational   
      & $\Omega_t$ & Set of operational scenarios at strategic stage $t \in \mathcal{H}$. & $\begin{aligned}[t] 
       \Omega_t =  \{ 
\textnormal{AB},\textnormal{AC},\textnormal{AD},\textnormal{AE}\} \, \forall t \end{aligned}$ \\
       & $\Omega_n$ & Set of operational scenarios at strategic node $n \in \mathcal{N}_t,  t \in \mathcal{H}$; an arbitrary element of $\Omega_n$ is denoted by $\omega_n$. & $\begin{aligned}[t] 
       \Omega_7 =  \{ 
\textnormal{AB},\textnormal{AC},\textnormal{AD},\textnormal{AE}\} \end{aligned}$ \\
& $\pi_{\omega_{n,i}}$  & Probability of operational scenario $\omega_{n,i}$ associated with strategic node $n \in \mathcal{N}_t,  t \in \mathcal{H}$. & $
           \pi_{\omega_{4,1}} = \frac{2}{5}, \text{ where } \omega_{4,1}=\{\text{AB}\} 
       $ \\
       &  $\pi_{\mathbf{s},t}^{\omega}$ & Probability of operational scenario $\omega$ at strategic scenario $\mathbf{s} \in \mathcal{S}$ and strategic stage $t \in \mathcal{H}$. & $\pi_{\{1,2,4\},2}^{\{\text{AB}\}} = \frac{2}{5}$ \\
       \hline
\caption{Notations for constructing
multi-horizon scenario tree approximations.}
\label{tb:basicnotation}
\end{longtable}
\endgroup

In the rest of the paper, we refer to the collection of operational scenarios across {\it all} strategic nodes as the {\it operational subtree}. We continue to use {\it scenario tree} for the strategic component. See Figure \ref{Fig:Common_Tree} and observe that the strategic nodes $\mathcal{N}=\{1,2,3,4,5,6,7\}$ form the strategic scenario tree, whereas the remainder of the multi-horizon scenario tree $\mathfrak{T}(7)$ forms the operational subtree.

As the size of the multi-horizon tree increases, providing a solution to the MHSP in \eqref{MHSP1} becomes computationally challenging. For this reason, bounding techniques that replace the original problem with simpler ones and allow to bound the optimal value are of crucial importance. 
A classical method to obtain lower bounds consists of replacing the stochastic processes by their expectations. In a multi-horizon setting, by considering both strategic and operational uncertainties, different simplification approaches may be introduced. 
In \cite{maggioni2020bounds}, lower bounds have been obtained by replacing either all uncertain parameters or only operational uncertain parameters.
In \ref{sec:LB_exp}, we recall these two
approaches and introduce a third approach, based on the replacement of only strategic uncertain parameters by their expectations.
We will use these traditional expectation-based bounds as benchmarks in our computational results. 
We are now ready to discuss the novel 
bounding techniques derived in this paper based on dissection and recombination of scenario trees for MHSPs in the following section.
\vspace*{-0.09in}

\section{Lower Bounds for Multi-Horizon Problems via Scenario Tree Decomposition}
\label{sec:LBmeasure_diss}
\vspace*{-0.05in}

In this section, we propose novel methods to derive lower bounds to the optimal value of MHSP by dissecting the strategic scenario tree and/or the operational subtree of the multi-horizon scenario tree $\mathfrak{T}$.  
Throughout the paper, we use regular font (e.g., MHSP) to denote a problem and italics (e.g., \textit{MHSP}) to denote its optimal value; recall formulation \eqref{MHSP1}.  
We solve the resulting subgroup problems independently and combine their optimal values appropriately to obtain a chain of lower bounds on the optimal value \textit{MHSP}. The dissection into subgroups ensures that the computation can be easily performed in a parallel fashion.
We dissect (i) \textit{only} the operational subtree, or (ii) \textit{only} the strategic tree, or (iii) \textit{both} the operational subtree and strategic scenario tree simultaneously. In addition, we present two different approaches for combining the optimal values of the resulting subgroups: (i) expectation-based or (ii) node-based. The dissection and combination approaches together dictate the number of subgroups to be constructed. In the subsequent sections, we show that when dissecting operational subtrees, the expectation-based combination approach yields a larger number of subgroups than the node-based combination. However, both approaches are equivalent when dissecting \textit{only} the strategic tree.  The advantage of a smaller number of subgroups carries over to the simultaneous dissection because it involves the operational component.\vspace*{-0.05in}

\subsection{Operational Subtree Dissection}\label{sec:OSD_theory}

\label{sec:op_dissec}
Our first approach dissects {\it only} the operational subtree, while considering the whole set of strategic scenarios. 
For a given strategic node $n\in\mathcal{N}_t$, $t\in\mathcal{H}$, 
we  construct a collection of subsets of its corresponding original operational support $\Omega_n$ with a refinement chain as follows:\vspace*{-0.15in}
\begin{gather}
\Omega_n, \nonumber \\
\vdots  \nonumber  \\
(\Omega_{n,1}^{(\gamma)},\Omega_{n,2}^{(\gamma)},\ldots,\Omega_{n,m_\gamma}^{(\gamma)} ), \label{refinement_chain_operational_common_prob} \\
\vdots  \nonumber  \\
(\{\omega_{n,1}\},\ldots,\{\omega_{n,|\Omega_n|}\} ), \nonumber 
\end{gather}
where each row is a collection of subsets of the support $\Omega_n$, $n\in\mathcal{N}_t$, $t\in \mathcal{H}$ 
with the following properties:
\begin{enumerate}[label = ({\roman*})]
    \item \label{OG_subset_number} Each set  $\Omega_{n,j}^{(\gamma)}$ for $j=1,\ldots, m_{\gamma}$ has the same number of scenarios, denoted by $g_{\gamma}$, of which $f$ is the number of fixed scenarios such that $g_{\gamma} > f$. Thus, the total number of subsets $ m_\gamma = (|\Omega_n| - f)/(g_{\gamma} - f)$ is  an integer;
    \item \label{union_OG_support} The union of the subsets covers the whole support, i.e.,  $\Omega_n = \mathop{\cup}_{j=1}^{m_{\gamma}}\Omega_{n,j}^{(\gamma)} $; 
    \item \label{item:refinement} Each set $\Omega_{n,j}^{(\gamma)}$ at level $\gamma$ is the union of sets from the next more refined collection $\Omega_{n,j'}^{(\gamma-1)}$, i.e.,\vspace*{-0.05in}
    \begin{equation}
    \Omega_{n,j}^{(\gamma)}= \bigcup_{ \{ \Omega_{n,j'}^{(\gamma-1)}\subseteq \Omega_{n,j}^{(\gamma)}\}}  \Omega_{n,j'}^{(\gamma-1)},\vspace*{-0.07in}
    \label{eq_subsets_op}
    \end{equation}
    where the union is taken over all subsets $\Omega_{n,j'}^{(\gamma-1)}$ at the refined level $\gamma - 1$ whose union yields $\Omega_{n,j}^{(\gamma)}$  at level $\gamma$. Specifically, the shorthand notation $\{\Omega_{n,j'}^{(\gamma-1)}\subseteq \Omega_{n,j}^{(\gamma)}\}$ represents the set $\big\{j' \in [m_{\gamma-1}]\, \mid\,  \Omega_{n,j'}^{(\gamma-1)}\subseteq \Omega_{n,j}^{(\gamma)}\big\}$ with
     $[m_{\gamma-1}] = \{1,2, \ldots, m_{\gamma-1}\}. $ For brevity, this shorthand notation is used throughout the paper.
    \end{enumerate}
At the highest level of the refinement chain, we have a single group containing the entire operational scenario set associated with strategic node $n$. 
As we move down in the refinement chain, we observe a finer partitioning of the operational scenario set $\Omega_n$, such that every $\Omega_{n,j}^{(\gamma)}$ consists of groups, each being the union of sets from the next more refined collection. 
The most refined partition (i.e., $\gamma=1$) contains each operational scenario on its own. 
    \begin{remark}
 In general, each set $\Omega_{n,j}^{(\gamma)}$ for $j=1,\ldots, m_{\gamma}$ at level $\gamma$ can have different cardinalities. However, for simplicity, in this paper,  we assume that it  has the same number $g_{\gamma}$ of scenarios, of which $f$ are fixed, appearing in all sets. Similarly, for ease of exposition, we further assume that all strategic nodes $n \in\mathcal{N}_t$, $t\in \mathcal{H}$, have the same number and structure of operational scenarios.
 \end{remark}

We consider groups of operational scenarios and derive lower bounds to the optimal value \textit{MHSP} by solving multiple subproblems, each involving a reduced number of operational scenarios. In the subsequent sections, we discuss two different combination approaches.
For clarity, each combination approach is structured into three parts. We begin by describing in detail how the trees are dissected (\textit{Dissection}). We then define the corresponding subgroup problems, referred to simply as `subproblems', to obtain lower bounds  (\textit{Subproblems \& Bound}).  We finally present the theoretical properties, where we derive a monotonic chain of lower bounds to the optimal value \textit{MHSP} (\textit{Properties}).  This structure is repeated throughout Section~\ref{sec:LBmeasure_diss} to aid readability.  
 
\subsubsection{Expectation-Based Combination}
\label{sssec:diss_oper} 
We assign to each subgroup a probabilistic weight derived from the probability measures of its operational scenarios. Subsequently, we combine the optimal values of the subgroups into a weighted sum (i.e., calculate an expectation), thereby obtaining a lower bound to the optimal value \textit{MHSP}. Furthermore, we show that these bounds lead to a monotonic chain of inequalities as the number of operational scenarios in each subgroup grows. 
In the sequel, we refer to a set of operational scenarios corresponding to a specific strategic node as an {\it operational subset}.

\paragraph{Dissection}
A refinement at level $\gamma $ generates $m_{\gamma}$ operational subsets for each strategic node. Applying the same refinement across all strategic nodes generates a total of $M_{\gamma} = m_{\gamma}^{\vert \mathcal{N} \vert}$ subgroups, where $m_{\gamma}^{\vert \mathcal{N} \vert }$ denotes the $\vert\mathcal{N} \vert^{th}$ power of $m_{\gamma}$. Let $[M_{\gamma}] = \{1,2, \ldots, M_{\gamma}\}$. For a given $k \in [M_{\gamma}]$, we denote the $k^{th}$ subgroup at level $\gamma$ by $\Omega_{\text{OG},k}^{(\gamma)}$. Let the operational subsets of the subgroups $\Omega_{\text{OG},k}^{(\gamma)}$, $k \in [M_{\gamma}]$, be denoted by $\Omega_{n,j}^{(\gamma)  }$, $j \in [m_{\gamma}]$, for each $n\in\mathcal{N}_t$ and $t\in \mathcal{H}$.
Along with properties \ref{OG_subset_number}--\ref{item:refinement} described above, we construct subgroups having the following additional property:

    \begin{enumerate}[label = ({\roman*}),resume]
    \item 
    \label{union_intersection_condition_common_prob_OP}
    For a given level $\gamma$ and strategic node $n' \in \mathcal{N}_t$, $t \in \mathcal{H}$, let the operational subsets at all nodes $n \neq n'$ be fixed and denoted by $\bar{\Omega}_n$. Define 
    \begin{align*}
    \bar{\mathcal{V}}_{n'}^{(\gamma)} = \Big\{ k \in [M_{\gamma}]\, \Bigm|\, \Omega_{n,j}^{(\gamma)  } = \bar{\Omega}_n, \ \forall \ n \neq n', \ \Omega_{n,j}^{(\gamma)  } \in \Omega_{\text{OG},k}^{(\gamma)},\ \forall \ n \in \mathcal{N}_t, 
 t \in \mathcal{H}
 \Big\}.
 \end{align*}
 In words, set $\bar{\mathcal{V}}_{n'}^{(\gamma)}$  is the index set of subgroups at level $\gamma$ whose operational subsets $\Omega_{n,j}^{(\gamma)}$ at strategic nodes $n \neq n'$ are equal to $\bar{\Omega}_n$. For simplicity, we remove the dependence of $\bar{\Omega}_n$ from the set $\bar{\mathcal{V}}_{n'}^{(\gamma)}$ and note that the property can be verified equivalently with different $\bar{\Omega}_n$. Then, $\Omega_{n',j}^{(\gamma)}$ is the union of sets from the next more refined collection, i.e.,\vspace*{-0.05in} 
 $$
 \Omega_{n',j}^{(\gamma)} = \bigcup_{\{k \in \bar{\mathcal{V}}_{n'}^{(\gamma-1)}\, \mid\, \Omega_{n',j'}^{(\gamma-1)  } \in \Omega_{\text{OG},k}^{(\gamma-1)}\}} \Omega_{n',j'}^{(\gamma-1)}, \vspace*{-0.1in}
 $$
 such that  
 $$
 \Omega_{n',j_1}^{(\gamma-1)}  \cap \Omega^{(\gamma-1)}_{n',j_2} = \Omega_f,\ \forall \ j_1 \neq j_2,
 $$
 where $\Omega_f = \{\omega_{n,i}\}_{i = 1}^f$ denotes the set of fixed operational scenarios.
\end{enumerate}
\noindent
Let $o_{\gamma} = g_{\gamma} - f$, and  define the set of operational scenarios at a given $\Omega^{(\gamma)}_{n,j}$ as
\begin{equation}    \label{eq:set_op_sc}
\Omega_{n,j}^{(\gamma)} = \{\omega_{n,1},\ldots,\omega_{n,f},\omega_{n,f+(j-1) o_{\gamma} + 1},\ldots,\omega_{n,f+j  o_{\gamma}} \}, \ j \in [m_\gamma].  
\end{equation}
Then the associated probabilities of these scenarios with support $\Omega^{(\gamma)}_{n,j}$ are given by
\begin{align}
    \label{eq:prob_op_sc} 
\hat{\pi}^{(\gamma)}_{\omega_{n,i}} = 
\left\{ \begin{array}{ll}
\pi_{\omega_{n,i}}, & \textrm{$i \in [f]$},  \\
\pi_{\omega_{n,i}} \frac{1 - \sum_{i=1}^f \pi_{\omega_{n,i}}}{\sum_{i=f+(j-1) o_{\gamma} +1}^{f+j o_{\gamma}} \pi_{\omega_{n,i}}} , & \textrm{$i = f+(j-1) o_{\gamma} +1, \ldots, f+j  o_{\gamma}$,}
\end{array} \right.
\end{align}
where $\pi_{\omega_{n,i}}$ denotes  the  probability   of the operational scenario $\omega_{n,i}$ with  support $\Omega_n$. 
 Additionally, the probabilistic weight associated with subgroup $ \Omega_{\text{OG},k}^{(\gamma)}$,   denoted by $ \phi_{\text{OG},k}^{(\gamma)}$, is defined as 
$$
\displaystyle \phi_{\text{OG},k}^{(\gamma)} = \prod_{n \in \mathcal{N}_t, t\in \mathcal{H}}\phi_{n,j}^{(\gamma)}, \vspace*{-0.13in}
$$
using the operational subsets $j$ associated with that subgroup, where
\begin{align}\label{eq:prob_node_op_sc}
\phi_{n,j}^{(\gamma)}=
   \frac{ \sum_{i=f+(j-1) o_{\gamma} +1}^{f+j  o_{\gamma}}\pi_{\omega_{n,i}}}{1 - \sum_{i=1}^f \pi_{\omega_{n,i}}}, \ j \in [m_\gamma].
   \end{align}
   For a given $j \in [m_\gamma]$ and $j' \in [m_{\gamma-1}]$, viewing
 $\phi_{n,j}^{(\gamma)}$ and $\phi_{n,j'}^{(\gamma-1)}$ as the probabilities associated with $\Omega_{n,j}^{(\gamma)}$ and $\Omega_{n,j'}^{(\gamma-1)}$, respectively, where $\Omega_{n,j'}^{(\gamma-1)} \subseteq \Omega_{n,j}^{(\gamma)}$, property \ref{union_intersection_condition_common_prob_OP} establishes a relation between the probability $\phi_{n,j}^{(\gamma)}$ at level $\gamma$ and the probabilities $\phi_{n,j'}^{(\gamma-1)}$, $j' \in [m_{\gamma-1}]$,  at the more refined level $\gamma-1$ for each $n \in \mathcal{N}_t,  t \in \mathcal{H}$. We formalize this in the following lemma. 

\begin{lemma}\label{sum_condition_lemma_common_prob}
Let the strategic node $n \in \mathcal{N}_t,  t \in \mathcal{H}$ be fixed. Suppose for a given  $j \in [m_\gamma]$, we have $\Omega_{n,j}^{(\gamma)}= \bigcup_{ \{ \Omega_{n,j'}^{(\gamma-1)}\subseteq \Omega_{n,j}^{(\gamma)}\}} \Omega_{n,j'}^{(\gamma-1)}$. Then, we obtain\vspace*{-0.13in}
\begin{align*}
\phi_{n,j}^{(\gamma)} = \sum_{ \{ \Omega_{n,j'}^{(\gamma-1)}\subseteq \Omega_{n,j}^{(\gamma)}\}  } \phi_{n,j'}^{(\gamma-1)}. 
\end{align*}
\end{lemma}
\begin{proof}
    For each $n \in \mathcal{N}_t,  t \in \mathcal{H}$, it follows from condition \ref{union_intersection_condition_common_prob_OP} that 
    $
 \Omega_{n,j_1}^{(\gamma-1)}  \cap \Omega^{(\gamma-1)}_{n,j_2} = \Omega_f$ for all  $j_1 \neq j_2. $ Consequently, each non-fixed operational scenario at level $\gamma-1$ occurs exactly once in the set of all subsets of $\Omega_{n,j}^{(\gamma)}$, thus, we obtain\vspace*{-0.09in} 
\begin{align*}
     \sum_{ \{\Omega_{n,j'}^{(\gamma-1)}\subseteq \Omega_{n,j}^{(\gamma)}\}  } \phi_{n,j'}^{(\gamma-1)} &= \sum_{ \{ \Omega_{n,j'}^{(\gamma-1)}\subseteq \Omega_{n,j}^{(\gamma)}\}  } \sum_{i=f+(j'-1) o_{\gamma-1} +1}^{f+j'  o_{\gamma-1}}\frac{\pi_{\omega_{n,i}}}{1 - \sum_{i=1}^f \pi_{\omega_{n,i}}}\\
     &= \sum_{i=f+(j-1) o_{\gamma} +1}^{f+j  o_{\gamma}} \frac{\pi_{\omega_{n,i}}}{1 - \sum_{i=1}^f \pi_{\omega_{n,i}}} \ = \phi_{n,j}^{(\gamma)}. \qedhere
 \end{align*} 
\end{proof}

To illustrate this construction, we introduce the following example.
\begin{myex}\label{example_dissect_operational_common_prob}
Consider the multi-horizon scenario tree $\mathfrak{T}(3)$ at level $\gamma=3$ in Figure \ref{MHEOGSIOPT_common_prob_figure_example}, where at each strategic node $n\in \{1,2,3\}$, there are four operational scenarios, i.e., $\Omega_{n,1}^{(3)}=\{ 
\textnormal{AB,AC,AD,AE} \}$, and thus $g_{3} = 4$ and $m_3=1$. 
The probability of each operational scenario is given below it, e.g., the probability of $\{\textnormal{AB}\}$ at $n=1$ is given by $\pi_{\omega_{1,1}}=\frac{1}{5}$. 
Going from level $\gamma=3$ to $\gamma-1=2$, the four operational scenarios are split into two subsets (i.e, $m_2=2$) of size two (i.e, $g_2=2$) with no fixed scenarios (i.e, $f=0$ and $\Omega_f=\emptyset$), given by $\{ 
\textnormal{AB,AC}\}$ and $\{ 
\textnormal{AD,AE} \}$, leading to $2^3=8$ subgroups $\Omega_{\text{OG},k}^{(2)}$, $k=1,\ldots, 8$.  
For each subgroup, the probability of operational scenarios at level $2$ are calculated according to \eqref{eq:prob_op_sc}; e.g., the probability of $\{\textnormal{AB}\}$ at $n=1$ is given by $\hat{\pi}^{(2)}_{\omega_{1,1}} = \frac{1}{2} =\frac{1/5}{2/5}$. 
Each subgroup can then be solved independently.
Property \ref{union_intersection_condition_common_prob_OP} can be verified by fixing $n'=1$ and $\bar{\Omega}_n
= \{\textnormal{AB,AC}\}$ at strategic nodes $n=2,3$. Thus,   
$\bar{\mathcal{V}}_{1}^{(2)} = \{1,5\}$ and $\Omega_{1,1}^{(3)} = \Omega_{1,1}^{(2)} \cup \Omega_{1,2}^{(2)}$, where $\Omega_{1,1}^{(2)} = \{\textnormal{AB,AC}\}$ and $\Omega_{1,2}^{(2)} = \{\textnormal{AD,AE}\}$. Moreover, $\phi_{1,1}^{(2)} = \sum_{i=1}^2 \pi_{\omega_{1,i}} = \frac{1}{5}+\frac{1}{5}=\frac{2}{5}$, where $\omega_{1,1} = \{\textnormal{AB}\}$ and $\omega_{1,2} = \{\textnormal{AC}\}$. Similarly, it can be derived for subgroup $1$ that $\phi_{2,1}^{(2)} =\frac{2}{6}$ and  $\phi_{3,1}^{(2)} = \frac{4}{7}$ for strategic nodes $2$ and $3$, respectively.  Consequently, the probabilistic weight of the first subgroup can be found by $\phi_{\textnormal{OG},1}^{(2)} = \phi_{1,1}^{(2)}\times \phi_{2,1}^{(2)} \times \phi_{3,1}^{(2)} = \frac{2}{5} \times \frac{2}{6} \times \frac{4}{7}$. Others are calculated similarly.\vspace*{-0.15in}
\end{myex}
\input{MHEOG_common_probability_less_trees}\vspace*{-0.15in}

\paragraph{Subproblems \& Bound} Given the above dissection of the operational subtrees, we define the corresponding operational subproblems to derive a lower bound to the optimal value of MHSP. 
\begin{mydef}\label{def:MHOG}
The \textit{Multi-Horizon Operational Group subproblem} with index $k \in [M_{\gamma}]$ at level $\gamma$, 
denoted by $\text{MHOG}\big( \Omega_{\text{OG},k}^{(\gamma)} \big)$, is defined as\vspace*{-0.1in} 
\begin{align*}  
\textit{MHOG}\big( \Omega_{\text{OG},k}^{(\gamma)} \big)
& :=  \min\limits_{\mathbf{x},\mathbf{y}}\sum\limits_{t \in \mathcal{H}}\sum\limits_{n \in \mathcal{N}_t} \pi_{n} \Big( c_n x_n + \sum\limits_{\omega \in \Omega_{n,j}^{(\gamma)}} \hat{\pi}^{(\gamma)}_{\omega} \sum\limits_{\tau \in \mathcal{T}_t}  q_{n}^{\omega,\tau} y_{n}^{\omega,\tau} \Big) 
\nonumber  \\
&\quad\ \ \textrm{ s.t.  }   Ax_{n}=h_{n}, \  n \in  \mathcal{N}_0, \nonumber \\
& \qquad\ \ \ \quad T_{n} x_{a({n})} + W_{n} x_{n}=h_{n},\   n \in  \mathcal{N}_t,\ t\in\mathcal{H}\setminus\{0\},\nonumber \\
& \qquad\ \ \ \quad T_{n}^{\omega,1} x_{n} + W_{n}^{\omega,1} y_{n}^{\omega,1} = h_{n}^{\omega,1},\   n \in \mathcal{N}_t,\ t\in\mathcal{H}, \ \omega \in \Omega_{n,j}^{(\gamma)}, \ j \in [m_\gamma], \nonumber\\
& \qquad\ \ \ \quad T_{n}^{\omega,\tau} y_{n}^{\omega,\tau -1} + W_{n}^{\omega,\tau} y_{n}^{\omega,\tau} = h_{n}^{\omega,\tau},\   n \in \mathcal{N}_t, \tau \in \mathcal{T}_t\setminus\left\{1\right\},\ t\in\mathcal{H}, \nonumber \\ 
& \hspace{6.4cm} \omega \in \Omega_{n,j}^{(\gamma)}, \ j \in [m_\gamma].
\end{align*}
\end{mydef}
Combining the optimal values of all  subproblems $\textit{MHOG}(\Omega_{\textnormal{OG},k}^{(\gamma)})$, $k \in [M_{\gamma}]$ with their weights allows us to introduce the following quantity.
\begin{mydef}
    \label{def:MHEOG}
       The  \textit{Multi-Horizon Expected Operational Group bound} at operational level $\gamma$, denoted $\textit{MHEOG}(\gamma)$, is defined as follows:\vspace*{-0.15in}
    \begin{align*} 
  \textit{MHEOG}(\gamma) 
 : =  
 \sum_{k \in [M_{\gamma}]} \phi^{(\gamma)}_{\text{OG},k} \textit{MHOG}\big( \Omega_{\text{OG},k}^{(\gamma)} \big).  
  \end{align*}
  \end{mydef}

\paragraph{Properties}  We now show that $\textit{MHEOG}(\gamma)$ leads to a monotonic chain of lower bounds to the optimal value of  \textit{MHSP}  as $\gamma$ increases (or, as we get closer to the original MHSP).
 
\begin{myprop} \label{Prop_MHEOG}
Consider \textit{MHSP} given in \eqref{MHSP1}.
The following chain of inequalities holds true:\vspace*{-0.03in}
\begin{eqnarray}
     \textit{MHEOG}(1) \leq \textit{MHEOG}(2) \leq \ldots  \leq \textit{MHEOG}(\gamma)  \leq \VV{\textit{MHEOG}(\gamma+1)  \leq} \ldots \leq \textit{MHSP}. \nonumber
\end{eqnarray}
\end{myprop}
\proof
We prove that for two consecutive levels $\gamma$ and $\gamma+1$ of the refinement chain \eqref{refinement_chain_operational_common_prob}, we have  $\textit{MHEOG}(\gamma) \leq 
\textit{MHEOG}(\gamma+1)$.
 Consider a multi-horizon operational group subproblem $\text{MHOG}\big( \Omega_{\text{OG},k}^{(\gamma+1)} \big)$ and let 
 $(\hat{x}_{n},\hat{y}_{n}^{\omega,\tau})\footnote{ The solution $(\hat{x}_{n},\hat{y}_{n}^{\omega,\tau})$ is  appropriately indexed for all $n \in \mathcal{N}_t, \tau \in \mathcal{T}_t, \omega \in \Omega_t$, and  $t\in\mathcal{H}$. For brevity, we suppress writing out these indices in the remainder of the paper whenever the indexing of such solutions is clear from the context.}$ 
 be its optimal solution, where we ignore the dependence of this solution on subgroup $k$ for notational simplicity. Let $\Omega_{\text{OG},k'}^{(\gamma)} \subseteq \Omega_{\text{OG},k}^{(\gamma+1)}$ and  the operational
subsets of $\Omega_{\text{OG},k'}^{(\gamma)}$ and $\Omega_{\text{OG},k}^{(\gamma+1)}$ at strategic node $n$ be denoted by $\Omega_{n,j'}^{(\gamma)  }$ and $\Omega_{n,j}^{(\gamma+1)  }$, respectively.  We have due to suboptimality that\vspace*{-0.1in} 
\begin{align*}
\textit{MHOG}\big(\Omega_{\text{OG},k'}^{(\gamma)} \big)
    \leq &
    \sum\limits_{t \in \mathcal{H}}\sum\limits_{n \in \mathcal{N}_t} \pi_{n} \Big( c_n \hat{x}_n + \sum\limits_{\omega \in \Omega_{n,j'}^{(\gamma)}} \hat{\pi}^{(\gamma)}_{\omega} \sum\limits_{\tau \in \mathcal{T}_t}  q_{n}^{\omega,\tau} \hat{y}_{n}^{\omega,\tau} \Big).\vspace*{-0.2in} 
\end{align*} 
Multiplying the above inequality by 
$\phi^{(\gamma)}_{\text{OG},k'}$ and summing over $k' \in [M_{\gamma}]$ such that  $\Omega_{\text{OG},k'}^{(\gamma)} \subseteq \Omega_{\text{OG},k}^{(\gamma+1)}$,  we obtain the following inequality: 
 \begin{align} \label{inequality_sum_subgroup_level}
\V{\sum_{\{\Omega_{\text{OG},k'}^{(\gamma)} \subseteq \Omega_{\text{OG},k}^{(\gamma+1)}\}}} &   \!\!\!\!\!\!\!\!\!\!\!\!\! \phi^{(\gamma)}_{\text{OG},k'}
\textit{MHOG}\big(\Omega_{\text{OG},k'}^{(\gamma)} \big) \nonumber\\
 \leq &
\V{\sum_{\{ \Omega_{\text{OG},k'}^{(\gamma)} \subseteq \Omega_{\text{OG},k}^{(\gamma+1)}\}}}
   \!\!\!\!\!\!\!\!\!\!\!\!\! \phi^{(\gamma)}_{\text{OG},k'}
 \bigg( \sum\limits_{t \in \mathcal{H}}\sum\limits_{n \in \mathcal{N}_t} \pi_{n} \Big( c_n \hat{x}_n + \sum\limits_{\omega \in \Omega_{n,j'}^{(\gamma)}} \hat{\pi}^{(\gamma)}_{\omega} \sum\limits_{\tau \in \mathcal{T}_t}  q_{n}^{\omega,\tau} \hat{y}_{n}^{\omega,\tau} \Big) \bigg).
\end{align}
 From  \Cref{sum_condition_lemma_common_prob},  
\V{$\sum_{ \{\Omega_{n,j'}^{(\gamma)}\subseteq \Omega_{n,j}^{(\gamma+1)}\}  }
     \phi_{n,j'}^{(\gamma)} 
     = \phi_{n,j}^{(\gamma+1)}$,} thus the first term and the second term with respect to only the fixed scenarios $\Omega_f $,  on the right side of \eqref{inequality_sum_subgroup_level} become\vspace*{-0.1in}
 \begin{align}\label{MHEOG_SOPT_common_prob_derivation_with_fixed_scenarios}  \V{\sum_{\{\Omega_{\text{OG},k'}^{(\gamma)} \subseteq \Omega_{\text{OG},k}^{(\gamma+1)}\}}}
   \!\!\!\!\!\!\!\!\!\!\!\!\! \phi^{(\gamma)}_{\text{OG},k'}
  \bigg( \sum\limits_{t \in \mathcal{H}}\sum\limits_{n \in \mathcal{N}_t} \pi_{n} \Big( c_n \hat{x}_n + \sum\limits_{\omega \in \Omega_f} \pi_{\omega} \sum\limits_{\tau \in \mathcal{T}_t}  q_{n}^{\omega,\tau} \hat{y}_{n}^{\omega,\tau} \Big) \bigg) &
 \nonumber \\
 =
   \phi_{\text{OG},k}^{(\gamma+1)} 
  \bigg( \sum\limits_{t \in \mathcal{H}}\sum\limits_{n \in \mathcal{N}_t} \pi_{n} \Big( c_n \hat{x}_n + \sum\limits_{\omega \in \Omega_f} \pi_{\omega} \sum\limits_{\tau \in \mathcal{T}_t}  q_{n}^{\omega,\tau} \hat{y}_{n}^{\omega,\tau} \Big) \bigg)&.\vspace*{-0.1in}
\end{align}
  On the other hand, the remaining terms on the right side of  \eqref{inequality_sum_subgroup_level} become\vspace*{-0.1in}
\begin{align}
\label{MHEOG_SOPT_common_prob_derivation_without_fixed_scenarios}
    & \sum_{\{\Omega_{\text{OG},k'}^{(\gamma)} \subseteq \Omega_{\text{OG},k}^{(\gamma+1)}\}}
   \!\!\!\!\!\!\!\!\!\!\!\!\! \phi^{(\gamma)}_{\text{OG},k'}
 \bigg( \sum\limits_{t \in \mathcal{H}} \sum\limits_{n \in \mathcal{N}_t} \pi_{n} \Big( \sum\limits_{\omega \in \Omega_{n,j'}^{(\gamma)} \backslash \Omega_f } \hat{\pi}^{(\gamma)}_{\omega} \sum\limits_{\tau \in \mathcal{T}_t}  q_{n}^{\omega,\tau} \hat{y}_{n}^{\omega,\tau} \Big) \bigg)
 \nonumber 
 \\
 = \quad & 
 \sum_{\substack{ \{ \Omega_{\text{OG},k'}^{(\gamma)} \backslash \Omega_{n',j'}^{(\gamma)} \\ \subseteq \Omega_{\text{OG},k}^{(\gamma+1)} \backslash  \Omega_{n',j}^{(\gamma+1)} \}}} \!\!\!\! 
  \Big( \prod_{n, t; n \neq n'}\phi_{n,j'}^{(\gamma)} \Big)
 \bigg( \sum_{\substack{ \{ \Omega_{n',j'}^{(\gamma)}\subseteq \Omega_{n',j}^{(\gamma+1)}\} } }
\!\!\!\!\!\!\!\! \phi_{n',j'}^{(\gamma)} 
 \mathop{\sum\limits_{t \in \mathcal{H}}\sum\limits_{n \in \mathcal{N}_t}}_{n \neq n'} \pi_{n} \Big( \sum\limits_{\omega \in \Omega_{n,j'}^{(\gamma)} \backslash \Omega_f } \!\!\!\!\!\!\!\!\hat{\pi}^{(\gamma)}_{\omega} \sum\limits_{\tau \in \mathcal{T}_t} q_{n}^{\omega,\tau} \hat{y}_{n}^{\omega,\tau} \Big) 
 \nonumber
 \\
 & \hspace{4.75cm} {}+ 
  \vphantom{\mathop{\sum\limits_{t \in \mathcal{H}}\sum\limits_{n \in \mathcal{N}_t}}_{n \neq n'}}
\sum_{\substack{ \{ \Omega_{n',j'}^{(\gamma)}\subseteq \Omega_{n',j}^{(\gamma+1)}\} } }
\!\!\!\!\!\!\!\! \phi_{n',j'}^{(\gamma)} 
 \pi_{n'} \Big( \sum\limits_{\omega \in \Omega_{n',j'}^{(\gamma)} \backslash \Omega_f } \!\!\!\!\!\!\!\! \hat{\pi}^{(\gamma)}_{\omega} \sum\limits_{\tau \in \mathcal{T}_t}  q_{n'}^{\omega,\tau} \hat{y}_{n'}^{\omega,\tau} \Big)  \bigg)
 \nonumber \\
 = \quad & 
 \sum_{\substack{ \{ \Omega_{\text{OG},k'}^{(\gamma)} \backslash \Omega_{n',j'}^{(\gamma)} \\ \subseteq \Omega_{\text{OG},k}^{(\gamma+1)} \backslash  \Omega_{n',j}^{(\gamma+1)} \}}}
    \Big( \prod_{n, t; n \neq n'}\phi_{n,j'}^{(\gamma)} \Big)
 \bigg(
 \phi_{n',j}^{(\gamma+1)} 
 \mathop{\sum\limits_{t \in \mathcal{H}}\sum\limits_{n \in \mathcal{N}_t}}_{n \neq n'} \pi_{n} \Big( \sum\limits_{\omega \in \Omega_{n,j'}^{(\gamma)} \backslash \Omega_f} \hat{\pi}^{(\gamma)}_{\omega} \sum\limits_{\tau \in \mathcal{T}_t}  q_{n}^{\omega,\tau} \hat{y}_{n}^{\omega,\tau} \Big) 
 \\
 & \hspace{5.75cm} {}+ 
 \pi_{n'} \Big( \sum\limits_{\omega \in \Omega_{n',j}^{(\gamma+1)} \backslash \Omega_f } \pi_{\omega} \sum\limits_{\tau \in \mathcal{T}_t}  q_{n'}^{\omega,\tau} \hat{y}_{n'}^{\omega,\tau} \Big) 
 \bigg)
\nonumber \\
& \begin{aligned}
 = \sum_{\substack{ \{ \Omega_{\text{OG},k'}^{(\gamma)} \backslash \Omega_{n',j'}^{(\gamma)} \\ \subseteq \Omega_{\text{OG},k}^{(\gamma+1)} \backslash  \Omega_{n',j}^{(\gamma+1)} \}}}
   \Big( \prod_{n, t; n \neq n'}\phi_{n,j'}^{(\gamma)} \Big)
 \bigg(
 \phi_{n',j}^{(\gamma+1)} 
 \mathop{\sum\limits_{t \in \mathcal{H}}\sum\limits_{n \in \mathcal{N}_t}}_{n \neq n'} \pi_{n} \Big( \sum\limits_{\omega \in \Omega_{n,j'}^{(\gamma)} \backslash \Omega_f } \hat{\pi}^{(\gamma)}_{\omega} \sum\limits_{\tau \in \mathcal{T}_t}  q_{n}^{\omega,\tau} \hat{y}_{n}^{\omega,\tau} \Big) 
 \\
  {}+ 
   \phi_{n',j}^{(\gamma+1)} 
\pi_{n'} \Big( 
 \sum\limits_{\omega \in \Omega_{n',j}^{(\gamma+1)} \backslash \Omega_f } \hat{\pi}^{(\gamma+1)}_{\omega}  \sum\limits_{\tau \in \mathcal{T}_t}  q_{n'}^{\omega,\tau} \hat{y}_{n'}^{\omega,\tau} \Big) 
 \bigg)
 \end{aligned}
 \nonumber \\
& \begin{aligned}
 = \sum_{\substack{ \{ \Omega_{\text{OG},k'}^{(\gamma)} \backslash \Omega_{n',j'}^{(\gamma)} \\ \subseteq \Omega_{\text{OG},k}^{(\gamma+1)} \backslash  \Omega_{n',j}^{(\gamma+1)} \}}}
   \Big( \prod_{n, t; n \neq n'}\phi_{n,j'}^{(\gamma)} \Big)
 \bigg(
 \phi_{n',j}^{(\gamma+1)} 
 \mathop{\sum\limits_{t \in \mathcal{H}}\sum\limits_{n \in \mathcal{N}_t}}_{n \neq n'} \pi_{n} \Big( \sum\limits_{\omega \in \Omega_{n,j'}^{(\gamma)} \backslash \Omega_f } \hat{\pi}^{(\gamma)}_{\omega} \sum\limits_{\tau \in \mathcal{T}_t}  q_{n}^{\omega,\tau} \hat{y}_{n}^{\omega,\tau} \Big) 
 \bigg)
 \\
{}+ \phi^{(\gamma+1)}_{\text{OG},k}
 \pi_{n'} \Big(    
 \sum\limits_{\omega \in \Omega_{n',j}^{(\gamma+1)} \backslash \Omega_f } \hat{\pi}^{(\gamma+1)}_{\omega}  \sum\limits_{\tau \in \mathcal{T}_t}  q_{n'}^{\omega,\tau} \hat{y}_{n'}^{\omega,\tau} \Big) 
 \end{aligned}
\nonumber   \\
& \begin{aligned} 
 {}= \phi^{(\gamma+1)}_{\text{OG},k}\mathop{\sum\limits_{t \in \mathcal{H}}\sum\limits_{n \in \mathcal{N}_t}}
 \pi_{n} \Big(    
 \sum\limits_{\omega \in \Omega_{n,j}^{(\gamma+1)} \backslash \Omega_f } \hat{\pi}^{(\gamma+1)}_{\omega}  \sum\limits_{\tau \in \mathcal{T}_t}  q_{n}^{\omega,\tau} \hat{y}_{n}^{\omega,\tau} \Big),\vspace*{-0.17in}  
 \end{aligned}
\end{align}
 where the first equality is obtained by distributing the term $
 \phi_{n',j'}^{(\gamma)}$ across the summands and summing over \V{$j' \in [m_{\gamma}]$ such that }  $\Omega_{n',j'}^{(\gamma)}  
    \subseteq
\Omega_{n',j}^{(\gamma+1)}   $, for an arbitrarily fixed strategic node $n'$, the second and third equalities are obtained
from condition \ref{union_intersection_condition_common_prob_OP} and Lemma  \ref{sum_condition_lemma_common_prob}. These results are repeatedly used for all nodes in the last two equalities.
 Summing the right side of \eqref{inequality_sum_subgroup_level} over all subgroups at level $\gamma+1$ gives $\textit{MHEOG}(\gamma+1)$. This is observed by combining \eqref{MHEOG_SOPT_common_prob_derivation_with_fixed_scenarios} and \eqref{MHEOG_SOPT_common_prob_derivation_without_fixed_scenarios}. On the other hand, a similar operation on the left side of \eqref{inequality_sum_subgroup_level} gives the weighted sum of the optimal values of all the stochastic programs at level $\gamma$, i.e., we obtain\vspace*{-0.07in}
 \begin{align*} 
 \sum_{k\in [M_{\gamma+1}]}    
\V{\sum_{\{\Omega_{\text{OG},k'}^{(\gamma)} \subseteq \Omega_{\text{OG},k}^{(\gamma+1)}\}}} \!\!\!\!\!\!\!\!\!\!\!\!\! \phi^{(\gamma)}_{\text{OG},k'}
\textit{MHOG}\big(\Omega_{\text{OG},k'}^{(\gamma)} \big)
 =
 \textit{MHEOG}\left(\gamma \right). 
\end{align*} 
This is due to the fact that the subsets at any given level are disjoint, and thus, $\textit{MHEOG}(\gamma) \leq \textit{MHEOG}(\gamma+1)$. 
At the highest level, because we have the full set of operational scenarios $\Omega_n$ at each strategic node  $n\in\mathcal{N}_t$, $t\in\mathcal{H}$, we obtain the optimal value \textit{MHSP}.
\qed
\smallskip

The main limitation of the presented approach is that the number of subgroups grows exponentially with the number of strategic nodes, i.e., $M_{\gamma} = m_{\gamma}^{\vert \mathcal{N} \vert}$.
In the following section, we introduce a novel approach that substantially reduces the number of subgroups---and thus the computational effort---by decomposing and recombining the operational scenarios on a node-wise basis.
\vspace*{-0.1in}

\subsubsection{Node-Based Combination}
\label{sec:dissecting_oper2}
To reduce the number of subproblems, we propose a second approach that replaces property \ref{union_intersection_condition_common_prob_OP} with a new condition (property \ref{add_prop_op} below). This method combines the optimal solutions of the operational subgroups through nodal weights and aggregates the corresponding objective function values into a node-based sum, yielding a lower bound to  \textit{MHSP}. We first present the new property.

\paragraph{Dissection} We form a refinement chain as in \eqref{refinement_chain_operational_common_prob} with properties \ref{OG_subset_number}--\ref{item:refinement} defined in \Cref{sec:op_dissec}, additionally construct subgroups having the following alternative property instead of property \ref{union_intersection_condition_common_prob_OP}:  
\begin{enumerate}[label = ({\roman*}),resume]
    \item \label{add_prop_op} The intersection of subsets, in general, is  not empty but consists only of fixed scenarios, i.e. ${\bigcap}_{j=1}^{m_{\gamma}}\Omega_{n,j}^{(\gamma)} = 
    \Omega_f, \ f\geq 0$ for each operational level $\gamma$. If $f=0$, the subsets are disjoint.
\end{enumerate}
As before, we define the set of operational scenarios at a given $\Omega^{(\gamma)}_{n,j}$ as in \eqref{eq:set_op_sc},  the associated probabilities of these scenarios with support $\Omega^{(\gamma)}_{n,j}$ as in \eqref{eq:prob_op_sc}, and  the probability $\phi_{n,j}^{(\gamma)}$ associated with $\Omega_{n,j}^{(\gamma)}$ as in \eqref{eq:prob_node_op_sc}, with $o_{\gamma} = g_{\gamma} - f$. 
Now, with this dissection, for each partition of the support $\Omega_n$, there corresponds a dissection of its probability measure $P_n$  into probability measures $P_{n,j}^{(\gamma)}=\sum_{\omega_{n,i} \in \Omega^{(\gamma)}_{n,j}}\hat{\pi}^{(\gamma)}_{\omega_{n,i}}\delta_{\omega_{n,i}}$ with $\delta_{\omega_{n,i}}$ denoting the Dirac measure at scenario $\omega_{n,i}$, for $j \in [m_\gamma]$ with the following properties:
\begin{enumerate}[label = ({\roman*})]
    \item $P_{n,j}^{(\gamma)}$ has support $\Omega_{n,j}^{(\gamma)}$;
    \item $P_n$ can be expressed as $P_n = \sum_{j=1}^{m_\gamma} \phi_{n,j}^{(\gamma)} P_{n,j}^{(\gamma)}$ with $\sum_{j=1}^{m_\gamma} \phi_{n,j}^{(\gamma)} = 1$ and $\phi_{n,j}^{(\gamma)} \ge 0$ for all $j\in [m_\gamma]$;
    \item Each $P_{n,j}^{(\gamma)}$ can be expressed as a convex combination of measures 
    from the refined collection $\{P_{n,j'}^{(\gamma-1)}\}$, i.e.,
    \begin{equation}    
    P_{n,j}^{(\gamma)} = \!\!\!\!\! \! \!\!\sum\limits_{\{\Omega_{n,j'}^{(\gamma-1)} \subseteq \Omega_{n,j}^{(\gamma)}\} }\!\!\!\!\!\!\!\!\!\! \phi_{n,j',j}^{(\gamma -1, \gamma)}  P_{n,j'}^{(\gamma-1)} \! \! \textrm{ with } \!\!\!\!\!\!\!\!\!\! \sum\limits_{\{\Omega_{n,j'}^{(\gamma-1)} \subseteq \Omega_{n,j}^{(\gamma)} \} } \!\!\!\!\!\!\!\!\!\! \phi_{n,j',j}^{(\gamma -1, \gamma)}  = 1 \textrm{ and } \phi_{n,j',j}^{(\gamma -1, \gamma)} \ge 0 \textrm{ for all } j', j.
   \label{eq:convex_combo_operational}
   \end{equation}
   Substituting \eqref{eq:convex_combo_operational} into $P_n = \sum_{j=1}^{m_\gamma} \phi_{n,j}^{(\gamma)} P_{n,j}^{(\gamma)}$ and observing $P_n = \sum_{j'=1}^{m_{\gamma-1}} \phi_{n,j'}^{(\gamma-1)} P_{n,j'}^{(\gamma-1)}$, we obtain  $ \phi_{n,j',j}^{(\gamma -1, \gamma)} = \frac{\phi_{n,j'}^{(\gamma -1)}}{\phi_{n,j}^{(\gamma)}}$. Consequently, we can write
    \begin{equation}  
    \!\!\!\!\!\!P_{n,j}^{(\gamma)}\!=\!\!\!\!\! \! \!\!\sum\limits_{\{\Omega_{n,j'}^{(\gamma-1)} \subseteq \Omega_{n,j}^{(\gamma)}\} }\!\!\!\!\frac{\phi_{n,j'}^{(\gamma -1)}}{\phi_{n,j}^{(\gamma)}}P_{n,j'}^{(\gamma-1)} \! \!\textrm{ with}\sum\limits_{\{\Omega_{n,j'}^{(\gamma-1)} \subseteq \Omega_{n,j}^{(\gamma)}\} }\! \frac{\phi_{n,j'}^{(\gamma -1)}}{\phi_{n,j}^{(\gamma)}} \!= 1  \textrm{ and } \! \! \frac{\phi_{n,j'}^{(\gamma -1)}}{\phi_{n,j}^{(\gamma)}}\! \ge\! 0 \textrm{ for all } j', j.\! \!
    \label{eq:probdissect_op}\vspace*{-0.1in}
    \end{equation}
\end{enumerate}
The above formulation characterizes $\frac{\phi_{n,j'}^{(\gamma -1)}}{\phi_{n,j}^{(\gamma)}}$ as the probability associated with moving from level $\gamma-1$ to $\gamma$.
This will be used to show the monotonicity of the resulting bounds.
\smallskip

This dissection results in significantly fewer subgroups, as illustrated below.   

\begin{myex}
\label{Ex:NodalOG}
Consider again the multi-horizon scenario tree $\mathfrak{T}(3)$ at level $\gamma=3$ from Figure \ref{MHEOGSIOPT_common_prob_figure_example}.
Under this new dissection, there are only $m_2=2$ subgroups (instead of $8$) at level $\gamma-1=2$, as shown in Figure \ref{fig:MHNOG_example}. These are the first and eighth subgroups in Figure \ref{MHEOGSIOPT_common_prob_figure_example}, in which every strategic node contains either the operational scenarios $\{\textnormal{AB,AC}\}$ or $\{\textnormal{AD,AE}\}$. The probability of operational scenarios within each subgroup are obtained the same way, and the two subproblems are solved independently.  The nodal weights (i.e., $\phi_{1,1}^{(2)} = \frac{2}{5},\   \phi_{2,1}^{(2)} = \frac{2}{6},$ and $\phi_{3,1}^{(2)} =\frac{4}{7}$) are also calculated similarly, and they will be used to form the bound after the two subproblems are solved. 
\end{myex}
\input{MHNOG_simple}\vspace*{-0.15in}

\paragraph{Subproblems \& Bound} Given the above dissection of the operational subtree, we define our operational subproblems as in Definition \ref{def:MHOG} and use it to obtain a lower bound to the optimal value \textit{MHSP}.   Given that we now only have $m_{\gamma}$ subgroups (instead of $m_{\gamma}^{|\mathcal{N}|}$), we can no longer use an expectation-based combination. We need to combine them in a different way, as defined below. 

\begin{mydef}
    \label{def:MHNOG}
Let 
$(\hat{x}_{n,j},\hat{y}_{n,j}^{\omega,\tau}) $ 
be the optimal solution of multi-horizon operational group subproblem ${\text{MHOG}\big( \Omega_{\text{OG},j}^{(\gamma)} \big)}$ for $j \in [m_{\gamma}]$ based on the dissection with property \ref{add_prop_op}. 
The \textit{Multi-Horizon Nodal Operational Group bound} at operational level $\gamma$, denoted ${\textit{MHNOG}}(\gamma)$, is defined as follows:\vspace*{-0.15in}
    \begin{align} \label{MHNOG_def}
  {\textit{MHNOG}}(\gamma) 
 : =  
 \sum\limits_{t \in \mathcal{H}}\sum\limits_{n \in \mathcal{N}_t}  \sum_{j=1}^{m_{\gamma}} 
 \phi_{n,j}^{(\gamma)}
 \pi_{n} \Big( c_n \hat{x}_{n,j} + \sum\limits_{\omega \in \Omega_{n,j}^{(\gamma)}} {\hat{\pi}^{(\gamma)}_{\omega}} \sum\limits_{\tau \in \mathcal{T}_t}  q_{n}^{\omega,\tau} \hat{y}_{n,j}^{\omega,\tau} \Big) .  
  \end{align}
\end{mydef}

Let us compare the above node-based bound with the expectation-based bound introduced in Definition~\ref{def:MHEOG}. 
In the expectation-based bound $\textit{MHEOG}(\gamma)$, the optimal values of subproblems  $\textit{MHOG}\big( \Omega_{\text{OG},k}^{(\gamma)} \big)$ are combined with their  probabilistic weights $\phi_{\text{OG},k}^{(\gamma)}$. In contrast,  $\textit{MHNOG}(\gamma)$  
does not combine the optimal values $\textit{MHOG}\big( \Omega_{\text{OG},j}^{(\gamma)} \big)$, $j \in [m_{\gamma}]$ directly. Instead, it 
takes the optimal solution from each subgroup problem and aggregates their objective contributions at the strategic-node level. Specifically, for each strategic node $n$ in subgroup $j \in [m_{\gamma}]$, which is associated with operational subset $\Omega_{n,j}^{(\gamma)}$, 
the corresponding objective value is weighted by $\phi_{n,j}^{(\gamma)}$ (see Figure \ref{fig:MHNOG_example}). To see this, compare \eqref{MHNOG_def} with the objective function of the operational subproblems in Definition~\ref{def:MHOG} and observe the additional summation over $\sum_{j=1}^{m_{\gamma}} 
 \phi_{n,j}^{(\gamma)}$.  We further illustrate this node-based dissection and combination by Example \ref{ex:MHNOG} in \ref{Appendix:OSD}, where we also summarize the notation.

\paragraph{Properties} We show that $\textit{MHNOG}(\gamma)$ also yields a monotonic chain of lower bounds to the optimal value \textit{MHSP} as $\gamma$ increases.
\begin{myprop} \label{Prop4}
Consider \textit{MHSP} given in \eqref{MHSP1}.
The following chain of inequalities holds true:\vspace*{-0.07in}
\begin{eqnarray}
     \textit{MHNOG}(1) \leq \textit{MHNOG}(2) \leq \ldots  \leq \textit{MHNOG}(\gamma)  \leq \textit{MHNOG}(\gamma + 1) \leq \ldots \leq \textit{MHSP}. \nonumber
\end{eqnarray}
\end{myprop}

\proof{
We prove that for two consecutive levels $\gamma$ and $\gamma+1$ of the refinement chain \eqref{refinement_chain_operational_common_prob}, we have  $\textit{MHNOG}(\gamma)\leq 
\textit{MHNOG}(\gamma+1)$. At the highest level, we recover the optimal value \textit{MHSP} because we have the full set of operational scenarios $\Omega_n$ at each strategic node  $n\in\mathcal{N}_t$, $t\in\mathcal{H}$. Consider a multi-horizon operational group subproblem $\text{MHOG}\big( \Omega_{\text{OG},j}^{(\gamma+1)} \big)$ for $j \in [m_{\gamma}]$ and let 
$(\hat{x}_{n,j},\hat{y}_{n,j}^{\omega,\tau})$ be its optimal solution. Let $\Omega_{\text{OG},j'}^{(\gamma)} \subseteq \Omega_{\text{OG},j}^{(\gamma+1)}$ and  the operational
subsets of $\Omega_{\text{OG},j'}^{(\gamma)}$ and $\Omega_{\text{OG},j}^{(\gamma+1)}$ at strategic node $n$ be denoted by $\Omega_{n,j'}^{(\gamma)  }$ and $\Omega_{n,j}^{(\gamma+1)  }$, respectively.
We have
\begin{align}
\textit{MHOG}(\Omega_{\text{OG},j'}^{(\gamma)}) & =
\sum\limits_{t \in \mathcal{H}}\sum\limits_{n \in \mathcal{N}_t} \pi_{n} \big( c_n \overline{x}_{n,j'} + \sum\limits_{\omega \in \Omega_{n,j'}^{(\gamma)}} \hat{\pi}^{(\gamma)}_{\omega} \sum\limits_{\tau \in \mathcal{T}_t}  q_{n}^{\omega,\tau} \overline{y}_{n,j'}^{\omega,\tau} \big) 
\nonumber 
\\
& \leq \sum\limits_{t \in \mathcal{H}}\sum\limits_{n \in \mathcal{N}_t} \pi_{n} \big( c_n \hat{x}_{n,j} + \sum\limits_{\omega \in \Omega_{n,j'}^{(\gamma)}} \hat{\pi}^{(\gamma)}_{\omega} \sum\limits_{\tau \in \mathcal{T}_t}  q_{n}^{\omega,\tau} \hat{y}_{n,j}^{\omega,\tau} \big), 
\label{eq_30}
\end{align}
\vspace*{-0.15in}

\noindent
where 
$(\overline{x}_{n,j'},\overline{y}_{n,j'}^{\omega,\tau})  $
denotes the optimal solution of the multi-horizon operational group subproblem $\text{MHOG}(\Omega_{\text{OG},j'}^{(\gamma)})$.  Multiplying in inequality \eqref{eq_30} the cost of each strategic node $n$ by $\phi_{n,j'}^{(\gamma)}$ and summing up for all $\Omega_{n,j}^{(\gamma+1)}$, $j \in [m_{\gamma+1}]$, satisfying \eqref{eq_subsets_op} we obtain\vspace*{-0.1in}
\begin{align*}
 \textit{MHNOG}(\gamma)  & = 
\sum\limits_{t \in \mathcal{H}}\sum\limits_{n \in \mathcal{N}_t} \sum_{j=1}^{m_{\gamma+1}}\V{\sum_{ \{ \Omega_{n,j'}^{(\gamma)}\subseteq \Omega_{n,j}^{(\gamma+1)}\}  }
     } \!\!\!\!\!\!\!\!\!\!\!  \phi_{n,j'}^{(\gamma)} \pi_n \!
 \big( c_n \overline{x}_{n,j'} + \sum\limits_{\omega \in \Omega_{n,j'}^{(\gamma)}} \hat{\pi}^{(\gamma)}_{\omega} \sum\limits_{\tau \in \mathcal{T}_t}  q_{n}^{\omega,\tau} \overline{y}_{n,j'}^{\omega,\tau} \big) \\
 & \leq 
\sum\limits_{t \in \mathcal{H}}\sum\limits_{n \in \mathcal{N}_t} \sum_{j=1}^{m_{\gamma+1}}\V{\sum_{ \{ \Omega_{n,j'}^{(\gamma)}\subseteq \Omega_{n,j}^{(\gamma+1)}\}  }
     } \!\!\!\!\!\!\!\!\!\!\! \phi_{n,j'}^{(\gamma)}
 \pi_{n} \! \big( c_n \hat{x}_{n,j} + \sum\limits_{\omega \in \Omega_{n,j'}^{(\gamma)}} \hat{\pi}^{(\gamma)}_{\omega} \sum\limits_{\tau \in \mathcal{T}_t}  q_{n}^{\omega,\tau} \hat{y}_{n,j}^{\omega,\tau} \big) \\
& =  
\sum\limits_{t \in \mathcal{H}}\sum\limits_{n \in \mathcal{N}_t} \! \! \sum_{j=1}^{m_{\gamma+1}} 
\! \! 
\V{\sum_{ \substack{\{ \Omega_{n,j'}^{(\gamma)}\subseteq \Omega_{n,j}^{(\gamma+1)}\} } }
     }
\!\!\!\!\!\!\!\!\! \!\!\phi_{n,j'}^{(\gamma)} \pi_{n}
c_n \hat{x}_{n,j} + \!\! 
\sum\limits_{t \in \mathcal{H}}\sum\limits_{n \in \mathcal{N}_t} 
\sum_{j=1}^{m_{\gamma+1}}
\pi_n 
\V{\sum_{ \substack{\{ \Omega_{n,j'}^{(\gamma)}\subseteq \Omega_{n,j}^{(\gamma+1)}\}}  }
     } 
\Big(
\sum\limits_{\omega \in \Omega_{n,j'}^{(\gamma)}} \!\! \!\! \phi_{n,j'}^{(\gamma)} \hat{\pi}^{(\gamma)}_{\omega} \\
& \qquad \qquad \qquad \qquad \qquad \qquad \qquad \qquad \qquad \qquad  \qquad \qquad \qquad 
\qquad \qquad  
\sum\limits_{\tau \in \mathcal{T}_t}  q_{n}^{\omega,\tau} \hat{y}_{n,j}^{\omega,\tau}
\Big)
\\
& =  
\sum\limits_{t \in \mathcal{H}}\sum\limits_{n \in \mathcal{N}_t} \! \! \sum_{j=1}^{m_{\gamma+1}} 
\! \! 
\V{\sum_{ \substack{\{ \Omega_{n,j'}^{(\gamma)}\subseteq \Omega_{n,j}^{(\gamma+1)}\} } }
     }
\!\!\!\!\!\!\!\!\! \!\!\phi_{n,j'}^{(\gamma)} \pi_{n}
c_n \hat{x}_{n,j} + 
\sum\limits_{t \in \mathcal{H}}\sum\limits_{n \in \mathcal{N}_t} 
\sum_{j=1}^{m_{\gamma+1}} \pi_n
\! \!\! \!\! 
\sum\limits_{\omega \in \Omega_{n,j}^{(\gamma+1)}} \!\!\!\! \!\! \phi_{n,j}^{(\gamma+1)} \hat{\pi}^{(\gamma+1)}_{\omega} \sum\limits_{\tau \in \mathcal{T}_t}  q_{n}^{\omega,\tau} \hat{y}_{n,j}^{\omega,\tau},
\label{eq:MHEOG_gamma_bound}
\end{align*}
where the last equality follows from \eqref{eq:probdissect_op}.
Since for every strategic node $n \in \mathcal{N}_t, \ t \in \mathcal{H}$, \VV{it follows from \Cref{sum_condition_lemma_common_prob} that }
\V{$\sum_{j=1}^{m_{\gamma+1}} \sum_{ \{ \Omega_{n,j'}^{(\gamma)}\subseteq \Omega_{n,j}^{(\gamma+1)}\}  }
     \phi_{n,j'}^{(\gamma)} =
\sum_{j=1}^{m_{\gamma+1}} 
\phi_{n,j}^{(\gamma+1)}$,}
we get\vspace*{-0.15in} 
\begin{align*}
& \! \! \textit{MHNOG}(\gamma) \! \! \leq \!\sum\limits_{t \in \mathcal{H}}\sum\limits_{n \in \mathcal{N}_t} \! \sum_{j=1}^{m_{\gamma+1}} \! \! 
\phi_{n,j}^{(\gamma+1)} \pi_{n}
\Big( c_n \hat{x}_{n,j}\!+ \! \!
\! \!\sum\limits_{\omega \in \Omega_{n,j}^{(\gamma+1)}} \! \! \hat{\pi}^{(\gamma+1)}_{\omega} \sum\limits_{\tau \in \mathcal{T}_t}  q_{n}^{\omega,\tau} \hat{y}_{n,j}^{\omega,\tau} \Big) \! = \! \textit{MHNOG}(\gamma+1). \qedhere 
\end{align*}
}

\subsection{Strategic Scenario Tree Dissection}\label{sec:Strategic_P_M}

An alternative strategy consists in dissecting {\it only the strategic} scenario tree, while retaining the full set of operational scenarios at each strategic node. Specifically, we consider groups of strategic scenarios and derive lower bounds to the optimal value \textit{MHSP} by solving multiple subproblems, each involving a reduced number of strategic scenarios.\vspace*{-0.05in}

\subsubsection{Expectation-Based Combination} 
\label{sec:dissecting_stra}
Similar to the operational subtree dissection, we assign to each strategic subproblem, a probabilistic weight derived from the probability measures of its strategic scenarios. 
We then combine the optimal values of the subgroups into a weighted sum, thereby obtaining a lower bound to the optimal value  \textit{MHSP}. Moreover, we show that these bounds also form a monotonic chain of inequalities as the number of strategic scenarios within each subgroup increases. 

\paragraph{Dissection} We construct a collection of subsets of the original strategic support $\mathcal{S}$,
with a refinement chain as follows:\vspace*{-0.25in}
\begin{gather}
\mathcal{S},\nonumber\\
\vdots\nonumber\\
(\mathcal{S}_{1}^{(\ell)},\mathcal{S}_{2}^{(\ell)},\dots,\mathcal{S}_{\hat{m}_{\ell}}^{(\ell)}),\label{eq:ref_chain}\\
\vdots\nonumber\\
(\{\mathbf{s}_{1}\},\{\mathbf{s}_{2}\},\dots,\{\mathbf{s}_{|\mathcal{S}|}\}),\vspace*{-0.1in}\nonumber
\end{gather}
where each row is a collection of subsets of the support
$\mathcal{S}$
with the following properties:
\begin{enumerate}[label = ({\roman*})]
    \item  Each set ${\mathcal{S}}_i^{(\ell)}$ for $i=1,\ldots,\hat{m}_{\ell}$ has the same number of scenarios, denoted by $\hat{g}_{\ell}$, of which $\hat{f}$ \VV{is the number of } fixed scenarios such that $\hat{g}_{\ell} > \hat{f}$. Thus, the total number of subgroups $ \hat{m}_{\ell} = (\vert \mathcal{S}\vert - \hat{f})/(\hat{g}_{\ell} - \hat{f})$ is an integer;\vspace*{-0.03in} 
    
    \item  The union of the subsets covers the whole support, i.e., $\mathcal{S} = \mathop{\cup}_{i=1}^{\hat{m}_{\ell}} {\mathcal{S}}_i^{(\ell)}$;\vspace*{-0.03in}
    
    \item Each set $\mathcal{S}_i^{(\ell)}$ at level $\ell$ is the union of sets from the next more refined collection ${\mathcal{S}}_{\V{i'}}^{(\ell-1)}$, i.e.,\vspace*{-0.05in} 
    \begin{equation}
    \V{{\mathcal{S}}_i^{(\ell)}= \bigcup_{\{  \mathcal{S}_{i'}^{(\ell -1)}\subseteq \mathcal{S}_i^{(\ell)} \}} \mathcal{S}_{i'}^{(\ell -1)}};
    \label{eq_subsets}
    \end{equation}
    \item  The intersection of the subsets, in general, is not empty but consists only of fixed scenarios, i.e. ${\bigcap}_{i=1}^{\hat{m}_{\ell}}\mathcal{S}_{i}^{(\ell)} =
    \hat{\Omega}_{\hat{f}} = 
    \{ \mathbf{s}_{1},\ldots,\mathbf{s}_{\hat{f}} \}, \ \hat{f} \geq 0$ for each strategic level $\ell$. If $\hat{f} = 0$, the subsets are disjoint. 
\end{enumerate}
\begin{remark}
In general, each set $\mathcal{S}_i^{(\ell)}$ at level $\ell$ can have different cardinalities. However, for simplicity,  we assume that it has the same number $\hat{g}_{\ell}$ of scenarios, of which $\hat{f}$ are fixed, appearing in all subsets.
\end{remark}

At the highest level,  we have a single group containing the entire strategic scenario set. 
Moving down the refinement chain yields progressively finer partitions, and the most refined partition (i.e., $\ell=1$) contains each strategic scenario on its own. 
For each partition of the support, there corresponds a dissection of the probability measure $\Pi$ into probability measures $\Pi_i^{(\ell)}, i \in [\hat{m}_{\ell}]$ with the following properties:
\begin{enumerate}[label = ({\roman*})]
    \item $\Pi_i^{(\ell)}$ has support $S_i^{(\ell)}$;
    \item $\Pi$ can be expressed as $\Pi = \sum_{i=1}^{\hat{m}_{\ell}} \phi_i^{(\ell)} \Pi_i^{(\ell)}$, with $\sum_{i=1}^{\hat{m}_{\ell}} \phi_i^{(\ell)} = 1$ and $\phi_i^{(\ell)} \ge 0$ for all $i \in [\hat{m}_{\ell}]$;
    \item Each $\Pi_i^{(\ell)}$ can be expressed as a convex combination of
    measures from the refined collection $\{\Pi_{i'}^{(\ell-1)}\}$, similar to the formulation in  \eqref{eq:convex_combo_operational}. Consequently, we can write\vspace*{-0.1in}
    \begin{equation}    \!\!\!\!\Pi_i^{(\ell)}\!\!=\!\!\!\!\!\sum\limits_{\{\mathcal{S}_{i'}^{(\ell-1)} \subseteq \mathcal{S}_i^{(\ell)}\} }\!\!\!\frac{\phi_{i'}^{(\ell -1)}}{\phi_i^{(\ell)}}\Pi_{i'}^{(\ell -1)}  \textrm{ with}  \sum\limits_{\{\mathcal{S}_{i'}^{(\ell-1)} \subseteq \mathcal{S}_i^{(\ell)} \} }\! \frac{\VVP{\phi_{i'}^{(\ell -1)}}}{\phi_i^{(\ell)}} \!= \! 1 \textrm{ and } \frac{\VVP{\phi_{i'}^{(\ell -1)}}}{\phi_i^{(\ell)}}\! \! \ge\! 0 \textrm{ for all }i', i.\vspace*{-0.15in}
    \label{eq:probdissect}
    \end{equation}
\end{enumerate}

Let $\hat{o}_{\ell} = \hat{g}_{\ell} - \hat{f}$ and the set of strategic scenarios at a given $\mathcal{S}_i^{(\ell)}$ be defined as\vspace*{-0.1in} 
\begin{equation}
    \mathcal{S}_i^{(\ell)} = \{\mathbf{s}_1,\ldots,\mathbf{s}_{\hat{f}},\mathbf{s}_{\hat{f} + (i-1) \hat{o}_{\ell} + 1}, \ldots, \mathbf{s}_{\hat{f} + i \hat{o}_{\ell}} \}, \ i \in [\hat{m}_{\ell}].\vspace*{-0.05in} \nonumber
\end{equation}
Then the associated probabilities of these scenarios with support  $\mathcal{S}_i^{(\ell)}$ are given by\vspace*{-0.15in}
\begin{displaymath}
\hat{\Pi}_{i,\mathbf{s}_j}^{(\ell)}  = 
\left\{ \begin{array}{ll}
\pi^{\mathbf{s}_j},  & \textrm{$j \in [\hat{f}]$}, \\
\pi^{\mathbf{s}_j} \frac{1 - \sum_{j=1}^{\hat{f}} \pi^{\mathbf{s}_j}}{\sum_{j = \hat{f} + (i-1) \hat{o}_{\ell} +1}^{\hat{f} + i \hat{o}_{\ell}} \pi^{\mathbf{s}_j}}, & \textrm{$j = \hat{f}+(i-1)  \hat{o}_{\ell} +1, \ldots, \hat{f} + i \hat{o}_{\ell}$,}
\end{array} \right.
\end{displaymath}
where $\pi^{\mathbf{s}_j}$ denotes the probability of the strategic scenario $\mathbf{s}_j$ with support $\mathcal{S}$.  
With this, the probability measures $\Pi_i^{(\ell)}$ are given by $\Pi_i^{(\ell)}=\sum_{\mathbf{s}_j \in \mathcal{S}_i^{(\ell)}} \hat{\Pi}_{i,\mathbf{s}_j}^{(\ell)} \delta_{\mathbf{s}_j}$, with $\delta_{\mathbf{s}_j}$ denoting the Dirac measure at strategic scenario $\mathbf{s}_j$, for $i \in [\hat{m}_{\ell}]$.
Additionally, the probabilistic weight associated with subset $\mathcal{S}_i^{(\ell)}$, denoted by  $\phi_i^{(\ell)}$, is defined as\vspace*{-0.15in} 
\begin{equation}
    \phi_i^{({\ell})} = \frac{\sum_{j = \hat{f} + (i-1) \hat{o}_{\ell} +1}^{\hat{f} + i \hat{o}_{\ell}} \pi^{\mathbf{s}_j}}  
    {1 - \sum_{j=1}^{\hat{f}} \pi^{\mathbf{s}_j}}, \  i \in [\hat{m}_{\ell}].\vspace*{-0.05in}  \label{eq:stra_weights}
\end{equation}
For the case in which the subgroups are disjoint, we set $\hat{f}=0$ in the formulas above.

To illustrate this dissection of the strategic scenario tree, we consider the following example.
\begin{myex}\label{Ex:Str_Exp}
    Consider the multi-horizon scenario tree $\mathfrak{T}(7)$ at level $\ell=3$ from Figure \ref{Fig:Common_Tree}, where there are $\hat{g}_3=4$ strategic scenarios and $\hat{m}_3=1$ strategic group. 
    In Figure \ref{MHESG_2_with_7_nodes},  going from level $\ell=3$ to $\ell-1=2$, the four strategic scenarios are split into two subgroups (i.e, $\hat{m}_2=2$) of size two (i.e, $\hat{g}_2=2$) with no fixed scenarios (i.e, $\hat{f}=0$ and $\hat{\Omega}_{\hat{f}}=\emptyset$).   
    The weight of the first subgroup $\mathcal{S}_1^{(2)}$ is given by $\phi_1^{(2)}=\sum_{j=1}^2 \pi^{\mathbf{s}_j}=\frac{1}{3}$, where $\mathbf{s}_1=\{1,2,4\}$, $\mathbf{s}_2=\{1,2,5\}$, $\pi^{\mathbf{s}_1}=\frac{1}{3} \times \frac{2}{3}$, and $\pi^{\mathbf{s}_2}=\frac{1}{3} \times \frac{1}{3}$. The probability of $\mathbf{s}_1$ at level $\ell=2$ is $\hat{\Pi}_{1,\mathbf{s}_1}^{(2)}  = (\frac{1}{3}\times \frac{2}{3}) \times (\frac{1}{1/3}) = \frac{2}{3} $, so $\hat{\Pi}_{1,\mathbf{s}_2}^{(2)}  = \frac{1}{3}$. The second strategic subgroup $\mathcal{S}_2^{(2)}$'s values are calculated similarly, and each subgroup is solved independently.
\end{myex}
\vspace*{-0.05in}
\input{Tree_MHESG_Fig}

\paragraph{Subproblems \& Bound}
Given the above dissection of the strategic tree, we define the corresponding strategic subproblems to obtain a lower bound to the optimal value \textit{MHSP}.
\begin{mydef}
    \label{def:MHSG_scen}
The \textit{Multi-Horizon Strategic Group subproblem} with index $i \in [\hat{m}_{\ell}]$ at level $\ell$, denoted by    $\text{MHSG}(\mathcal{S}_i^{(\ell)})$, is defined as\vspace*{-0.05in}
\begin{align*} 
\textit{MHSG}(\mathcal{S}_i^{(\ell)}) &  : = \min\limits_{\mathbf{x},\mathbf{y}}\sum\limits_{t \in \mathcal{H}} \sum\limits_{\mathbf{s} \in \mathcal{S}_i^{(\ell)} }\hat{\Pi}_{i,\mathbf{s}}^{(\ell)}\Big( c^{\mathbf{s}}_t x^{\mathbf{s}}_t + \sum\limits_{\omega \in \Omega_t} \pi^{\omega}_{\mathbf{s},t}\sum\limits_{\tau \in \mathcal{T}_t}  q_{\mathbf{s},t}^{\omega,\tau} y_{\mathbf{s},t}^{\omega,\tau} \Big)  \nonumber \\
\ &\quad\ \textrm{ s.t. }    A x_{0}=h_{0}, \nonumber \\
&\qquad\quad\ T_{t}^{\mathbf{s}} x_{t-1}^{\mathbf{s}} + W_{t}^{\mathbf{s}} x_{t}^{\mathbf{s}} =h_{t}^{\mathbf{s}},\    t\in\mathcal{H}\setminus\left\{0\right\},\nonumber \\
&\qquad\quad\ T_{\mathbf{s},t}^{\omega,1} x_{t}^{\mathbf{s}}+W_{\mathbf{s},t}^{\omega,1} y_{\mathbf{s},t}^{\omega,1} = h_{\mathbf{s},t}^{\omega,1},\  \mathbf{s}\in\mathcal{S}_i^{(\ell)}, \  \omega \in \Omega_t,\ t \in \mathcal{H}, \nonumber\\
&\qquad\quad\ T_{\mathbf{s},t}^{\omega,\tau} y_{\mathbf{s},t}^{\omega,\tau -1} +W_{\mathbf{s},t}^{\omega,\tau} y_{\mathbf{s},t}^{\omega,\tau} = h_{\mathbf{s},t}^{\omega,\tau},\   \mathbf{s}\in\mathcal{S}_i^{(\ell)},\ \omega \in \Omega_t,\ \tau \in \mathcal{T}_t\setminus\left\{1\right\}, t \in \mathcal{H},  \nonumber\\
&\qquad\quad\ x_{t}^{\mathbf{s}^{\prime}}=x_{t}^{\mathbf{s}^{\prime\prime}},\ \forall \mathbf{s}^{\prime},\mathbf{s}^{\prime\prime}\in\mathcal{S}_i^{(\ell)},\textrm{ for which } \mathbf{s}^{\prime}=\mathbf{s}^{\prime\prime}, \textrm{ up to  stage } t, \nonumber \\
&\qquad \quad\ y_{\mathbf{s},t}^{\omega^{\prime},\tau}=y_{\mathbf{s},t}^{\omega^{\prime\prime},\tau},\ \forall {\omega}^{\prime},{\omega}^{\prime\prime} \in \Omega_t\textrm{ for which } {\omega}^{\prime}={\omega}^{\prime\prime}, \textrm{ up to  stage } \tau.
\end{align*}
\end{mydef}

\noindent Combining the objective function values of all  subproblems $\textit{MHSG}(\mathcal{S}_i^{(\ell)})$, $i\in  [\hat{m}_\ell]$ with their corresponding probabilistic weights associated with $\mathcal{S}_i^{(\ell)}$ allows us to introduce the following quantity. 
\begin{mydef}
       The \textit{Multi-Horizon Expected Strategic Group bound} at strategic level $\ell$, denoted $\textit{MHESG}(\ell)$, is defined as follows:\vspace*{-0.17in}
    \begin{align} \label{MHESG_def}
  \textit{MHESG}(\ell) 
 : =  
  \sum_{i=1}^{\hat{m}_{\ell}} \phi_i^{(\ell)} \textit{MHSG}(\mathcal{S}_i^{(\ell)}).  
  \end{align}
\end{mydef}

\paragraph{Properties} We now show that $\textit{MHESG}(\ell)$ leads to a monotonic chain of lower bounds to the optimal value \textit{MHSP} as $\ell$ increases.
\begin{myprop} \label{Prop3}
Consider \textit{MHSP} given in \eqref{MHSP1}.
The following chain of inequalities holds true:\vspace*{-0.05in}
\begin{eqnarray}
    \textit{MHESG}(1) \leq \textit{MHESG}(2) \leq \ldots \leq \textit{MHESG}(\ell) \leq \textit{MHESG}(\ell+1) \leq \ldots \leq \textit{MHSP}. \nonumber
\end{eqnarray}
\end{myprop}

\proof{We  prove that for two consecutive levels $\ell$ and $\ell+1$ of the refinement chain \eqref{eq:ref_chain}, we have  $\textit{MHESG}(\ell)\leq 
\textit{MHESG}(\ell+1)$. Consider a multi-horizon strategic group subproblem $\text{MHSG}(\mathcal{S}_i^{(\ell+1)})$ and let 
$(\hat{x}_t^{\mathbf{s}},\hat{y}_{\mathbf{s},t}^{\omega,\tau})$
be its optimal solution, where we ignore the dependence of this solution on subgroup $i$ for notational simplicity. Due to \eqref{eq_subsets} and \eqref{eq:probdissect} 
 we have\vspace*{-0.07in}
\begin{align*}
\textit{MHSG}(\mathcal{S}_i^{(\ell+1)})= &\sum\limits_{t \in \mathcal{H}} \sum\limits_{\mathbf{s} \in \mathcal{S}_i^{(\ell+1)} }\hat{\Pi}_{i,\mathbf{s}}^{(\ell+1)}\Big( c^{\mathbf{s}}_t \hat{x}^{\mathbf{s}}_t + \sum\limits_{\omega \in \Omega_t} \pi^{\omega}_{\mathbf{s},t}\sum\limits_{\tau \in \mathcal{T}_t}  q_{\mathbf{s},t}^{\omega,\tau} \hat{y}_{\mathbf{s},t}^{\omega,\tau} \Big)\\
=&\sum\limits_{t \in \mathcal{H}} \sum\limits_{\{  \mathcal{S}_{i'}^{(\ell)} \subseteq \mathcal{S}_i^{(\ell+1)}\} } \sum\limits_{\mathbf{s} \in \mathcal{S}_{i'}^{(\ell)} } 
\frac{\phi_{i'}^{(\ell )}}{\phi_i^{(\ell+1)}}\hat{\Pi}_{i',\mathbf{s}}^{(\ell)}
\Big( c^{\mathbf{s}}_t \hat{x}^{\mathbf{s}}_t + \sum\limits_{\omega \in \Omega_t} \pi^{\omega}_{\mathbf{s},t}\sum\limits_{\tau \in \mathcal{T}_t}  q_{\mathbf{s},t}^{\omega,\tau} \hat{y}_{\mathbf{s},t}^{\omega,\tau} \Big).\vspace*{-0.05in}
\end{align*}
\vspace*{-0.2in}

\noindent
Multiplying both sides by $\phi_i^{(\ell+1)}$ we get\vspace*{-0.1in}
\begin{equation*}
\phi_i^{(\ell+1)}\textit{MHSG}(\mathcal{S}_i^{(\ell+1)})=\sum\limits_{t \in \mathcal{H}} \sum\limits_{\{  \mathcal{S}_{i'}^{(\ell)} \subseteq \mathcal{S}_i^{(\ell+1)}\} } \sum\limits_{\mathbf{s} \in \mathcal{S}_{i'}^{(\ell)} } 
\V{{\phi_{i'}^{(\ell)}}\hat{\Pi}_{i',\mathbf{s}}^{(\ell)}}
\Big( c^{\mathbf{s}}_t \hat{x}^{\mathbf{s}}_t + \sum\limits_{\omega \in \Omega_t} \pi^{\omega}_{\mathbf{s},t}\sum\limits_{\tau \in \mathcal{T}_t}  q_{\mathbf{s},t}^{\omega,\tau} \hat{y}_{\mathbf{s},t}^{\omega,\tau} \Big),
\end{equation*}
\vspace*{-0.2in}

\noindent
and summing up for all $\mathcal{S}_i^{(\ell+1)}$,  $i \in [\hat{m}_{\ell+1}]$, we get\vspace*{-0.05in}
    \begin{align*} 
  \textit{MHESG}(\ell+1) \!  
  = \! \!
  \sum_{i=1}^{\hat{m}_{\ell+1}}  \phi_i^{(\ell+1)} \textit{MHSG}(\mathcal{S}_i^{(\ell+1)})
  \geq 
  \! \sum_{i=1}^{\hat{m}_{\ell+1}} \! \!
  \V{\sum\limits_{\{  \mathcal{S}_{i'}^{(\ell)} \subseteq \mathcal{S}_i^{(\ell+1)}\} }} \!\!\!\!\!\! \phi_{i'}^{(\ell)}\textit{MHSG}(\mathcal{S}_{\V{i'}}^{(\ell)}) = \textit{MHESG}(\ell), 
  \end{align*}
  \vspace*{-0.2in}

  \noindent 
where the inequality follows from the fact that 
$(\hat{x}_t^{\mathbf{s}},\hat{y}_{\mathbf{s},t}^{\omega,\tau})$
is suboptimal for each $\mathcal{S}_{\V{i'}}^{(\ell)} \subseteq \mathcal{S}_i^{(\ell+1)}$.
At the highest level, optimal value \textit{MHSP} is recovered from the full set of strategic scenarios $\mathcal{S}$.\qed
}

\subsubsection{Node-Based Combination}
\label{stra_node}
Alternatively, we formulate the $\text{MHSG}$ problem through a node-based representation, assigning to each subproblem a probabilistic weight derived from the probability of its strategic nodes. The optimal values of the subgroups are then aggregated into a weighted sum, yielding a lower bound to the optimal value \textit{MHSP}, in analogy with the previous section. 
\vspace*{-0.05in}

\paragraph{Dissection}
We recall that a strategic scenario $\mathbf{s} \in \mathcal{S}$ is a path from the root node at $t=0$ to a strategic leaf node at $t=H$.
Let $ \mathcal{N}^{\mathbf{s}} \subseteq \mathcal{N} $ denote the subset of strategic nodes associated with scenario $ \mathbf{s} $, and define 
$ \mathcal{N}_i^{(\ell)} := \bigcup_{\mathbf{s} \in \mathcal{S}_i^{(\ell)}} \mathcal{N}^{\mathbf{s}} $ for $ i\in  [\hat{m}_\ell]$ 
as the set of strategic nodes within strategic subgroup $i$ at strategic level $\ell$. Breaking this up by stage, for each strategic stage $t \in \mathcal{H}$, let  $\mathcal{N}_{i,t}^{(\ell)} \subseteq \mathcal{N}_i^{(\ell)} $ comprise the nodes at stage $t$. Thus, we also have $\mathcal{N}_i^{(\ell)} = \bigcup_{t \in \mathcal{H}} \mathcal{N}_{i,t}^{(\ell)}$.
Furthermore, let $ \mathcal{S}^n \subseteq \mathcal{S} $ represent the subset of strategic scenarios passing through node $ n \in \mathcal{N} $.
At level $\ell$, let $\mathcal{S}_{i,n}^{(\ell)}$ denote
the set of strategic scenarios included in the $i^{th}$ subgroup passing through node $n$,
and define 
$ \hat{\Pi}_{i,n}^{(\ell)} := \sum_{\mathbf{s} \in \mathcal{S}_{i,n}^{(\ell)}} \hat{\Pi}_{i,\mathbf{s}}^{(\ell)} $ as the updated probability of node $ n \in \mathcal{N}_i^{(\ell)}$ within subgroup $i$, for $ i \in [\hat{m}_{\ell}]$.
The dissection procedure remains the same as that described in the previous section, except that we adopt a node-based notation in this section.\smallskip

 \noindent {\bf Example \ref{Ex:Str_Exp} Continued.} 
The subgroup $\mathcal{S}_1^{(2)}$ has strategic nodes $\mathcal{N}_1^{(2)}=\{1,2,4,5\}$, while the subgroup $\mathcal{S}_2^{(2)}$ has strategic nodes $\mathcal{N}_2^{(2)}=\{1,3,6,7\}$. Each strategic subgroup contains the original full operational subtree. Recall from \Cref{Ex:Str_Exp} that $\mathcal{S}_1^{(2)}$ contains strategic scenarios $\mathbf{s}_1$  and $\mathbf{s}_2$. Since $\mathbf{s}_1$  passes through the nodes $\{1,2,4\}$, $\mathcal{N}^{\mathbf{s}_1} = \{1,2,4\}$. Similarly, $\mathcal{N}^{\mathbf{s}_2} = \{1,2,5\}$ and $\mathcal{N}_1^{(2)}=\bigcup_{\mathbf{s} \in \mathcal{S}_1^{(2)}} \mathcal{N}^{\mathbf{s}}.$  Moreover, the nodes at different strategic stages of $\mathcal{S}_1^{(2)}$ can be written as $\mathcal{N}_{1,0}^{(2)} = \{1\},$ $\mathcal{N}_{1,1}^{(2)} = \{2\},$ and $\mathcal{N}_{1,2}^{(2)} = \{4,5\}.$   Scenarios $\mathbf{s}_1$ and $\mathbf{s}_2$ pass through node $2$ in the first subgroup; thus, $\mathcal{S}_{1,2}^{(2)} = \{\mathbf{s}_1, \mathbf{s}_2\}.$ Consequently, $\hat{\Pi}_{1,2}^{(2)} = \sum_{\mathbf{s} \in \mathcal{S}_{1,2}^{(2)}} \hat{\Pi}_{1,\mathbf{s}}^{(2)} = 1.$ Similarly, $\hat{\Pi}_{1,4}^{(2)} =\frac{2}{3}.$

\paragraph{Subproblems \& Bound}
Given the above definition, we define our strategic subproblems, which we then use to obtain a lower bound to the optimal value \textit{MHSP}. 
\begin{mydef}
    \label{def:MHSG}
The node formulation of the \textit{Multi-Horizon Strategic Group subproblem} with index $i \in [\hat{m}_{\ell}]$ at level $\ell$, denoted by $\text{MHSG}(\mathcal{N}_i^{(\ell)})$,  is defined as\vspace*{-0.05in} 
\begin{align} \label{MHSG_node}
\textit{MHSG}(\mathcal{N}_i^{(\ell)})  : = 
& \min \limits_{\mathbf{x},\mathbf{y}} 
\sum_{t\in\mathcal{T}}
\sum\limits_{n \in \mathcal{N}_{i,t}^{(\ell)}} \hat{\Pi}_{i,n}^{(\ell)} \Big( c_n x_n + \sum\limits_{\omega \in \Omega_n} \pi_{\omega} \sum\limits_{\tau \in \mathcal{T}_t}  q_{n}^{\omega,\tau} y_{n}^{\omega,\tau} \Big) 
\\
& \textrm{ s.t.  }  Ax_{n}=h_{n},\  n \in  \mathcal{N}_0, \nonumber \\
& \qquad \ T_{n} x_{a({n})} + W_{n} x_{n}=h_{n},\   n \in  \mathcal{N}_i^{(\ell)}\setminus\{\mathcal{N}_0\},\nonumber \\
& \qquad \  T_{n}^{\omega,1} x_{n} + W_{n}^{\omega,1} y_{n}^{\omega,1} = h_{n}^{\omega,1},\  n \in \mathcal{N}_i^{(\ell)},\  \omega \in \Omega_n,
\nonumber\\
& \qquad \  T_{n}^{\omega,\tau} y_{n}^{\omega,\tau -1} + W_{n}^{\omega,\tau} y_{n}^{\omega,\tau} = h_{n}^{\omega,\tau},\   n \in \mathcal{N}_{i,t}^{(\ell)}, \tau \in \mathcal{T}_t\setminus\left\{1\right\}, 
\omega \in \Omega_n,\ t\in\mathcal{H}.
\nonumber
\end{align}
\end{mydef}

Let us briefly discuss the differences between the strategic and operational dissections. Observe that the operational component of MHSP is a {\it subtree} embedded under the strategic nodes. Therefore, when operational scenarios are split for the expectation-based bound, it is necessary to enumerate all operational subgroup combinations to obtain subgroup weights that sum to $1$. In contrast, the strategic component is already a {\it scenario tree}. Hence, it can be divided into fewer subgroups whose weights sum to $1$. To see this, compare Figures \ref{MHEOGSIOPT_common_prob_figure_example} and \ref{MHESG_2_with_7_nodes}.

Due to this fundamental structural difference between the strategic and operational components 
of the multi-horizon tree, the dissection for the node-based combination differs from the expectation-based combination only in the operational case. In the strategic case, the two dissections coincide.  
Consequently, the optimal values of the nodal subgroup problems in \Cref{def:MHSG} are identical to those of the scenario-based subproblems in \Cref{def:MHSG_scen}. Therefore, the expectation-based bound $\textit{MHESG}(\ell)$ can also be written as a convex combination of the node-based subproblem optimal values using the subgroup weights $\phi_i^{({\ell})}$:
$\textit{MHESG}(\ell) = \sum_{i=1}^{\hat{m}_\ell} \phi_i^{(\ell)} \textit{MHSG}(\mathcal{N}_i^{(\ell)}).$

Nevertheless, it is useful to formulate a strategic node-based bound that is analogous to its operational counterpart in \eqref{MHNOG_def}. Rather than assigning a single probability to each strategic subgroup, this lower bound is obtained by weighting the contribution of each strategic node within the subgroup using nodal weights. Shortly, we will show that, in the strategic case, this node-based formulation is equivalent to the expectation-based bound.

\paragraph{Properties}
The following proposition establishes the equivalence of the node-based and expectation-based formulations for the strategic scenario tree dissection. As a consequence, the resulting bounds also form a monotonic chain as the number of strategic nodes within each subgroup increases.

\begin{myprop} \label{prop:Nodal_MHESG}
    Let $\mathcal{S}_{\hat{f}}$ denote the subset of fixed strategic scenarios. For each strategic subgroup $i \in [\hat{m}_{\ell}]$ at level $\ell$, define $\hat{\phi}_{n,i}^{(\ell)}$ as the weight associated with its strategic node $n$, given by\vspace*{-0.1in}
 \begin{align*}
     \hat{\phi}_{n,i}^{(\ell)}= \frac{
\phi_i^{(\ell)}\sum_{\mathbf{s} \in \mathcal{S}^n \cap \mathcal{S}_{\hat{f}}}
\pi^{\mathbf{s}} 
+
\sum_{\mathbf{s} \in \mathcal{S}^n \cap (\mathcal{S}_i^{(\ell)} \setminus \mathcal{S}_{\hat{f}}) } \pi^{\mathbf{s}}}{\sum_{\mathbf{s} \in \mathcal{S}^n} \pi^{\mathbf{s}}}.\vspace*{-0.05in}
 \end{align*}
Furthermore, for $i \in [\hat{m}_{\ell}]$, let $(\hat{x}_{n, i},\hat{y}_{n, i}^{\omega,\tau})$ 
be the optimal solution of the problem $\textnormal{MHSG}(\mathcal{N}_i^{(\ell)})$ given in \eqref{MHSG_node}. Then, the following holds true:\vspace*{-0.1in}
    \begin{align} \label{node_MHSEG}
  \textit{MHESG}(\ell) =  
  \sum_{i=1}^{\hat{m}_\ell} \phi_i^{(\ell)} \textit{MHSG}(\mathcal{N}_i^{(\ell)})  = \sum\limits_{t \in \mathcal{H}}\sum\limits_{n \in \mathcal{N}_t} \sum_{i=1}^{\hat{m}_\ell} 
\hat{\phi}_{n,i}^{(\ell)} \pi_n 
\big( c_n \hat{x}_{n,i} + \sum\limits_{\omega \in \Omega_t} \pi_{\omega} \sum\limits_{\tau \in \mathcal{T}_t}  q_{n}^{\omega,\tau} \hat{y}_{n,i}^{\omega,\tau} \big).  
  \end{align}
\end{myprop}
\vspace*{-0.1in}

\proof{
Using \eqref{MHSG_node}, the first equality in  \eqref{node_MHSEG}, and noting $ \mathcal{N}_i^{(\ell)} = \bigcup_{t \in \mathcal{H}} \mathcal{N}_{i,t}^{(\ell)}$, we obtain\vspace*{-0.15in}
\begin{align}
  \textit{MHESG}(\ell) 
 : & =  
  \sum_{i=1}^{\hat{m}_\ell} \phi_i^{(\ell)} 
\sum\limits_{n \in \mathcal{N}_i^{(\ell)}}  \hat{\Pi}_{i,n}^{(\ell)} \big( c_n \hat{x}_{n,i} + \sum\limits_{\omega \in \Omega_t} \pi_{\omega} \sum\limits_{\tau \in \mathcal{T}_t}  q_{n}^{\omega,\tau} \hat{y}_{n,i}^{\omega,\tau} \big) \nonumber \\
& =
\sum_{i=1}^{\hat{m}_\ell}  
\sum\limits_{n \in \mathcal{N}_i^{(\ell)}}  
\frac{\phi_i^{(\ell)}\hat{\Pi}_{i,n}^{(\ell)}}{\pi_n} 
\pi_n 
\big( c_n \hat{x}_{n,i} + \sum\limits_{\omega \in \Omega_t} \pi_{\omega} \sum\limits_{\tau \in \mathcal{T}_t}  q_{n}^{\omega,\tau} \hat{y}_{n,i}^{\omega,\tau} \big) \nonumber\\
& =
\sum_{i=1}^{\hat{m}_\ell}  
\sum\limits_{n \in \mathcal{N}_i^{(\ell)}}  
\frac{\phi_i^{(\ell)}
\sum_{\mathbf{s} \in  \mathcal{S}^n \cap \mathcal{S}_i^{(\ell)}}\hat{\Pi}_{i,\mathbf{s}}^{(\ell)}}
{\sum_{\mathbf{s} \in  \mathcal{S}^n} \pi^{\mathbf{s}}} 
\pi_n 
\big( c_n \hat{x}_{n,i} + \sum\limits_{\omega \in \Omega_t} \pi_{\omega} \sum\limits_{\tau \in \mathcal{T}_t}  q_{n}^{\omega,\tau} \hat{y}_{n,i}^{\omega,\tau} \big) \nonumber\\
& =
\sum_{i=1}^{\hat{m}_\ell}  
\sum\limits_{n \in \mathcal{N}_i^{(\ell)}}  
\frac{\phi_i^{(\ell)} \!
\Big( \! 
\sum_{\mathbf{s} \in  \mathcal{S}^n \cap \mathcal{S}_{\VV{\hat{f}}}   \! \!  }
\! 
\hat{\Pi}_{i,\mathbf{s}}^{(\ell)} \! \!
+
\! \!
\sum_{\mathbf{s} \in \mathcal{S}^n \cap (\mathcal{S}_i^{(\ell)} \setminus  \mathcal{S}_{\VV{\hat{f}}})}
\!
\hat{\Pi}_{i,\mathbf{s}}^{(\ell)}
\!
\Big)
}
{\sum_{\mathbf{s} \in  \mathcal{S}^n} \pi^{\mathbf{s}}} 
\!
\pi_n
\big( c_n \hat{x}_{n,i} + \sum\limits_{\omega \in \Omega_t} \pi_{\omega} 
\! 
\sum\limits_{\tau \in \mathcal{T}_t}  
\! \!
q_{n}^{\omega,\tau} \hat{y}_{n,i}^{\omega,\tau} \big) \nonumber\\
& =
\sum_{i=1}^{\hat{m}_\ell}  
\sum\limits_{n \in \mathcal{N}_i^{(\ell)}}  
\frac{
\phi_i^{(\ell)} 
\sum_{\mathbf{s} \in  \mathcal{S}^n \cap \mathcal{S}_{\hat{f}}} \pi^{\mathbf{s}}
+ 
\sum_{\mathbf{s} \in \mathcal{S}^n \cap (\mathcal{S}_i^{(\ell)} \setminus \mathcal{S}_{\hat{f}} )}
\pi^{\mathbf{s}}
}
{\sum_{\mathbf{s} \in  \mathcal{S}^n} \pi^{\mathbf{s}}} 
\pi_n 
\big( c_n \hat{x}_{n,i} + \sum\limits_{\omega \in \Omega_t} \pi_{\omega} \sum\limits_{\tau \in \mathcal{T}_t}  q_{n}^{\omega,\tau} \hat{y}_{n,i}^{\omega,\tau} \big) \nonumber\\
& =
\sum_{i=1}^{\hat{m}_\ell} 
\sum\limits_{n \in \mathcal{N}_i^{(\ell)}}  
\hat{\phi}_{n,i}^{(\ell)} \pi_n 
\big( c_n \hat{x}_{n,i} + \sum\limits_{\omega \in \Omega_t} \pi_{\omega} \sum\limits_{\tau \in \mathcal{T}_t}  q_{n}^{\omega,\tau} \hat{y}_{n,i}^{\omega,\tau} \big). \tag*{\qed }   
\end{align}
}

Computing $\textit{MHESG}(\ell)$ by assigning an overall probability to each strategic subgroup (see \eqref{MHESG_def}) is equivalent to calculating the lower bound by weighting the cost of each strategic node \( n \in \mathcal{N} \) within subgroup \( i \), for \( i \in  [\hat{m}_\ell] \), using the nodal weights \( \hat{\phi}_{n,i}^{(\ell)} \), as done for the operational counterpart in \eqref{MHNOG_def}.  
Compare the rightmost expression in \eqref{node_MHSEG} with \eqref{MHNOG_def}.
Specifically, the equivalence between the two formulations follows from the equivalence between the terms $\phi_{i}^{(\ell)} \hat{\Pi}_{i,n}^{(\ell)}$ and $\hat{\phi}_{n,i}^{(\ell)} \pi_n$.
We illustrate this further by Example \ref{ex:stra_nodal} in \ref{Appendix:SSD}, where we also summarize the notation.

\subsection{Simultaneous Strategic Scenario Tree and Operational Subtree Dissection}
A third approach dissects {\it both} the strategic scenario tree and the operational subtree.  In such an approach, we consider groups of strategic and operational scenarios and derive lower bounds to the optimal value \textit{MHSP} by solving several subproblems, each involving a reduced number of strategic and operational scenarios.   
\vspace*{-0.07in}

\subsubsection{Expectation-Based Combination}\label{sec:Str_op_exp_dissec}
\paragraph{Dissection}
We construct a collection of subgroups by dissecting both strategic and operational supports with the corresponding refinement chains described as in \eqref{eq:ref_chain} and \eqref{refinement_chain_operational_common_prob}, respectively. 
Let $\mathcal{S}_i^{(\ell)}$ be an arbitrarily chosen strategic subgroup at level $\ell$ with nodes $\mathcal{N}_i^{(\ell)}$. Let $\Omega_{\text{OG},k}^{(\ell,\gamma)}$ denote its corresponding operational subgroup, with operational subset $\Omega_{n,j}^{(\ell, \gamma)}$ at each strategic node $n$.  We denote this subgroup by $ (\mathcal{N}_i^{(\ell)}, \Omega_{\text{OG},k}^{(\ell,\gamma)})$ and its associated probability  by $\phi_{G,i,k}^{(\ell,\gamma)}$, defined as \vspace*{-0.05in} 
\begin{equation}
    \label{eq:SO_weight}
    \phi_{G,i,k}^{(\ell,\gamma)} = \phi_i^{(\ell)} \prod_{n \in \mathcal{N}_i^{(\ell)}} \phi_{n,j}^{(\gamma)},\vspace*{-0.05in} 
\end{equation}
where $\phi_i^{(\ell)}$ and $\phi_{n,j}^{(\gamma)}$ are defined in \eqref{eq:stra_weights} and \eqref{eq:prob_node_op_sc}, respectively.
The probabilities of the associated strategic and operational scenarios are defined similarly to those in Sections \ref{sec:dissecting_stra} and \ref{sssec:diss_oper}, respectively. For a subgroup at a given operational level, a strategic refinement to a level $\ell$ generates $\hat{m}_{\ell}$  subgroups, while for a subgroup at a fixed strategic level $\ell$ with $\vert \mathcal{N}_i^{(\ell)} \vert$ nodes, an operational refinement to a level $\gamma$ generates $M_{i,\ell,\gamma} = m_{\gamma}^{\vert \mathcal{N}_i^{(\ell)} \vert}$ subgroups. Thus, dissecting a subgroup at both strategic and operational components to levels $\ell$ and $\gamma$, respectively, generates $M_{\ell, \gamma} = \sum_{i=1}^{\hat{m}_{\ell}} m_{\gamma}^{\vert \mathcal{N}_i^{(\ell)} \vert } $ subgroups. 
For brevity, an illustrative example of this dissection is provided in
Example \ref{ex:joint_dissection} in \ref{Appendix:Simultaneous_dissection},  where we also summarize the notation.

 \paragraph{Subproblems \& Bound}
 Given the above simultaneous dissection of the strategic scenario tree and the operational subtree, we define the corresponding subproblems as follows.
\begin{mydef}\label{def:MHG}
The \textit{Multi-Horizon Group subproblem}  with strategic index $i \in [\hat{m}_\ell]$ at level $\ell$ and operational index $k \in [M_{i,\ell,\gamma} ]$ at level $\gamma$, denoted by $\text{MHG}\big(\mathcal{N}_i^{(\ell)},\Omega_{\text{OG},k}^{(\ell, \gamma)}\big)$, is defined as\vspace*{-0.05in} 
\begin{align*} 
\textit{MHG}\big(\mathcal{N}_i^{(\ell)},\Omega_{\text{OG},k}^{(\ell, \gamma)}\big)  := 
& \min \limits_{\mathbf{x},\mathbf{y}} 
\sum_{t\in\mathcal{H}}
\sum\limits_{n \in \mathcal{N}_i^{(\ell)} \cap \mathcal{N}_t} \hat{\Pi}_{i,n}^{(\ell)} \Big( c_n x_n + \sum\limits_{\omega \in \Omega_{n,j}^{(\ell,\gamma)}} \hat{\pi}_{\omega}^{(\gamma)} \sum\limits_{\tau \in \mathcal{T}_t}  q_{n}^{\omega,\tau} y_{n}^{\omega,\tau} \Big) 
\nonumber  \qquad \\
& \textrm{ s.t. }  Ax_{n}=h_{n},\  n \in  \mathcal{N}_0, \nonumber \\
& \qquad \ T_{n} x_{a({n})} + W_{n} x_{n}=h_{n},\   n \in  \mathcal{N}_i^{(\ell)}\setminus\{\mathcal{N}_0\},\nonumber \\
& \qquad \ T_{n}^{\omega,1} x_{n} + W_{n}^{\omega,1} y_{n}^{\omega,1} = h_{n}^{\omega,1},\   n \in \mathcal{N}_i^{(\ell)},\  \omega \in  \Omega_{n,j}^{(\ell,\gamma)},
\nonumber\\
& \qquad \ T_{n}^{\omega,\tau} y_{n}^{\omega,\tau -1} + W_{n}^{\omega,\tau} y_{n}^{\omega,\tau} = h_{n}^{\omega,\tau},\   n \in \mathcal{N}_i^{(\ell)}\cap\mathcal{N}_t, \tau \in \mathcal{T}_t\setminus\left\{1\right\}, \nonumber \\
& \hspace{2in}
\qquad \omega \in  \Omega_{n,j}^{(\ell,\gamma)}, t\in\mathcal{H}.
\end{align*}
\end{mydef}
\noindent
Combining the objective function values of all subproblems  $\textit{MHG}\big(\mathcal{N}_i^{(\ell)},\Omega_{\text{OG},k}^{(\ell, \gamma)}\big)$, $k \in [M_{i,\ell,\gamma} ], i \in [\hat{m}_\ell]$, with their associated probabilistic weights $\phi_{G,i,k}^{(\ell,\gamma)}$  allows us to introduce the following quantity. 
\begin{mydef}
\label{def:MHEG}
The  \textit{Multi-Horizon Expected Group bound}, denoted $\textit{MHEG}(\ell, \gamma)$, is defined as 
\vspace*{-0.1in}
\begin{align*}
    \textit{MHEG}(\ell, \gamma)= \sum_{i \in [\hat{m}_\ell]} \sum_{k \in [M_{i,\ell,\gamma} ]} \phi_{G,i,k}^{(\ell,\gamma)}  \textit{MHG}\big(\mathcal{N}_i^{(\ell)},\Omega_{\text{OG},k}^{(\ell, \gamma)}\big).
\end{align*}
\end{mydef}
\paragraph{Properties}  
We show that $\textit{MHEG}(\ell, \gamma)$ leads to a monotonic chain of lower bounds to the optimal value \textit{MHSP}  as $\ell$ and $\gamma$ increase. For brevity, proof of the below result is presented in \ref{Proofs}.
\begin{myprop} \label{Prop_MHEG}  
Consider \textit{MHSP} given in \eqref{MHSP1}.
The following chain of inequalities holds true:\vspace*{-0.1in}
\begin{enumerate}[label={(\alph*)}]
\item      $\textit{MHEG}(1,\gamma) \leq \textit{MHEG}(2,\gamma) \leq \ldots  \leq \textit{MHEG}(\ell,\gamma)  \leq \ldots \leq \textit{MHEOG}(\gamma)$, for a fixed operational level $\gamma$,\vspace*{-0.03in} 
\item      $\textit{MHEG}(\ell,1) \leq \textit{MHEG}(\ell,2) \leq \ldots  \leq \textit{MHEG}(\ell,\gamma)  \leq \ldots \leq \textit{MHESG}(\ell)$, for a fixed strategic level $\ell$.
\end{enumerate}
\end{myprop}

\subsubsection{Node-Based Combination}
For node-based combination of simultaneous dissection, we assign to each subproblem, a weight derived from the probability 
 of its strategic scenarios, and to each of its strategic node, a weight derived from the probability  of its operational scenarios. The objective function contributions of the optimal solutions of subgroups are then aggregated into a weighted sum, yielding a lower bound to  the optimal value \textit{MHSP}. Once again, we show that these bounds lead to a monotonic chain of inequalities as the number of strategic and operational scenarios in each subgroup grows.  As in \Cref{sec:dissecting_oper2}, this novel combination approach enables us to significantly reduce the number of subgroups formed from the multi-horizon scenario tree.

\paragraph{Dissection}
We construct a collection of subgroups by dissecting both strategic and operational supports with properties as in \Cref{sec:Str_op_exp_dissec}, where the operational scenarios are dissected according to \Cref{sec:dissecting_oper2}. For a subgroup at a given operational level, a strategic refinement to level $\ell$ generates $\hat{m}_{\ell}$  subgroups, while for a subgroup
 at a fixed strategic level, an operational refinement to level $\gamma$ generates $m_{\gamma}$ subgroups. Thus,  in this case, dissecting a subgroup at both strategic and operational supports to levels $\ell$ and $\gamma$, respectively, generates $\hat{m}_{\ell} m_{\gamma}$ subgroups. 

For an arbitrary strategic subgroup $S_i^{(\ell)}$, \VV{$i \in [\hat{m}_\ell]$,} we recall the $\gamma$-level operational subgroup as the collection of the operational subgroups derived for $S_i^{(\ell)}$ as follows:
$\Omega_{\text{OG},j}^{(\ell,\gamma)}:=\{\Omega_{n,j}^{(\ell, \gamma)}\}_{n\in \mathcal{N}_i^{(\ell)}}$, $j \in [m_\gamma]$,
satisfying the properties \ref{OG_subset_number}--\ref{item:refinement} described in \Cref{sssec:diss_oper}.
 To illustrate this construction, we introduce the following example.
 
\begin{myex}\label{Ex:Node_Str}
Consider again the multi-horizon scenario tree  $\mathfrak{T}(7)$ at level $\ell=3$ from Figure \ref{Fig:Common_Tree}. In Figure \ref{Tree_MHEG(2,2)}, going from strategic level $\ell=3$ to $\ell-1=2$, the four strategic scenarios are split into two subgroups containing strategic nodes $\mathcal{N}_1^{(2)}=\{1,2,4,5\}$ and $\mathcal{N}_2^{(2)}=\{1,3,6,7\}$. 
Now, going from operational level $\gamma=3$ to $\gamma-1=2$,  the operational scenarios at each node of  $\mathcal{N}_1^{(2)}$ are partitioned into two subgroups $ (\mathcal{N}_1^{(2)}, \Omega_{\text{OG},1}^{(2,2)})$ and $ (\mathcal{N}_1^{(2)}, \Omega_{\text{OG},2}^{(2,2)})$; similarly for $\mathcal{N}_2^{(2)}$.
Each strategic node of $\mathcal{N}_1^{(2)}$ and $\mathcal{N}_2^{(2)}$ now contains either the operational subset $\{ 
\textnormal{AB,AC}\}$ or $\{ 
\textnormal{AD,AE} \}$. Overall, this creates $\hat{m}_{2}\times m_{2}=2\times2=4$ smaller multi-horizon problems. The optimal solutions of these
subproblems can be computed independently and then combined to derive a lower bound. Unlike the expectation-based combination, which assigns a probabilistic weight to each subgroup (e.g., $\phi_{G,1,1}^{(2,2)} = \frac{1}{3}(\frac{2}{5} \times \frac{2}{6} \times \frac{3}{5} \times \frac{2}{4})$), the node-based combination instead assigns nodal weights, avoiding the need to enumerate all operational subgroup combinations and thereby significantly reducing the number of subgroups to be solved. 
\end{myex}
\vspace*{-0.1in}
\input{Tree_MHEG_Fig_with_node_weights}

\paragraph{Subproblems \& Bound}
Given the above joint dissection of the strategic and operational scenarios, the corresponding subproblems remain the same as in Definition \ref{def:MHG}. 
However, here, we combine the contributions to the objective function using the nodal weights, as defined below.
\begin{mydef}
Let 
$(\hat{x}_{n,{i,j}},\hat{y}_{n,{i,j}}^{\omega,\tau}) $ 
be the optimal solution of the multi-horizon group subproblem $\text{MHG}\big(\mathcal{N}_i^{(\ell)},\Omega_{\text{OG},j}^{(\ell, \gamma)}\big)$ {for $i \in [\hat{m}_\ell]$ and $j \in [m_\gamma]$}. 
The \textit{Multi-Horizon Nodal Group bound}, denoted  $\textit{MHNG}(\ell,\gamma)$, is defined as\vspace*{-0.05in} 
    \begin{align*} 
  \textit{MHNG}(\ell,\gamma) 
  \!:= \!\!\sum_{i=1}^{\hat{m}_{\ell}} \phi^{(\ell)}_{i}
    \Big(
\sum_{t\in\mathcal{H}}
\sum\limits_{n \in \mathcal{N}_i^{(\ell)}\cap\mathcal{N}_t}
  \sum_{j=1}^{m_{\gamma}}
  \phi_{n,j}^{(\gamma)}
  \hat{\Pi}_{i,n}^{(\ell)} \big( c_n \hat{x}_{n,{i,j}} \!+\! \!\!\sum\limits_{\omega \in \Omega_{n,j}^{(\ell,\gamma)}} \hat{\pi}_{\omega}^{(\gamma)} \sum\limits_{\tau \in \mathcal{T}_t}  q_{n}^{\omega,\tau} \hat{y}_{n,{i,j}}^{\omega,\tau} \big)
 \Big).  
  \end{align*}
\end{mydef}

Note once again the differences between the node- and expectation-based bounds. In the expectation-based bound presented in \Cref{def:MHEG}, multiplication of $\phi_i^{(\ell)}$ and $\phi_{n,j}^{(\gamma)}$ determine the subproblem weights $\phi_{G,i,k}^{(\ell,\gamma)}$ (recall \eqref{eq:SO_weight}), which are then employed to compute an expectation. In contrast, in the node-based bound $\textit{MHNG}(\ell,\gamma)$, the coefficients $\phi_{n,j}^{(\gamma)}$ serve as nodal weights that aggregate the objective value contributions from each subproblem's optimal solutions, which are then further combined using the strategic weights $\phi_i^{(\ell)}$.

\allowdisplaybreaks

\paragraph{Properties}
We show that $\textit{MHNG}(\ell, \gamma)$  leads to a monotonic chain of lower bounds to the optimal value \textit{MHSP}  as $\ell$ and $\gamma$ increase. The proof is presented in \ref{Proofs} for brevity. 
\begin{myprop} \label{Prop_MHNG}
Consider \textit{MHSP} given in \eqref{MHSP1}.
The following chain of inequalities holds true:\vspace*{-0.05in}
\begin{enumerate}[label={(\alph*)}]
\item      $\textit{MHNG}(1,\gamma) \leq \textit{MHNG}(2,\gamma) \leq \ldots  \leq \textit{MHNG}(\ell,\gamma)  \leq \ldots \leq \textit{MHNOG}(\gamma)$,  for a fixed operational level $\gamma$,\vspace*{-0.03in} 
\item     $\textit{MHNG}(\ell,1) \leq \textit{MHNG}(\ell,2) \leq \ldots  \leq \textit{MHNG}(\ell,\gamma)  \leq \ldots \leq \textit{MHESG}(\ell)$,  for a fixed strategic level $\ell$.
\end{enumerate}
\end{myprop}
\vspace*{-0.15in}

\section{Numerical Results}
\label{sec:numres}
\vspace*{-0.05in}

In this section, we present an extensive set of computational experiments to assess the proposed methodology.
All tests were performed on an ASUS laptop equipped with a 3.0 GHz Intel Core i7-5500U processor and 8 GB of RAM, using the GAMS 24.7.4 environment with the Gurobi solver.
Each run was limited to six hours.
Section~\ref{sec:model} describes our multi-horizon investment planning model for electricity generation and transmission, which serves as the main framework for testing.
Section~\ref{sec:data} summarizes the main features of the case study used to generate the experimental instances.
Finally, Section~\ref{sec:LB_groups} presents the resulting lower bounds.

\subsection{Problem Description}
\label{sec:model}

In this section, we introduce a multi-horizon stochastic optimization problem designed to test the bounds developed in the previous sections.
The formulation is inspired by the \textit{Generation and Transmission Expansion Planning} (GTEP) model of \cite{Micheli2020two},
which is a two-stage model that combines long-term investment and short-term operational decisions. We extend it to a multi-horizon setting, where both the strategic and operational components have a multistage structure.
The problem aims to determine optimal investment decisions in generation, transmission, and storage facilities to satisfy electricity demand at minimum total cost while complying with decarbonization targets.

The power system consists of zones interconnected by transmission lines with limited transfer capacity.
The generation system includes power plants without storage capabilities (e.g., gas-fired, coal-fired, nuclear, solar, and wind), each characterized by an initial installed capacity. Renewable power plants within this set are non-programmable.
Each generating unit is defined by an emission coefficient and hourly capacity factors, representing the maximum fraction of installed capacity converted into energy output. 
Capacity factors equal one for programmable plants and range between zero and one for renewable technologies, reflecting the intermittent availability of solar and wind resources.
The storage system includes technologies (e.g., hydro pumped storage, lithium-ion batteries, sodium-ion batteries), characterized by initial installed capacity, energy-to-power ratio, hourly inflows, and charging/discharging loss factors.

Given stochastic investment and operating costs, the goal is to determine the expansion plan and operational strategy of the power system that minimize the expected total cost.
These decisions must satisfy zonal demand, capacity limits, and decarbonization constraints.
Policy requirements include a maximum daily limit on CO$_2$ emissions at each strategic stage and a minimum share of total generation capacity allocated to non-programmable renewable technologies.

The multi-horizon nature of the problem is as follows. 
At the strategic level, the model determines investments in generation, transmission, and storage capacity.
At the operational level, it optimizes generation dispatch, power flows, and the charging, discharging, and energy levels of storage facilities.
The mathematical formulation of the problem is provided in~\ref{sec:formulation}.

\subsection{The Data}
\label{sec:data}

The case study is inspired by the Italian power system.
The national territory is divided into seven geographical zones: North, Central-North, Central-South, South, Calabria, Sicily, and Sardinia.
These zones are interconnected by eight high-voltage transmission lines representing the main interzonal exchange corridors.
Each zone hosts four generation technologies (combined-cycle gas turbines, open-cycle gas turbines, photovoltaic power plants, and wind power plants), and three storage technologies (pumped-hydro storage, lithium-ion batteries, and sodium-ion batteries).
Technical parameters are taken from \cite{Micheli2020two}.

To assess the performance of the proposed bounding techniques across different problem sizes, we consider three types of multi-horizon scenario trees:
\begin{itemize}
\item \textit{Small Instance Type}.
This configuration includes a strategic scenario tree with two stages and 3 nodes.
Each strategic node is associated with an operational three-stage scenario tree comprising eight hourly periods per stage, resulting in a total of 64 operational scenarios.
\item \textit{Medium Instance Type}.
This instance features a strategic tree with three stages, including two nodes in the second stage and eight in the third, for a total of 11 strategic nodes.
Each strategic node is linked to a three-stage operational scenario tree identical to the one in the small instance, with eight hourly periods per stage and 64 operational scenarios. 
\item \textit{Large Instance Type}.
The large configuration consists of a five-stage strategic scenario tree, where each non-leaf node has two child nodes, resulting in 31 strategic nodes.
As before, each strategic node is associated with a three-stage operational scenario tree with eight hourly periods per stage and 64 operational scenarios.
\end{itemize}
Table~\ref{tab:size} summarizes the dimensions and solution times of the resulting MHSPs.
The large instance cannot be solved directly due to memory limitations. Classical expectation-based lower bounds (based on the replacement of the stochastic processes by their expectations) are given in  \ref{sec:LB_exp}.
Numerical results reveal the presence of very large optimality gaps, underscoring the limited effectiveness of classical approaches and motivating the development of more accurate bounding strategies, such as the dissection-based methods proposed in this work.

\begin{table}[ht!]
\footnotesize
    \centering
    \begin{tabular}{l|rrr}
    \hline
         Instance type&  Small&  Medium& Large\\
         \hline
         Strategic stages&  2&  3& 5\\
         Strategic nodes&  3&  11& 31\\
         Strategic scenarios&  2&  8& 16\\
         Operational scenarios&  64&  64& 64\\
         Number of constraints&  626,299&  2,298,517& 6,487,639\\
         Number of variables&  450,449&  1,653,148& 4,659,871\\
         Optimal objective function value (bln \euro) & 280.20 & 290.73 & --\\
         CPU time (s)&  4,635&  17,428& -- \\ \hline
    \end{tabular}
    \caption{Problem dimensions and solution times for the three instance types.}
    \label{tab:size}
\end{table}
\vspace*{-0.15in}

\subsection{Lower Bounds via Scenario Tree Dissection}
\label{sec:LB_groups}

This section reports the results for the lower bounds derived through the dissection of subproblem scenario trees.  
We first analyze the small instance in Section~\ref{sec:small_inst} followed by the large instance in Section~\ref{sec:large_inst}. Results on the medium instance are reported in \ref{sec:med_inst}.

\subsubsection{Results for the Small Instance}
\label{sec:small_inst}

We consider two variants of the operational refinement chain: disjoint groups and one fixed operational scenario across all levels.
This analysis enables us to quantify the effect of scenario fixing on bound accuracy. 
We consider the small instance with 64 operational scenarios at each of the 3 strategic nodes to construct an operational refinement chain with 7 levels. 

\textbf{Disjoint Strategic and Operational Subgroups.} At each level $\gamma$, the groups are disjoint and formed by merging pairs of groups from the previous level $\gamma-1$, resulting in group sizes $g_{\gamma}$ ranging from~1 at $\gamma=1$ to 64 at $\gamma=7$.
Table~\ref{tab:small_combined_aligned} reports the results for bounds $\textit{MHEG}(\ell,\gamma)$ and $\textit{MHNG}(\ell,\gamma)$ across different combinations of the strategic ($\ell$) and operational ($\gamma$) refinement levels.  
$\textit{MHEG}(2,7)$ corresponds to the original MHSP. Optimality gaps are reported relative to this model.  
CPU times (s) and speed-ups relative to the full MHSP are also reported.

\begin{table}[ht!]
\footnotesize
\centering
\setlength{\tabcolsep}{3pt} 
\begin{tabular}{lcccccccccccc}
\hline
\multirow{2}{*}{\makecell{Bounding\\scheme}} & \multirow{2}{*}{$\gamma$} & \multirow{2}{*}{$m_\gamma$} & \multirow{2}{*}{$g_\gamma$} &
\multicolumn{4}{c}{$\ell=2$ ($\hat{m}_{\ell}=1$)} & &
\multicolumn{4}{c}{$\ell=1$ ($\hat{m}_{\ell}=2$)} \\
\cline{5-8} \cline{10-13}
 & & & &
Bound & Gap (\%) & CPU (s) & Speed-up & & 
Bound & Gap (\%) & CPU (s) & Speed-up \\
\hline
\multirow{7}{*}{$\textit{MHEG}(\ell,\gamma)$}
 & 7 & 1 & 64 & 280.20 & -- & 4635 & -- & & 274.58 & -2.01 & 407 & 11.39 \\
 & 6 & 8 & 32 & 280.02 & -0.06 & 1944 & 2.38 & & 274.41 & -2.07 & 1512 & 3.06 \\
 & 5 & 64 & 16 & 279.79 & -0.15 & 1784 & 2.60 & & 274.18 & -2.15 & 1124 & 4.12 \\
 & 4 & 512 & 8 & 279.07 & -0.40 & 2705 & 1.71 & & 273.43 & -2.42 & 1905 & 2.43 \\
 & 3 & 4096 & 4 & 276.81 & -1.21 & 6581 & 0.70 & & 271.37 & -3.15 & 6488 & 0.71 \\
 & 2 & 32768 & 2 & -- & -- & -- & -- & & -- & -- & -- & -- \\
 & 1 & 262144 & 1 & -- & -- & -- & -- & & -- & -- & -- & -- \\
\hline
\multirow{6}{*}{$\textit{MHNG}(\ell,\gamma)$}
 & 6 & 2 & 32 & 279.95 & -0.09 & 518 & 8.95 & & 274.34 & -2.09 & 403 & 11.50 \\
 & 5 & 4 & 16 & 279.74 & -0.16 & 346 & 13.40 & & 274.09 & -2.18 & 218 & 21.26 \\
 & 4 & 8 & 8 & 278.44 & -0.63 & 201 & 23.06 & & 272.96 & -2.58 & 112 & 41.38 \\
 & 3 & 16 & 4 & 275.03 & -1.85 & 71 & 65.28 & & 269.33 & -3.88 & 70 & 66.21 \\
 & 2 & 32 & 2 & 269.46 & -3.83 & 62 & 74.76 & & 263.88 & -5.82 & 64 & 72.42 \\
 & 1 & 64 & 1 & 263.12 & -6.10 & 52 & 89.13 & & 257.38 & -8.14 & 73 & 63.49 \\
\hline
\end{tabular}
\caption{Results for bounding schemes $\textit{MHEG}(\ell,\gamma)$ and $\textit{MHNG}(\ell,\gamma)$ at varying levels of the strategic ($\ell$) and operational ($\gamma$) refinement chains, with disjoint strategic and operational subgroups in the small instance.  
Each block reports: \textit{bound (bln~\euro), optimality gap (\%), CPU time (s), and speed-up} relative to the full multi-horizon stochastic optimization model.}
\label{tab:small_combined_aligned}
\end{table}

As shown in Table~\ref{tab:small_combined_aligned}, the expectation-based combination $\textit{MHEG}(\ell,\gamma)$ provides very tight lower bounds at operational levels $3\leq\gamma\leq 6$.
For instance, at operational level~$\gamma=6$ and strategic level $\ell=2$, the bound equals 280.02~bln~\euro{}, only 0.06\% below the full multi-horizon optimum of 280.20~bln~\euro{}, while reducing the solution time from 4635~s to 1944~s (a speed-up of about~2.4).
At $\gamma=5$, the bound remains within 0.15\% of the optimum with a speed up of about 2.6.
Even when the strategic uncertainty is dissected in two groups (case $\ell=1$), the gaps remain small (around 2\%, with speed-ups between 3 and 4).
However, for lower refinement levels, the rapidly growing number of subgroups makes this approach computationally demanding and infeasible for $\gamma\leq2$.

Conversely, the node-based combination $\textit{MHNG
}(\ell,\gamma)$ achieves slightly looser bounds but with significantly lower computation times.
At level $\gamma=6$ and $\ell=2$, the bound is 279.95 bln \euro{} (with a gap 0.09\%), requiring only 518 s, almost a 9 times speed-up relative to the full model.
At lower levels ($\gamma=4$--$3$), the gap increases modestly (to 0.63\% and 1.85\%, respectively), while the solution time decreases sharply to just 201 s and 71 s (speed-ups of 23 and 65).
For the most dissected  configuration ($\gamma=1$), the computation time drops to about one minute, achieving an 89 times speed-up, with a corresponding bound of 263.12 bln \euro{} (gap $6.1\%$).

Overall, the $\textit{MHEG}$ approach delivers the most accurate estimates of the optimal value but at a higher computational cost, while the $\textit{MHNG}$ approach provides a good trade-off between accuracy and efficiency.
In particular, the node-based method can solve all levels of the operational refinement chain with speed-ups up to two orders of magnitude, while maintaining optimality gaps below 2\% for moderately refined levels ($\gamma\geq 3$).

\textbf{One Fixed Operational Scenario.}  To evaluate the effect of fixing operational scenarios on the quality of the lower bounds, we now analyze an alternative operational refinement chain in which one scenario remains fixed across all operational groups at each level.
The chain consists of four levels: at $\gamma=1$, 63 groups of cardinality~2 (one fixed and one variable scenario) are created; at $\gamma=2$, 9 groups of cardinality~8; at $\gamma=3$, 3 groups of cardinality~22; and at $\gamma=4$, a single group corresponding to the full operational subtree.
Three fixed-scenario selection strategies are tested: the scenario leading to the highest cost, the lowest cost, and the median cost.

Table~\ref{tab:oper_fix} lists the results of these experiments. Results at $\gamma=4$ coincide with those reported in Table~\ref{tab:small_combined_aligned} and are therefore omitted.
For the expectation-based approach, results at $\gamma=1$ are not reported due to the prohibitive number of subgroups (250,047).  Table~\ref{tab:oper_fix} reveals that the influence of the fixed scenario choice on the resulting bounds is limited for this problem.
For the expectation-based formulation ($\text{MHEG}$), the best results are obtained at $\gamma=3$ with bounds of 279.91~bln~\euro{} ($-0.10\%$ gap) for $\ell=2$ and 275.21 bln~\euro{} ($-1.81\%$ gap) for $\ell=1$ when fixing the highest-cost scenario.
At a lower level of operational aggregation ($\gamma=2$), the corresponding bounds slightly deteriorate to 278.62~bln~\euro{} ($-0.57\%$) and 273.12~bln~\euro{} ($-2.59\%$), respectively.
The results for the median-  and lowest-cost scenarios are nearly identical, confirming that the specific choice of the fixed scenario has only a marginal impact on the approximation quality.

A similar trend is observed for the node-based formulation ($\text{MHNG}$).
At $\gamma=3$, the median scenario yields a bound of 279.78~bln~\euro{} ($-0.15\%$) for $\ell=2$ and 274.16~bln~\euro{} ($-2.20\%$) for $\ell=1$.
When moving to $\gamma=2$, the bounds decrease slightly to 277.32~bln~\euro{} ($-1.03\%$) and 271.80~bln~\euro{} ($-3.09\%$).
At the most dissected level ($\gamma=1$), the node-based model achieves  bounds of 263.86~bln~\euro{} ($-6.19\%$) and 258.11~bln~\euro{} ($-8.56\%$), respectively.

Comparing these results with those obtained without fixed scenarios shows that the presence of a fixed scenario generally produces slightly looser bounds.
For instance, at $\gamma=2$, $\textit{MHEG}$ with one fixed and seven variable scenarios gives bounds of 278.62~bln~\euro{} and 273.12~bln~\euro{} (for $\ell=2$ and $\ell=1$), compared to 279.07~bln~\euro{} and 273.43~bln~\euro{} in the corresponding fully variable setting.

Overall, maintaining full heterogeneity in the operational scenario groups, rather than fixing one scenario, tends to yield tighter lower bounds, confirming that scenario diversity plays a positive role in accurately capturing the stochastic structure of the studied problem.

\begin{table}[ht!]
\footnotesize
\centering
\setlength{\tabcolsep}{3pt} 
\begin{tabular}{l|lcrr|rrrr}
\hline
 &  &  &  &  & \multicolumn{4}{c}{ $\ell$} \\
Bounding scheme & Fixed Scenario & $\gamma$ & $m_\gamma$ & $g_\gamma$ &
\multicolumn{2}{c}{2 ($\hat{m}_{\ell}=1$)} & \multicolumn{2}{c}{1 ($\hat{m}_{\ell}=2$)} \\
\cline{6-9}
 & & & & & Bound & Gap (\%) & Bound & Gap (\%) \\
\hline
\multirow{6}{*}{$\textit{MHEG}(\ell,\gamma)$} 
 & \multirow{2}{*}{Highest} & 3 & 27 & 22 & 279.91 & -0.10 & 275.21 & -1.81 \\
 &  & 2 & 729 & 8 & 278.62 & -0.57 & 273.12 & -2.59 \\
\cdashline{2-9}
 & \multirow{2}{*}{Least} & 3 & 27 & 22 & 279.87 & -0.12 & 275.14 & -1.84 \\
 &  & 2 & 729 & 8 & 278.51 & -0.61 & 272.97 & -2.65 \\
\cdashline{2-9}
 & \multirow{2}{*}{Median} & 3 & 27 & 22 & 279.89 & -0.11 & 275.18 & -1.82 \\
 &  & 2 & 729 & 8 & 278.38 & -0.65 & 273.03 & -2.63 \\
\hline
\multirow{9}{*}{$\textit{MHNG}(\ell,\gamma)$}
 & \multirow{3}{*}{Highest} & 3 & 3 & 22 & 279.79 & -0.15 & 274.17 & -2.20 \\
 &  & 2 & 9 & 8 & 277.64 & -0.92 & 272.22 & -2.93 \\
 &  & 1 & 63 & 2 & 263.72 & -6.25 & 257.94 & -8.63 \\
\cdashline{2-9}
 & \multirow{3}{*}{Least} & 3 & 3 & 22 & 279.78 & -0.15 & 274.15 & -2.21 \\
 &  & 2 & 9 & 8 & 277.45 & -0.99 & 271.99 & -3.02 \\
 &  & 1 & 63 & 2 & 263.79 & -6.22 & 258.08 & -8.57 \\
\cdashline{2-9}
 & \multirow{3}{*}{Median} & 3 & 3 & 22 & 279.78 & -0.15 & 274.16 & -2.20 \\
 &  & 2 & 9 & 8 & 277.32 & -1.03 & 271.80 & -3.09 \\
 &  & 1 & 63 & 2 & 263.86 & -6.19 & 258.11 & -8.56 \\
\hline
\end{tabular}
\caption{Bounds (in bln \euro) and corresponding optimality gaps (\%) for bounding schemes $\textit{MHEG}(\ell,\gamma)$ and $\textit{MHNG}(\ell,\gamma)$ at varying levels of the strategic ($\ell$) and operational ($\gamma$) refinement chains with one fixed operational scenario in the small instance.}
\label{tab:oper_fix}
\end{table}

\subsubsection{Results for the Large Instance}
\label{sec:large_inst}

Finally, to assess the scalability of the proposed bounding schemes, we consider a large-size instance with 5 strategic stages, 16 strategic scenarios, and 31 strategic nodes. 
Each strategic node is associated with the same 64 operational scenarios used in the previous experiments.
In this large-scale setting, solving the full multi-horizon stochastic program becomes computationally infeasible. 
Therefore, to establish a reference benchmark for comparison, we employ the rolling-horizon heuristic \textit{SFR3} introduced in \cite{escudero2021multistage} and described in \ref{sec:UB}. 
This heuristic yields a feasible solution with an objective value of 286.93 bln \euro{} (upper bound) obtained in 15,153 seconds.
Furthermore, because the expectation-based bounds are computationally intractable at this size, we report only the efficient node-based bounds in this section. 
We analyze three alternative configurations of the proposed bounding approach. The first  partitions both the strategic and operational scenarios into disjoint groups, the second fixes one strategic scenario across all strategic levels, and the third fixes both one strategic scenario across all strategic levels and one operational scenario across all operational levels.

\textbf{Disjoint Strategic and Operational Subgroups.}
At each strategic level $\ell$, the groups are disjoint and formed by merging pairs of groups from the previous level $\ell-1$, resulting in group sizes  ranging from 1 at $\ell=1$ to 16 at $\ell=5$. Results are reported in Table \ref{tab:large_combined}.

\begin{table}[ht!]
\scriptsize
\centering
\setlength{\tabcolsep}{2pt} 
\begin{tabular}{ccc|rr|rr|rr|rr|rr}
\hline
\multirow{2}{*}{$\gamma$} & \multirow{2}{*}{$m_\gamma$} & \multirow{2}{*}{$g_\gamma$} &
\multicolumn{2}{c|}{$\ell=5$} $(\hat{m}_\ell=1)$ &
\multicolumn{2}{c|}{$\ell=4$} $(\hat{m}_\ell=2)$ &
\multicolumn{2}{c|}{$\ell=3$} $(\hat{m}_\ell=4)$ &
\multicolumn{2}{c|}{$\ell=2$} $(\hat{m}_\ell=8)$ &
\multicolumn{2}{c}{$\ell=1$} $(\hat{m}_\ell=16)$ \\
\cline{4-13}
& & & Bound  & CPU  & Bound & CPU  & Bound  & CPU  & Bound  & CPU  & Bound  & CPU  \\
\hline
7 & 1  & 64 & -    & -   & 277.92 & 7629 & 273.20 & 6073 & 269.56 & 6113 & 268.88 & 6495 \\
6 & 2  & 32 & 285.97 & 5587 & 276.88 & 4128 & 272.26 & 2065 & 268.88 & 3742 & 267.92 & 3568 \\
5 & 4  & 16 & 284.54 & 3953 & 275.62 & 826  & 271.22 & 651  & 267.59 & 1024 & 266.83 & 1782 \\
4 & 8  & 8  & 281.87 & 1858 & 273.41 & 706  & 269.07 & 463  & 265.34 & 762  & 264.56 & 1278 \\
3 & 16 & 4  & 278.67 & 853  & 271.38 & 544  & 266.94 & 423  & 263.01 & 629  & 262.10 & 1108 \\
2 & 32 & 2  & 276.49 & 685  & 269.88 & 704  & 265.09 & 373  & 260.82 & 719  & 259.72 & 1023 \\
1 & 64 & 1  & 272.37 & 528  & 266.20 & 524  & 260.99 & 462  & 256.30 & 847  & 254.82 & 1395 \\
\hline
\end{tabular}
\caption{Results for bounding approach $\textit{MHNG}(\ell,\gamma)$ at varying levels of the strategic ($\ell$) and operational ($\gamma$) refinement chains with no fixed scenarios in the large instance.  
Each block reports: \textit{bound (bln~\euro),  CPU time (s)}.}
\label{tab:large_combined}
\end{table}
\vspace*{-0.05in}

The results confirm previous findings: the most accurate bounds are obtained when the full strategic uncertainty is retained ($\ell=5$) and grouping is applied only to the operational level. 
For example, $\textit{MHNG}(5,6)$ and $\textit{MHNG}(5,5)$ yield a percentage gap of  0.34\% and 0.84\% with respect to the upper bound obtained using the rolling-horizon heuristic \textit{SFR3}, while $\textit{MHNG}(4,7)$ and $\textit{MHNG}(3,7)$ exhibit larger gaps of 3.24\% and 5.03\%, respectively, reflecting the loss of solution quality when the strategic tree is dissected. 
Decreasing $\ell$ and $\gamma$ reduces subproblem size but increases their number, so the fastest solution times are not always at the lowest levels. 
For each $\gamma$, the fastest solution times are typically achieved at $\ell = 3$, which seems to strike the best balance between the number of strategic subproblems and their individual computational burden.

\textbf{One Fixed Stategic Scenario.}
To assess the impact of fixing a single strategic scenario, we consider a hierarchical chain where the median-cost scenario is fixed across all levels. 
At level $\ell=1$, 15 strategic subgroups are formed, each including the fixed scenario and one distinct variable scenario (groups of size 2). 
At level $\ell=2$, subgroups from level 1 are merged into 5 larger  subgroups (cardinality 4), and at level $\ell=3$, all subgroups are aggregated into a single group, representing the complete strategic scenario tree. 
The corresponding results are reported in Table~\ref{tab:stra_fix}. 
Results for $\ell=3$ are omitted, as they coincide with those already presented in Table~\ref{tab:large_combined} for $\ell=5$.

\begin{table}[ht!]
\scriptsize
\centering
\setlength{\tabcolsep}{3pt} 
\begin{tabular}{ccc|rr|rr}
\hline
\multirow{2}{*}{$\gamma$} & \multirow{2}{*}{$m_\gamma$} & \multirow{2}{*}{$g_\gamma$} &
\multicolumn{2}{c|}{$\ell=2$ ($\hat{m}_{\ell}=5$)} &
\multicolumn{2}{c}{$\ell=1$ ($\hat{m}_{\ell}=15$)} \\
\cline{4-7}
& & & Bound & CPU  & Bound & CPU  \\
\hline
7 & 1  & 64 & 274.87 & 10376 & 270.23 & 13838 \\
6 & 2  & 32 & 273.95 & 2971  & 269.27 & 5486  \\
5 & 4  & 16 & 272.83 & 1269  & 268.18 & 2541  \\
4 & 8  & 8  & 270.51 & 1116  & 265.90 & 1755  \\
3 & 16 & 4  & 267.98 & 1115  & 263.43 & 1642  \\
2 & 32 & 2  & 265.80 & 1021  & 261.09 & 1641  \\
1 & 64 & 1  & 261.46 & 851   & 256.25 & 1695  \\
\hline
\end{tabular}
\caption{Results for bounding approach $\textit{MHNG}(\ell,\gamma)$ at varying levels of the strategic ($\ell$) and operational ($\gamma$) refinement chains with one fixed strategic scenario in the large instance case.  
Each block reports: \textit{bound (bln~\euro)} and \textit{CPU time (s)}.}
\label{tab:stra_fix}
\end{table}

Table~\ref{tab:stra_fix} shows that fixing one strategic scenario leads to slightly tighter bounds compared to the case without any fixed scenario, when the same number of scenarios per subgroup is considered. 
However, this small improvement comes at the expense of longer computation times, due to the larger number of subproblems that must be solved. 
As observed earlier, the most accurate bounds are obtained when the full representation of strategic uncertainty is kept and grouping is applied only to the operational level. 
Hence, fixing a strategic scenario is not an effective strategy.

\textbf{Fixing One Strategic and One  Operational Scenario.}
We further analyze the case where both one strategic and one operational scenario are fixed. 
Specifically, the strategic refinement chain uses the \textit{median-cost strategic scenario} fixed across all levels, as introduced in the previous section, while the operational refinement chain employs the \textit{median-cost operational scenario} fixed across its four levels, as described in Section~\ref{sec:small_inst}. 
This combined configuration enables the evaluation of the joint impact of fixing scenarios in both the strategic and operational levels.

\begin{table}[ht!]
\footnotesize
\centering
\setlength{\tabcolsep}{3pt}
\begin{tabular}{ccr|rrrrrr}
\hline
 &  &  & \multicolumn{6}{c}{$\ell$} \\
$\gamma$ & $m_\gamma$ & $g_\gamma$ &
\multicolumn{2}{c}{3 ($\hat{m}_{\ell}=1$)} &
\multicolumn{2}{c}{2 ($\hat{m}_{\ell}=5$)} &
\multicolumn{2}{c}{1 ($\hat{m}_{\ell}=15$)} \\
\cline{4-9}
 & & & Bound & CPU  & Bound & CPU  & Bound & CPU  \\
\hline
4 & 1 & 64  & --     & --     & 274.87 & 10376 & 270.23 & 13838 \\
3 & 3 & 22  & 284.82 & 4782   & 272.85 & 3091  & 268.12 & 4027  \\
2 & 9 & 8   & 281.61 & 1770   & 270.48 & 1242  & 265.82 & 2233  \\
1 & 63 & 2  & 272.81 & 1674   & 261.92 & 2879  & 256.74 & 3602  \\
\hline
\end{tabular}
\caption{Bounds (in bln \euro) and total CPU times (in seconds) for the bou\VV{n}ding scheme $\textit{MHNG}(\ell,\gamma)$ at varying levels of the strategic ($\ell$) and operational ($\gamma$) refinement chains with one fixed strategic and one fixed operational scenario in the large instance case.}
\label{tab:stra_oper_fix}
\end{table}

As reported in Table~\ref{tab:stra_oper_fix}, fixing both a strategic and an operational scenario leads to slightly weaker bounds and longer solution times compared with the fully variable configuration.  
For instance, at $\ell=2$, the configuration with $\gamma=2$ (corresponding to groups of $g_\gamma=8$ operational scenarios) yields a bound of 270.48~bln~\euro{}, compared to 270.51~bln~\euro{} obtained at $\gamma=4$ in Table~\ref{tab:stra_fix}, which features the same subgroup cardinality.  
The associated CPU time increases from 1116 s to 1242 s.  
A similar pattern is observed across all levels of the refinement chains: the improvement obtained by fixing scenarios is marginal, whereas the computational cost tends to rise due to the increased number of subproblems.  
These results confirm that preserving full variability within both the strategic and operational scenario subgroups yields tighter bounds and a more accurate approximation of the stochastic problem structure.\vspace*{-0.1in}

\section{Conclusions}\vspace*{-0.03in}
\label{sec:conclusion}
This paper introduced novel bounding schemes for multi-horizon stochastic optimization problems, proposing new techniques to generate tight lower bounds by dissecting the multi-horizon scenario tree into smaller subgroups. The approach allows the strategic and operational scenario trees to be dissected either through disjoint partitions or by fixing specific scenarios within each subgroup.
Three alternative dissection strategies have been developed:
(i) Dissecting the operational subtrees while preserving all strategic scenarios;
(ii) Dissecting the strategic trees while preserving all operational scenarios;
(iii) Dissecting both simultaneously.
The resulting subproblems are then recombined using two aggregation schemes, expectation-based and node-based, each of which generates monotonic chains of lower bounds that become tighter as the subgroup size increases. A comparison with standard expectation-based bounds demonstrates the superior performance of the proposed schemes.
The proposed bounding schemes establish monotonic chains of inequalities that are valid for any multi-horizon stochastic program, thus providing a general framework for deriving tight lower approximations in complex multi-scale optimization problems.

The methodology was validated on an investment planning problem in electricity generation and transmission. Numerical experiments on small, medium, and large-scale instances confirm the effectiveness of the approach:
\begin{itemize}
\item \textit{Small instances:} The expectation-based method produces very tight bounds when dissecting operational uncertainty, but at a high computational cost due to the large number of subproblems. The node-based method, while slightly less accurate, offers significant computational savings.
Fixing operational scenarios across refinement levels has a minor impact on bound quality; selecting the median-cost scenario can slightly improve results, but more heterogeneous groupings often yield better approximations.
\item \textit{Medium instances:} The larger number of strategic nodes makes the expectation-based method computationally intractable, leaving the node-based scheme as the only viable option. 
\item \textit{Large instances:} Solving the full multi-horizon problem becomes computationally prohibitive, requiring a heuristic to obtain an upper bound. Across various grouping strategies, dissecting the operational subtree consistently produces tighter bounds than dissecting the strategic one, with full variability within subgroups proving most beneficial.
\end{itemize}
The results highlight the trade-off between accuracy and computational effort and demonstrate the robustness of the proposed bounding schemes across different problem scales.
Beyond methodological advances, the results show that the framework also offers a practical tool, enabling analysts to evaluate the quality of approximate solutions without solving the full-scale stochastic model.

Future research includes extending the proposed approach to multi-horizon Distributionally Robust Optimization, targeting settings in which the underlying probability distributions---for both long-term and short-term uncertainties---are partially known, a condition that frequently arises in real-world decision-making under uncertainty. Estimating optimal values of subgroups via surrogate models, such as neural networks, and levering these estimates to devise near-optimal partitioning strategies for a given multi-horizon problem also merits future work.

\bigskip
{\bf Acknowledgements:}
This research was supported in part by the U.S. Department of Energy, Office of Science, Office of Advanced Scientific Computing Research, under Grant DE-SC0023361.
The authors also acknowledge support from the Gruppo Nazionale per il Calcolo Scientifico (GNCS-INdAM).

\bibliographystyle{model5-apa}  
\bibliography{sample}

\newpage
\appendix

\setcounter{table}{0}
\renewcommand{\thetable}{A.\arabic{table}}
\section{Expectation-based Lower Bounds}
\label{sec:LB_exp}
This appendix presents the expectation-based lower bounds used as benchmarks in this study.
We first recall their formal definition and then report the corresponding numerical results for the case study.
We first recall the two methods from \cite{maggioni2020bounds}, and then present the additional variant based on replacing only strategic uncertain parameters with their expectations.

\subsection{Expectation-based Lower Bounds: Definitions}

\begin{mydef}
    The \textit{Multi-Horizon Expected Value problem}, denoted  \text{MHEV}, is obtained by replacing both strategic uncertain parameters $(c_n,h_n,T_n,W_n)$ and operational uncertain parameters $(q_n^{\omega,\tau},h_n^{\omega,\tau}, \allowbreak  T_n^{\omega,\tau},W_n^{\omega,\tau})$ in problem \eqref{MHSP1} by their expected values, i.e.
    \begin{align}
\Big( \sum_{n \in \mathcal{N}_{t}}\pi_n c_n, \sum_{n \in \mathcal{N}_{t}}\pi_n h_n, \sum_{n \in \mathcal{N}_{t}}\pi_n T_n, \sum_{n \in \mathcal{N}_{t}}\pi_n W_n \Big) := 
( \bar{c}_t, \bar{h}_t, \bar{T}_t, \bar{W}_t ), \nonumber
    \end{align}
and 
    \begin{align}
\Big( \sum_{\omega \in \Omega_t} \pi_\omega q_n^{\omega,\tau}, \sum_{\omega \in \Omega_t} \pi_\omega h_n^{\omega,\tau}, \sum_{\omega \in \Omega_t} \pi_\omega T_n^{\omega,\tau}, \sum_{\omega \in \Omega_t} \pi_\omega W_n^{\omega,\tau} \Big) := 
( \bar{q}_t^\tau, \bar{h}_t^\tau, \bar{T}_t^\tau, \bar{W}_t^\tau ), \nonumber
    \end{align}
    and solving the deterministic program
\begin{align*}
	\textit{MHEV} \vcentcolon = & \min\limits_{\bm{\mathrm{x}},\bm{\mathrm{y}}} 
 \sum_{t \in \mathcal{H}} \Big( \bar{c_t}\mathrm{x}_{t} + \sum_{\tau \in \mathcal{T}_t} \bar{q}_t^\tau\mathrm{y}_t^\tau \Big) \nonumber \\
& \textrm{ s.t. }   A \mathrm{x_0} = {h}_0,  \nonumber \\
 & \qquad \, \bar{T}_t\mathrm{x}_{t-1} + \bar{W}_t\mathrm{x}_{t} = \bar{h}_t, \ t \in \mathcal{H}\setminus \{0\},  \nonumber \\
& \qquad \, \bar{T}_{t}^1 \mathrm{x}_t + \bar{W}_t^1 \mathrm{y}_t^1 = \bar{h}_t^1, \  t \in \mathcal{H}, \nonumber\\ 
& \qquad \, \bar{T}_{t}^\tau \mathrm{y}_t^{\tau-1} + \bar{W}_t^1 \mathrm{y}_t^\tau = \bar{h}_t^\tau, \  \tau \in \mathcal{T}_t \setminus \{1\}, t \in \mathcal{H}. 
 	\end{align*}  
\end{mydef}

\smallskip
\begin{mydef}
    The \textit{Multi-Horizon Operational Expected Value problem}, denoted $\text{MHOEV}$, is obtained by replacing only operational uncertain parameters  $(q_n^{\omega,\tau},h_n^{\omega,\tau},T_n^{\omega,\tau},W_n^{\omega,\tau})$ in problem \eqref{MHSP1} by their expected values $( \bar{q}_n^\tau, \bar{h}_n^\tau, \bar{T}_n^\tau, \bar{W}_n^\tau ), n \in \mathcal{N}_t, \, \tau \in \mathcal{T}_t, \, t \in \mathcal{H},$ and solving the stochastic program
\begin{align*}
\textit{MHOEV} \vcentcolon= 
 & \min\limits_{\bm{\mathrm{x}},\bm{\mathrm{y}}} 
 \sum\limits_{t \in \mathcal{H}}\sum\limits_{n \in \mathcal{N}_t} \pi_{n} \big( c_n x_n+ \sum_{\tau \in \mathcal{T}_t} \bar{q}_n^\tau \mathrm{y}_n^\tau
\big) 
\nonumber  \\
  & \textrm{ s.t. }  A \mathrm{x_0} = {h}_0,  \nonumber \\
 & \qquad \ T_{n} x_{a({n})} + W_{n} x_{n} =h_{n},\   n \in  \mathcal{N}_t,\ t\in\mathcal{H}\setminus\{0\},\nonumber \\
 & \qquad \ \bar{T}_{n}^1 \mathrm{x}_n + \bar{W}_n^1 \mathrm{y}_n^1 = \bar{h}_n^1, \   n \in  \mathcal{N}_t,\ t \in \mathcal{H}, \nonumber\\
& \qquad \ \bar{T}_{n}^\tau \mathrm{y}_n^{\tau-1} + \bar{W}_n^1 \mathrm{y}_n^\tau = \bar{h}_n^\tau, \   n \in  \mathcal{N}_t,\ \tau \in \mathcal{T}_t \setminus \{1\}, t \in \mathcal{H}. 
 	\end{align*}  
\end{mydef}

\smallskip

\begin{mydef}
    The \textit{Multi-Horizon Strategic Expected Value problem}, denoted $\text{MHSEV}$, is obtained by replacing only strategic uncertain parameters $(c_n,h_n,T_n,W_n)$ in problem \eqref{MHSP1} by their expected values $( \bar{c}_t, \bar{h}_t, \bar{T}_t, \bar{W}_t ), t \in \mathcal{H}$, and solving the stochastic program
\begin{align*}
 \textit{MHSEV} \vcentcolon= &
\min\limits_{\bm{\mathrm{x}},\bm{\mathrm{y}}} 
 \sum_{t \in \mathcal{H}} \Big( \bar{c_t}\mathrm{x}_{t} 
+ \sum\limits_{\omega \in \Omega_t} \pi^{\omega}_{t} \sum\limits_{\tau \in \mathcal{T}_t}  q_{t}^{\omega,\tau} y_{t}^{\omega,\tau}
 \Big) \nonumber \\
& \textrm{ s.t. }  A \mathrm{x_0} = {h}_0,  \nonumber \\
 & \qquad\, \bar{T}_t\mathrm{x}_{t-1} + \bar{W}_t\mathrm{x}_{t} = \bar{h}_t, \  t \in \mathcal{H}\setminus \{0\},  \nonumber \\
&\qquad\, T_{t}^{\omega,1} x_{t} + W_{t}^{\omega,1} y_{t}^{\omega,1} = h_{t}^{\omega,1},\   \omega \in \Omega_t,\ t\in\mathcal{H}, \nonumber\\
&\qquad\, T_{t}^{\omega,\tau} y_{t}^{\omega,\tau -1} + W_{t}^{\omega,\tau} y_{t}^{\omega,\tau} = h_{t}^{\omega,\tau},\  \tau \in \mathcal{T}_t\setminus\left\{1\right\},\ \omega \in \Omega_t,\ t\in\mathcal{H}. 
\end{align*}
\end{mydef}
Compared to the original multi-horizon problem \eqref{MHSP1}, the \text{MHSEV} formulation introduces three key simplifications: (i) the strategic decision variables $\mathrm{x}_t$ are now deterministic, (ii) strategic costs are evaluated using their expected values $\bar{c}_t$, and (iii) the strategic constraints linking consecutive strategic stages, $\bar{T}_t \mathrm{x}_{t-1} + \bar{W}_t \mathrm{x}_t = \bar{h}_t$, are also deterministic. All operational decisions and constraints remain stochastic, preserving the uncertainty at the operational level while providing a lower bound for the full multi-horizon problem.

\subsection{Expectation-based Lower Bounds: Numerical Results}
\label{expectation-based}
We now compute classical lower bounds obtained by replacing the strategic, operational, or both stochastic processes with their expected values.
As reported in Table~\ref{tab:exp}, these expectation-based formulations, namely Multi-Horizon Expected Value Problem ($\textit{MHEV}$), Multi-Horizon Operational Expected Value Problem ($\textit{MHOEV}$), and Multi-Horizon Strategic Expected Value Problem ($\textit{MHSEV}$), produce significantly weaker bounds.  
For the small instance, the best expectation-based bound ($\textit{MHSEV}$) underestimates the optimal value by about 17.7\%, while for the medium and large instances, the corresponding optimality gaps remain substantial at 13.7\% and 7.8\%, respectively.  
Although the computation times are negligible compared to the full stochastic model (below 20 minutes even for the largest case), the results confirm that such approximations provide poor accuracy and are unsuitable for reliable strategic planning analyses.

\begin{table}[ht!]
\footnotesize
    \centering
    \begin{tabular}{ll|rrr}
    \hline
         Instance type &  &  $\textit{MHEV}$ & $\textit{MHOEV}$ & $\textit{MHSEV}$ \\ \hline
         \multirow{3}{*}{Small} 
         & Bound (bln~\euro) & 220.04 & 225.07 & 238.14 \\
         & Gap (\%) & $-27.34$ & $-24.49$ & $-17.66$ \\
         & CPU time (s) & 2 & 3 & 204 \\ \hdashline
         \multirow{3}{*}{Medium} 
         & Bound (bln~\euro) & 240.22 & 245.34 & 255.78 \\
         & Gap (\%) & $-21.03$ & $-18.50$ & $-13.66$ \\
         & CPU time (s) & 4 & 7 & 522 \\ \hdashline
         \multirow{3}{*}{Large} 
         & Bound (bln~\euro) & 256.22 & 259.59 & 266.15 \\
         & Gap (\%) & $-11.99$ & $-10.93$ & $-7.81$ \\
         & CPU time (s) & 5 & 9 & 1121 \\ \hline
    \end{tabular}
    \caption{Bound (in bln~\euro), optimality gap (in~\%), and CPU time (in seconds) of lower bounds computed via expectation-based approximations for the three instance types.}
    \label{tab:exp}
\end{table}

\setcounter{table}{0}
\renewcommand{\thetable}{B.\arabic{table}}
\setcounter{figure}{0}
\renewcommand{\thefigure}{B.\arabic{figure}}
\section{Notation and Additional Illustrative Examples}
\label{Example}
This Appendix complements Section~\ref{sec:LBmeasure_diss} by providing a list of notations used along with some illustrative examples.
\subsection{Operational Subtree Dissection}\label{Appendix:OSD}
In Table~\ref{tb:operational_dissection}, we summarize the notations introduced for dissecting \textit{only} the operational subtrees in a multi-horizon scenario tree based on both expectation and node-based combinations. The column ``Example'' in this table is presented from \Cref{MHEOGSIOPT_common_prob_figure_example}.
\begingroup
\renewcommand\arraystretch{1.2}
\footnotesize
\begin{longtable}{p{1.3cm} p{8.5cm}Sl}
\hline
        Notation  & \multicolumn{1}{c}{Explanation}  & Example \\
        \hline 
        \endfirsthead
        \hline
          Notation  & \multicolumn{1}{c}{Explanation} & Example \\
        \hline
        \endhead
\midrule
    \multicolumn{3}{r}{\footnotesize\itshape Continued on the next page}
\endfoot
\endlastfoot 
        $\Omega_{n,j}^{(\gamma)} $ & $j^{\text{th}}$-operational subset at operational level $\gamma$ and strategic node $n \in \mathcal{N}_t,  t \in \mathcal{H}$. & $\Omega_{1,1}^{(2)} = \{\text{AB,AC}\}$ \\ 
        $g_{\gamma} $ & Number of operational scenarios in $\Omega_{n,j}^{(\gamma)}$. & $g_2 = 2$ \\ 
  $m_{\gamma}$ & Total number of operational subsets at operational
level $\gamma$. & $m_2=2$ \\
  $\Omega_{\text{OG},k}^{(\gamma)} $ & $k^{th}$ operational subgroup at operational level $\gamma$. & \\
  $\Omega_f$ & The set of fixed operational scenarios. &$\Omega_f = \varnothing$ \\
  $\hat{\pi}^{(\gamma)}_{\omega}$ & Probability  of the operational scenario $\omega$ at operational level $\gamma$. & $\hat{\pi}^{(2)}_{\{\text{AB}\} } = \frac{1}{2} $ \VV{($\{\text{AB}\} \in \Omega_{1,1}^{(2)}$)} \\  
        $\phi_{\text{OG},k}^{(\gamma)}$ & Weight of the subgroup $\Omega_{\text{OG},k}^{(\gamma)}$ at operational level $\gamma$.  & $ \phi_{\text{OG},1}^{(2)} = \frac{2}{5} \times \frac{2}{6} \times \frac{4}{7}$  \\
        $\phi_{n,j}^{(\gamma)}$ & Weight of the subset $\Omega_{n,j}^{(\gamma)}$. & $\phi_{2,1}^{(2)} = \frac{2}{6}$ \\
      \hline
    \caption{Notations for dissecting \textit{only} the operational subtree subproblem in a multi-horizon scenario tree.}
    \label{tb:operational_dissection}
\end{longtable}
\endgroup

We introduce an example to illustrate the node-based combination approach for the operational subtree dissection. 
\input{Tree_MHEOG_Chain_Fig2}
\begin{myex}\label{ex:MHNOG}
Consider the multi-horizon scenario tree $\mathfrak{T}(3)$ depicted in Figure \ref{Chain_MHESG_2_with_7_nodes_2}, where at each strategic node, the operational scenarios are defined as in Example \ref{example_dissect_operational_common_prob}. 
We construct a collection of subsets of the original operational support $\Omega_n, \, n \in \mathcal{N}_t,\,  t \in \mathcal{H}$, forming a refinement chain as follows:
\begin{itemize}
    \item At $\gamma=1$, each operational tree is dissected into its atoms. Therefore, we obtain $m_1=4$. $\Omega_{n,1}^{(1)}=\{\omega_{n,1}\} = \{\textnormal{AB}\}$, $\Omega_{n,2}^{(1)}=\{\omega_{n,2}\}= \{\textnormal{AC}\}$, $\Omega_{n,3}^{(1)}=\{\omega_{n,3}\}= \{\textnormal{AD}\}$, $\Omega_{n,4}^{(1)}=\{\omega_{n,4}\}= \{\textnormal{AE}\}$, $n \in \mathcal{N}_t,\,  t \in \mathcal{H}$, with weights $\phi_{n,j}^{(1)}, \ j= 1, 2, 3, 4$, given in Table \ref{tab:phi_computation}. 
    \item At $\gamma=2$, each operational tree is dissected into $m_2 = 2$ \VV{subsets}, with  $\Omega_{n,1}^{(2)}=\{\omega_{n,1},\omega_{n,2}\}$, $\Omega_{n,2}^{(2)}=\{\omega_{n,3},\omega_{n,4}\}$, $n \in \mathcal{N}_t,\,  t \in \mathcal{H}$, with weights $\phi_{n,j}^{(2)}, \ j= 1, 2$, given in Table \ref{tab:phi_computation}. 
    \item At $\gamma=3$, in each strategic node $n \in \mathcal{N}_t,\,  t \in \mathcal{H}$, we have a single operational \VV{subset} $\Omega_{n,1}^{(3)}=\{\omega_{n,1},\omega_{n,2},\omega_{n,3},\omega_{n,4}\}$, corresponding to the original operational scenario tree.
\end{itemize}
Unlike the previous example, wherein we assigned a probabilistic weight to each subgroup, we forgo this step and adopt a novel combination approach, consequently reducing the number of operational group subproblems. 
For example, at $\gamma=2$ of the operational refinement chain, only 2 subproblems need to be solved, whereas calculating the lower bound with the previous approach required solving 8 subproblems.

\renewcommand{\arraystretch}{1.5}
\begin{table} [ht!]
\footnotesize
    \centering
    \begin{tabular}{llcccc}
    \hline
           $\gamma$&$j$&$\Omega_{n,j}^{(\gamma)}$&  $\phi_{1,j}^{(\gamma)}$&  $\phi_{2,j}^{(\gamma)}$&  $\phi_{3,j}^{(\gamma)}$ \\
           \hline
           1&1
&$\{\omega_{n,1}\}$&  $\frac{1}{5}$&  $\frac{1}{6}$&  $\frac{1}{7}$\\
           1&2&$\{\omega_{n,2}\}$&  $\frac{1}{5}$ & $\frac{1}{6}$ & $\frac{3}{7}$ \\
           1&3&$\{\omega_{n,3}\}$&  $\frac{2}{5}$&  $\frac{2}{6}$&  $\frac{2}{7}$  \\
           1&4&$\{\omega_{n,4}\}$&  $\frac{1}{5}$&  $\frac{2}{6}$&  $\frac{1}{7}$  \\
           \hline
           2&1&$\{\omega_{n,1},\omega_{n,2}\}$&  $\frac{2}{5}$ & $\frac{2}{6}$ & $\frac{4}{7}$ \\
           2&2&  $\{\omega_{n,3},\omega_{n,4}\}$& $\frac{3}{5}$ & $\frac{4}{6}$ & $\frac{3}{7}$ \\
           \hline
           3&1&  $\{\omega_{n,1},\omega_{n,2},\omega_{n,3},\omega_{n,4}\}$ & 1 & 1 & 1  \\
           \hline
    \end{tabular}
    \caption{Value of weights $\phi_{n,j}^{(\gamma)}$, $n \in \mathcal{N}_t$, $t \in \mathcal{H}$, of different operational subsets $j \in [m_{\gamma}]$ at different levels $\gamma$ of the refinement chain.}
    \label{tab:phi_computation}
\end{table}
\renewcommand{\arraystretch}{1}
\end{myex}

\subsection{Strategic Scenario Tree Dissection}\label{Appendix:SSD}
In Table~\ref{tb:strategic_dissection}, we summarize the notations introduced for dissecting \textit{only} the strategic scenario trees in a multi-horizon scenario tree based on both expectation and node-based combinations. The column ``Example'' in this table is presented from Figure  \ref{MHESG_2_with_7_nodes}. 
\begingroup
\renewcommand\arraystretch{1.2}
\footnotesize
\begin{longtable}{p{1.5cm}p{1.1cm} p{7.9cm}Sl}
\hline
     Combina- tion type &   Notation  & \multicolumn{1}{c}{Explanation}  & Example \\
        \hline 
        \endfirsthead
        \hline
        Combina- tion type &   Notation  & \multicolumn{1}{c}{Explanation} & Example \\
        \hline
        \endhead
\midrule
    \multicolumn{4}{r}{\footnotesize\itshape Continued on the next page}
\endfoot
\endlastfoot 
      Expectati- on-based
      &  $\hat{m}_{\ell}$ & Total number of strategic subgroups at strategic level $\ell$. & $\hat{m}_2=2$  \\
        & $\mathcal{S}_{i}^{(\ell)}$ & {Set of strategic scenarios of subgroup $i$} at strategic level $\ell$.
        & $\mathcal{S}_1^{(2)} = \big\{ \{1,2,4\}, \{1,2,5\} \big\}$
        \\
        & $\Pi_i^{(\ell)}$ &  Probability measure of strategic scenarios at subgroup $i$ and strategic level $\ell$.  & $\Pi_1^{(2)}(\{1,2,4\}) = \frac{2}{3}$ \\
       & $\phi_i^{(\ell)}$ & Weight of the strategic subgroup $i$, at strategic level $\ell$. & $  \phi_1^{(2)} = \frac{1}{3}, \phi_2^{(2)} = \frac{2}{3}  $ \\
       \midrule
Node-based & $\mathcal{N}^{\mathbf{s}}$ & Set of strategic nodes associated with scenario $\mathbf{s}$. & $\mathcal{N}^{\{1,2,4\}} = \{1,2,4\}$ \\
       & $\mathcal{N}_i^{(\ell)}$ & Set of strategic nodes at strategic subgroup $ i $ and strategic level $\ell$. & $ \begin{aligned}[t] \mathcal{N}_1^{(2)} &= \{ 
1,2,4,5 \}, \\ \mathcal{N}_2^{(2)} &= \{ 
1,3,6,7 \}  \end{aligned}$ \\
& $\mathcal{N}_{i,t}^{(\ell)}$ & Set of strategic nodes at strategic subgroup $ i $ and strategic level $\ell$ at strategic stage $t$. & $\mathcal{N}_{1,2}^{(2)} = \{4,5\}$ \\
& $\mathcal{S}^n $ & Set of strategic scenarios passing through node $n \in \mathcal{N}$. & $\mathcal{S}^4 = \big\{ \{1,2,4\} \big\}$ \\
       & $\hat{\Pi}_{i,n}^{(\ell)}$ & Probability of node $ n \in \mathcal{N}_i^{(\ell)} $  at strategic subgroup $ i $ and strategic level $\ell$.  & $\hat{\Pi}_{1,4}^{(2)} = \frac{2}{3}$\\
        & $\hat{\phi}_{n,i}^{(\ell)}$ & Weight of strategic node $n$ in strategic subgroup $i$. & $\hat{\phi}_{1,1}^{(2)} = \frac{1}{3}$\\ 
      \hline
    \caption{Notations for dissecting \textit{only} the strategic tree in a multi-horizon scenario tree.}
    \label{tb:strategic_dissection}
\end{longtable}
\endgroup

We illustrate through an example that the expectation-based and node-based combination approaches are equivalent when only the strategic subproblem trees are dissected.
\begin{myex}
\label{ex:stra_nodal}
Consider the multi-horizon scenario tree $\mathfrak{T}(10)$ depicted in Figure ~\ref{MHESG_2_groups_4_scen_2_fixed}, which comprises 10 strategic nodes and 3 operational scenarios, divided into two strategic subgroups of size 4 each, with two fixed scenarios. 
In Table~\ref{tab:phi_stra} we show the equivalence between the expectation-based and the node-based approach when only the strategic tree is dissected.
Note that such an equivalence arises from the identity between the terms $\phi_{i}^{(\ell)} \hat{\Pi}_{i,n}^{(\ell)}$ and $\hat{\phi}_{n,i}^{(\ell)} \pi_n$.

\input{Tree_MHESG_Phi_Fixed_Fig}

\renewcommand{\arraystretch}{1.5}
\begin{table} [ht!]
\footnotesize
    \centering
    \begin{tabular}{c | cccc|c| cccc}
    \hline
         $n \in \mathcal{N}$ &  
         $\hat{\Pi}_{1,n}^{(2)}$ &
         $\hat{\Pi}_{2,n}^{(2)}$&  
         $\phi_1^{(2)} \hat{\Pi}_{1,n}^{(2)}$ & 
         $\phi_2^{(2)} \hat{\Pi}_{2,n}^{(2)}$ & 
         $\pi_n$ & 
         $\hat{\phi}_{n,1}^{(2)}$ & 
         $\hat{\phi}_{n,2}^{(2)}$ &
         $\hat{\phi}_{n,1}^{(2)} \pi_n $ &
         $\hat{\phi}_{n,2}^{(2)} \pi_n$
         \\
         \hline
         
1 &  1 & 1 & $\frac{4}{7}$ & $\frac{3}{7}$ & 1
& $\frac{4}{7}$ & $\frac{3}{7}$ & $\frac{4}{7}$ & $\frac{3}{7}$
\\

2 &  $\frac{5}{8}$ & $\frac{4}{12}$ & $\frac{5}{14}$ & $\frac{1}{7}$ & $\frac{2}{4}$
& $\frac{5}{7}$ & $\frac{2}{7}$ & $\frac{5}{14}$ & $\frac{1}{7}$
\\

3 &  $\frac{3}{8}$ & $\frac{1}{12}$ & $\frac{3}{14}$ & $\frac{1}{28}$ & $\frac{1}{4}$
& $\frac{6}{7}$ & $\frac{1}{7}$ & $\frac{3}{14}$ & $\frac{1}{28}$
\\

4 &  0 & $\frac{7}{12}$ & 0 & $\frac{3}{12}$ & $\frac{1}{4}$
& 0 & 1 & 0 & $\frac{3}{12}$
\\

5 &  $\frac{7}{24}$ & 0 & $\frac{1}{6}$ & 0 & $\frac{2}{12}$ 
& 1 & 0 & $\frac{1}{6}$ & 0
\\

6 &  $\frac{4}{12}$ & $\frac{4}{12}$ & $\frac{4}{21}$ & $\frac{1}{7}$ & $\frac{4}{12}$
& $\frac{4}{7}$ & $\frac{3}{7}$ & $\frac{4}{21}$ & $\frac{1}{7}$ 
\\

7 &  $\frac{1}{12}$ & $\frac{1}{12}$ & $\frac{1}{21}$ & $\frac{1}{28}$ & $\frac{1}{12}$
& $\frac{4}{7}$ & $\frac{3}{7}$ & $\frac{1}{21}$ & $\frac{1}{28}$
\\

8 &  $\frac{7}{24}$ & 0 & $\frac{1}{6}$ & 0 & $\frac{2}{12}$ 
& 1 & 0 & $\frac{1}{6}$ & 0
\\

9 &  0 & $\frac{14}{36}$ & 0 & $\frac{2}{12}$ & $\frac{2}{12}$
& 0 & 1  & 0 & $\frac{2}{12}$
\\

10 &  0 & $\frac{7}{36}$ & 0 & $\frac{1}{12}$ & $\frac{1}{12}$
& 0 & 1  & 0 & $\frac{1}{12}$ \\
\hline
    \end{tabular}
    \caption{An example of equivalence in constructing the $\textit{MHESG}(\ell)$ for the scenario tree illustrated in Figure \ref{MHESG_2_groups_4_scen_2_fixed}. 
    }
    \label{tab:phi_stra}
\end{table}
\renewcommand{\arraystretch}{1}
\end{myex}
\subsection{Simultaneous Strategic Scenario Tree and Operational Subtree Dissection}\label{Appendix:Simultaneous_dissection}
In Table~\ref{tb:strategic_operational_dissection}, we  summarize the notations introduced for dissecting \textit{both} strategic and operational subproblem trees in a multi-horizon scenario tree based on both expectation and node-based combinations. The column ``Example'' in this table refers to the tree $\mathfrak{T}(3)$ presented in Figure \ref{Figure_dissect_strategic_then_operational}.
\begingroup
\renewcommand\arraystretch{1.2}
\begin{table}[ht!]
\centering
\footnotesize
\begin{tabular}{p{2cm} p{8cm} S l}
\hline
Notation & \multicolumn{1}{c}{Explanation} & Example \\
\hline
$(\mathcal{N}_i^{(\ell)},  \Omega_{\text{OG},k}^{(\ell,\gamma)})$ & Subgroup with set of strategic nodes $\mathcal{N}_i^{(\ell)}$ and operational subgroup $\Omega_{\text{OG},k}^{(\gamma)}$. & $(\mathcal{N}_1^{(3)},  \Omega_{\text{OG},1}^{(3,3)}) = \mathfrak{T}(3)$ \\
$\phi_{G,i,k}^{(\ell,\gamma)}$ & Weight of the subgroup $(\mathcal{N}_i^{(\ell)}, \Omega_{\text{OG},k}^{(\ell,\gamma)})$. & $\phi_{G,1,1}^{(3,3)} = 1$ \\
\hline
\end{tabular}
\caption{Notations for dissecting \textit{both} strategic and operational subproblem trees in a multi-horizon scenario tree.}
\label{tb:strategic_operational_dissection}
\end{table}
\endgroup

We   present an additional example to illustrate the approach to compute lower bounds by dissecting both strategic scenario tree and operational subtree according to the expectation-based combination.  
\begin{myex}
\label{ex:joint_dissection}
    Consider the multi-horizon scenario tree $\mathfrak{T}(3)$ depicted in Figure \ref{Figure_dissect_strategic_then_operational}, where the strategic scenarios $\{1,2\}$ and $\{1,3\}$ occur with probabilities $\frac{1}{3}$ and $\frac{2}{3}$, respectively. Dissecting the strategic scenario results in two subgroups, $(\mathcal{N}_1^{(1)},\Omega_{\textnormal{OG},1}^{(1,3)})$ and $(\mathcal{N}_2^{(1)},\Omega_{\textnormal{OG},1}^{(1,3)})$, occurring with probabilities $\frac{1}{3}$ and $\frac{2}{3}$, respectively. For each subgroup, we further dissect the operational scenarios, thereby resulting in four subgroups each. For example, dissecting    $(\mathcal{N}_1^{(1)},\Omega_{\textnormal{OG},1}^{(1,3)})$ results in subgroups, $(\mathcal{N}_1^{(1)},\Omega_{\textnormal{OG},1}^{(1,2)})$, $(\mathcal{N}_1^{(1)},\Omega_{\textnormal{OG},2}^{(1,2)})$, $(\mathcal{N}_1^{(1)},\Omega_{\textnormal{OG},3}^{(1,2)})$, and $(\mathcal{N}_1^{(1)},\Omega_{\textnormal{OG},4}^{(1,2)})$. The subgroup $(\mathcal{N}_1^{(1)},\Omega_{\textnormal{OG},1}^{(1,2)})$ occurs with a probability $\phi_{G,1,1}^{(1,2)} = \phi_{1}^{(1)}(\phi_{1,1}^{(2)} \times \phi_{2,1}^{(2)}) = \frac{1}{3}\left(\frac{2}{5} \times \frac{2}{6}\right)$. Similarly, all other values are derived. As a result, we obtain
  eight multi-horizon scenario trees, which can be solved independently. 
\end{myex} 
\input{MHEG_Strategic_Operational_dissection.tex}

\section{Proofs of Theoretical Results}
\label{Proofs}
In this Appendix, we provide the proofs of  Propositions \ref{Prop_MHEG} and \ref{Prop_MHNG}.

\setcounter{myprop}{4}

\begin{myprop}
Consider \textit{MHSP} given in \eqref{MHSP1}.
The following chain of inequalities holds true:
\begin{enumerate}[label={(\alph*)}]
\item      $\textit{MHEG}(1,\gamma) \leq \textit{MHEG}(2,\gamma) \leq \ldots  \leq \textit{MHEG}(\ell,\gamma)  \leq \ldots \leq \textit{MHEOG}(\gamma)$, for a fixed operational level $\gamma$, 
\item      $\textit{MHEG}(\ell,1) \leq \textit{MHEG}(\ell,2) \leq \ldots  \leq \textit{MHEG}(\ell,\gamma)  \leq \ldots \leq \textit{MHESG}(\ell)$, for a fixed strategic level $\ell$.
\end{enumerate}
\end{myprop}
\proof{
\begin{enumerate}[label={(\alph*)}]
    \item Consider a multi-horizon group subproblem $\text{MHG}\big(\mathcal{N}_i^{(\ell+1)},\Omega_{\text{OG},k}^{(\ell+1, \gamma)}\big)$ and let 
    $(\hat{x}_n,\hat{y}_n^{\omega,\tau})$ be its optimal solution, where we ignore the dependence of this solution on subgroup $(i,k)$ for notational simplicity.
    Let $(\V{\mathcal{N}_{i'}^{(\ell)}}, \Omega_{\text{OG},k}^{(\ell, \gamma)})$ be its refinement to level $\ell  $. By definition, we obtain
    \begin{align*}
\textit{MHG}\big(\mathcal{N}_i^{(\ell+1)},\Omega_{\text{OG},k}^{(\ell+1, \gamma)}\big)
    & = 
    \sum_{t\in\mathcal{H}}
\sum\limits_{n \in \mathcal{N}_i^{(\ell+1)} \cap \mathcal{N}_t} \hat{\Pi}_{i,n}^{(\ell+1)} \Big( c_n \hat{x}_n + \sum\limits_{\omega \in \Omega_{n,j}^{(\ell+1,\gamma)}} \hat{\pi}_{\omega}^{(\gamma)} \sum\limits_{\tau \in \mathcal{T}_t}  q_{n}^{\omega,\tau} \hat{y}_{n}^{\omega,\tau} \Big) \\
& = \sum\limits_{t \in \mathcal{H}} \sum\limits_{\mathbf{s} \in \mathcal{S}_i^{(\ell+1)} }\hat{\Pi}_{i,\mathbf{s}}^{(\ell+1)}\Big( c^{\mathbf{s}}_t \hat{x}^{\mathbf{s}}_t + \sum\limits_{\omega \in \Omega_t} \hat{\pi}^{\omega}_{\mathbf{s},t}\sum\limits_{\tau \in \mathcal{T}_t}  q_{\mathbf{s},t}^{\omega,\tau} \hat{y}_{\mathbf{s},t}^{\omega,\tau} \Big) \\
& = \sum\limits_{t \in \mathcal{H}}                
\V{\sum\limits_{\{  \mathcal{S}_{i'}^{(\ell)} \subseteq \mathcal{S}_i^{(\ell+1)}\} }} \V{\sum\limits_{\mathbf{s} \in \mathcal{S}_{i'}^{(\ell)} }} 
\V{\frac{\phi_{i'}^{(\ell )}}{\phi_i^{(\ell+1)}}\hat{\Pi}_{i',\mathbf{s}}^{(\ell)}}
\Big( c^{\mathbf{s}}_t \hat{x}^{\mathbf{s}}_t + \sum\limits_{\omega \in \Omega_t} \hat{\pi}^{\omega}_{\mathbf{s},t}\sum\limits_{\tau \in \mathcal{T}_t}  q_{\mathbf{s},t}^{\omega,\tau} \hat{y}_{\mathbf{s},t}^{\omega,\tau} \Big).
\end{align*} 
Multiplying the above inequality by \VV{$\phi_{G,i,k}^{(\ell+1,\gamma)} = \phi_i^{(\ell+1)} \prod_{n \in \mathcal{N}_i^{(\ell+1)}} \phi_{n,j}^{(\gamma)}$} and  summing over $k \in [M_{i, \ell+1, \gamma}]$, we obtain 
\begin{align*}
    &  \sum_{k \in [M_{i,\ell+1,\gamma} ]} \phi_{G,i,k}^{(\ell+1,\gamma)}  \textit{MHG}\big(\mathcal{N}_i^{(\ell+1)},\Omega_{\text{OG},k}^{(\ell+1, \gamma)}\big) \\
      &= 
      \sum_{k \in [M_{i,\ell+1,\gamma} ]}  \Big( \prod_{n \in \mathcal{N}_i^{(\ell+1)}} \phi_{n,j}^{(\gamma)} \Big) \bigg( \sum\limits_{t \in \mathcal{H}} \V{\sum\limits_{\{ \mathcal{S}_{i'}^{(\ell)} \subseteq \mathcal{S}_i^{(\ell+1)}\} }} \V{\sum\limits_{\mathbf{s} \in \mathcal{S}_{i'}^{(\ell)} }} 
 \V{\phi_{i'}^{(\ell )}}\V{\hat{\Pi}_{i,',\mathbf{s}}^{(\ell)}}
\Big( c^{\mathbf{s}}_t \hat{x}^{\mathbf{s}}_t + \sum\limits_{\omega \in \Omega_t} \hat{\pi}^{\omega}_{\mathbf{s},t}\sum\limits_{\tau \in \mathcal{T}_t}  q_{\mathbf{s},t}^{\omega,\tau} \hat{y}_{\mathbf{s},t}^{\omega,\tau} \Big) \bigg) \\
&= 
         \sum\limits_{t \in \mathcal{H}} \V{\sum\limits_{\{  \mathcal{S}_{i'}^{(\ell)} \subseteq \mathcal{S}_i^{(\ell+1)}\} }} \sum_{k \in [M_{i,\ell+1,\gamma} ]}  \Big( \prod_{n \in \mathcal{N}_i^{(\ell+1)}} \phi_{n,j}^{(\gamma)} \Big) \bigg(\V{\sum\limits_{\mathbf{s} \in \mathcal{S}_{i'}^{(\ell)} }} 
\V{\phi_{i'}^{(\ell )}}\V{\hat{\Pi}_{i,',\mathbf{s}}^{(\ell)}}
\Big( c^{\mathbf{s}}_t \hat{x}^{\mathbf{s}}_t + \sum\limits_{\omega \in \Omega_t} \hat{\pi}^{\omega}_{\mathbf{s},t}\sum\limits_{\tau \in \mathcal{T}_t}  q_{\mathbf{s},t}^{\omega,\tau} \hat{y}_{\mathbf{s},t}^{\omega,\tau} \Big) \bigg) \\
&= 
         \sum\limits_{t \in \mathcal{H}} \V{\sum\limits_{\{ \mathcal{S}_{i'}^{(\ell)} \subseteq \mathcal{S}_i^{(\ell+1)}\} }}  \sum_{k \in [M_{\V{i'},\ell,\gamma} ]}  \Big( \prod_{n \in \mathcal{N}_{\V{i'}}^{(\ell)}} \phi_{n,j}^{(\gamma)} \Big) \bigg(\V{\sum\limits_{\mathbf{s} \in \mathcal{S}_{i'}^{(\ell)} }} 
\V{\phi_{i'}^{(\ell )}}\V{\hat{\Pi}_{i,',\mathbf{s}}^{(\ell)}}
\Big( c^{\mathbf{s}}_t \hat{x}^{\mathbf{s}}_t + \sum\limits_{\omega \in \Omega_t} \hat{\pi}^{\omega}_{\mathbf{s},t}\sum\limits_{\tau \in \mathcal{T}_t}  q_{\mathbf{s},t}^{\omega,\tau} \hat{y}_{\mathbf{s},t}^{\omega,\tau} \Big) \bigg) \\ 
&= 
         \sum\limits_{t \in \mathcal{H}} \V{\sum\limits_{\{  \mathcal{S}_{i'}^{(\ell)} \subseteq \mathcal{S}_i^{(\ell+1)}\} }} \sum_{k \in [M_{\V{i'},\ell,\gamma} ]}  \phi_{G,\V{i'},k}^{(\ell,\gamma)} \bigg(\V{\sum\limits_{\mathbf{s} \in \mathcal{S}_{i'}^{(\ell)} } 
  \hat{\Pi}_{i',\mathbf{s}}^{(\ell)}}
\Big( c^{\mathbf{s}}_t \hat{x}^{\mathbf{s}}_t + \sum\limits_{\omega \in \Omega_t} \hat{\pi}^{\omega}_{\mathbf{s},t}\sum\limits_{\tau \in \mathcal{T}_t}  q_{\mathbf{s},t}^{\omega,\tau} \hat{y}_{\mathbf{s},t}^{\omega,\tau} \Big) \bigg),
\end{align*}
where the third equality is obtained by combining identical subgroups obtained by refining strategic level $\ell+1$ to $\ell$. Specifically, the \V{${i'}^{th}$} subgroup at strategic level $\ell$ obtained by dissecting $i^{th}$ subgroup at strategic level $\ell\VV{+1}$ repeats \V{$m_{\gamma}^{\vert \mathcal{N}_{i,i'} \vert }$} times, where \V{$\mathcal{N}_{i,i'} =\mathcal{N}_{i}^{(\ell+1)} \backslash \mathcal{N}_{i'}^{(\ell)}$}. Summing the above equality over $i \in [\hat{m}_{\ell+1}]$ and using the fact that 
$(\hat{x}_n,\hat{y}_n^{\omega,\tau}) $
is suboptimal for the problem $\text{MHG}\big(\V{\mathcal{N}_{i'}^{(\ell)}},\Omega_{\text{OG},k}^{(\ell, \gamma)}\big)$, we obtain the required inequality.   Finally, at the highest strategic level, we obtain $\textit{MHEOG}(\gamma)$ because we have the full set of strategic nodes $\mathcal{N}$ of the scenario tree and each strategic node contains operational scenarios at level $\gamma$. 
\item The result follows from similar arguments as in the previous part using Proposition \ref{Prop_MHEOG}. \qedhere
\end{enumerate}
}

\begin{myprop}
Consider \textit{MHSP} given in \eqref{MHSP1}.
The following chain of inequalities holds true:
\begin{enumerate}[label={(\alph*)}]
\item      $\textit{MHNG}(1,\gamma) \leq \textit{MHNG}(2,\gamma) \leq \ldots  \leq \textit{MHNG}(\ell,\gamma)  \leq \ldots \leq \textit{MHNOG}(\gamma)$,  for a fixed operational level $\gamma$, 
\item     $\textit{MHNG}(\ell,1) \leq \textit{MHNG}(\ell,2) \leq \ldots  \leq \textit{MHNG}(\ell,\gamma)  \leq \ldots \leq \textit{MHESG}(\ell)$,  for a fixed strategic level $\ell$.
\end{enumerate}
\end{myprop}

\proof{
\begin{enumerate}[label={(\alph*)}]
\item We demonstrate that for two consecutive strategic levels $\ell$ and $\ell+1$ of the refinement chain \eqref{eq:ref_chain}, we have  $\textit{MHNG}(\ell,\gamma)\leq 
\textit{MHNG}(\ell+1,\gamma)$, for a fixed operational level $\gamma$.
Let 
$(\hat{x}_{n,i',j},\hat{y}_{n,i',j}^{\omega,\tau})$
and 
$(\tilde{x}_{n,i,j},\tilde{y}_{n,i,j}^{\omega,\tau})$ 
denote the optimal solutions of the multi-horizon subproblems
$\text{MHG}(\V{\mathcal{N}_{i'}^{(\ell)}},\Omega_{\text{OG}, j}^{(\ell,\gamma)})$ and $\text{MHG}(\mathcal{N}_i^{(\ell+1)},\Omega_{\text{OG},j}^{(\ell+1,\gamma)})$, respectively.
From Proposition \ref{Prop3} it follows that 
\begin{align}
& 
\V{\sum_{i'=1}^{\hat{m}_{\ell}}
\phi_{i'}^{(\ell)} \textit{MHG}(\mathcal{N}_{i'}^{(\ell)},\Omega_{\text{OG}, j}^{(\ell,\gamma)})}
=
\V{\sum_{i'=1}^{\hat{m}_{\ell}}
\phi_{i'}^{(\ell)}}
\sum\limits_{t \in \mathcal{H}}\V{\sum\limits_{n \in \mathcal{N}_{i'}^{(\ell)} \cap
\mathcal{N}_t} \!\!\!\!\!\!  
\hat{\Pi}_{i',n}^{(\ell)}}
\Big( c_n \hat{x}_{n,{i',j}}  + \!\!\!\! 
\! \sum\limits_{\omega \in \Omega_{n,j}^{(\ell,\gamma)}} \! \!\!\! \hat{\pi}_{\omega}^{(\gamma)} \sum\limits_{\tau \in \mathcal{T}_t}  q_{n}^{\omega,\tau} \hat{y}_{n,{i',j}}^{\omega,\tau} \Big)
\leq \nonumber \\
& 
\sum_{i=1}^{\hat{m}_{\ell+1}}
\phi_i^{(\ell+1)} \textit{MHG}(\mathcal{N}_i^{(\ell+1)},\Omega_{\text{OG}, j}^{(\ell+1,\gamma)})
= \! \!
\sum_{i=1}^{\hat{m}_{\ell+1}}
\phi_i^{(\ell+1)}
\sum\limits_{t \in \mathcal{H}}
\sum\limits_{n \in \mathcal{N}_i^{(\ell+1)} \cap
\mathcal{N}_t} \!\!\!\!\!\!\! 
\hat{\Pi}_{i,n}^{(\ell+1)}
\Big( c_n \tilde{x}_{n,{i,j}} + \! 
\! \sum\limits_{\omega \in \Omega_{n,j}^{(\ell+1,\gamma)}} \! \! \hat{\pi}_{\omega}^{(\gamma)} \nonumber \\
& \qquad \sum\limits_{\tau \in \mathcal{T}_t}  q_{n}^{\omega,\tau} \tilde{y}_{n,{i,j}}^{\omega,\tau} \Big), \  j \in [m_\gamma].
\label{eq_44}
\end{align}
Multiplying the cost of each strategic node $n$ by $\phi_{n,j}^{(\gamma)}$ and summing up for all $j \in [m_{\gamma}]$ in \eqref{eq_44},
we get
\begin{align*}
& 
\textit{MHNG}(\ell,\gamma) =
\V{\sum_{i'=1}^{\hat{m}_{\ell}}
\phi_{i'}^{(\ell)}}
\sum\limits_{t \in \mathcal{H}}
\V{\sum\limits_{n \in \mathcal{N}_{i'}^{(\ell)} \cap
\mathcal{N}_t}} \!
\sum_{j=1}^{m_{\gamma}}
\phi_{n,j}^{(\gamma)}
\V{\hat{\Pi}_{i',n}^{(\ell)}}
\Big( c_n \hat{x}_{n,{i',j}} + \! 
\! \sum\limits_{\omega \in \Omega_{n,j}^{(\ell,\gamma)}} \! \! \hat{\pi}_{\omega}^{(\gamma)} \sum\limits_{\tau \in \mathcal{T}_t}  q_{n}^{\omega,\tau} \hat{y}_{n,{i',j}}^{\omega,\tau} \Big)
\leq \nonumber \\
& 
\sum_{i=1}^{\hat{m}_{\ell+1}}
\phi_i^{(\ell+1)}
\sum\limits_{t \in \mathcal{H}}
\sum\limits_{n \in \mathcal{N}_i^{(\ell+1)} \cap
\mathcal{N}_t}
\sum_{j=1}^{m_{\gamma}}
\phi_{n,j}^{(\gamma)}
\hat{\Pi}_{i,n}^{(\ell+1)}
\Big( c_n \tilde{x}_{n,{i,j}}  +
\sum\limits_{\omega \in \Omega_{n,j}^{(\ell+1,\gamma)}} \! \! \hat{\pi}_{\omega}^{(\gamma)} \sum\limits_{\tau \in \mathcal{T}_t}  q_{n}^{\omega,\tau} \tilde{y}_{n,{i,j}}^{\omega,\tau} \Big)
= \\
& \textit{MHNG}(\ell+1,\gamma),
\end{align*}
which proves the thesis.
At the highest strategic level, we obtain $\textit{MHNOG}(\gamma)$  because we have the full set of strategic nodes $\mathcal{N}$ of the multi-horizon scenario tree and each strategic node contains operational scenarios at level $\gamma$. 

\item 
We demonstrate that for two consecutive operational levels $\gamma$ and $\gamma+1$ of the refinement chain \eqref{refinement_chain_operational_common_prob}, we have  $\textit{MHNG}(\ell,\gamma)\leq 
\textit{MHNG}(\ell,\gamma+1)$, for a fixed strategic level $\ell$.
Let $(\hat{x}_{n,{i,j'}},\hat{y}_{n,{i,j'}}^{\omega,\tau})$ 
and 
$(\bar{x}_{n,{i,j}},\bar{y}_{n,{i,j}}^{\omega,\tau})$ 
denote the optimal solutions of the multi-horizon subproblems 
$\textit{MHG}(\mathcal{N}_i^{(\ell)},\Omega_{\text{OG},j'}^{(\ell,\gamma)})$ and $\textit{MHG}(\mathcal{N}_i^{(\ell)},\Omega_{\text{OG},j}^{(\ell,\gamma+1)})$, respectively.
From Proposition \ref{Prop4} it follows that 
\begin{align*}
& 
\sum\limits_{t \in \mathcal{H}}\sum\limits_{n \in \mathcal{N}_i^{(\ell)} \cap
\mathcal{N}_t} \! 
\V{\sum_{j'=1}^{m_{\gamma}}
\phi_{n,j'}^{(\gamma)}} \hat{\Pi}_{i,n}^{(\ell)}
\Big( c_n \hat{x}_{n,{i,j'}} + \!\!\!\!\! 
\! \sum\limits_{\omega \in \Omega_{n,j'}^{(\ell,\gamma)}} \! \!\!\!\! \hat{\pi}_{\omega}^{(\gamma)} \sum\limits_{\tau \in \mathcal{T}_t}  q_{n}^{\omega,\tau} \hat{y}_{n,{i,j'}}^{\omega,\tau} \Big)
\leq
\sum\limits_{t \in \mathcal{H}}\sum\limits_{n \in \mathcal{N}_i^{(\ell)} \cap
\mathcal{N}_t} \! 
\sum_{j=1}^{m_{\gamma+1}} \!\!
\phi_{n,j}^{(\gamma+1)} \hat{\Pi}_{i,,n}^{(\ell)}
\nonumber \\
&\qquad \Big( c_n \bar{x}_{n,{i,j}} + \sum\limits_{\omega \in \Omega_{n,j}^{(\ell,\gamma+1)}} \! \! \hat{\pi}_{\omega}^{(\gamma+1)} \sum\limits_{\tau \in \mathcal{T}_t}  q_{n}^{\omega,\tau} \bar{y}_{n,{i,j}}^{\omega,\tau} \Big), \  i \in [\hat{m}_\ell].
\end{align*}
Multiplying both sides by $\phi_i^{(\ell)}$ 
and summing up for all  $i \in [\hat{m}_{\ell+1}]$,
we get $\textit{MHNG}(\ell,\gamma)\leq \textit{MHNG}(\ell,\gamma+1)$.
{Furthermore, at the highest operational level, we obtain $\textit{MHESG}(\ell)$  because we have the full set of operational scenarios $\Omega_n$ at each strategic node $n$ of strategic level $\ell$. }
\qedhere
\end{enumerate}
}

\newpage

\setcounter{table}{0}
\renewcommand{\thetable}{D.\arabic{table}}
\section{Mathematical Formulation of the GTEP Problem}
\label{sec:formulation}
In this Appendix\VV{,} we provide the formulation of the multi-horizon stochastic optimization model, which serves as the main testing framework.
The notation for the problem is reported in Tables \ref{tab:sets} and \ref{tab:param}.

\begin{table}[ht!]
\footnotesize
    \centering
    \renewcommand{\arraystretch}{1}
    \begin{tabular}{ll}
        \toprule
        {Set} & {Description} \\
        \midrule
        $\mathcal{H} = \{t\,|\, t=0,\dots,H\}$ & Set of strategic stages \\
        $\mathcal{N}_t = \{n\,|\,n=1,\dots,|\mathcal{N}_t|\}$ & Set of ordered nodes of the strategic tree at stage $t \in \mathcal{H}$ \\
        $\mathcal{T}_t = \{\tau\,|\,\tau=1,\dots,O_t\}$ & Set of operational stages at strategic stage $t \in \mathcal{H}$ \\
        $\Omega_n = \{\omega\,|\,\omega=1,\dots,|\Omega_n|\}$ & Set of operational scenarios at strategic node $n$ \\
        $\mathcal{Z} = \{z\,|\, z=1,\dots,Z\}$ & Set of power system zones \\
        $\mathcal{L} = \{l\,|\, l=1,\dots,L\}$ & Set of transmission lines \\
        $BS_z \subset \mathcal{L}$ & Set of transmission lines entering zone $z \in \mathcal{Z}$ \\
        $FS_z \subset \mathcal{L}$ & Set of transmission lines leaving zone $z \in \mathcal{Z}$ \\
        $\mathcal{K} = \{k\,|\, k=1,\dots,K\}$ & Set of power plants (with no storage capacity) \\
        $\mathcal{K}^{\text{RES}} \subset \mathcal{K}$ & Set of non-programmable renewable power plants \\
        $\mathcal{K}_z \subset \mathcal{K}$ & Set of power plants in zone $z \in \mathcal{Z}$ \\
        $\mathcal{B} = \{b\,|\, b=1,\dots,B\}$ & Set of storage facilities \\
        $\mathcal{B}_z \subset \mathcal{B}$ & Set of storage facilities in zone $z \in \mathcal{Z}$ \\
        \bottomrule
    \end{tabular}
    \caption{Sets of the GTEP problem.}
    \label{tab:sets}
\end{table}

\renewcommand{\arraystretch}{1}
\begingroup
\centering
\footnotesize
\begin{longtable}{>{\raggedright}p{2cm} >{\raggedright}p{2cm} p{10cm}}
\toprule
{Symbol} & {Unit} & {Description} \\
\midrule
\multicolumn{3}{l}{{Deterministic Parameters}} \\
\midrule
$\hat{F}_l$ & [MW] & Existing capacity of transmission line $l \in \mathcal{L}$ at the beginning of the planning period \\
$\overline{F}_l$ & [MW] & Maximum capacity of transmission line $l \in \mathcal{L}$ at the end of the planning period \\
$\hat{V}_k$ & [MW] & Existing capacity of power plant $k \in \mathcal{K}$ at the beginning of the planning period \\
$\overline{V}_k$ & [MW] & Maximum capacity of power plant $k \in \mathcal{K}$ at the end of the planning period \\
$M_k$ & [ton/MW] & CO$_2$ emission coefficient of power plant $k \in \mathcal{K}$ \\
$\tilde{c}_{t,b}$ & [\euro/MW] & Unit operational cost of storage facility $b \in \mathcal{B}$ at strategic stage $t \in \mathcal{H}$ \\
$\hat{U}_b$ & [MW] & Existing capacity of storage facility $b \in \mathcal{B}$ at the beginning of the planning period \\
$\overline{U}_b$ & [MW] & Maximum capacity of storage facility $b \in \mathcal{B}$ at the end of the planning period \\
$E_b$ & [MW] & Energy-to-power ratio of storage facility $b \in \mathcal{B}$ \\
$\hat{w}_b^0$ & [MW] & Initial energy level of storage facility $b \in \mathcal{B}$ \\
$\hat{w}_b^F$ & [MW] & Minimum energy level of storage facility $b \in \mathcal{B}$ at the end of each representative day \\
$\lambda_b^+$ & [MW] & Loss coefficient for energy injected into storage facility $b \in \mathcal{B}$ \\
$\lambda_b^-$ & [MW] & Loss coefficient for energy withdrawn from storage facility $b \in \mathcal{B}$ \\
$D_{t,z}^{\tau}$ & [MW] & Electricity demand in zone $z \in \mathcal{Z}$ in hour $\tau \in \mathcal{T}_t$, $t \in \mathcal{H}$ \\
$\overline{M}_t$ & [ton] & Maximum daily CO$_2$ emissions at strategic stage $t \in \mathcal{H}$ \\
$\varphi_t$ & [-] & Minimum share of the total generation capacity that must be allocated to non-programmable renewable power plants at stage $t \in \mathcal{H}$ \\
$ND_t$ & [-] & Number of days in strategic stage $t \in \mathcal{H}$ \\
\midrule
\multicolumn{3}{l}{{Strategic Stochastic Parameters}} \\
\midrule
$C_{n,l}$ & [\euro/MW] & Capacity expansion cost for transmission line $l \in\mathcal{L}$ at strategic node $n \in \mathcal{N}_t$, $t \in \mathcal{H}$ \\
$\hat{C}_{n,k}$ & [\euro/MW] & Capacity expansion cost for power plant $k \in\mathcal{K}$ at strategic node $n \in \mathcal{N}_t$, $t \in \mathcal{H}$ \\
$\hat{c}_{n,k}$ & [\euro/MW] & Unit production cost of power plant $k \in \mathcal{K}$ at strategic node $n \in \mathcal{N}_t$, $t \in \mathcal{H}$ \\
$\tilde{C}_{n,b}$ & [\euro/MW] & Capacity expansion cost for storage facility $b \in \mathcal{B}$ at strategic node $n \in \mathcal{N}_t$, $t \in \mathcal{H}$ \\
$a(n)$ & [-] & Ancestor of node $n \in \mathcal{N}_t$, $t \in \mathcal{H} \setminus \{0\}$ in the strategic scenario tree \\
$\pi_n$ & [-] & Probability of strategic node $n \in \mathcal{N}_t$, $t \in \mathcal{H}$ \\
\midrule
\multicolumn{3}{l}{{Operational Stochastic Parameters}} \\
\midrule
$R_{n,k}^{\omega,\tau}$ & [-] & Capacity factor of power plant $k \in \mathcal{K}$ at hour $\tau \in \mathcal{T}_t$, scenario $\omega \in \Omega_n$, derived from strategic node $n \in \mathcal{N}_t$, $t \in \mathcal{H}$ \\
$I_{n,s}^{\omega,\tau}$ & [MW] & Energy natural inflow in storage facility $b \in \mathcal{B}$ at hour $\tau \in \mathcal{T}_t$, scenario $\omega \in \Omega_n$, derived from strategic node $n \in \mathcal{N}_t$, $t \in \mathcal{H}$ \\
$\pi_{\omega}$ & [-] & Probability of scenario $\omega \in \Omega_n$, derived from strategic node $n \in \mathcal{N}_t$, $t \in \mathcal{H}$ \\
\midrule
\multicolumn{3}{l}{{Strategic Decision Variables}} \\
\midrule
$f_{n,l} \geq 0$ & [MW] & New installed capacity of transmission line $l \in \mathcal{L}$ in strategic node $n \in \mathcal{N}_t$, $t \in \mathcal{H}$ \\
$F_{n,l} \geq 0$ & [MW] & Total installed capacity of transmission line $l \in \mathcal{L}$ in strategic node $n \in \mathcal{N}_t$, $t \in \mathcal{H}$ \\
$v_{n,k} \geq 0$ & [MW] & New installed capacity of power plant $k \in \mathcal{K}$ in strategic node $n \in \mathcal{N}_t$, $t \in \mathcal{H}$ \\
$V_{n,k} \geq 0$ & [MW] & Total installed capacity of power plant $k \in \mathcal{K}$ in strategic node $n \in \mathcal{N}_t$, $t \in \mathcal{H}$ \\
$u_{n,b} \geq 0$ & [MW] & New installed capacity of storage facility $b \in \mathcal{B}$ in strategic node $n \in \mathcal{N}_t$, $t \in \mathcal{H}$ \\
$U_{n,b} \geq 0$ & [MW] & Total installed capacity of storage facility $b \in \mathcal{B}$ in strategic node $n \in \mathcal{N}_t$, $t \in \mathcal{H}$ \\
\midrule
\multicolumn{3}{l}{{Operational Decision Variables}} \\
\midrule
$d_{n,l}^{\omega,\tau} \geq 0$ & [MW] & Power flow on transmission line $l \in \mathcal{L}$ at time $\tau \in \mathcal{T}_t$ in scenario $\omega \in \Omega_n$, $t \in \mathcal{H}$ \\
$p_{n,k}^{\omega,\tau} \geq 0$ & [MW] & Power production of power plant $k \in \mathcal{K}$ at time $\tau \in \mathcal{T}_t$ in scenario $\omega \in \Omega_n$, $t \in \mathcal{H}$ \\
$w_{n,b,+}^{\omega,\tau} \geq 0$ & [MW] & Amount of electricity charged into storage $b \in \mathcal{B}$ at time $\tau \in \mathcal{T}_t$ in scenario $\omega \in \Omega_n$, $t \in \mathcal{H}$ \\
$w_{n,b,-}^{\omega,\tau} \geq 0$ & [MW] & Amount of electricity discharged from storage $b \in \mathcal{B}$ at time $\tau \in \mathcal{T}_t$ in scenario $\omega \in \Omega_n$, $t \in \mathcal{H}$ \\
$w_{n,b}^{\omega,\tau} \geq 0$ & [MWh] & Amount of electricity stored in storage $b \in \mathcal{B}$ at time $\tau \in \mathcal{T}_t$ in scenario $\omega \in \Omega_n$, $t \in \mathcal{H}$ \\
\bottomrule
\caption{Parameters and variables of the GTEP problem.}\\
\label{tab:param}
\end{longtable}
\endgroup

\noindent The MHSP derived for the GTEP problem is formulated as follows: 
\begin{subequations}
\begin{align}
 \min_{\mathbf{f}, \mathbf{v}, \mathbf{u}, \mathbf{d}, \mathbf{p},\mathbf{w}}\limits & \sum  \limits_{t=0}^{H}\sum\limits_{n \in \mathcal{N}_t} \pi_{n}
 \Bigg( 
\sum_{l \in \mathcal{L}} C_{n,l} \ f_{n,l} +
\sum_{k \in \mathcal{K}} \hat{C}_{n,k} \ v_{n,k} +
\sum_{b \in \mathcal{B}} \tilde{C}_{n,b} \ u_{n,b} + \nonumber \\
& \ \qquad \quad \quad + ND_{t} \sum\limits_{\omega \in \Omega_n} \pi_{\omega}\sum\limits_{\tau \in \mathcal{T}_t} 
\Big(
\sum_{k \in \mathcal{K}} \hat{c}_{n,k} \ p_{n,k}^{\omega,\tau} +
\sum_{b \in \mathcal{B}} \tilde{c}_{t,b} \  w_{n,b,-}^{\omega,\tau}
\Big)
\Bigg)  \label{obj} \\
\textrm{ s.t. }
&  F_{0,l} = \hat{F}_l + f_{0,l}, \quad l \in \mathcal{L}, \label{line_inv0}\\
&  F_{n,l} = F_{a(n),l} + f_{n,l}, \quad l \in \mathcal{L}, \ n \in \mathcal{N} \setminus \{0\}, \label{line_inv}\\
& F_{n,l} \leq \overline{F}_{l}, \quad l \in \mathcal{L}, \ n \in \mathcal{N}_H, \label{max_lin} \\
&  V_{0,k} = \hat{V}_k + v_{0,k}, \quad k \in \mathcal{K}, \label{gen_inv0}\\
&  V_{n,k} = V_{a(n),k} + v_{n,k}, \quad k \in \mathcal{K}, \ n \in \mathcal{N} \setminus \{0\}, \label{gen_inv}\\
& V_{n,k} \leq \overline{V}_{k}, \quad k \in \mathcal{K}, \ n \in \mathcal{N}_H, \label{max_cap} \\
& U_{0,b} = \hat{U}_{b} + u_{0,b}, \quad b \in \mathcal{B}, \label{sto_inv0}\\
& U_{n,b} = U_{a(n),b} + u_{n,b}, \quad b \in \mathcal{B},\ n \in \mathcal{N} \setminus \{ 0 \}, \label{sto_inv}\\
& U_{n,b} \leq \overline{U}_{b}, \quad b \in \mathcal{B}, \ n \in \mathcal{N}_H, \label{max_sto} \\
& \sum_{k \in \mathcal{K}_z} p_{n,k}^{\omega,\tau} +
\sum_{b \in \mathcal{B}_z} w_{n,b,-}^{\omega,\tau} +
\sum_{l \in BS_z} d_{n,l}^{\omega,\tau}  
= 
D_{t,z}^{\tau} +
\sum_{b \in \mathcal{B}_z} w_{n,b,+}^{\omega,\tau} +
\sum_{l \in FS_z} d_{n,l}^{\omega,\tau}, \nonumber \\
& z \in \mathcal{Z}, \ n \in \mathcal{N}_t,\  \tau \in \mathcal{T}_t,\  \omega \in \Omega_n,\ t\in\mathcal{H}, \label{eq:balance} \\
&  d_{n,l}^{\omega,\tau} \leq F_{n,l}, \quad l \in \mathcal{L}, n \in \mathcal{N}_t,\  \tau \in \mathcal{T}_t,\  \omega \in \Omega_n,\ t\in\mathcal{H}, \label{line_flow}\\
&  p_{n,k}^{\omega,\tau} \leq R_{n,k}^{\omega,\tau} V_{n,k}, \quad k \in \mathcal{K}, \ n \in \mathcal{N}_t,\  \tau \in \mathcal{T}_t,\  \omega \in \Omega_n,\ t\in\mathcal{H}, \label{gen_prod}\\
& w_{n,b}^{\omega,\tau} \leq E_b U_{n,b}, \quad b \in \mathcal{B}, \ n \in \mathcal{N}_t,\  \tau \in \mathcal{T}_t,\  \omega \in \Omega_n,\ t\in\mathcal{H}, \label{max_store} \\
& w_{n,b,-}^{\omega,\tau} \leq U_{n,b}, \quad b \in \mathcal{B}, \ n \in \mathcal{N}_t,\  \tau \in \mathcal{T}_t,\  \omega \in \Omega_n,\ t\in\mathcal{H}, \label{max_dis}\\
& w_{n,b,+}^{\omega,\tau} \leq U_{n,b}, \quad b \in \mathcal{B}, \ n \in \mathcal{N}_t,\  \tau \in \mathcal{T}_t,\  \omega \in \Omega_n,\ t\in\mathcal{H}, \label{max_ch}\\
& w_{n,b}^{\omega,\tau} =  w_{n,b}^{\omega,\tau-1} + \lambda_{b}^+ w_{n,b,+}^{\omega,\tau} - \lambda_{b}^- w_{n,b,-}^{\omega,\tau} + I_{n,b}^{\omega,\tau},  \,\, b \in \mathcal{B}, n \in \mathcal{N}_t,  \tau \in \mathcal{T}_t\setminus\{1\}, \nonumber \\  
& \quad \omega \in \Omega_n, t\in\mathcal{H}, \label{Bat_bal}\\
& w_{n,b}^{\omega,1} =  \hat{w}_b^0  +  \lambda_{b}^+ w_{n,b,+}^{\omega,1} -  \lambda_{b}^- w_{n,b,-}^{\omega,1} + I_{n,b}^{\omega,1} ,  \,\, b \in \mathcal{B},  \omega \in \Omega_n, n \in \mathcal{N}_t, t \in \mathcal{H},\label{Bat_bal_01}\\
& w_{n,b}^{\omega,O_t} \geq  \hat{w}_b^F,  \quad b \in \mathcal{B},  \ \omega \in \Omega_n, \ n \in \mathcal{N}_t, \ t \in \mathcal{H}, \label{final_sto}\\
& \sum_{k \in \mathcal{K}} 
\sum_{\tau \in \mathcal{T}_t} M_k p_{n,k}^{\omega,\tau} 
\leq \overline{M}_t,\quad  n \in \mathcal{N}_t, \ t\in\mathcal{H}, \ \omega \in \Omega_n,  \label{max_CO2} \\
& \sum_{k \in \mathcal{K}^{\text{RES}}} 
 V_{n,k} \geq \varphi_t \sum_{k \in \mathcal{K}} V_{n,k},\quad  n \in \mathcal{N}_t, \ t\in\mathcal{H}. \label{min_RES}
\end{align}
\end{subequations}
The objective function \eqref{obj} minimizes the expected installation costs of transmission lines, power plants and storage facilities, plus the expected operational costs of generation and storage facilities. 
Constraint \eqref{line_inv0}$-$\eqref{max_sto} control the investment decisions for the power system.
Specifically, equation \eqref{line_inv0} defines the total installed capacity $F_{0,l}$ of the power line $l$ in the root node as the sum of the existing capacity $\hat{F}_l$ and the new installed capacity $f_{0,l}$. Constraint \eqref{line_inv} works similarly for the subsequent nodes $n \in \mathcal{N}\setminus\{0\}$. 
Constraint \eqref{max_lin} imposes an upper bound $\overline{U}_b$ for the transmission capacity of line $l$ at the horizon.
Equations \eqref{gen_inv0} and \eqref{gen_inv} compute the total installed capacity of the power plants in the root node and in the subsequent nodes respectively, while constraint \eqref{max_cap} limits the installed capacity for power plants.
Constraints \eqref{sto_inv0}, \eqref{sto_inv} and \eqref{max_sto} work similarly for the storage facilities. 
Constraint \eqref{eq:balance} enforces the hourly energy balance in each zone $z$ of the power system:
the electricity available is the sum of 
(i) the power output $\sum_{k \in \mathcal{K}_z} p_{n,k}^{\omega,\tau}$ of generation plants,
(ii) the power withdrawn from storage facilities $\sum_{b \in \mathcal{B}_z} w_{n,b,-}^{\omega,\tau}$, and 
(iii) the import $\sum_{l \in BS_z} d_{n,l}^{\omega,\tau}$ from other zones; 
the electricity used is the sum of 
(i) the zonal load $D_{t,z}^{\tau}$,
(ii) the electricity charged in storage facilities $\sum_{b \in \mathcal{B}_z} w_{n,b,+}^{\omega,\tau}$, and 
(iii) the export $\sum_{l \in FS_z} d_{n,l}^{\omega,\tau}$ to other zones. 
Constraints \eqref{line_flow}$-$\eqref{max_store} link the operational and the investment decisions.
In each operational period, the electricity flow on transmission lines is bounded above by the capacity installed in the corresponding strategic node, as stated by constraint \eqref{line_flow}.
Equation \eqref{gen_prod} ensures that the hourly production from generating plants does not exceed the percentage $R_{n,k}^{\omega,\tau}$ of the corresponding available capacity.
For each storage facility, the energy level is bounded above by the product of the available capacity times the energy-to-power ratio $E_b$, as stated by constraint \eqref{max_store}, while the discharge and the charge are bounded above by the installed capacity, as enforced by equations \eqref{max_dis} and \eqref{max_ch}, respectively. 
Equation \eqref{Bat_bal} is the battery electricity balance constraint for any operational period $\tau \in \mathcal{T}_t \setminus \{1\}$, stating that the storage level at operational period $\tau$ is the sum of the storage level in the previous operational period $\tau - 1$, plus the electricity charged in the storage at period $\tau$, minus the storage discharge at period $\tau$, plus the natural inflow at period $\tau$.
Equation \eqref{Bat_bal_01} enforces the storage electricity balance for the first operational period $\tau = 1$. 
Constraint \eqref{final_sto} ensures a minimum storage level at the end of each representative day.
Constraints \eqref{max_CO2} and \eqref{min_RES} impose policy goals that must be met by the power system. Specifically, constraint \eqref{max_CO2} limits the total CO$_2$ emissions in each operational scenario $\omega \in\Omega_n$, $t \in \mathcal{H}$. Constraint \eqref{min_RES} represents targets on renewable penetration: at strategic stage $t \in \mathcal{H}$ the total installed non-programmable renewable generation capacity must represent at least a fraction $\varphi_{t}$ of the total generation capacity. 

\setcounter{table}{0}
\renewcommand{\thetable}{E.\arabic{table}}
\section{Results for the Medium Instance Case}
\label{sec:med_inst}

In this Appendix, we further validate the proposed bounding schemes using a medium-size instance defined by eight strategic scenarios and sixty-four operational scenarios per strategic node.
We adopt the configuration with disjoint groups at both the strategic and operational levels.
Due to the substantially larger number of strategic nodes, the expectation-based formulation becomes computationally intractable; hence, only the node-based approach ($\textit{MHNG}$) is evaluated.
The corresponding results are summarized in Table~\ref{tab:medium_compact}.
The node-based approach produces high-quality solutions with substantial computational savings.
When the full strategic uncertainty is preserved  ($\ell=4$) and only the operational scenarios are grouped, the resulting objective values remain remarkably close to the optimal solution of the full multi-horizon stochastic program.
As expected, the quality of the lower bounds systematically improves with the level of operational refinement $\gamma$.
At the lowest operational level  ($\gamma=1$), the bound equals 269.47 bln \euro{} for $\ell=4$ and 262.12 bln \euro{} for $\ell=1$, corresponding to optimality gaps of 7.9\% and 10.9\%, respectively.
As the operational chain is refined, the approximation accuracy rapidly increases: at $\gamma=5$, the bound reaches 290.02 bln \euro{} ($-0.24$\%) for $\ell=4$, while remaining above 281 bln \euro{} ($-3.46$\%) for $\ell=1$.
The tightest bound is obtained for $\gamma=6$, where $\textit{MHNG
}(4,6)$ yields 290.43 bln \euro{}, only 0.1\% below the full multi-horizon optimum of 290.73 bln \euro{}.
Decreasing the level of strategic refinement from $\ell=4$ to $\ell=3$ leads to a marked deterioration in bound quality, with the optimality gap increasing to 2.96\% at $\gamma=7$.
Further reductions in $\ell$ produce only marginal additional losses, with gaps stabilizing around 3-3.5\%.
Conversely, refining the operational chain consistently improves the bound accuracy across all strategic levels, confirming the effectiveness of the dissection structure in capturing short-term uncertainty.

In terms of computational effort, the proposed node-based dissection method provides substantial savings compared with the complete model.
The full multi-horizon instance ($\gamma=7$, $\ell=4$) requires 17,428~s, whereas the partially aggregated configuration $\textit{MHNG}(4,6)$ converges in 957~s, representing an 18.2 speed-up.
Further aggregation enhances computational efficiency at the expense of accuracy: for example, at $\gamma=4$, solution times drop to 306~s (56.9 times faster), while the optimality gap increases to 0.64\%.
For $\ell=1$, the smallest problems achieve extreme reductions in runtime, down to 269-402~s for $\gamma=4$ and $\gamma=1$  (speed-ups between 65 and 43), with moderate degradation of bound quality (3.9-10.9\%).

Overall, the medium-size experiments confirm the strong scalability of the node-based dissection approach.
Dissecting only the operational uncertainty\VV{,} it consistently produces tight lower bounds within 1\% of optimality for $\gamma\geq4$, while reducing computation times by more than an order of magnitude compared to the full multi-horizon stochastic program.

\begin{table}[ht!]
\scriptsize
\centering
\setlength{\tabcolsep}{2.2pt} 
\begin{tabular}{cc|rrrr|rrrr|rrrr|rrrr}
\hline
\multirow{2}{*}{$\gamma$} & \multirow{2}{*}{$m_\gamma$} &
\multicolumn{4}{c|}{$\ell=4$ ($\hat{m}_{\ell}=1$)} &
\multicolumn{4}{c|}{$\ell=3$ ($\hat{m}_{\ell}=2$)} &
\multicolumn{4}{c|}{$\ell=2$ ($\hat{m}_{\ell}=4$)} &
\multicolumn{4}{c}{$\ell=1$ ($\hat{m}_{\ell}=8$)} \\
\cline{3-6}\cline{7-10}\cline{11-14}\cline{15-18}
& & Bound & Gap & CPU & Spd & Bound & Gap & CPU & Spd &
   Bound & Gap & CPU & Spd & Bound & Gap & CPU & Spd \\
\hline
7 & 1  & 290.73 & -   & 17428 & -   & 282.37 & -2.96 & 1468 & 11.9 & 281.58 & -3.25 & 1436 & 12.1 & 281.37 & -3.33 & 1656 & 10.5 \\
6 & 2  & 290.43 & -0.1 & 957 & 18.2 & 282.15 & -3.04 & 520 & 33.5 & 281.32 & -3.34 & 709 & 24.6 & 281.08 & -3.43 & 774 & 22.5 \\
5 & 4  & 290.02 & -0.2 & 641 & 27.2 & 282.01 & -3.09 & 413 & 42.2 & 281.22 & -3.38 & 353 & 49.4 & 281.01 & -3.46 & 381 & 45.7 \\
4 & 8  & 288.89 & -0.6 & 306 & 56.9 & 280.79 & -3.54 & 176 & 99.0 & 280.02 & -3.82 & 189 & 92.2 & 279.82 & -3.90 & 269 & 64.8 \\
3 & 16 & 285.51 & -1.8 & 361 & 48.3 & 277.39 & -4.81 & 163 &107. & 276.45 & -5.17 & 161 &108. & 276.27 & -5.23 & 232 & 75.1 \\
2 & 32 & 278.80 & -4.3 & 307 & 56.8 & 271.08 & -7.25 & 210 & 83.0 & 270.14 & -7.62 & 227 & 76.8 & 270.00 & -7.68 & 342 & 51.0 \\
1 & 64 & 269.47 & -7.9 & 236 & 73.9 & 262.92 &-10.6 & 229 & 76.1 & 262.19 &-10.9 & 259 & 67.3 & 262.12 &-10.9 & 402 & 43.4 \\
\hline
\end{tabular}
\caption{Results for bounding schemes $\textit{MHNG}(\ell,\gamma)$ at varying levels of the strategic ($\ell$) and operational ($\gamma$) refinement chains in the medium instance.  
Each block reports: \textit{bound (bln~\euro), gap (\%), CPU time (s), and speed-up}.}
\label{tab:medium_compact}
\end{table}

\setcounter{table}{0}
\renewcommand{\thetable}{E.\arabic{table}}
\section{A Heuristic Approach to Compute Upper Bounds for Computationally Intractable Instances}
\label{sec:UB}
In large-scale instances, directly solving the full multi-horizon problem is often computationally prohibitive. 
This Appendix presents a rolling-horizon heuristic, \textit{SFR3}, introduced in \cite{escudero2021multistage}, which can be used to obtain high-quality feasible solutions efficiently.
The method is controlled by four user-defined parameters: 
\begin{itemize}
\item $\hat{h}$, the number of non-relaxed stages; 
\item $\hat{h}^R$, the number of relaxation stages; 
\item $\psi_\text{ST}$, the sample probability for strategic nodes; 
\item $\psi_\text{OP}$, the sample probability for operational nodes. 
\end{itemize}
At iteration $\kappa \in {0,1,\ldots,H+1-\hat{h}}$, representing the current position of the rolling-horizon window, the full multi-horizon problem is replaced by a subproblem defined on a subset of nodes. These nodes are selected by partitioning the set of strategic stages $\mathcal{H}$ into three disjoint groups:
\begin{enumerate}
    \item $\mathcal{H}_\text{NR}(\kappa)=\{\kappa,\ldots,\kappa+\hat{h}-1\}$, the \textit{non-relaxed stages} where all strategic and operational nodes are fully included in the subproblem;
    \item $\mathcal{H}_\text{R}(\kappa)=\{\kappa+\hat{h},\ldots,\kappa+\hat{h}+\hat{h}^R-1\}$, the \textit{relaxation stages} where strategic nodes are sampled with probability $\psi_\text{ST}$ and associated operational nodes with probability $\psi_\text{OP}$; 
    \item $\mathcal{H}_\text{RM}(\kappa)=\{\kappa+\hat{h}+\hat{h}^R,\ldots,H\}$, the \textit{removed stages} which are excluded from the subproblem.  
\end{enumerate}
After solving the subproblem, the decisions corresponding to the first non-relaxed stage are fixed, the horizon is advanced by one stage ($\kappa \leftarrow \kappa+1$), and the procedure is repeated until all stages in $\mathcal{H}$ have been processed. 

Table~\ref{tab:UB} reports results obtained for different parameter configurations. 
\begin{table}[ht!]
\footnotesize
    \centering
    \begin{tabular}{cccc|cr}
    \hline
    \rule{0pt}{12pt}
         $\hat{h}$&  $\hat{h}^R$&  $\psi_\text{ST}$&  $\psi_\text{OP}$&  Objective function value (bln \euro)& CPU time (s)\\ \hline
         2&  0&  --&  --&  294.63& 4128\\ [4pt]
         2&  3&  $\frac{1}{2}$&  $\frac{1}{8}$&  289.83& 6260\\[4pt]
         3&  2&  $\frac{1}{3}$&  $\frac{1}{32}$&  287.73& 9226\\[4pt]
         3&  2&  $\frac{1}{3}$&  $\frac{1}{8}$&  287.67& 13186\\[4pt]
         3&  2&  $\frac{1}{2}$&  $\frac{1}{8}$&  286.93& 15153\\ [4pt] \hline
    \end{tabular}
    \caption{Results of the heuristic method to compute feasible solutions on the large instance for different configurations of the parameters $\hat{h}$, $\hat{h}^R$, $\psi_\text{ST}$, and $\psi_\text{OP}$.}
    \label{tab:UB}
\end{table}
The first row corresponds to a pure two-stage rolling horizon with no relaxation, giving the weakest feasible solution. Adding relaxation stages (second row) significantly improves the bound, while enlarging the non-relaxed window (third row) further strengthens results. 
Increasing the selection probabilities for both strategic and operational nodes (last two rows) yields the tightest feasible solution, with $\hat{h}=3$, $\hat{h}^R=2$, $\psi_\text{ST}=\tfrac{1}{2}$, and $\psi_\text{OP}=\tfrac{1}{8}$ achieving an objective value of 286.93 bln~\euro \,at the cost of higher CPU time.

\end{document}

%% file: Common_Tree.tex

 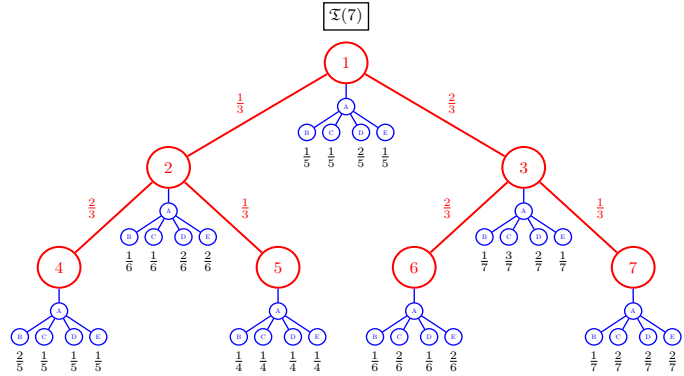
\begin{figure}[ht] 
	 \begin{center}
  \resizebox{0.6\textwidth}{0.33\textwidth}{%
	\begin{tikzpicture}[-, >=stealth', auto, thick, node distance = 1.1 cm]
	\node[state, red, very thick] (1) {$1$};
 \node at (0,1) [rectangle,draw] (MHSG) {$\mathfrak{T}(7)$};
 \node[state, below left=1.6cm and 3cm of 1, red, very thick ] (2)               {$2$};
 \node[below left = 0.2cm and 1.6cm of 1, red, very thick ] (2_probability)  {$\frac{1}{3}$};
\node[state, below right=1.6cm and 3cm of 1, red, very thick] (3)               {$3$};
\node[below right = 0.2cm and 1.6cm of 1, red, very thick ] (3_probability)  {$\frac{2}{3}$};
 \node[state, below left=1.5cm and 1.6cm of 2, red, very thick ] (4)               {$4$};
\node[state, below right=1.5cm and 1.6cm of 2, red, very thick] (5)               {$5$};
 \node[state, below left=1.5cm and 1.6cm  of 3, red, very thick ] (6)               {$6$};
\node[state, below right=1.5cm and 1.6cm of 3, red, very thick] (7)               {$7$};

\node[below left = 0.2cm and 1cm of 2, red, very thick ] (24_probability)  {$\frac{2}{3}$};
\node[below right = 0.2cm and 1cm of 2, red, very thick ] (25_probability)  {$\frac{1}{3}$};
\node[below left = 0.2cm and 1cm of 3, red, very thick ] (36_probability)  {$\frac{2}{3}$};
\node[below right = 0.2cm and 1cm of 3, red, very thick ] (37_probability)  {$\frac{1}{3}$};

\node[state, blue,  below=0.3cm of 1, scale = 0.4] (A1) {A};
\node[state, blue,  below left =0.3cm and 0.55cm of A1, scale = 0.4] (B1) {B};
\node[state, blue,  below left =0.3cm and 0.05cm of A1, scale = 0.4] (C1) {C};
\node[state, blue,   below right =0.3cm and 0.05cm of A1, scale = 0.4] (D1) {D};
\node[state, blue,   below right =0.3cm and 0.55cm of A1, scale = 0.4] (E1) {E};
\node[below =0.005cm of B1 ] (B1_probability)  {$\frac{1}{5}$};
\node[below =0.005cm of C1 ] (C1_probability)  {$\frac{1}{5}$};	
\node[below =0.005cm of D1 ] (D1_probability)  {$\frac{2}{5}$};	
\node[below =0.005cm of E1 ] (E1_probability)  {$\frac{1}{5}$};	

\node[state, blue,  below=0.3cm of 2, scale = 0.4] (A2) {A};
\node[state, blue,  below left =0.3cm and 0.55cm of A2, scale = 0.4] (B2) {B};
\node[state, blue,  below left =0.3cm and 0.05cm of A2, scale = 0.4] (C2) {C};
\node[state, blue,   below right =0.3cm and 0.05cm of A2, scale = 0.4] (D2) {D};
\node[state, blue,   below right =0.3cm and 0.55cm of A2, scale = 0.4] (E2) {E};
\node[below =0.005cm of B2 ] (B2_probability)  {$\frac{1}{6}$};
\node[below =0.005cm of C2 ] (C2_probability)  {$\frac{1}{6}$};	
\node[below =0.005cm of D2 ] (D2_probability)  {$\frac{2}{6}$};	
\node[below =0.005cm of E2 ] (E2_probability)  {$\frac{2}{6}$};		

\node[state, blue,  below=0.3cm of 3, scale = 0.4] (A3) {A};
\node[state, blue,  below left =0.3cm and 0.55cm of A3, scale = 0.4] (B3) {B};
\node[state, blue,  below left =0.3cm and 0.05cm of A3, scale = 0.4] (C3) {C};
\node[state, blue,   below right =0.3cm and 0.05cm of A3, scale = 0.4] (D3) {D};
\node[state, blue,   below right =0.3cm and 0.55cm of A3, scale = 0.4] (E3) {E};
\node[below =0.005cm of B3 ] (B3_probability)  {$\frac{1}{7}$};
\node[below =0.005cm of C3 ] (C3_probability)  {$\frac{3}{7}$};	
\node[below =0.005cm of D3 ] (D3_probability)  {$\frac{2}{7}$};	
\node[below =0.005cm of E3 ] (E3_probability)  {$\frac{1}{7}$};	

\node[state, blue,  below=0.3cm of 4, scale = 0.4] (A4) {A};
\node[state, blue,  below left =0.3cm and 0.55cm of A4, scale = 0.4] (B4) {B};
\node[state, blue,  below left =0.3cm and 0.05cm of A4, scale = 0.4] (C4) {C};
\node[state, blue,   below right =0.3cm and 0.05cm of A4, scale = 0.4] (D4) {D};
\node[state, blue,   below right =0.3cm and 0.55cm of A4, scale = 0.4] (E4) {E};
\node[below =0.005cm of B4 ] (B4_probability)  {$\frac{2}{5}$};
\node[below =0.005cm of C4 ] (C4_probability)  {$\frac{1}{5}$};	
\node[below =0.005cm of D4 ] (D4_probability)  {$\frac{1}{5}$};	
\node[below =0.005cm of E4 ] (E4_probability)  {$\frac{1}{5}$};	

\node[state, blue,  below=0.3cm of 5, scale = 0.4] (A5) {A};
\node[state, blue,  below left =0.3cm and 0.55cm of A5, scale = 0.4] (B5) {B};
\node[state, blue,  below left =0.3cm and 0.05cm of A5, scale = 0.4] (C5) {C};
\node[state, blue,   below right =0.3cm and 0.05cm of A5, scale = 0.4] (D5) {D};
\node[state, blue,   below right =0.3cm and 0.55cm of A5, scale = 0.4] (E5) {E};
\node[below =0.005cm of B5 ] (B5_probability)  {$\frac{1}{4}$};
\node[below =0.005cm of C5 ] (C5_probability)  {$\frac{1}{4}$};	
\node[below =0.005cm of D5 ] (D5_probability)  {$\frac{1}{4}$};	
\node[below =0.005cm of E5 ] (E5_probability)  {$\frac{1}{4}$};	

\node[state, blue,  below=0.3cm of 6, scale = 0.4] (A6) {A};
\node[state, blue,  below left =0.3cm and 0.55cm of A6, scale = 0.4] (B6) {B};
\node[state, blue,  below left =0.3cm and 0.05cm of A6, scale = 0.4] (C6) {C};
\node[state, blue,   below right =0.3cm and 0.05cm of A6, scale = 0.4] (D6) {D};
\node[state, blue,   below right =0.3cm and 0.55cm of A6, scale = 0.4] (E6) {E};
\node[below =0.005cm of B6 ] (B6_probability)  {$\frac{1}{6}$};
\node[below =0.005cm of C6 ] (C6_probability)  {$\frac{2}{6}$};	
\node[below =0.005cm of D6 ] (D6_probability)  {$\frac{1}{6}$};	
\node[below =0.005cm of E6 ] (E6_probability)  {$\frac{2}{6}$};	

\node[state, blue,  below=0.3cm of 7, scale = 0.4] (A7) {A};
\node[state, blue,  below left =0.3cm and 0.55cm of A7, scale = 0.4] (B7) {B};
\node[state, blue,  below left =0.3cm and 0.05cm of A7, scale = 0.4] (C7) {C};
\node[state, blue,   below right =0.3cm and 0.05cm of A7, scale = 0.4] (D7) {D};
\node[state, blue,   below right =0.3cm and 0.55cm of A7, scale = 0.4] (E7) {E};
\node[below =0.005cm of B7 ] (B7_probability)  {$\frac{1}{7}$};
\node[below =0.005cm of C7 ] (C7_probability)  {$\frac{2}{7}$};	
\node[below =0.005cm of D7 ] (D7_probability)  {$\frac{2}{7}$};	
\node[below =0.005cm of E7 ] (E7_probability)  {$\frac{2}{7}$};	

	\path
	(1) edge[red, very thick] (2)
    (1) edge[red, very thick] (3)
    	(2) edge[red, very thick] (4)
    (2) edge[red, very thick] (5)
    	(3) edge[red, very thick] (6)
    (3) edge[red, very thick] (7)
	(1) edge[blue] (A1)
    (A1) edge[blue] (B1)
    (A1) edge[blue] (C1)
    (A1) edge[blue] (D1)
    (A1) edge[blue] (E1)
  
    (2) edge[blue] (A2)
    (A2) edge[blue] (B2)
    (A2) edge[blue] (C2)
    (A2) edge[blue] (D2)
    (A2) edge[blue] (E2)

    (3) edge[blue] (A3)
    (A3) edge[blue] (B3)
    (A3) edge[blue] (C3)
    (A3) edge[blue] (D3)
    (A3) edge[blue] (E3)

    (4) edge[blue] (A4)
    (A4) edge[blue] (B4)
    (A4) edge[blue] (C4)
    (A4) edge[blue] (D4)
    (A4) edge[blue] (E4)
    
    (5) edge[blue] (A5)
    (A5) edge[blue] (B5)
    (A5) edge[blue] (C5)
    (A5) edge[blue] (D5)  
    (A5) edge[blue] (E5)
    
    (6) edge[blue] (A6)
    (A6) edge[blue] (B6)
    (A6) edge[blue] (C6)
    (A6) edge[blue] (D6)
    (A6) edge[blue] (E6)

    (7) edge[blue] (A7)
    (A7) edge[blue] (B7)
    (A7) edge[blue] (C7)
    (A7) edge[blue] (D7) 
    (A7) edge[blue] (E7)
    ;
	\end{tikzpicture}
}
  \end{center}
  \caption {Example of a multi-horizon scenario tree with 3 strategic stages, 4 strategic scenarios, 2 operational stages, and 4 operational scenarios at each strategic node.}
\label{Fig:Common_Tree}
  \end{figure}

%% file: MHEOG_common_probability_less_trees.tex
\begin{figure}[ht] 
	 \begin{center}
  \resizebox{0.85\textwidth}{0.55\textwidth}{%
	\begin{tikzpicture}[-, >=stealth', auto, thick, node distance = 1.1 cm]
	\node[state, red, very thick] (1) {$1$};
 \node at (0,1) [rectangle,draw] (MHSG) {$\mathfrak{T}(3) $};
 \node[state, below left= of 1, red, very thick ] (2)    {$2$};
 \node[below left = -0.2cm and 0.6cm of 1, red, very thick ] (2_probability)  {$\frac{1}{3}$};
\node[state, below right=of 1, red, very thick] (3)               {$3$};
\node[below right = -0.2cm and 0.6cm of 1, red, very thick ] (3_probability)  {$\frac{2}{3}$};

\node[state, blue,  below=0.2cm of 1, scale = 0.4] (A1) {A};
\node[state, blue,  below left =0.2cm and 0.45cm of A1, scale = 0.4] (B1) {B};
\node[state, blue,  below left =0.2cm and 0.01cm of A1, scale = 0.4] (C1) {C};
\node[state, blue,   below right =0.2cm and 0.01cm of A1, scale = 0.4] (D1) {D};
\node[state, blue,   below right =0.2cm and 0.45cm of A1, scale = 0.4] (E1) {E};
\node[below =0.005cm of B1 ] (B1_probability)  {$\frac{1}{5}$};
\node[below =0.005cm of C1 ] (C1_probability)  {$\frac{1}{5}$};	
\node[below =0.005cm of D1 ] (D1_probability)  {$\frac{2}{5}$};	
\node[below =0.005cm of E1 ] (E1_probability)  {$\frac{1}{5}$};

\node[state, blue,  below=0.2cm of 2, scale = 0.4] (A2) {A};
\node[state, blue,  below left =0.2cm and 0.45cm of A2, scale = 0.4] (B2) {B};
\node[state, blue,  below left =0.2cm and 0.01cm of A2, scale = 0.4] (C2) {C};
\node[state, blue,   below right =0.2cm and 0.01cm of A2, scale = 0.4] (D2) {D};
\node[state, blue,   below right =0.2cm and 0.45cm of A2, scale = 0.4] (E2) {E};
\node[below =0.005cm of B2 ] (B2_probability)  {$\frac{1}{6}$};
\node[below =0.005cm of C2 ] (C2_probability)  {$\frac{1}{6}$};	
\node[below =0.005cm of D2 ] (D2_probability)  {$\frac{2}{6}$};	
\node[below =0.005cm of E2 ] (E2_probability)  {$\frac{2}{6}$};

\node[state, blue,  below=0.2cm of 3, scale = 0.4] (A3) {A};
\node[state, blue,  below left =0.2cm and 0.45cm of A3, scale = 0.4] (B3) {B};
\node[state, blue,  below left =0.2cm and 0.01cm of A3, scale = 0.4] (C3) {C};
\node[state, blue,   below right =0.2cm and 0.01cm of A3, scale = 0.4] (D3) {D};
\node[state, blue,   below right =0.2cm and 0.45cm of A3, scale = 0.4] (E3) {E};
\node[below =0.005cm of B3 ] (B3_probability)  {$\frac{1}{7}$};
\node[below =0.005cm of C3 ] (C3_probability)  {$\frac{3}{7}$};	
\node[below =0.005cm of D3 ] (D3_probability)  {$\frac{2}{7}$};	
\node[below =0.005cm of E3 ] (E3_probability)  {$\frac{1}{7}$};	

	\path
    (1) edge[red, very thick] (2)
    (1) edge[red, very thick] (3)

    (1)  edge[blue] (A1)
    (A1) edge[blue] (B1)
    (A1) edge[blue] (C1)
    (A1) edge[blue] (D1)
    (A1) edge[blue] (E1)
    
    (2) edge[blue] (A2)
    (A2) edge[blue] (B2)
    (A2) edge[blue] (C2)
    (A2) edge[blue] (D2)
    (A2) edge[blue] (E2)
    
    (3) edge[blue] (A3)
    (A3) edge[blue] (B3)
    (A3) edge[blue] (C3)
    (A3) edge[blue] (D3)
    (A3) edge[blue] (E3)

    ;
 \draw[->](-9.6,-4) -- (-9.6,-4.5); 
\draw[->](9.6,-4) -- (9.6,-4.5); 
\draw[->](-3.2,-4) -- (-3.2,-4.5); 
\draw[->](3.2,-4) -- (3.2,-4.5); 
\draw[-](-9.6,-4) -- (9.6, -4);
\node at (7,-3.6)[rectangle,draw] {Subgroups of cardinality $g_2=2$};
\begin{scope}[shift={(-9.6, -6)}]
 \node at (0,1) [rectangle,draw] (SG_1^{12}) {$\displaystyle \Omega_{\text{OG},1}^{(2)}:\phi_{\text{OG},1}^{(2)}= \frac{2}{5} \times \frac{2}{6} \times \frac{4}{7}$};
	\node[state, red, very thick, scale = 0.7] (1) {$1$};
 \node[state, below left=of 1, red, very thick, scale = 0.7 ] (2)               {$2$};
 \node[below left = -0.2cm and 0.6cm of 1, red, very thick ] ()  {$\frac{1}{3}$};
\node[state, below right=of 1, red, very thick, scale = 0.7] (3)               {$3$};
\node[below right = -0.2cm and 0.6cm of 1, red, very thick ] ()  {$\frac{2}{3}$};
\node[state, blue,  below=0.3cm of 1, scale = 0.4] (A1) {A};
\node[state, white,  below=0.3cm of A1, scale = 0.4] (TR) {TR};
\node[state, blue,   left =0.05cm of TR, scale = 0.4] (B1) {B};
\node[state, blue,   right =0.05cm of TR, scale = 0.4] (C1) {C};
\node[below =0.005cm of C1 ] (C1_probability)  {$\frac{1}{2}$};	
\node[below =0.005cm of B1 ] (B1_probability)  {$\frac{1}{2}$};	

\node[state, blue,  below=0.3cm of 2, scale = 0.4] (A2) {A};
\node[state, white,  below=0.3cm of A2, scale = 0.4] (TR) {TR};
\node[state, blue,   left =0.05cm of TR, scale = 0.4] (B2) {B};
\node[state, blue,   right =0.05cm of TR, scale = 0.4] (C2) {C};
\node[below =0.005cm of B2 ] (B2_probability)  {$\frac{1}{2}$};	
\node[below =0.005cm of C2 ] (C2_probability)  {$\frac{1}{2}$};	

\node[state, blue,  below=0.3cm of 3, scale = 0.4] (A3) {A};
\node[state, white,  below=0.3cm of A3, scale = 0.4] (TR) {TR};
\node[state, blue,   left =0.05cm of TR, scale = 0.4] (B3) {B};
\node[state, blue,   right =0.05cm of TR, scale = 0.4] (D3) {C};
\node[below =0.005cm of B3 ] (B3_probability)  {$\frac{1}{4}$};
\node[below =0.005cm of D3 ] (D3_probability)  {$\frac{3}{4}$};
	\path
	(1) edge[red, very thick] (2)
    (1) edge[red, very thick] (3)
	(1) edge[blue] (A1)
    (A1) edge[blue] (B1)
    (A1) edge[blue] (C1)

    (2) edge[blue] (A2)
    (A2) edge[blue] (B2)
    (A2) edge[blue] (C2)

    (3) edge[blue] (A3)
    (A3) edge[blue] (B3)
    (A3) edge[blue] (D3);

\end{scope}

\begin{scope}[shift={(-3.2, -6)}]
  \node at (0,1) [rectangle,draw] (SG_1^{12}) {$\displaystyle \Omega_{\text{OG},2}^{(2)}: \phi_{\text{OG},2}^{(2)} = \frac{2}{5} \times \frac{2}{6} \times \frac{3}{7}$};
	\node[state, red, very thick, scale = 0.7] (1) {$1$};
 \node[state, below left=of 1, red, very thick, scale = 0.7 ] (2)               {$2$};
 \node[below left = -0.2cm and 0.6cm of 1, red, very thick ] ()  {$\frac{1}{3}$};
\node[state, below right=of 1, red, very thick, scale = 0.7] (3)               {$3$};
\node[below right = -0.2cm and 0.6cm of 1, red, very thick ] ()  {$\frac{2}{3}$};
\node[state, blue,  below=0.3cm of 1, scale = 0.4] (A1) {A};
\node[state, white,  below=0.3cm of A1, scale = 0.4] (TR) {TR};
\node[state, blue,   left =0.05cm of TR, scale = 0.4] (B1) {B};
\node[state, blue,   right =0.05cm of TR, scale = 0.4] (C1) {C};
\node[below =0.005cm of C1 ] (C1_probability)  {$\frac{1}{2}$};	
\node[below =0.005cm of B1 ] (B1_probability)  {$\frac{1}{2}$};	

\node[state, blue,  below=0.3cm of 2, scale = 0.4] (A2) {A};
\node[state, white,  below=0.3cm of A2, scale = 0.4] (TR) {TR};
\node[state, blue,   left =0.05cm of TR, scale = 0.4] (B2) {B};
\node[state, blue,   right =0.05cm of TR, scale = 0.4] (C2) {C};
\node[below =0.005cm of B2 ] (B2_probability)  {$\frac{1}{2}$};	
\node[below =0.005cm of C2 ] (C2_probability)  {$\frac{1}{2}$};	

\node[state, blue,  below=0.3cm of 3, scale = 0.4] (A3) {A};
\node[state, white,  below=0.3cm of A3, scale = 0.4] (TR) {TR};
\node[state, blue,   left =0.05cm of TR, scale = 0.4] (B3) {D};
\node[state, blue,   right =0.05cm of TR, scale = 0.4] (D3) {E};
\node[below =0.005cm of B3 ] (B3_probability)  {$\frac{2}{3}$};
\node[below =0.005cm of D3 ] (D3_probability)  {$\frac{1}{3}$};
	\path
	(1) edge[red, very thick] (2)
    (1) edge[red, very thick] (3)
	(1) edge[blue] (A1)
    (A1) edge[blue] (B1)
    (A1) edge[blue] (C1)

    (2) edge[blue] (A2)
    (A2) edge[blue] (B2)
    (A2) edge[blue] (C2)

    (3) edge[blue] (A3)
    (A3) edge[blue] (B3)
    (A3) edge[blue] (D3);
\end{scope}

\begin{scope}[shift={(3.2, -6)}]
  \node at (0,1) [rectangle,draw] (SG_1^{12}) {$\displaystyle \Omega_{\text{OG},3}^{(2)}: \phi_{\text{OG},3}^{(2)}= \frac{2}{5} \times \frac{4}{6} \times \frac{4}{7}$};
	\node[state, red, very thick, scale = 0.7] (1) {$1$};
 \node[state, below left=of 1, red, very thick, scale = 0.7 ] (2)               {$2$};
 \node[below left = -0.2cm and 0.6cm of 1, red, very thick ] ()  {$\frac{1}{3}$};
\node[state, below right=of 1, red, very thick, scale = 0.7] (3)               {$3$};
\node[below right = -0.2cm and 0.6cm of 1, red, very thick ] ()  {$\frac{2}{3}$};
\node[state, blue,  below=0.3cm of 1, scale = 0.4] (A1) {A};
\node[state, white,  below=0.3cm of A1, scale = 0.4] (TR) {TR};
\node[state, blue,   left =0.05cm of TR, scale = 0.4] (B1) {B};
\node[state, blue,   right =0.05cm of TR, scale = 0.4] (C1) {C};
\node[below =0.005cm of C1 ] (C1_probability)  {$\frac{1}{2}$};	
\node[below =0.005cm of B1 ] (B1_probability)  {$\frac{1}{2}$};	

\node[state, blue,  below=0.3cm of 2, scale = 0.4] (A2) {A};
\node[state, white,  below=0.3cm of A2, scale = 0.4] (TR) {TR};
\node[state, blue,   left =0.05cm of TR, scale = 0.4] (B2) {D};
\node[state, blue,   right =0.05cm of TR, scale = 0.4] (C2) {E};
\node[below =0.005cm of B2 ] (B2_probability)  {$\frac{2}{4}$};	
\node[below =0.005cm of C2 ] (C2_probability)  {$\frac{2}{4}$};	

\node[state, blue,  below=0.3cm of 3, scale = 0.4] (A3) {A};
\node[state, white,  below=0.3cm of A3, scale = 0.4] (TR) {TR};
\node[state, blue,   left =0.05cm of TR, scale = 0.4] (B3) {B};
\node[state, blue,   right =0.05cm of TR, scale = 0.4] (D3) {C};
\node[below =0.005cm of B3 ] (B3_probability)  {$\frac{1}{4}$};
\node[below =0.005cm of D3 ] (D3_probability)  {$\frac{3}{4}$};
	\path
	(1) edge[red, very thick] (2)
    (1) edge[red, very thick] (3)
	(1) edge[blue] (A1)
    (A1) edge[blue] (B1)
    (A1) edge[blue] (C1)

    (2) edge[blue] (A2)
    (A2) edge[blue] (B2)
    (A2) edge[blue] (C2)

    (3) edge[blue] (A3)
    (A3) edge[blue] (B3)
    (A3) edge[blue] (D3);
\end{scope}

\begin{scope}[shift={(9.5, -6)}]
  \node at (0,1) [rectangle,draw] (SG_1^{12}) {$\displaystyle \Omega_{\text{OG},4}^{(2)}:\phi_{\text{OG},4}^{(2)}= \frac{2}{5} \times \frac{4}{6} \times \frac{3}{7}$};
	\node[state, red, very thick, scale = 0.7] (1) {$1$};
 \node[state, below left=of 1, red, very thick, scale = 0.7 ] (2)               {$2$};
 \node[below left = -0.2cm and 0.6cm of 1, red, very thick ] ()  {$\frac{1}{3}$};
\node[state, below right=of 1, red, very thick, scale = 0.7] (3)               {$3$};
\node[below right = -0.2cm and 0.6cm of 1, red, very thick ] ()  {$\frac{2}{3}$};
\node[state, blue,  below=0.3cm of 1, scale = 0.4] (A1) {A};
\node[state, white,  below=0.3cm of A1, scale = 0.4] (TR) {TR};
\node[state, blue,   left =0.05cm of TR, scale = 0.4] (B1) {B};
\node[state, blue,   right =0.05cm of TR, scale = 0.4] (C1) {C};
\node[below =0.005cm of C1 ] (C1_probability)  {$\frac{1}{2}$};	
\node[below =0.005cm of B1 ] (B1_probability)  {$\frac{1}{2}$};	

\node[state, blue,  below=0.3cm of 2, scale = 0.4] (A2) {A};
\node[state, white,  below=0.3cm of A2, scale = 0.4] (TR) {TR};
\node[state, blue,   left =0.05cm of TR, scale = 0.4] (B2) {D};
\node[state, blue,   right =0.05cm of TR, scale = 0.4] (C2) {E};
\node[below =0.005cm of B2 ] (B2_probability)  {$\frac{2}{4}$};	
\node[below =0.005cm of C2 ] (C2_probability)  {$\frac{2}{4}$};	

\node[state, blue,  below=0.3cm of 3, scale = 0.4] (A3) {A};
\node[state, white,  below=0.3cm of A3, scale = 0.4] (TR) {TR};
\node[state, blue,   left =0.05cm of TR, scale = 0.4] (B3) {D};
\node[state, blue,   right =0.05cm of TR, scale = 0.4] (D3) {E};
\node[below =0.005cm of B3 ] (B3_probability)  {$\frac{2}{3}$};
\node[below =0.005cm of D3 ] (D3_probability)  {$\frac{1}{3}$};
	\path
	(1) edge[red, very thick] (2)
    (1) edge[red, very thick] (3)
	(1) edge[blue] (A1)
    (A1) edge[blue] (B1)
    (A1) edge[blue] (C1)

    (2) edge[blue] (A2)
    (A2) edge[blue] (B2)
    (A2) edge[blue] (C2)

    (3) edge[blue] (A3)
    (A3) edge[blue] (B3)
    (A3) edge[blue] (D3);
\end{scope}

\begin{scope}[shift={(-9.6, -12)}]
  \node at (0,1) [rectangle,draw] (SG_1^{12}) {$\displaystyle \Omega_{\text{OG},5}^{(2)}:\phi_{\text{OG},5}^{(2)}= \frac{3}{5} \times \frac{2}{6} \times \frac{4}{7}$};
	\node[state, red, very thick, scale = 0.7] (1) {$1$};
 \node[state, below left=of 1, red, very thick, scale = 0.7 ] (2)               {$2$};
 \node[below left = -0.2cm and 0.6cm of 1, red, very thick ] ()  {$\frac{1}{3}$};
\node[state, below right=of 1, red, very thick, scale = 0.7] (3)               {$3$};
\node[below right = -0.2cm and 0.6cm of 1, red, very thick ] ()  {$\frac{2}{3}$};
\node[state, blue,  below=0.3cm of 1, scale = 0.4] (A1) {A};
\node[state, white,  below=0.3cm of A1, scale = 0.4] (TR) {TR};
\node[state, blue,   left =0.05cm of TR, scale = 0.4] (B1) {D};
\node[state, blue,   right =0.05cm of TR, scale = 0.4] (C1) {E};
\node[below =0.005cm of C1 ] (C1_probability)  {$\frac{1}{3}$};	
\node[below =0.005cm of B1 ] (B1_probability)  {$\frac{2}{3}$};	

\node[state, blue,  below=0.3cm of 2, scale = 0.4] (A2) {A};
\node[state, white,  below=0.3cm of A2, scale = 0.4] (TR) {TR};
\node[state, blue,   left =0.05cm of TR, scale = 0.4] (B2) {B};
\node[state, blue,   right =0.05cm of TR, scale = 0.4] (C2) {C};
\node[below =0.005cm of B2 ] (B2_probability)  {$\frac{1}{2}$};	
\node[below =0.005cm of C2 ] (C2_probability)  {$\frac{1}{2}$};	

\node[state, blue,  below=0.3cm of 3, scale = 0.4] (A3) {A};
\node[state, white,  below=0.3cm of A3, scale = 0.4] (TR) {TR};
\node[state, blue,   left =0.05cm of TR, scale = 0.4] (B3) {B};
\node[state, blue,   right =0.05cm of TR, scale = 0.4] (D3) {C};
\node[below =0.005cm of B3 ] (B3_probability)  {$\frac{1}{4}$};
\node[below =0.005cm of D3 ] (D3_probability)  {$\frac{3}{4}$};
	\path
	(1) edge[red, very thick] (2)
    (1) edge[red, very thick] (3)
	(1) edge[blue] (A1)
    (A1) edge[blue] (B1)
    (A1) edge[blue] (C1)

    (2) edge[blue] (A2)
    (A2) edge[blue] (B2)
    (A2) edge[blue] (C2)

    (3) edge[blue] (A3)
    (A3) edge[blue] (B3)
    (A3) edge[blue] (D3);

\end{scope}

\begin{scope}[shift={(-3.2, -12)}]
 \node at (0,1) [rectangle,draw] (SG_1^{12}) {$\displaystyle \Omega_{\text{OG},6}^{(2)}:\phi_{\text{OG},6}^{(2)}= \frac{3}{5} \times \frac{2}{6} \times \frac{3}{7}$};
	\node[state, red, very thick, scale = 0.7] (1) {$1$};
 \node[state, below left=of 1, red, very thick, scale = 0.7 ] (2)               {$2$};
 \node[below left = -0.2cm and 0.6cm of 1, red, very thick ] ()  {$\frac{1}{3}$};
\node[state, below right=of 1, red, very thick, scale = 0.7] (3)               {$3$};
\node[below right = -0.2cm and 0.6cm of 1, red, very thick ] ()  {$\frac{2}{3}$};
\node[state, blue,  below=0.3cm of 1, scale = 0.4] (A1) {A};
\node[state, white,  below=0.3cm of A1, scale = 0.4] (TR) {TR};
\node[state, blue,   left =0.05cm of TR, scale = 0.4] (B1) {D};
\node[state, blue,   right =0.05cm of TR, scale = 0.4] (C1) {E};
\node[below =0.005cm of C1 ] (C1_probability)  {$\frac{1}{3}$};	
\node[below =0.005cm of B1 ] (B1_probability)  {$\frac{2}{3}$};	

\node[state, blue,  below=0.3cm of 2, scale = 0.4] (A2) {A};
\node[state, white,  below=0.3cm of A2, scale = 0.4] (TR) {TR};
\node[state, blue,   left =0.05cm of TR, scale = 0.4] (B2) {B};
\node[state, blue,   right =0.05cm of TR, scale = 0.4] (C2) {C};
\node[below =0.005cm of B2 ] (B2_probability)  {$\frac{1}{2}$};	
\node[below =0.005cm of C2 ] (C2_probability)  {$\frac{1}{2}$};	

\node[state, blue,  below=0.3cm of 3, scale = 0.4] (A3) {A};
\node[state, white,  below=0.3cm of A3, scale = 0.4] (TR) {TR};
\node[state, blue,   left =0.05cm of TR, scale = 0.4] (B3) {D};
\node[state, blue,   right =0.05cm of TR, scale = 0.4] (D3) {E};
\node[below =0.005cm of B3 ] (B3_probability)  {$\frac{2}{3}$};
\node[below =0.005cm of D3 ] (D3_probability)  {$\frac{1}{3}$};
	\path
	(1) edge[red, very thick] (2)
    (1) edge[red, very thick] (3)
	(1) edge[blue] (A1)
    (A1) edge[blue] (B1)
    (A1) edge[blue] (C1)

    (2) edge[blue] (A2)
    (A2) edge[blue] (B2)
    (A2) edge[blue] (C2)

    (3) edge[blue] (A3)
    (A3) edge[blue] (B3)
    (A3) edge[blue] (D3);
\end{scope}

\begin{scope}[shift={(3.2, -12)}]
 \node at (0,1) [rectangle,draw] (SG_1^{12}) {$\displaystyle \Omega_{\text{OG},7}^{(2)}:\phi_{\text{OG},7}^{(2)}= \frac{3}{5} \times \frac{4}{6} \times \frac{4}{7}$};
	\node[state, red, very thick, scale = 0.7] (1) {$1$};
 \node[state, below left=of 1, red, very thick, scale = 0.7 ] (2)               {$2$};
 \node[below left = -0.2cm and 0.6cm of 1, red, very thick ] ()  {$\frac{1}{3}$};
\node[state, below right=of 1, red, very thick, scale = 0.7] (3)               {$3$};
\node[below right = -0.2cm and 0.6cm of 1, red, very thick ] ()  {$\frac{2}{3}$};
\node[state, blue,  below=0.3cm of 1, scale = 0.4] (A1) {A};
\node[state, white,  below=0.3cm of A1, scale = 0.4] (TR) {TR};
\node[state, blue,   left =0.05cm of TR, scale = 0.4] (B1) {D};
\node[state, blue,   right =0.05cm of TR, scale = 0.4] (C1) {E};
\node[below =0.005cm of C1 ] (C1_probability)  {$\frac{1}{3}$};	
\node[below =0.005cm of B1 ] (B1_probability)  {$\frac{2}{3}$};	

\node[state, blue,  below=0.3cm of 2, scale = 0.4] (A2) {A};
\node[state, white,  below=0.3cm of A2, scale = 0.4] (TR) {TR};
\node[state, blue,   left =0.05cm of TR, scale = 0.4] (B2) {D};
\node[state, blue,   right =0.05cm of TR, scale = 0.4] (C2) {E};
\node[below =0.005cm of B2 ] (B2_probability)  {$\frac{2}{4}$};	
\node[below =0.005cm of C2 ] (C2_probability)  {$\frac{2}{4}$};	

\node[state, blue,  below=0.3cm of 3, scale = 0.4] (A3) {A};
\node[state, white,  below=0.3cm of A3, scale = 0.4] (TR) {TR};
\node[state, blue,   left =0.05cm of TR, scale = 0.4] (B3) {B};
\node[state, blue,   right =0.05cm of TR, scale = 0.4] (D3) {C};
\node[below =0.005cm of B3 ] (B3_probability)  {$\frac{1}{4}$};
\node[below =0.005cm of D3 ] (D3_probability)  {$\frac{3}{4}$};
	\path
	(1) edge[red, very thick] (2)
    (1) edge[red, very thick] (3)
	(1) edge[blue] (A1)
    (A1) edge[blue] (B1)
    (A1) edge[blue] (C1)

    (2) edge[blue] (A2)
    (A2) edge[blue] (B2)
    (A2) edge[blue] (C2)

    (3) edge[blue] (A3)
    (A3) edge[blue] (B3)
    (A3) edge[blue] (D3);
\end{scope}

\begin{scope}[shift={(9.5, -12)}]
 \node at (0,1) [rectangle,draw] (SG_1^{12}) {$\displaystyle \Omega_{\text{OG},8}^{(2)}:\phi_{\text{OG},8}^{(2)}= \frac{3}{5} \times \frac{4}{6} \times \frac{3}{7}$};
	\node[state, red, very thick, scale = 0.7] (1) {$1$};
 \node[state, below left=of 1, red, very thick, scale = 0.7 ] (2)               {$2$};
 \node[below left = -0.2cm and 0.6cm of 1, red, very thick ] ()  {$\frac{1}{3}$};
\node[state, below right=of 1, red, very thick, scale = 0.7] (3)               {$3$};
\node[below right = -0.2cm and 0.6cm of 1, red, very thick ] ()  {$\frac{2}{3}$};
\node[state, blue,  below=0.3cm of 1, scale = 0.4] (A1) {A};
\node[state, white,  below=0.3cm of A1, scale = 0.4] (TR) {TR};
\node[state, blue,   left =0.05cm of TR, scale = 0.4] (B1) {D};
\node[state, blue,   right =0.05cm of TR, scale = 0.4] (C1) {E};
\node[below =0.005cm of C1 ] (C1_probability)  {$\frac{1}{3}$};	
\node[below =0.005cm of B1 ] (B1_probability)  {$\frac{2}{3}$};	

\node[state, blue,  below=0.3cm of 2, scale = 0.4] (A2) {A};
\node[state, white,  below=0.3cm of A2, scale = 0.4] (TR) {TR};
\node[state, blue,   left =0.05cm of TR, scale = 0.4] (B2) {D};
\node[state, blue,   right =0.05cm of TR, scale = 0.4] (C2) {E};
\node[below =0.005cm of B2 ] (B2_probability)  {$\frac{2}{4}$};	
\node[below =0.005cm of C2 ] (C2_probability)  {$\frac{2}{4}$};	

\node[state, blue,  below=0.3cm of 3, scale = 0.4] (A3) {A};
\node[state, white,  below=0.3cm of A3, scale = 0.4] (TR) {TR};
\node[state, blue,   left =0.05cm of TR, scale = 0.4] (B3) {D};
\node[state, blue,   right =0.05cm of TR, scale = 0.4] (D3) {E};
\node[below =0.005cm of B3 ] (B3_probability)  {$\frac{2}{3}$};
\node[below =0.005cm of D3 ] (D3_probability)  {$\frac{1}{3}$};
	\path
	(1) edge[red, very thick] (2)
    (1) edge[red, very thick] (3)
	(1) edge[blue] (A1)
    (A1) edge[blue] (B1)
    (A1) edge[blue] (C1)

    (2) edge[blue] (A2)
    (A2) edge[blue] (B2)
    (A2) edge[blue] (C2)

    (3) edge[blue] (A3)
    (A3) edge[blue] (B3)
    (A3) edge[blue] (D3);
\end{scope}

	\end{tikzpicture}
}
  \end{center}
  \caption {Visual representation of 8 multi-horizon subgroups obtained by  dissecting the operational scenarios at each strategic node of $\mathfrak{T}(3)$ into 2 disjoint subsets, each of cardinality 2, for the expectation-based bound.}
\label{MHEOGSIOPT_common_prob_figure_example} 
  \end{figure}
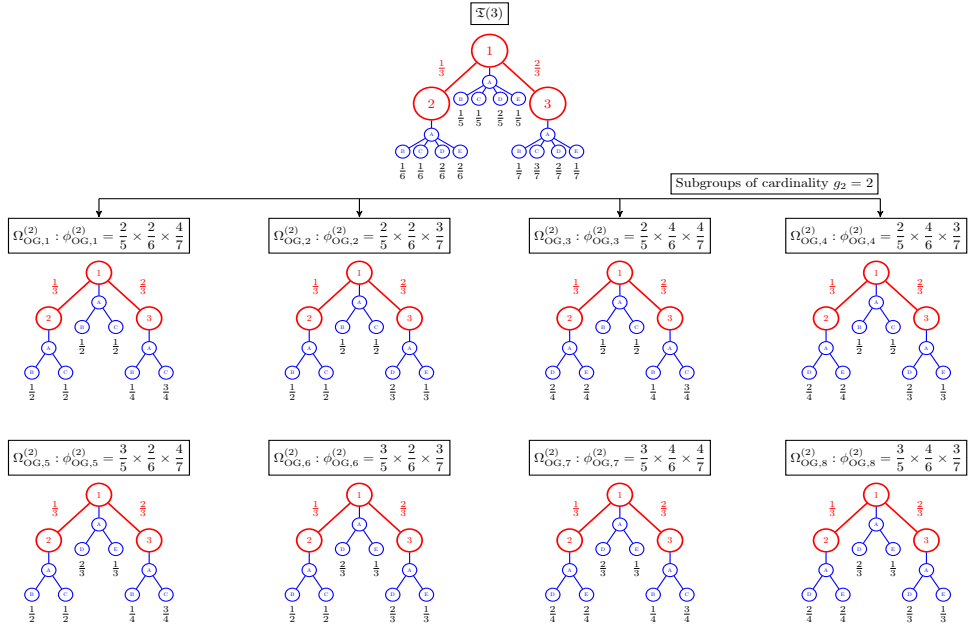

%% file: MHNOG_simple.tex
 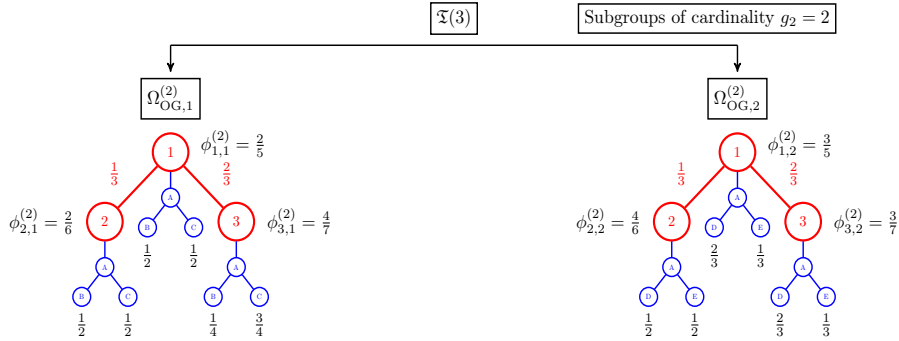
\begin{figure}[htbp] 
	 \begin{center}
  \resizebox{0.8\textwidth}{0.3\textwidth}{%
	\begin{tikzpicture}[-, >=stealth', auto, thick, node distance = 1.1 cm]
 \node at (0,1) [rectangle,draw] (MHSG) {$\mathfrak{T}(3)$};

\draw[->](-5.6,0.5) -- (-5.6,0); 
\draw[->](5.6,0.5) -- (5.6,0); 
\draw[-](-5.6,0.5) -- (5.6, 0.5);
\node at (5,1)[rectangle,draw] {Subgroups of cardinality $g_2=2$};
\begin{scope}[shift={(-5.6, -1.5)}]
 \node at (0,1) [rectangle,draw] (SSG_2^1) {$\Omega_{\text{OG},1}^{(2)}$};
	\node[state, red, very thick, scale = 0.8] (1) {$1$};
 \node[state, below left=of 1, red, very thick, scale = 0.8 ] (2)               {$2$};
 \node[left = -0.2cm and 0.1cm of 2, black, very thick] (2_N_probability)  {$\phi_{2,1}^{(2)} = \frac{2}{6}$};
\node[state, below right=of 1, red, very thick, scale = 0.8] (3)               {$3$};
\node[right = -0.2cm and 0.1cm of 3, black, very thick ] (3_N_probability)  {$\phi_{3,1}^{(2)} = \frac{4}{7}$};
\node[above right = -0.5cm and 0.2cm of 1, black, very thick] (1_N_probability)  {$\phi_{1,1}^{(2)} = \frac{2}{5}$};
 \node[below left = -0.2cm and 0.6cm of 1, red, very thick ] (2_probability)  {$\frac{1}{3}$};
\node[below right = -0.2cm and 0.6cm of 1, red, very thick ] (3_probability)  {$\frac{2}{3}$};

\node[state, blue,  below= 0.3cm of 1, scale = 0.4] (A1) {A};
\node[state, blue,  below right = 0.3cm and 0.2cm of A1, scale = 0.4] (C1) {C};
\node[state, blue,   below left = 0.3cm and 0.2cm of A1, scale = 0.4] (B1) {B};
\node[below =0.005cm of B1 ] (B1_probability)  {$\frac{1}{2}$};
\node[below =0.005cm of C1 ] (C1_probability)  {$\frac{1}{2}$};

\node[state, blue,  below= 0.3cm of 2, scale = 0.4] (A2) {A};
\node[state, blue,  below right = 0.3cm and 0.2cm of A2, scale = 0.4] (C2) {C};
\node[state, blue,   below left = 0.3cm and 0.2cm of A2, scale = 0.4] (B2) {B};
\node[below =0.005cm of B2 ] (B2_probability)  {$\frac{1}{2}$};
\node[below =0.005cm of C2 ] (C2_probability)  {$\frac{1}{2}$};		

\node[state, blue,  below= 0.3cm of 3, scale = 0.4] (A3) {A};
\node[state, blue,  below right = 0.3cm and 0.2cm of A3, scale = 0.4] (C3) {C};
\node[state, blue,   below left = 0.3cm and 0.2cm of A3, scale = 0.4] (B3) {B};
\node[below =0.005cm of B3 ] (B3_probability)  {$\frac{1}{4}$};
\node[below =0.005cm of C3 ] (C3_probability)  {$\frac{3}{4}$};	

	\path
    (1) edge[red, very thick] (2)
    (1) edge[red, very thick] (3)

     (1) edge[blue] (A1)
    (A1) edge[blue] (B1)
    (A1) edge[blue] (C1)
  
    (2) edge[blue] (A2)
    (A2) edge[blue] (B2)
    (A2) edge[blue] (C2)

    (3) edge[blue] (A3)
    (A3) edge[blue] (B3)
    (A3) edge[blue] (C3)
    ;
\end{scope}
\begin{scope}[shift={(5.6, -1.5)}]
 \node at (0,1) [rectangle,draw] (SSG_2^2) {${\Omega_{\text{OG},2}^{(2)}}$};
\node[state, red, very thick, scale = 0.8] (1) {$1$};
\node[state, below left=of 1, red, very thick, scale = 0.8 ] (2)               {$2$};
\node[state, below right=of 1, red, very thick, scale = 0.8] (3)               {$3$};
\node[left = -0.2cm and 0.1cm of 2, black, very thick] (2_N_probability)  {$\phi_{2,2}^{(2)} = \frac{4}{6}$};

\node[right = -0.2cm and 0.1cm of 3, black, very thick] (3_N_probability)  {$ \phi_{3,2}^{(2)} = \frac{3}{7}$};

\node[above right = -0.5cm and 0.2cm of 1, black, very thick ] (1_N_probability)  {$ \phi_{1,2}^{(2)} = \frac{3}{5}$};
 \node[below left = -0.2cm and 0.6cm of 1, red, very thick ] (2_probability)  {$\frac{1}{3}$};
\node[below right = -0.2cm and 0.6cm of 1, red, very thick ] (3_probability)  {$\frac{2}{3}$};

\node[state, blue,  below= 0.3cm of 1, scale = 0.4] (A1) {A};
\node[state, blue,  below right = 0.3cm and 0.2cm of A1, scale = 0.4] (C1) {E};
\node[state, blue,   below left = 0.3cm and 0.2cm of A1, scale = 0.4] (B1) {D};
\node[below =0.005cm of B1 ] (B1_probability)  {$\frac{2}{3}$};
\node[below =0.005cm of C1 ] (C1_probability)  {$\frac{1}{3}$};

\node[state, blue,  below= 0.3cm of 2, scale = 0.4] (A2) {A};
\node[state, blue,  below right = 0.3cm and 0.2cm of A2, scale = 0.4] (C2) {E};
\node[state, blue,   below left = 0.3cm and 0.2cm of A2, scale = 0.4] (B2) {D};
\node[below =0.005cm of B2 ] (B2_probability)  {$\frac{1}{2}$};
\node[below =0.005cm of C2 ] (C2_probability)  {$\frac{1}{2}$};		

\node[state, blue,  below= 0.3cm of 3, scale = 0.4] (A3) {A};
\node[state, blue,  below right = 0.3cm and 0.2cm of A3, scale = 0.4] (C3) {E};
\node[state, blue,   below left = 0.3cm and 0.2cm of A3, scale = 0.4] (B3) {D};
\node[below =0.005cm of B3 ] (B3_probability)  {$\frac{2}{3}$};
\node[below =0.005cm of C3 ] (C3_probability)  {$\frac{1}{3}$};	

	\path
    (1) edge[red, very thick] (2)
    (1) edge[red, very thick] (3)

     (1) edge[blue] (A1)
    (A1) edge[blue] (B1)
    (A1) edge[blue] (C1)
  
    (2) edge[blue] (A2)
    (A2) edge[blue] (B2)
    (A2) edge[blue] (C2)

    (3) edge[blue] (A3)
    (A3) edge[blue] (B3)
    (A3) edge[blue] (C3)
    ;
\end{scope}

	\end{tikzpicture}
}
  \end{center}
  \caption {Visual representation of 2 multi-horizon subgroups obtained by  dissecting the operational scenarios at each strategic node of $\mathfrak{T}(3)$ into 2 disjoint subsets, each of cardinality 2, for the node-based bound.} 
  \label{fig:MHNOG_example} 
  \end{figure}

%% file: Tree_MHESG_Fig.tex


\begin{figure}[t] 
	 \begin{center}
  \resizebox{0.82\textwidth}{0.37\textwidth}{%
	\begin{tikzpicture}[-, >=stealth', auto, thick, node distance = 2.5 cm]
 \node at (0,1) [rectangle,draw] (MHSG) {$\mathfrak{T}(7)$};
 
\draw[->](-5.6,0.5) -- (-5.6,0.1); 
\draw[->](5.6,0.5) -- (5.6,0.1); 
\draw[-](-5.6,0.5) -- (5.6,0.5);
\node at (5,1)[rectangle,draw] {Subgroups of cardinality $\hat{g}_{2}=2$};
\begin{scope}[shift={(-5.6, -1.5)}]
 \node at (0,1) [rectangle,draw] (SSG_2^1) {$\displaystyle {\mathcal{S}_{1}^{(2)}:\ } \phi_1^{(2)} = \frac{1}{3}$};
\node[state, red, very thick, scale=0.8] (1) {$1$};
 \node[state, below=2cm of 1, red, very thick, scale=0.8 ] (2)               {$2$};
\node[state, below left=of 2, red, very thick, scale=0.8] (4)               {$4$};
\node[state, below right=of 2, red, very thick, scale=0.8] (5)               {$5$};
 \node[left =0.005cm of 4, red, very thick ] (4_probability)  {$\hat{\Pi}_{1,\mathbf{s}_1}^{(2)} =   \frac{2}{3}$};
  \node[right =0.005cm of 5, red, very thick ] (5_probability)  {$\hat{\Pi}_{1,\mathbf{s}_2}^{(2)} = \frac{1}{3}$};

\node[state, blue,  below left =0.3cm and 0.75cm of 1, scale = 0.4] (A1) {A};
\node[state, blue,  below left =0.3cm and 0.55cm of A1, scale = 0.4] (B1) {B};
\node[state, blue,  below left =0.3cm and 0.05cm of A1, scale = 0.4] (C1) {C};
\node[state, blue,   below right =0.3cm and 0.05cm of A1, scale = 0.4] (D1) {D};
\node[state, blue,   below right =0.3cm and 0.55cm of A1, scale = 0.4] (E1) {E};
\node[below =0.005cm of B1 ] (B1_probability)  {$\frac{1}{5}$};
\node[below =0.005cm of C1 ] (C1_probability)  {$\frac{1}{5}$};	
\node[below =0.005cm of D1 ] (D1_probability)  {$\frac{2}{5}$};	
\node[below =0.005cm of E1 ] (E1_probability)  {$\frac{1}{5}$};	

\node[state, blue,  below=0.3cm of 2, scale = 0.4] (A2) {A};
\node[state, blue,  below left =0.3cm and 0.55cm of A2, scale = 0.4] (B2) {B};
\node[state, blue,  below left =0.3cm and 0.05cm of A2, scale = 0.4] (C2) {C};
\node[state, blue,   below right =0.3cm and 0.05cm of A2, scale = 0.4] (D2) {D};
\node[state, blue,   below right =0.3cm and 0.55cm of A2, scale = 0.4] (E2) {E};
\node[below =0.005cm of B2 ] (B2_probability)  {$\frac{1}{6}$};
\node[below =0.005cm of C2 ] (C2_probability)  {$\frac{1}{6}$};	
\node[below =0.005cm of D2 ] (D2_probability)  {$\frac{2}{6}$};	
\node[below =0.005cm of E2 ] (E2_probability)  {$\frac{2}{6}$};		

\node[state, blue,  below=0.3cm of 4, scale = 0.4] (A4) {A};
\node[state, blue,  below left =0.3cm and 0.55cm of A4, scale = 0.4] (B4) {B};
\node[state, blue,  below left =0.3cm and 0.05cm of A4, scale = 0.4] (C4) {C};
\node[state, blue,   below right =0.3cm and 0.05cm of A4, scale = 0.4] (D4) {D};
\node[state, blue,   below right =0.3cm and 0.55cm of A4, scale = 0.4] (E4) {E};
\node[below =0.005cm of B4 ] (B4_probability)  {$\frac{2}{5}$};
\node[below =0.005cm of C4 ] (C4_probability)  {$\frac{1}{5}$};	
\node[below =0.005cm of D4 ] (D4_probability)  {$\frac{1}{5}$};	
\node[below =0.005cm of E4 ] (E4_probability)  {$\frac{1}{5}$};	

\node[state, blue,  below=0.3cm of 5, scale = 0.4] (A5) {A};
\node[state, blue,  below left =0.3cm and 0.55cm of A5, scale = 0.4] (B5) {B};
\node[state, blue,  below left =0.3cm and 0.05cm of A5, scale = 0.4] (C5) {C};
\node[state, blue,   below right =0.3cm and 0.05cm of A5, scale = 0.4] (D5) {D};
\node[state, blue,   below right =0.3cm and 0.55cm of A5, scale = 0.4] (E5) {E};
\node[below =0.005cm of B5 ] (B5_probability)  {$\frac{1}{4}$};
\node[below =0.005cm of C5 ] (C5_probability)  {$\frac{1}{4}$};	
\node[below =0.005cm of D5 ] (D5_probability)  {$\frac{1}{4}$};	
\node[below =0.005cm of E5 ] (E5_probability)  {$\frac{1}{4}$};

	\path
	(1) edge[red, very thick] (2)
	(2) edge[red, very thick] (4)
    (2) edge[red, very thick] (5)
     (1) edge[blue] (A1)
    (A1) edge[blue] (B1)
    (A1) edge[blue] (C1)
    (A1) edge[blue] (D1)
    (A1) edge[blue] (E1)
  
    (2) edge[blue] (A2)
    (A2) edge[blue] (B2)
    (A2) edge[blue] (C2)
    (A2) edge[blue] (D2)
    (A2) edge[blue] (E2)

    (4) edge[blue] (A4)
    (A4) edge[blue] (B4)
    (A4) edge[blue] (C4)
    (A4) edge[blue] (D4)
    (A4) edge[blue] (E4)
    
    (5) edge[blue] (A5)
    (A5) edge[blue] (B5)
    (A5) edge[blue] (C5)
    (A5) edge[blue] (D5)  
    (A5) edge[blue] (E5)
    
    ;

\end{scope}
\begin{scope}[shift={(5.6, -1.5)}]
 \node at (0,1) [rectangle,draw] (SSG_2^2) {$\displaystyle\ {\mathcal{S}_{2}^{(2)}:\ } \phi_2^{(2)} = \frac{2}{3}$};
	\node[state, red, very thick, scale=0.8] (1) {$1$};
 \node[state, below=2cm of 1, red, very thick, scale=0.8 ] (3)               {$3$};
\node[state, below left=of 3, red, very thick, scale=0.8] (6)               {$6$};
\node[state, below right=of 3, red, very thick, scale=0.8] (7)               {$7$};

 \node[left =0.005cm of 6, red, very thick ] (6_probability)  {$\hat{\Pi}_{2,\mathbf{s}_3}^{(2)} =   \frac{2}{3}$};
  \node[right =0.005cm of 7, red, very thick ] (7_probability)  {$\hat{\Pi}_{2,\mathbf{s}_4}^{(2)} = \frac{1}{3}$};

\node[state, blue,  below right=0.3cm and 0.75cm of 1, scale = 0.4] (A1) {A};
\node[state, blue,  below left =0.3cm and 0.55cm of A1, scale = 0.4] (B1) {B};
\node[state, blue,  below left =0.3cm and 0.05cm of A1, scale = 0.4] (C1) {C};
\node[state, blue,   below right =0.3cm and 0.05cm of A1, scale = 0.4] (D1) {D};
\node[state, blue,   below right =0.3cm and 0.55cm of A1, scale = 0.4] (E1) {E};
\node[below =0.005cm of B1 ] (B1_probability)  {$\frac{1}{5}$};
\node[below =0.005cm of C1 ] (C1_probability)  {$\frac{1}{5}$};	
\node[below =0.005cm of D1 ] (D1_probability)  {$\frac{2}{5}$};	
\node[below =0.005cm of E1 ] (E1_probability)  {$\frac{1}{5}$};

\node[state, blue,  below=0.3cm of 3, scale = 0.4] (A3) {A};
\node[state, blue,  below left =0.3cm and 0.55cm of A3, scale = 0.4] (B3) {B};
\node[state, blue,  below left =0.3cm and 0.05cm of A3, scale = 0.4] (C3) {C};
\node[state, blue,   below right =0.3cm and 0.05cm of A3, scale = 0.4] (D3) {D};
\node[state, blue,   below right =0.3cm and 0.55cm of A3, scale = 0.4] (E3) {E};
\node[below =0.005cm of B3 ] (B3_probability)  {$\frac{1}{7}$};
\node[below =0.005cm of C3 ] (C3_probability)  {$\frac{3}{7}$};	
\node[below =0.005cm of D3 ] (D3_probability)  {$\frac{2}{7}$};	
\node[below =0.005cm of E3 ] (E3_probability)  {$\frac{1}{7}$};		

\node[state, blue,  below=0.3cm of 6, scale = 0.4] (A6) {A};
\node[state, blue,  below left =0.3cm and 0.55cm of A6, scale = 0.4] (B6) {B};
\node[state, blue,  below left =0.3cm and 0.05cm of A6, scale = 0.4] (C6) {C};
\node[state, blue,   below right =0.3cm and 0.05cm of A6, scale = 0.4] (D6) {D};
\node[state, blue,   below right =0.3cm and 0.55cm of A6, scale = 0.4] (E6) {E};
\node[below =0.005cm of B6 ] (B6_probability)  {$\frac{1}{6}$};
\node[below =0.005cm of C6 ] (C6_probability)  {$\frac{2}{6}$};	
\node[below =0.005cm of D6 ] (D6_probability)  {$\frac{1}{6}$};	
\node[below =0.005cm of E6 ] (E6_probability)  {$\frac{2}{6}$};	

\node[state, blue,  below=0.3cm of 7, scale = 0.4] (A7) {A};
\node[state, blue,  below left =0.3cm and 0.55cm of A7, scale = 0.4] (B7) {B};
\node[state, blue,  below left =0.3cm and 0.05cm of A7, scale = 0.4] (C7) {C};
\node[state, blue,   below right =0.3cm and 0.05cm of A7, scale = 0.4] (D7) {D};
\node[state, blue,   below right =0.3cm and 0.55cm of A7, scale = 0.4] (E7) {E};
\node[below =0.005cm of B7 ] (B7_probability)  {$\frac{1}{7}$};
\node[below =0.005cm of C7 ] (C7_probability)  {$\frac{2}{7}$};	
\node[below =0.005cm of D7 ] (D7_probability)  {$\frac{2}{7}$};	
\node[below =0.005cm of E7 ] (E7_probability)  {$\frac{2}{7}$};	

	\path
    (1) edge[red, very thick] (3)
    	(3) edge[red, very thick] (6)
    (3) edge[red, very thick] (7)
	(1) edge[blue] (A1)
    (A1) edge[blue] (B1)
    (A1) edge[blue] (C1)
    (A1) edge[blue] (D1)
    (A1) edge[blue] (E1)

    (3) edge[blue] (A3)
    (A3) edge[blue] (B3)
    (A3) edge[blue] (C3)
    (A3) edge[blue] (D3)
    (A3) edge[blue] (E3)
    
    (6) edge[blue] (A6)
    (A6) edge[blue] (B6)
    (A6) edge[blue] (C6)
    (A6) edge[blue] (D6)
    (A6) edge[blue] (E6)

    (7) edge[blue] (A7)
    (A7) edge[blue] (B7)
    (A7) edge[blue] (C7)
    (A7) edge[blue] (D7) 
    (A7) edge[blue] (E7)
    ;
    
    ;
\end{scope}
	\end{tikzpicture}
}
  \end{center}
  \caption {Visual representation of 2 multi-horizon subgroups obtained by dissecting the strategic scenarios of $\mathfrak{T}(7)$ into 2 disjoint subgroups of cardinality 2, each with the whole operational sample space at each strategic node.}
\label{MHESG_2_with_7_nodes}
  \end{figure}
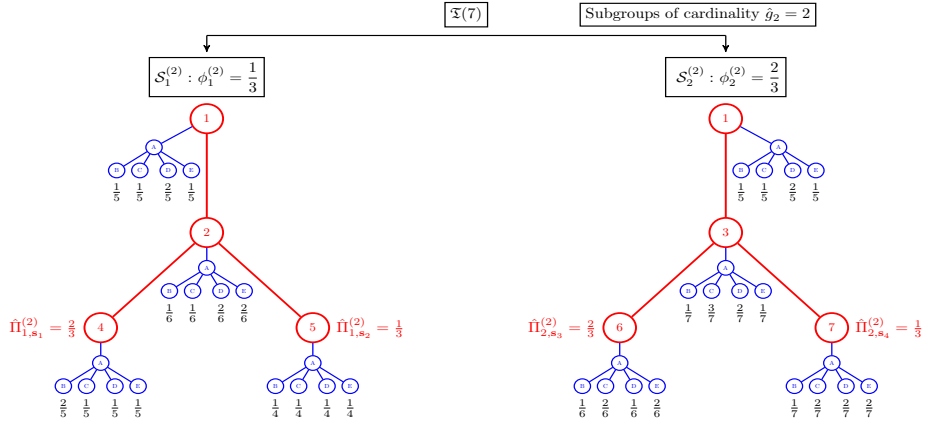

%% file: Tree_MHEG_Fig_with_node_weights.tex


\begin{figure}[ht] 
	 \begin{center}
  \resizebox{0.9\textwidth}{0.33\textwidth}{%
	\begin{tikzpicture}[-, >=stealth', auto, thick, node distance = 2.5 cm]
\node at (0,-7.3) 
 [rectangle,draw] (MHSG) {$\mathfrak{T}(7)$};
\draw[-](-9.6,-7.8) -- (9.6,-7.8); 
\draw[->](-9.6,-7.8) -- (-9.6,-8.3); 
\draw[->](-3.2,-7.8) -- (-3.2,-8.3); 
\draw[->](9.6,-7.8) -- (9.6,-8.3); 
\draw[->](3.2,-7.8) -- (3.2,-8.3); 
\node at (5,-7.3)[rectangle,draw] {Subgroups of cardinality $g_2=2$ and $\hat{g}_2=2$};
\begin{scope}[shift={(-9.6, -9.7)}]
 \node at (0,1) [rectangle,draw] (SSG_2.2^1) {$(\mathcal{N}_1^{(2)},{\Omega_{\text{OG},1}^{(2,2)}): \phi_1^{(2)} = \frac{1}{3}}$};
	\node[state, red, very thick, scale=0.8] (1) {$1$};
    \node[right = -0.2cm and 0.1cm of 1, black, very thick] (111)  {$\phi_{1,1}^{(2)} = \frac{2}{5}$};
 \node[state, below=1.6cm of 1, red, very thick, scale=0.8 ] (2)               {$2$};
 \node[right = -0.2cm and 0.1cm of 2, black, very thick] (211)  {$\phi_{2,1}^{(2)} = \frac{2}{6}$};
\node[state, below left=1.4cm and 1cm of 2, red, very thick, scale=0.8] (4)               {$4$};
\node[above left = 0.1cm and -0.6cm of 4, black, very thick] (411)  {$\phi_{4,1}^{(2)} = \frac{3}{5} $};
\node[state, below right=1.4cm and 1cm of 2, red, very thick, scale=0.8] (5)               {$5$};
\node[above right = 0.1cm and -0.6cm of 5, black, very thick] (511)  {$\phi_{5,1}^{(2)} = \frac{2}{4}$};
 \node[left =0.005cm of 4, red, very thick ] (4_probability)  {$\frac{2}{3}$};
  \node[left =0.005cm of 5, red, very thick ] (5_probability)  {$\frac{1}{3}$};
  
\node[state, blue,  below left= 0.2cm and 0.3cm of 1, scale = 0.4] (A1) {A};
\node[state, blue,  below left=0.3cm and 0.05 cm of A1, scale = 0.4] (B1) {B};
\node[state, blue,  below right=0.3cm and 0.05 cm of A1, scale = 0.4] (C1) {C};
\node[below =0.005cm of B1 ] (B1_probability)  {$\frac{1}{2}$};
\node[below =0.005cm of C1 ] (C1_probability)  {$\frac{1}{2}$};	

\node[state, blue,  below= 0.3cm of 2, scale = 0.4] (A2) {A};
\node[state, blue,  below left=0.3cm and 0.05 cm of A2, scale = 0.4] (B2) {B};
\node[state, blue,  below right=0.3cm and 0.05 cm of A2, scale = 0.4] (C2) {C};
\node[below =0.005cm of B2 ] (B2_probability)  {$\frac{1}{2}$};
\node[below =0.005cm of C2 ] (C2_probability)  {$\frac{1}{2}$};	

\node[state, blue,  below= 0.3cm of 4, scale = 0.4] (A4) {A};
\node[state, blue,  below left=0.3cm and 0.05 cm of A4, scale = 0.4] (B4) {B};
\node[state, blue,  below right=0.3cm and 0.05 cm of A4, scale = 0.4] (C4) {C};
\node[below =0.005cm of B4 ] (B4_probability)  {$\frac{2}{3}$};
\node[below =0.005cm of C4 ] (C4_probability)  {$\frac{1}{3}$};	

\node[state, blue,  below= 0.3cm of 5, scale = 0.4] (A5) {A};
\node[state, blue,  below left=0.3cm and 0.05 cm of A5, scale = 0.4] (B5) {B};
\node[state, blue,  below right=0.3cm and 0.05 cm of A5, scale = 0.4] (C5) {C};
\node[below =0.005cm of B5 ] (B5_probability)  {$\frac{1}{2}$};
\node[below =0.005cm of C5 ] (C5_probability)  {$\frac{1}{2}$};	

	\path
    (1) edge[red, very thick] (2)
    (2) edge[red, very thick] (4)
    (2) edge[red, very thick] (5)

     (1) edge[blue] (A1)
    (A1) edge[blue] (B1)
    (A1) edge[blue] (C1)
  
    (2) edge[blue] (A2)
    (A2) edge[blue] (B2)
    (A2) edge[blue] (C2)

    (4) edge[blue] (A4)
    (A4) edge[blue] (B4)
    (A4) edge[blue] (C4)
    
    (5) edge[blue] (A5)
    (A5) edge[blue] (B5)
    (A5) edge[blue] (C5) 
    ;
\end{scope}

\begin{scope}[shift={(-3.2, -9.7)}]
 \node at (0,1) [rectangle,draw] (SSG_2.2^2) {$(\mathcal{N}_1^{(2)},{\Omega_{\text{OG},2}^{(2,2)}): \phi_1^{(2)} = \frac{1}{3}}$};
	\node[state, red, very thick, scale=0.8] (1) {$1$};
    \node[right = -0.2cm and 0.1cm of 1, black, very thick] (122)  {$\phi_{1,2}^{(2)} = \frac{3}{5}$};
 \node[state, below=1.6cm of 1, red, very thick, scale=0.8 ] (2)               {$2$};
 \node[right = -0.2cm and 0.1cm of 2, black, very thick] (222)  {$\phi_{2,2}^{(2)} = \frac{4}{6}$};
\node[state, below left=1.4cm and 1cm of 2, red, very thick, scale=0.8] (4)               {$4$};
\node[above left = 0.1cm and -0.6cm of 4, black, very thick] (422)  {$\phi_{4,2}^{(2)} = \frac{2}{5} $};
\node[state, below right=1.4cm and 1cm of 2, red, very thick, scale=0.8] (5)               {$5$};
\node[above right = 0.1cm and -0.6cm of 5, black, very thick] (522)  {$\phi_{5,2}^{(2)} = \frac{2}{4}$};
 \node[left =0.005cm of 4, red, very thick ] (4_probability)  {$\frac{2}{3}$};
  \node[left =0.005cm of 5, red, very thick ] (5_probability)  {$\frac{1}{3}$};
  
\node[state, blue,  below left= 0.2cm and 0.3cm of 1, scale = 0.4] (A1) {A};
\node[state, blue,  below left=0.3cm and 0.05 cm of A1, scale = 0.4] (B1) {D};
\node[state, blue,  below right=0.3cm and 0.05 cm of A1, scale = 0.4] (C1) {E};
\node[below =0.005cm of B1 ] (B1_probability)  {$\frac{2}{3}$};
\node[below =0.005cm of C1 ] (C1_probability)  {$\frac{1}{3}$};	

\node[state, blue,  below= 0.3cm of 2, scale = 0.4] (A2) {A};
\node[state, blue,  below left=0.3cm and 0.05 cm of A2, scale = 0.4] (B2) {D};
\node[state, blue,  below right=0.3cm and 0.05 cm of A2, scale = 0.4] (C2) {E};
\node[below =0.005cm of B2 ] (B2_probability)  {$\frac{2}{4}$};
\node[below =0.005cm of C2 ] (C2_probability)  {$\frac{2}{4}$};	

\node[state, blue,  below= 0.3cm of 4, scale = 0.4] (A4) {A};
\node[state, blue,  below left=0.3cm and 0.05 cm of A4, scale = 0.4] (B4) {D};
\node[state, blue,  below right=0.3cm and 0.05 cm of A4, scale = 0.4] (C4) {E};
\node[below =0.005cm of B4 ] (B4_probability)  {$\frac{1}{2}$};
\node[below =0.005cm of C4 ] (C4_probability)  {$\frac{1}{2}$};	

\node[state, blue,  below= 0.3cm of 5, scale = 0.4] (A5) {A};
\node[state, blue,  below left=0.3cm and 0.05 cm of A5, scale = 0.4] (B5) {D};
\node[state, blue,  below right=0.3cm and 0.05 cm of A5, scale = 0.4] (C5) {E};
\node[below =0.005cm of B5 ] (B5_probability)  {$\frac{1}{2}$};
\node[below =0.005cm of C5 ] (C5_probability)  {$\frac{1}{2}$};	

	\path
    (1) edge[red, very thick] (2)
    (2) edge[red, very thick] (4)
    (2) edge[red, very thick] (5)

     (1) edge[blue] (A1)
    (A1) edge[blue] (B1)
    (A1) edge[blue] (C1)
  
    (2) edge[blue] (A2)
    (A2) edge[blue] (B2)
    (A2) edge[blue] (C2)

    (4) edge[blue] (A4)
    (A4) edge[blue] (B4)
    (A4) edge[blue] (C4)
    
    (5) edge[blue] (A5)
    (A5) edge[blue] (B5)
    (A5) edge[blue] (C5) 
    ;
\end{scope}

\begin{scope}[shift={(3.2, -9.7)}]
 \node at (0,1) [rectangle,draw] (SSG_2.2^3) {$(\mathcal{N}_2^{(2)},{\Omega_{\text{OG},1}^{(2,2)}): \phi_2^{(2)} = \frac{2}{3}}$};
	\node[state, red, very thick, scale=0.8] (1) {$1$};
    \node[right = -0.2cm and 0.1cm of 1, black, very thick] (113)  {$\phi_{1,1}^{(2)} = \frac{2}{5}$};
 \node[state, below=1.6cm of 1, red, very thick, scale=0.8 ] (2)               {$3$};
 \node[right = -0.2cm and 0.1cm of 2, black, very thick] (313)  {$\phi_{3,1}^{(2)} = \frac{4}{7}$};
\node[state, below left=1.4cm and 1cm of 2, red, very thick, scale=0.8] (4)               {$6$};
\node[above left = 0.1cm and -0.6cm of 4, black, very thick] (614)  {$\phi_{6,1}^{(2)} = \frac{3}{6} $};
\node[state, below right=1.4cm and 1cm of 2, red, very thick, scale=0.8] (5)               {$7$};
\node[above right = 0.1cm and -0.6cm of 5, black, very thick] (711)  {$\phi_{7,1}^{(2)} = \frac{3}{7}$};
 \node[left =0.005cm of 4, red, very thick ] (4_probability)  {$\frac{1}{3}$};
  \node[left =0.005cm of 5, red, very thick ] (5_probability)  {$\frac{2}{3}$};
  
\node[state, blue,  below left= 0.2cm and 0.3cm of 1, scale = 0.4] (A1) {A};
\node[state, blue,  below left=0.3cm and 0.05 cm of A1, scale = 0.4] (B1) {B};
\node[state, blue,  below right=0.3cm and 0.05 cm of A1, scale = 0.4] (C1) {C};
\node[below =0.005cm of B1 ] (B1_probability)  {$\frac{1}{2}$};
\node[below =0.005cm of C1 ] (C1_probability)  {$\frac{1}{2}$};	

\node[state, blue,  below= 0.3cm of 2, scale = 0.4] (A2) {A};
\node[state, blue,  below left=0.3cm and 0.05 cm of A2, scale = 0.4] (B2) {B};
\node[state, blue,  below right=0.3cm and 0.05 cm of A2, scale = 0.4] (C2) {C};
\node[below =0.005cm of B2 ] (B2_probability)  {$\frac{1}{4}$};
\node[below =0.005cm of C2 ] (C2_probability)  {$\frac{3}{4}$};	

\node[state, blue,  below= 0.3cm of 4, scale = 0.4] (A4) {A};
\node[state, blue,  below left=0.3cm and 0.05 cm of A4, scale = 0.4] (B4) {B};
\node[state, blue,  below right=0.3cm and 0.05 cm of A4, scale = 0.4] (C4) {C};
\node[below =0.005cm of B4 ] (B4_probability)  {$\frac{1}{3}$};
\node[below =0.005cm of C4 ] (C4_probability)  {$\frac{2}{3}$};	

\node[state, blue,  below= 0.3cm of 5, scale = 0.4] (A5) {A};
\node[state, blue,  below left=0.3cm and 0.05 cm of A5, scale = 0.4] (B5) {B};
\node[state, blue,  below right=0.3cm and 0.05 cm of A5, scale = 0.4] (C5) {C};
\node[below =0.005cm of B5 ] (B5_probability)  {$\frac{1}{3}$};
\node[below =0.005cm of C5 ] (C5_probability)  {$\frac{2}{3}$};	

	\path
    (1) edge[red, very thick] (2)
    (2) edge[red, very thick] (4)
    (2) edge[red, very thick] (5)

     (1) edge[blue] (A1)
    (A1) edge[blue] (B1)
    (A1) edge[blue] (C1)
  
    (2) edge[blue] (A2)
    (A2) edge[blue] (B2)
    (A2) edge[blue] (C2)

    (4) edge[blue] (A4)
    (A4) edge[blue] (B4)
    (A4) edge[blue] (C4)
    
    (5) edge[blue] (A5)
    (A5) edge[blue] (B5)
    (A5) edge[blue] (C5) 
    ;
\end{scope}

\begin{scope}[shift={(9.6, -9.7)}]
 \node at (0,1) [rectangle,draw] (SSG_2.2^4) {$(\mathcal{N}_2^{(2)},{\Omega_{\text{OG},2}^{(2,2)}): \phi_2^{(2)} = \frac{2}{3}}$};
	\node[state, red, very thick, scale=0.8] (1) {$1$};
    \node[right = -0.2cm and 0.1cm of 1, black, very thick] (124)  {$\phi_{1,2}^{(2)} = \frac{3}{5}$};
 \node[state, below=1.6cm of 1, red, very thick, scale=0.8 ] (2)               {$3$};
  \node[right = -0.2cm and 0.1cm of 2, black, very thick] (324)  {$\phi_{3,2}^{(2)} = \frac{3}{7}$};
\node[state, below left=1.4cm and 1cm of 2, red, very thick, scale=0.8] (4)               {$6$};
\node[above left = 0.1cm and -0.6cm of 4, black, very thick] (624)  {$\phi_{6,2}^{(2)} = \frac{3}{6} $};
\node[state, below right=1.4cm and 1cm of 2, red, very thick, scale=0.8] (5)               {$7$};
\node[above right = 0.1cm and -0.6cm of 5, black, very thick] (724)  {$\phi_{7,2}^{(2)} = \frac{4}{7}$};
 \node[left =0.005cm of 4, red, very thick ] (4_probability)  {$\frac{1}{3}$};
  \node[left =0.005cm of 5, red, very thick ] (5_probability)  {$\frac{2}{3}$};

\node[state, blue,  below left= 0.2cm and 0.3cm of 1, scale = 0.4] (A1) {A};
\node[state, blue,  below left=0.3cm and 0.05 cm of A1, scale = 0.4] (B1) {D};
\node[state, blue,  below right=0.3cm and 0.05 cm of A1, scale = 0.4] (C1) {E};
\node[below =0.005cm of B1 ] (B1_probability)  {$\frac{2}{3}$};
\node[below =0.005cm of C1 ] (C1_probability)  {$\frac{1}{3}$};	

\node[state, blue,  below= 0.3cm of 2, scale = 0.4] (A2) {A};
\node[state, blue,  below left=0.3cm and 0.05 cm of A2, scale = 0.4] (B2) {D};
\node[state, blue,  below right=0.3cm and 0.05 cm of A2, scale = 0.4] (C2) {E};
\node[below =0.005cm of B2 ] (B2_probability)  {$\frac{2}{3}$};
\node[below =0.005cm of C2 ] (C2_probability)  {$\frac{1}{3}$};	

\node[state, blue,  below= 0.3cm of 4, scale = 0.4] (A4) {A};
\node[state, blue,  below left=0.3cm and 0.05 cm of A4, scale = 0.4] (B4) {D};
\node[state, blue,  below right=0.3cm and 0.05 cm of A4, scale = 0.4] (C4) {E};
\node[below =0.005cm of B4 ] (B4_probability)  {$\frac{1}{3}$};
\node[below =0.005cm of C4 ] (C4_probability)  {$\frac{2}{3}$};	

\node[state, blue,  below= 0.3cm of 5, scale = 0.4] (A5) {A};
\node[state, blue,  below left=0.3cm and 0.05 cm of A5, scale = 0.4] (B5) {D};
\node[state, blue,  below right=0.3cm and 0.05 cm of A5, scale = 0.4] (C5) {E};
\node[below =0.005cm of B5 ] (B5_probability)  {$\frac{2}{4}$};
\node[below =0.005cm of C5 ] (C5_probability)  {$\frac{2}{4}$};	

    \path
    (1) edge[red, very thick] (2)
    (2) edge[red, very thick] (4)
    (2) edge[red, very thick] (5)

     (1) edge[blue] (A1)
    (A1) edge[blue] (B1)
    (A1) edge[blue] (C1)
  
    (2) edge[blue] (A2)
    (A2) edge[blue] (B2)
    (A2) edge[blue] (C2)

    (4) edge[blue] (A4)
    (A4) edge[blue] (B4)
    (A4) edge[blue] (C4)
    
    (5) edge[blue] (A5)
    (A5) edge[blue] (B5)
    (A5) edge[blue] (C5) 
    ;
\end{scope}

	\end{tikzpicture}
}
  \end{center}
  \caption {Visual representation of 4 multi-horizon subgroups obtained by dissecting the strategic and operational scenarios of $\mathfrak{T}(7)$ into 2 disjoint subgroups, each of cardinality 2, for the node-based bound.}
\label{Tree_MHEG(2,2)}
  \end{figure}
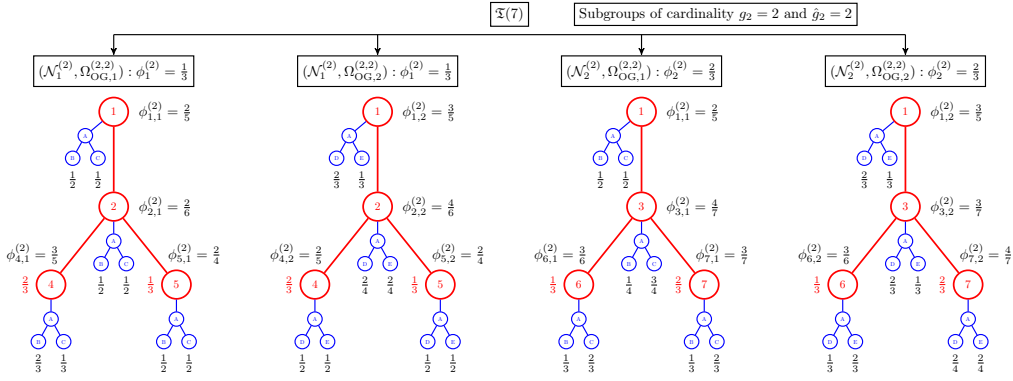

%% file: Tree_MHEOG_Chain_Fig2.tex

\begin{figure}[ht!] 
	 \begin{center}
  \resizebox{0.8\textwidth}{0.7\textwidth}{%
	\begin{tikzpicture}[-, >=stealth', auto, thick, node distance = 1.1 cm]
	\node[state, red, very thick] (1) {$1$};
 \node at (0,1) [rectangle,draw] (MHSG) {$\mathfrak{T}(3)$};
 \node[state, below left= of 1, red, very thick ] (2)               {$2$};
 \node[below left = -0.2cm and 0.6cm of 1, red, very thick ] (2_probability)  {$\frac{1}{3}$};
\node[state, below right=of 1, red, very thick] (3)               {$3$};
\node[below right = -0.2cm and 0.6cm of 1, red, very thick ] (3_probability)  {$\frac{2}{3}$};

\node[state, blue,  below=0.2cm of 1, scale = 0.4] (A1) {A};
\node[state, blue,  below left =0.2cm and 0.45cm of A1, scale = 0.4] (B1) {B};
\node[state, blue,  below left =0.2cm and 0.01cm of A1, scale = 0.4] (C1) {C};
\node[state, blue,   below right =0.2cm and 0.01cm of A1, scale = 0.4] (D1) {D};
\node[state, blue,   below right =0.2cm and 0.45cm of A1, scale = 0.4] (E1) {E};
\node[below =0.005cm of B1 ] (B1_probability)  {$\frac{1}{5}$};
\node[below =0.005cm of C1 ] (C1_probability)  {$\frac{1}{5}$};	
\node[below =0.005cm of D1 ] (D1_probability)  {$\frac{2}{5}$};	
\node[below =0.005cm of E1 ] (E1_probability)  {$\frac{1}{5}$};

\node[state, blue,  below=0.2cm of 2, scale = 0.4] (A2) {A};
\node[state, blue,  below left =0.2cm and 0.45cm of A2, scale = 0.4] (B2) {B};
\node[state, blue,  below left =0.2cm and 0.01cm of A2, scale = 0.4] (C2) {C};
\node[state, blue,   below right =0.2cm and 0.01cm of A2, scale = 0.4] (D2) {D};
\node[state, blue,   below right =0.2cm and 0.45cm of A2, scale = 0.4] (E2) {E};
\node[below =0.005cm of B2 ] (B2_probability)  {$\frac{1}{6}$};
\node[below =0.005cm of C2 ] (C2_probability)  {$\frac{1}{6}$};	
\node[below =0.005cm of D2 ] (D2_probability)  {$\frac{2}{6}$};	
\node[below =0.005cm of E2 ] (E2_probability)  {$\frac{2}{6}$};

\node[state, blue,  below=0.2cm of 3, scale = 0.4] (A3) {A};
\node[state, blue,  below left =0.2cm and 0.45cm of A3, scale = 0.4] (B3) {B};
\node[state, blue,  below left =0.2cm and 0.01cm of A3, scale = 0.4] (C3) {C};
\node[state, blue,   below right =0.2cm and 0.01cm of A3, scale = 0.4] (D3) {D};
\node[state, blue,   below right =0.2cm and 0.45cm of A3, scale = 0.4] (E3) {E};
\node[below =0.005cm of B3 ] (B3_probability)  {$\frac{1}{7}$};
\node[below =0.005cm of C3 ] (C3_probability)  {$\frac{3}{7}$};	
\node[below =0.005cm of D3 ] (D3_probability)  {$\frac{2}{7}$};	
\node[below =0.005cm of E3 ] (E3_probability)  {$\frac{1}{7}$};	

	\path
    (1) edge[red, very thick] (2)
    (1) edge[red, very thick] (3)

    (1)  edge[blue] (A1)
    (A1) edge[blue] (B1)
    (A1) edge[blue] (C1)
    (A1) edge[blue] (D1)
    (A1) edge[blue] (E1)
    
    (2) edge[blue] (A2)
    (A2) edge[blue] (B2)
    (A2) edge[blue] (C2)
    (A2) edge[blue] (D2)
    (A2) edge[blue] (E2)
    
    (3) edge[blue] (A3)
    (A3) edge[blue] (B3)
    (A3) edge[blue] (C3)
    (A3) edge[blue] (D3)
    (A3) edge[blue] (E3)

    ;
\draw[->](-5.6,-4.6) -- (-5.6,-5.1); 
\draw[->](5.6,-4.6) -- (5.6,-5.1); 
\draw[-](-5.6,-4.6) -- (5.6, -4.6);
\node at (5,-4.2)[rectangle,draw] {Subgroups of cardinality $g_2=2$};
\begin{scope}[shift={(-5.6, -6.5)}]
 \node at (0,1) [rectangle,draw] (SSG_2^1) {$\Omega_{\text{OG},1}^{(2)}$};
	\node[state, red, very thick, scale = 0.8] (1) {$1$};
 \node[state, below left=of 1, red, very thick, scale = 0.8 ] (2)               {$2$};
\node[state, below right=of 1, red, very thick, scale = 0.8] (3)               {$3$};
 \node[below left = -0.2cm and 0.6cm of 1, red, very thick ] (2_probability)  {$\frac{1}{3}$};
\node[below right = -0.2cm and 0.6cm of 1, red, very thick ] (3_probability)  {$\frac{2}{3}$};

\node[state, blue,  below= 0.3cm of 1, scale = 0.4] (A1) {A};
\node[state, blue,  below right = 0.3cm and 0.2cm of A1, scale = 0.4] (C1) {C};
\node[state, blue,   below left = 0.3cm and 0.2cm of A1, scale = 0.4] (B1) {B};
\node[below =0.005cm of B1 ] (B1_probability)  {$\frac{1}{2}$};
\node[below =0.005cm of C1 ] (C1_probability)  {$\frac{1}{2}$};

\node[state, blue,  below= 0.3cm of 2, scale = 0.4] (A2) {A};
\node[state, blue,  below right = 0.3cm and 0.2cm of A2, scale = 0.4] (C2) {C};
\node[state, blue,   below left = 0.3cm and 0.2cm of A2, scale = 0.4] (B2) {B};
\node[below =0.005cm of B2 ] (B2_probability)  {$\frac{1}{2}$};
\node[below =0.005cm of C2 ] (C2_probability)  {$\frac{1}{2}$};		

\node[state, blue,  below= 0.3cm of 3, scale = 0.4] (A3) {A};
\node[state, blue,  below right = 0.3cm and 0.2cm of A3, scale = 0.4] (C3) {C};
\node[state, blue,   below left = 0.3cm and 0.2cm of A3, scale = 0.4] (B3) {B};
\node[below =0.005cm of B3 ] (B3_probability)  {$\frac{1}{4}$};
\node[below =0.005cm of C3 ] (C3_probability)  {$\frac{3}{4}$};	

	\path
    (1) edge[red, very thick] (2)
    (1) edge[red, very thick] (3)

     (1) edge[blue] (A1)
    (A1) edge[blue] (B1)
    (A1) edge[blue] (C1)
  
    (2) edge[blue] (A2)
    (A2) edge[blue] (B2)
    (A2) edge[blue] (C2)

    (3) edge[blue] (A3)
    (A3) edge[blue] (B3)
    (A3) edge[blue] (C3)
    ;
\end{scope}
\begin{scope}[shift={(5.6, -6.5)}]
 \node at (0,1) [rectangle,draw] (SSG_2^2) {${\Omega_{\text{OG},2}^{(2)}}$};
\node[state, red, very thick, scale = 0.8] (1) {$1$};
\node[state, below left=of 1, red, very thick, scale = 0.8 ] (2)               {$2$};
\node[state, below right=of 1, red, very thick, scale = 0.8] (3)               {$3$};
 \node[below left = -0.2cm and 0.6cm of 1, red, very thick ] (2_probability)  {$\frac{1}{3}$};
\node[below right = -0.2cm and 0.6cm of 1, red, very thick ] (3_probability)  {$\frac{2}{3}$};

\node[state, blue,  below= 0.3cm of 1, scale = 0.4] (A1) {A};
\node[state, blue,  below right = 0.3cm and 0.2cm of A1, scale = 0.4] (C1) {E};
\node[state, blue,   below left = 0.3cm and 0.2cm of A1, scale = 0.4] (B1) {D};
\node[below =0.005cm of B1 ] (B1_probability)  {$\frac{2}{3}$};
\node[below =0.005cm of C1 ] (C1_probability)  {$\frac{1}{3}$};

\node[state, blue,  below= 0.3cm of 2, scale = 0.4] (A2) {A};
\node[state, blue,  below right = 0.3cm and 0.2cm of A2, scale = 0.4] (C2) {E};
\node[state, blue,   below left = 0.3cm and 0.2cm of A2, scale = 0.4] (B2) {D};
\node[below =0.005cm of B2 ] (B2_probability)  {$\frac{1}{2}$};
\node[below =0.005cm of C2 ] (C2_probability)  {$\frac{1}{2}$};		

\node[state, blue,  below= 0.3cm of 3, scale = 0.4] (A3) {A};
\node[state, blue,  below right = 0.3cm and 0.2cm of A3, scale = 0.4] (C3) {E};
\node[state, blue,   below left = 0.3cm and 0.2cm of A3, scale = 0.4] (B3) {D};
\node[below =0.005cm of B3 ] (B3_probability)  {$\frac{2}{3}$};
\node[below =0.005cm of C3 ] (C3_probability)  {$\frac{1}{3}$};	

	\path
    (1) edge[red, very thick] (2)
    (1) edge[red, very thick] (3)

     (1) edge[blue] (A1)
    (A1) edge[blue] (B1)
    (A1) edge[blue] (C1)
  
    (2) edge[blue] (A2)
    (A2) edge[blue] (B2)
    (A2) edge[blue] (C2)

    (3) edge[blue] (A3)
    (A3) edge[blue] (B3)
    (A3) edge[blue] (C3)
    ;
\end{scope}

\draw[->](-8.6,-10.9) -- (-8.6,-11.4); 
\draw[->](-3,-10.9) -- (-3,-11.4); 
\draw[->](3,-10.9) -- (3,-11.4); 
\draw[->](8.6,-10.9) -- (8.6,-11.4); 
\draw[-](-8.6,-10.9) -- (8.6, -10.9);
\node at (5,-10.5)[rectangle,draw] {Subgroups of cardinality $g_1=1$};

\begin{scope}[shift={(-8.6, -12.8)}]
 \node at (0,1) [rectangle,draw] (SSG_1^1) {${\Omega_{\text{OG},1}^{(1)}}$};
\node[state, red, very thick, scale = 0.8] (1) {$1$};
\node[state, below left=of 1, red, very thick, scale = 0.8 ] (2)               {$2$};
\node[state, below right=of 1, red, very thick, scale = 0.8] (3)               {$3$};
 \node[below left = -0.2cm and 0.6cm of 1, red, very thick ] (2_probability)  {$\frac{1}{3}$};
\node[below right = -0.2cm and 0.6cm of 1, red, very thick ] (3_probability)  {$\frac{2}{3}$};

\node[state, blue,  below = 0.3cm of 1, scale = 0.4] (A1) {A};
\node[state, blue,  below=0.3cm of A1, scale = 0.4] (B1) {B};
\node[below =0.005cm of B1 ] (B1_probability)  {$1$};

\node[state, blue,  below = 0.3cm of 2, scale = 0.4] (A2) {A};
\node[state, blue,  below=0.3cm of A2, scale = 0.4] (B2) {B};
\node[below =0.005cm of B2 ] (B2_probability)  {$1$};

\node[state, blue,  below = 0.3cm of 3, scale = 0.4] (A3) {A};
\node[state, blue,  below=0.3cm of A3, scale = 0.4] (B3) {B};
\node[below =0.005cm of B3 ] (B3_probability)  {$1$};

	\path
    (1) edge[red, very thick] (2)
    (1) edge[red, very thick] (3)

     (1) edge[blue] (A1)
    (A1) edge[blue] (B1)
  
    (2) edge[blue] (A2)
    (A2) edge[blue] (B2)

    (3) edge[blue] (A3)
    (A3) edge[blue] (B3)    
    ;
\end{scope}
\begin{scope}[shift={(-3, -12.8)}]
\node at (0,1) [rectangle,draw] (SSG_2^1) {${\Omega_{\text{OG},2}^{(1)}}$};
\node[state, red, very thick, scale = 0.8] (1) {$1$};
\node[state, below left=of 1, red, very thick, scale = 0.8 ] (2)               {$2$};
\node[state, below right=of 1, red, very thick, scale = 0.8] (3)               {$3$};
 \node[below left = -0.2cm and 0.6cm of 1, red, very thick ] (2_probability)  {$\frac{1}{3}$};
\node[below right = -0.2cm and 0.6cm of 1, red, very thick ] (3_probability)  {$\frac{2}{3}$};

\node[state, blue,  below = 0.3cm of 1, scale = 0.4] (A1) {A};
\node[state, blue,  below=0.3cm of A1, scale = 0.4] (B1) {C};
\node[below =0.005cm of B1 ] (B1_probability)  {$1$};

\node[state, blue,  below = 0.3cm of 2, scale = 0.4] (A2) {A};
\node[state, blue,  below=0.3cm of A2, scale = 0.4] (B2) {C};
\node[below =0.005cm of B2 ] (B2_probability)  {$1$};

\node[state, blue,  below = 0.3cm of 3, scale = 0.4] (A3) {A};
\node[state, blue,  below=0.3cm of A3, scale = 0.4] (B3) {C};
\node[below =0.005cm of B3 ] (B3_probability)  {$1$};

	\path
    (1) edge[red, very thick] (2)
    (1) edge[red, very thick] (3)

     (1) edge[blue] (A1)
    (A1) edge[blue] (B1)
  
    (2) edge[blue] (A2)
    (A2) edge[blue] (B2)

    (3) edge[blue] (A3)
    (A3) edge[blue] (B3)    
    ;
\end{scope}
\begin{scope}[shift={(3, -12.8)}]
 \node at (0,1) [rectangle,draw] (SSG_3^1) {${\Omega_{\text{OG},3}^{(1)}}$};
\node[state, red, very thick, scale = 0.8] (1) {$1$};
\node[state, below left=of 1, red, very thick, scale = 0.8 ] (2)               {$2$};
\node[state, below right=of 1, red, very thick, scale = 0.8] (3)               {$3$};
 \node[below left = -0.2cm and 0.6cm of 1, red, very thick ] (2_probability)  {$\frac{1}{3}$};
\node[below right = -0.2cm and 0.6cm of 1, red, very thick ] (3_probability)  {$\frac{2}{3}$};

\node[state, blue,  below = 0.3cm of 1, scale = 0.4] (A1) {A};
\node[state, blue,  below=0.3cm of A1, scale = 0.4] (B1) {D};
\node[below =0.005cm of B1 ] (B1_probability)  {$1$};

\node[state, blue,  below = 0.3cm of 2, scale = 0.4] (A2) {A};
\node[state, blue,  below=0.3cm of A2, scale = 0.4] (B2) {D};
\node[below =0.005cm of B2 ] (B2_probability)  {$1$};

\node[state, blue,  below = 0.3cm of 3, scale = 0.4] (A3) {A};
\node[state, blue,  below=0.3cm of A3, scale = 0.4] (B3) {D};
\node[below =0.005cm of B3 ] (B3_probability)  {$1$};

	\path
    (1) edge[red, very thick] (2)
    (1) edge[red, very thick] (3)

     (1) edge[blue] (A1)
    (A1) edge[blue] (B1)
  
    (2) edge[blue] (A2)
    (A2) edge[blue] (B2)

    (3) edge[blue] (A3)
    (A3) edge[blue] (B3)    
    ;
    \end{scope}
\begin{scope}[shift={(8.6, -12.8)}]
 \node at (0,1) [rectangle,draw] (SSG_4^1) {${\Omega_{\text{OG},4}^{(1)}}$};
\node[state, red, very thick, scale = 0.8] (1) {$1$};
\node[state, below left=of 1, red, very thick, scale = 0.8 ] (2)               {$2$};
\node[state, below right=of 1, red, very thick, scale = 0.8] (3)               {$3$};
 \node[below left = -0.2cm and 0.6cm of 1, red, very thick ] (2_probability)  {$\frac{1}{3}$};
\node[below right = -0.2cm and 0.6cm of 1, red, very thick ] (3_probability)  {$\frac{2}{3}$};

\node[state, blue,  below = 0.3cm of 1, scale = 0.4] (A1) {A};
\node[state, blue,  below=0.3cm of A1, scale = 0.4] (B1) {E};
\node[below =0.005cm of B1 ] (B1_probability)  {$1$};

\node[state, blue,  below = 0.3cm of 2, scale = 0.4] (A2) {A};
\node[state, blue,  below=0.3cm of A2, scale = 0.4] (B2) {E};
\node[below =0.005cm of B2 ] (B2_probability)  {$1$};

\node[state, blue,  below = 0.3cm of 3, scale = 0.4] (A3) {A};
\node[state, blue,  below=0.3cm of A3, scale = 0.4] (B3) {E};
\node[below =0.005cm of B3 ] (B3_probability)  {$1$};

	\path
    (1) edge[red, very thick] (2)
    (1) edge[red, very thick] (3)

     (1) edge[blue] (A1)
    (A1) edge[blue] (B1)
  
    (2) edge[blue] (A2)
    (A2) edge[blue] (B2)

    (3) edge[blue] (A3)
    (A3) edge[blue] (B3)    
    ;
\end{scope}
	\end{tikzpicture}
}
  \end{center}
  \caption {Visual representation of 2 multi-horizon subgroups obtained at level $2$ and 4 multi-horizon subgroups obtained at level $1$ by dissecting the operational scenarios at each strategic node of $\mathfrak{T}(3)$ into 2 and 1 disjoint subsets, respectively, for the node-based bound.} 
\label{Chain_MHESG_2_with_7_nodes_2}
  \end{figure}
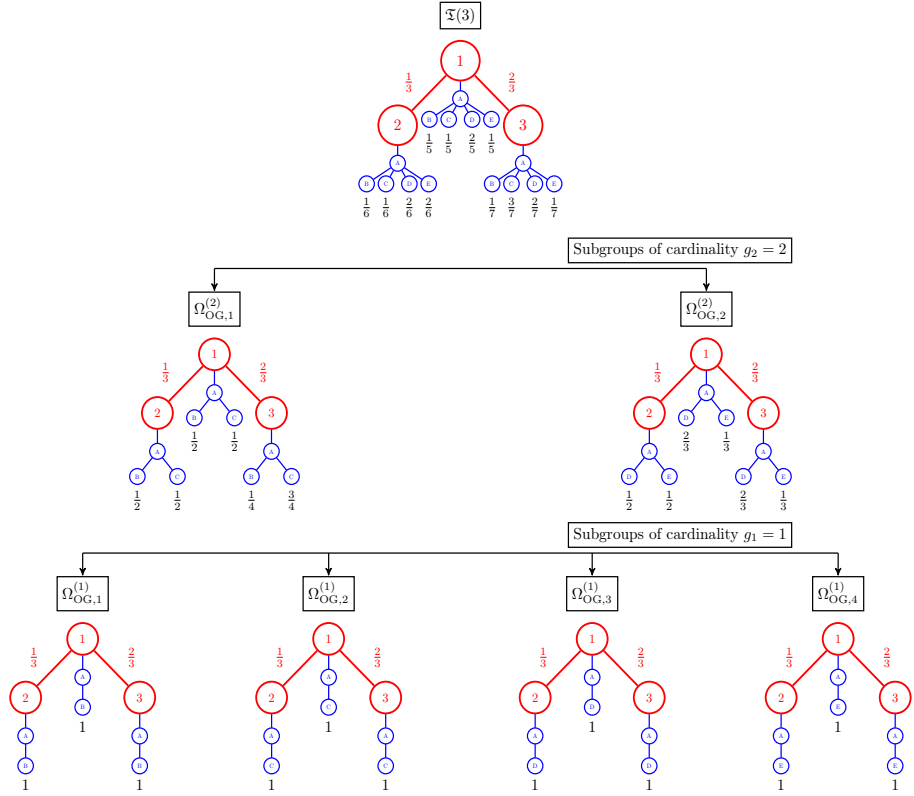

%% file: Tree_MHESG_Phi_Fixed_Fig.tex
 
 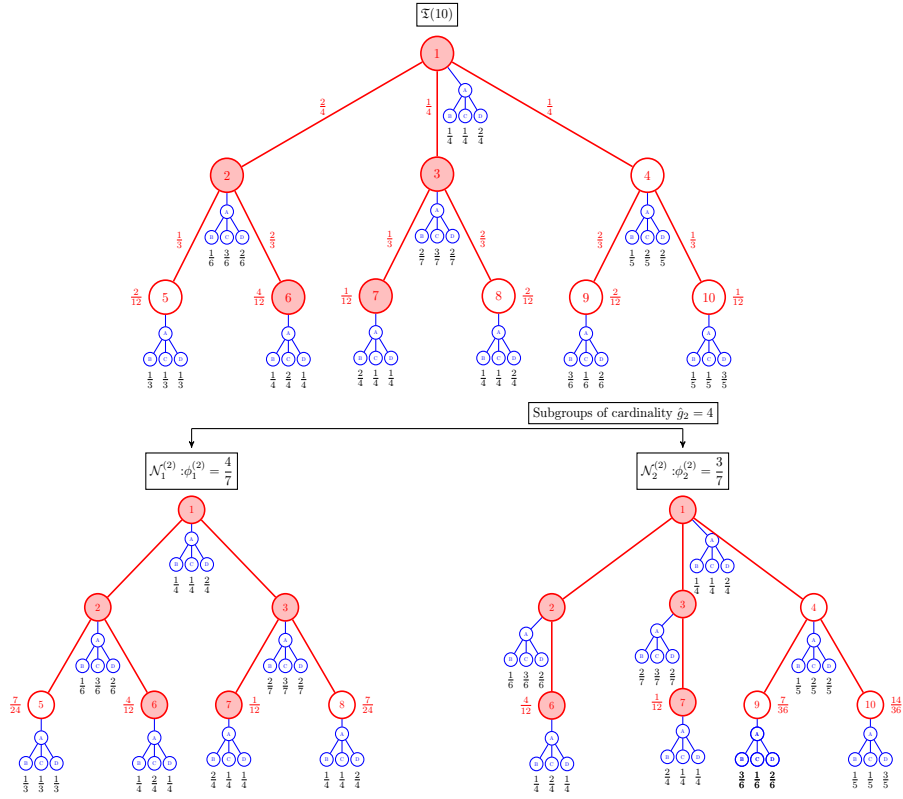
\begin{figure}[ht!] 
	 \begin{center}
  \resizebox{0.8\textwidth}{0.7\textwidth}{%
	\begin{tikzpicture}[-, >=stealth', auto, thick, node distance = 2.5 cm]
	\node[state, red, very thick,fill=pink] (1) {$1$};
 \node at (0,1) [rectangle,draw] (MHSG) {$\mathfrak{T}(10)$};
 \node[state, below left=2.5cm and 5cm of 1, red, very thick,fill=pink ] (2) {$2$};
  \node[below left = 0.7cm and 2.5cm of 1, red, very thick ] (2_probability)  {$\frac{2}{4}$};
 \node[state, below =2.2cm of 1, red, very thick,fill=pink ] (3) {$3$}; 
  \node[right = 2.35cm of 2_probability, red, very thick ] (3_probability)  {$\frac{1}{4}$};
\node[state, below right=2.5cm and 5 cm of 1, red, very thick] (4) {$4$};
\node[right = 2.75cm of 3_probability, red, very thick ] (4_probability)  {$\frac{1}{4}$};
 \node[state, below left=2.5cm and 1cm of 2, red, very thick,
 ] (5)      {$5$};
 \node[left =0.005cm of 5, red, very thick ] (5_probability)  {$\frac{2}{12}$};
\node[state, below right=2.5cm and 1cm of 2, red, very thick,fill=pink] (6)            {$6$};
\node[left =0.005cm of 6, red, very thick ] (6_probability)  {$\frac{4}{12}$};
 \node[state, below left=2.5cm and 1cm  of 3, red, very thick,fill=pink ] (7)               {$7$};
 \node[left =0.005cm of 7, red, very thick ] (7_probability)  {$\frac{1}{12}$};
  \node[state, below right=2.5cm and 1cm  of 3, red, very thick ] (8)               {$8$};
 \node[right =0.005cm of 8, red, very thick ] (8_probability)  {$\frac{2}{12}$};
\node[state, below left=2.5cm and 1cm of 4, red, very thick] (9)               {$9$};
 \node[right =0.005cm of 9, red, very thick ] (9_probability)  {$\frac{2}{12}$};
 \node[state, below right=2.5cm and 1cm of 4, red, very thick] (10)               {$10$};
 \node[right =0.005cm of 10, red, very thick ] (10_probability)  {$\frac{1}{12}$};

\node[below left = 1cm and 0.7 cm of 2, red, very thick ] (25_probability)  {$\frac{1}{3}$};
\node[right = 2cm of 25_probability, red, very thick ] (26_probability)  {$\frac{2}{3}$};
\node[below left = 1cm and 0.7 cm of 3, red, very thick ] (37_probability)  {$\frac{1}{3}$};
\node[right = 2cm of 37_probability, red, very thick ] (38_probability)  {$\frac{2}{3}$};
 \node[below left = 1cm and 0.7 cm of 4, red, very thick ] (49_probability)  {$\frac{2}{3}$};
  \node[right = 2cm of 49_probability, red, very thick ] (410_probability)  {$\frac{1}{3}$};

\node[state, blue,  below right =0.5cm and 0.3cm of 1, scale = 0.4] (A1) {A};
\node[state, blue,  below=0.3cm of A1, scale = 0.4] (C1) {C};
\node[state, blue,   left =0.05cm of C1, scale = 0.4] (B1) {B};
\node[state, blue,   right =0.05cm of C1, scale = 0.4] (D1) {D};
\node[below =0.005cm of B1 ] (B1_probability)  {$\frac{1}{4}$};
\node[below =0.005cm of C1 ] (C1_probability)  {$\frac{1}{4}$};	
\node[below =0.005cm of D1 ] (D1_probability)  {$\frac{2}{4}$};	

\node[state, blue,  below=0.3cm of 2, scale = 0.4] (A2) {A};
\node[state, blue,  below=0.3cm of A2, scale = 0.4] (C2) {C};
\node[state, blue,   left =0.05cm of C2, scale = 0.4] (B2) {B};
\node[state, blue,   right =0.05cm of C2, scale = 0.4] (D2) {D};
\node[below =0.005cm of B2 ] (B2_probability)  {$\frac{1}{6}$};
\node[below =0.005cm of C2 ] (C2_probability)  {$\frac{3}{6}$};	
\node[below =0.005cm of D2 ] (D2_probability)  {$\frac{2}{6}$};	

\node[state, blue,  below=0.3cm of 3, scale = 0.4] (A3) {A};
\node[state, blue,  below=0.3cm of A3, scale = 0.4] (C3) {C};
\node[state, blue,   left =0.05cm of C3, scale = 0.4] (B3) {B};
\node[state, blue,   right =0.05cm of C3, scale = 0.4] (D3) {D};
\node[below =0.005cm of B3 ] (B3_probability)  {$\frac{2}{7}$};
\node[below =0.005cm of C3 ] (C3_probability)  {$\frac{3}{7}$};	
\node[below =0.005cm of D3 ] (D3_probability)  {$\frac{2}{7}$};

\node[state, blue,  below=0.3cm of 4, scale = 0.4] (A4) {A};
\node[state, blue,  below=0.3cm of A4, scale = 0.4] (C4) {C};
\node[state, blue,   left =0.05cm of C4, scale = 0.4] (B4) {B};
\node[state, blue,   right =0.05cm of C4, scale = 0.4] (D4) {D};
\node[below =0.005cm of B4 ] (B4_probability)  {$\frac{1}{5}$};
\node[below =0.005cm of C4 ] (C4_probability)  {$\frac{2}{5}$};	
\node[below =0.005cm of D4 ] (D4_probability)  {$\frac{2}{5}$};

\node[state, blue,  below=0.3cm of 5, scale = 0.4] (A5) {A};
\node[state, blue,  below=0.3cm of A5, scale = 0.4] (C5) {C};
\node[state, blue,   left =0.05cm of C5, scale = 0.4] (B5) {B};
\node[state, blue,   right =0.05cm of C5, scale = 0.4] (D5) {D};
\node[below =0.005cm of B5 ] (B5_probability)  {$\frac{1}{3}$};
\node[below =0.005cm of C5 ] (C5_probability)  {$\frac{1}{3}$};	
\node[below =0.005cm of D5 ] (D5_probability)  {$\frac{1}{3}$};

\node[state, blue,  below=0.3cm of 6, scale = 0.4] (A6) {A};
\node[state, blue,  below=0.3cm of A6, scale = 0.4] (C6) {C};
\node[state, blue,   left =0.05cm of C6, scale = 0.4] (B6) {B};
\node[state, blue,   right =0.05cm of C6, scale = 0.4] (D6) {D};
\node[below =0.005cm of B6 ] (B6_probability)  {$\frac{1}{4}$};
\node[below =0.005cm of C6 ] (C6_probability)  {$\frac{2}{4}$};	
\node[below =0.005cm of D6 ] (D6_probability)  {$\frac{1}{4}$};
\node[state, blue,  below=0.3cm of 7, scale = 0.4] (A7) {A};
\node[state, blue,  below=0.3cm of A7, scale = 0.4] (C7) {C};
\node[state, blue,   left =0.05cm of C7, scale = 0.4] (B7) {B};
\node[state, blue,   right =0.05cm of C7, scale = 0.4] (D7) {D};
\node[below =0.005cm of B7 ] (B7_probability)  {$\frac{2}{4}$};
\node[below =0.005cm of C7 ] (C7_probability)  {$\frac{1}{4}$};	
\node[below =0.005cm of D7 ] (D7_probability)  {$\frac{1}{4}$};

\node[state, blue,  below=0.3cm of 8, scale = 0.4] (A8) {A};
\node[state, blue,  below=0.3cm of A8, scale = 0.4] (C8) {C};
\node[state, blue,   left =0.05cm of C8, scale = 0.4] (B8) {B};
\node[state, blue,   right =0.05cm of C8, scale = 0.4] (D8) {D};
\node[below =0.005cm of B8 ] (B8_probability)  {$\frac{1}{4}$};
\node[below =0.005cm of C8 ] (C8_probability)  {$\frac{1}{4}$};	
\node[below =0.005cm of D8 ] (D8_probability)  {$\frac{2}{4}$};	

\node[state, blue,  below=0.3cm of 9, scale = 0.4] (A9) {A};
\node[state, blue,  below=0.3cm of A9, scale = 0.4] (C9) {C};
\node[state, blue,   left =0.05cm of C9, scale = 0.4] (B9) {B};
\node[state, blue,   right =0.05cm of C9, scale = 0.4] (D9) {D};
\node[below =0.005cm of B9 ] (B9_probability)  {$\frac{3}{6}$};
\node[below =0.005cm of C9 ] (C9_probability)  {$\frac{1}{6}$};	
\node[below =0.005cm of D9 ] (D9_probability)  {$\frac{2}{6}$};	

\node[state, blue,  below=0.3cm of 10, scale = 0.4] (A10) {A};
\node[state, blue,  below=0.3cm of A10, scale = 0.4] (C10) {C};
\node[state, blue,   left =0.05cm of C10, scale = 0.4] (B10) {B};
\node[state, blue,   right =0.05cm of C10, scale = 0.4] (D10) {D};
\node[below =0.005cm of B10 ] (B10_probability)  {$\frac{1}{5}$};
\node[below =0.005cm of C10 ] (C10_probability)  {$\frac{1}{5}$};	
\node[below =0.005cm of D10 ] (D10_probability)  {$\frac{3}{5}$};	

    \path
    (1) edge[red, very thick] (2)
    (1) edge[red, very thick] (3)
    (1) edge[red, very thick] (4)
    (2) edge[red, very thick] (5)
    (2) edge[red, very thick] (6)
    (3) edge[red, very thick] (7)
    (3) edge[red, very thick] (8)
    (4) edge[red, very thick] (9)
    (4) edge[red, very thick] (10)
    
     (1) edge[blue] (A1)
    (A1) edge[blue] (B1)
    (A1) edge[blue] (C1)
    (A1) edge[blue] (D1)
  
    (2) edge[blue] (A2)
    (A2) edge[blue] (B2)
    (A2) edge[blue] (C2)
    (A2) edge[blue] (D2)

    (3) edge[blue] (A3)
    (A3) edge[blue] (B3)
    (A3) edge[blue] (C3)
    (A3) edge[blue] (D3)

    (4) edge[blue] (A4)
    (A4) edge[blue] (B4)
    (A4) edge[blue] (C4)
    (A4) edge[blue] (D4)
    
    (5) edge[blue] (A5)
    (A5) edge[blue] (B5)
    (A5) edge[blue] (C5)
    (A5) edge[blue] (D5)   
    
    (6) edge[blue] (A6)
    (A6) edge[blue] (B6)
    (A6) edge[blue] (C6)
    (A6) edge[blue] (D6)

    (7) edge[blue] (A7)
    (A7) edge[blue] (B7)
    (A7) edge[blue] (C7)
    (A7) edge[blue] (D7)    

        (8) edge[blue] (A8)
    (A8) edge[blue] (B8)
    (A8) edge[blue] (C8)
    (A8) edge[blue] (D8) 

        (9) edge[blue] (A9)
    (A9) edge[blue] (B9)
    (A9) edge[blue] (C9)
    (A9) edge[blue] (D9) 

        (10) edge[blue] (A10)
    (A10) edge[blue] (B10)
    (A10) edge[blue] (C10)
    (A10) edge[blue] (D10) 
    ;
\draw[->](-6.6,-9.7) -- (-6.6,-10.2); 
\draw[->](6.6,-9.7) -- (6.6,-10.2); 
\draw[-](-6.6,-9.7) -- (6.6, -9.7);
\node at (5,-9.3)[rectangle,draw] {Subgroups of cardinality $\hat{g}_2=4$};
\begin{scope}[shift={(-6.6, -11.8)}]
 \node at (0,1) [rectangle,draw] (SSG_4^1) {$\displaystyle {\mathcal{N}_{1}^{(2)}:} \phi_1^{(2)} = \frac{4}{7}$};
	\node[state, red, very thick,fill=pink, scale=0.8] (1) {$1$};
 \node[state, below left=2cm and 2cm of 1, red, very thick,fill=pink, scale=0.8 ] (2)               {$2$};
  \node[state, below right=2cm and 2cm of 1, red, very thick,fill=pink, scale=0.8 ] (3)               {$3$};
\node[state, below left=2cm and 1cm of 2, red, very thick, scale=0.8] (4)               {$5$};
\node[state, below right=2cm and 1cm of 2, red, very thick,fill=pink, scale=0.8] (5)               {$6$};
\node[state, below left=2cm and 1 cm of 3, red, very thick,fill=pink, scale=0.8] (6)               {$7$};
\node[state, below right=2cm and 1 cm  of 3, red, very thick, scale=0.8] (8)               {$8$};

 \node[left =0.005cm of 4, red, very thick ] (4_probability)  {$\frac{7}{24}$};
  \node[left =0.005cm of 5, red, very thick ] (5_probability)  {$\frac{4}{12}$};
   \node[right =0.005cm of 6, red, very thick ] (6_probability)  {$\frac{1}{12}$};
\node[right =0.005cm of 8, red, very thick ] (8_probability)  {$\frac{7}{24}$};
  
\node[state, blue,  below = 0.2cm of 1, scale = 0.4] (A1) {A};
\node[state, blue,  below=0.3cm of A1, scale = 0.4] (C1) {C};
\node[state, blue,   left =0.05cm of C1, scale = 0.4] (B1) {B};
\node[state, blue,   right =0.05cm of C1, scale = 0.4] (D1) {D};
\node[below =0.005cm of B1 ] (B1_probability)  {$\frac{1}{4}$};
\node[below =0.005cm of C1 ] (C1_probability)  {$\frac{1}{4}$};	
\node[below =0.005cm of D1 ] (D1_probability)  {$\frac{2}{4}$};	

\node[state, blue,  below=0.3cm of 2, scale = 0.4] (A2) {A};
\node[state, blue,  below=0.3cm of A2, scale = 0.4] (C2) {C};
\node[state, blue,   left =0.05cm of C2, scale = 0.4] (B2) {B};
\node[state, blue,   right =0.05cm of C2, scale = 0.4] (D2) {D};
\node[below =0.005cm of B2 ] (B2_probability)  {$\frac{1}{6}$};
\node[below =0.005cm of C2 ] (C2_probability)  {$\frac{3}{6}$};	
\node[below =0.005cm of D2 ] (D2_probability)  {$\frac{2}{6}$};

\node[state, blue,  below= 0.3cm of 3, scale = 0.4] (A3) {A};
\node[state, blue,  below=0.3cm of A3, scale = 0.4] (C3) {C};
\node[state, blue,   left =0.05cm of C3, scale = 0.4] (B3) {B};
\node[state, blue,   right =0.05cm of C3, scale = 0.4] (D3) {D};
\node[below =0.005cm of B3 ] (B3_probability)  {$\frac{2}{7}$};
\node[below =0.005cm of C3 ] (C3_probability)  {$\frac{3}{7}$};	
\node[below =0.005cm of D3 ] (D3_probability)  {$\frac{2}{7}$};

\node[state, blue,  below=0.3cm of 4, scale = 0.4] (A4) {A};
\node[state, blue,  below=0.3cm of A4, scale = 0.4] (C4) {C};
\node[state, blue,   left =0.05cm of C4, scale = 0.4] (B4) {B};
\node[state, blue,   right =0.05cm of C4, scale = 0.4] (D4) {D};
\node[below =0.005cm of B4 ] (B4_probability)  {$\frac{1}{3}$};
\node[below =0.005cm of C4 ] (C4_probability)  {$\frac{1}{3}$};	
\node[below =0.005cm of D4 ] (D4_probability)  {$\frac{1}{3}$};

\node[state, blue,  below=0.3cm of 5, scale = 0.4] (A5) {A};
\node[state, blue,  below=0.3cm of A5, scale = 0.4] (C5) {C};
\node[state, blue,   left =0.05cm of C5, scale = 0.4] (B5) {B};
\node[state, blue,   right =0.05cm of C5, scale = 0.4] (D5) {D};
\node[below =0.005cm of B5 ] (B5_probability)  {$\frac{1}{4}$};
\node[below =0.005cm of C5 ] (C5_probability)  {$\frac{2}{4}$};	
\node[below =0.005cm of D5 ] (D5_probability)  {$\frac{1}{4}$};

\node[state, blue,  below = 0.2cm of 6, scale = 0.4] (A6) {A};
\node[state, blue,  below=0.3cm of A6, scale = 0.4] (C6) {C};
\node[state, blue,   left =0.05cm of C6, scale = 0.4] (B6) {B};
\node[state, blue,   right =0.05cm of C6, scale = 0.4] (D6) {D};
\node[below =0.005cm of B6 ] (B6_probability)  {$\frac{2}{4}$};
\node[below =0.005cm of C6 ] (C6_probability)  {$\frac{1}{4}$};	
\node[below =0.005cm of D6 ] (D6_probability)  {$\frac{1}{4}$};

\node[state, blue,  below = 0.2cm of 8, scale = 0.4] (A8) {A};
\node[state, blue,  below=0.3cm of A8, scale = 0.4] (C8) {C};
\node[state, blue,   left =0.05cm of C8, scale = 0.4] (B8) {B};
\node[state, blue,   right =0.05cm of C8, scale = 0.4] (D8) {D};
\node[below =0.005cm of B8 ] (86_probability)  {$\frac{1}{4}$};
\node[below =0.005cm of C8 ] (C8_probability)  {$\frac{1}{4}$};	
\node[below =0.005cm of D8 ] (D8_probability)  {$\frac{2}{4}$};

	\path
    (1) edge[red, very thick] (2)
    (1) edge[red, very thick] (3)
    (2) edge[red, very thick] (4)
    (2) edge[red, very thick] (5)
    (3) edge[red, very thick] (6)
    (3) edge[red, very thick] (8)

     (1) edge[blue] (A1)
    (A1) edge[blue] (B1)
    (A1) edge[blue] (C1)
    (A1) edge[blue] (D1)
  
    (2) edge[blue] (A2)
    (A2) edge[blue] (B2)
    (A2) edge[blue] (C2)
    (A2) edge[blue] (D2)

    (3) edge[blue] (A3)
    (A3) edge[blue] (B3)
    (A3) edge[blue] (C3)
    (A3) edge[blue] (D3)       

    (4) edge[blue] (A4)
    (A4) edge[blue] (B4)
    (A4) edge[blue] (C4)
    (A4) edge[blue] (D4)
    
    (5) edge[blue] (A5)
    (A5) edge[blue] (B5)
    (A5) edge[blue] (C5)
    (A5) edge[blue] (D5)   

    (6) edge[blue] (A6)
    (A6) edge[blue] (B6)
    (A6) edge[blue] (C6)
    (A6) edge[blue] (D6) 

    (8) edge[blue] (A8)
    (A8) edge[blue] (B8)
    (A8) edge[blue] (C8)
    (A8) edge[blue] (D8)
    ;
\end{scope}
\begin{scope}[shift={(6.6, -11.8)}]
 \node at (0,1) [rectangle,draw] (SSG_4^2) {$\displaystyle {\mathcal{N}_{2}^{(2)}:} \phi_{2}^{(2)} = \frac{3}{7}$};
	\node[state, red, very thick,fill=pink, scale=0.8] (1) {$1$};
 \node[state, below left=2cm and 3cm of 1, red, very thick,fill=pink, scale=0.8 ] (2)               {$2$};
\node[state, below =1.7cm of 1, red, very thick,fill=pink, scale=0.8 ] (3bis)               {$3$};
  \node[state, below right=2cm and 3cm of 1, red, very thick, scale=0.8 ] (3)               {$4$};
\node[state, below=1.8cm of 2, red, very thick,fill=pink, scale=0.8] (4)               {$6$};
\node[state, below=1.8cm of 3bis, red, very thick,fill=pink, scale=0.8] (5)               {$7$};
\node[state, below left=2cm and 1cm of 3, red, very thick, scale=0.8] (6)               {$9$};
\node[state, below right=2cm and 1cm of 3, red, very thick, scale=0.8] (7)               {$10$};

 \node[left =0.005cm of 4, red, very thick ] (4_probability)  {$\frac{4}{12}$};
  \node[left =0.005cm of 5, red, very thick ] (5_probability)  {$\frac{1}{12}$};
   \node[right =0.005cm of 6, red, very thick ] (6_probability)  {$\frac{7}{36}$};
  \node[right =0.005cm of 7, red, very thick ] (7_probability)  {$\frac{14}{36}$};  
  
\node[state, blue,  below right=0.4cm and 0.4cm of 1, scale = 0.4] (A1) {A};
\node[state, blue,  below=0.3cm of A1, scale = 0.4] (C1) {C};
\node[state, blue,   left =0.05cm of C1, scale = 0.4] (B1) {B};
\node[state, blue,   right =0.05cm of C1, scale = 0.4] (D1) {D};
\node[below =0.005cm of B1 ] (B1_probability)  {$\frac{1}{4}$};
\node[below =0.005cm of C1 ] (C1_probability)  {$\frac{1}{4}$};	
\node[below =0.005cm of D1 ] (D1_probability)  {$\frac{2}{4}$};	

\node[state, blue,  below left=0.3cm and 0.3cm of 2, scale = 0.4] (A2) {A};
\node[state, blue,  below=0.3cm of A2, scale = 0.4] (C2) {C};
\node[state, blue,   left =0.05cm of C2, scale = 0.4] (B2) {B};
\node[state, blue,   right =0.05cm of C2, scale = 0.4] (D2) {D};
\node[below =0.005cm of B2 ] (B2_probability)  {$\frac{1}{6}$};
\node[below =0.005cm of C2 ] (C2_probability)  {$\frac{3}{6}$};	
\node[below =0.005cm of D2 ] (D2_probability)  {$\frac{2}{6}$};

\node[state, blue,  below left=0.3cm and 0.3cm of 3bis, scale = 0.4] (A3b) {A};
\node[state, blue,  below=0.3cm of A3b, scale = 0.4] (C3b) {C};
\node[state, blue,   left =0.05cm of C3b, scale = 0.4] (B3b) {B};
\node[state, blue,   right =0.05cm of C3b, scale = 0.4] (D3b) {D};
\node[below =0.005cm of B3b ] (B3b_probability)  {$\frac{2}{7}$};
\node[below =0.005cm of C3b ] (C3b_probability)  {$\frac{3}{7}$};	
\node[below =0.005cm of D3b ] (D3b_probability)  {$\frac{2}{7}$};

\node[state, blue,  below=0.3cm of 3, scale = 0.4] (A3) {A};
\node[state, blue,  below=0.3cm of A3, scale = 0.4] (C3) {C};
\node[state, blue,   left =0.05cm of C3, scale = 0.4] (B3) {B};
\node[state, blue,   right =0.05cm of C3, scale = 0.4] (D3) {D};
\node[below =0.005cm of B3 ] (B3_probability)  {$\frac{1}{5}$};
\node[below =0.005cm of C3 ] (C3_probability)  {$\frac{2}{5}$};	
\node[below =0.005cm of D3 ] (D3_probability)  {$\frac{{2}}{5}$};

\node[state, blue,  below = 0.2cm of 5, scale = 0.4] (A5) {A};
\node[state, blue,  below=0.3cm of A5, scale = 0.4] (C5) {C};
\node[state, blue,   left =0.05cm of C5, scale = 0.4] (B5) {B};
\node[state, blue,   right =0.05cm of C5, scale = 0.4] (D5) {D};
\node[below =0.005cm of B5 ] (B5_probability)  {$\frac{2}{4}$};
\node[below =0.005cm of C5 ] (C5_probability)  {$\frac{1}{4}$};	
\node[below =0.005cm of D5 ] (D5_probability)  {$\frac{1}{4}$};

\node[state, blue,  below=0.3cm of 4, scale = 0.4] (A4) {A};
\node[state, blue,  below=0.3cm of A4, scale = 0.4] (C4) {C};
\node[state, blue,   left =0.05cm of C4, scale = 0.4] (B4) {B};
\node[state, blue,   right =0.05cm of C4, scale = 0.4] (D4) {D};
\node[below =0.005cm of B4 ] (B4_probability)  {$\frac{1}{4}$};
\node[below =0.005cm of C4 ] (C4_probability)  {$\frac{2}{4}$};	
\node[below =0.005cm of D4 ] (D4_probability)  {$\frac{1}{4}$};

\node[state, blue,  below = 0.2cm of 6, scale = 0.4] (A6) {A};
\node[state, blue,  below=0.3cm of A6, scale = 0.4] (C6) {C};
\node[state, blue,   left =0.05cm of C6, scale = 0.4] (B6) {B};
\node[state, blue,   right =0.05cm of C6, scale = 0.4] (D6) {D};
\node[below =0.005cm of B6 ] (B6_probability)  {$\frac{3}{6}$};
\node[below =0.005cm of C6 ] (C6_probability)  {$\frac{1}{6}$};	
\node[below =0.005cm of D6 ] (D6_probability)  {$\frac{2}{6}$};

\node[state, blue,  below = 0.2cm of 6, scale = 0.4] (A6) {A};
\node[state, blue,  below=0.3cm of A6, scale = 0.4] (C6) {C};
\node[state, blue,   left =0.05cm of C6, scale = 0.4] (B6) {B};
\node[state, blue,   right =0.05cm of C6, scale = 0.4] (D6) {D};
\node[below =0.005cm of B6 ] (B6_probability)  {$\frac{3}{6}$};
\node[below =0.005cm of C6 ] (C6_probability)  {$\frac{1}{6}$};	
\node[below =0.005cm of D6 ] (D6_probability)  {$\frac{2}{6}$};

\node[state, blue,  below = 0.2cm of 7, scale = 0.4] (A7) {A};
\node[state, blue,  below=0.3cm of A7, scale = 0.4] (C7) {C};
\node[state, blue,   left =0.05cm of C7, scale = 0.4] (B7) {B};
\node[state, blue,   right =0.05cm of C7, scale = 0.4] (D7) {D};
\node[below =0.005cm of B7 ] (B7_probability)  {$\frac{1}{5}$};
\node[below =0.005cm of C7 ] (C7_probability)  {$\frac{1}{5}$};	
\node[below =0.005cm of D7 ] (D7_probability)  {$\frac{3}{5}$};

	\path
    (1) edge[red, very thick] (2)
    (1) edge[red, very thick] (3)
    (1) edge[red, very thick] (3bis)
    (2) edge[red, very thick] (4)
    (3bis) edge[red, very thick] (5)
    (3) edge[red, very thick] (6)
    (3) edge[red, very thick] (7)

     (1) edge[blue] (A1)
    (A1) edge[blue] (B1)
    (A1) edge[blue] (C1)
    (A1) edge[blue] (D1)
  
    (2) edge[blue] (A2)
    (A2) edge[blue] (B2)
    (A2) edge[blue] (C2)
    (A2) edge[blue] (D2)

    (3) edge[blue] (A3)
    (A3) edge[blue] (B3)
    (A3) edge[blue] (C3)
    (A3) edge[blue] (D3)  

    (3bis) edge[blue] (A3b)
    (A3b) edge[blue] (B3b)
    (A3b) edge[blue] (C3b)
    (A3b) edge[blue] (D3b) 

    (4) edge[blue] (A4)
    (A4) edge[blue] (B4)
    (A4) edge[blue] (C4)
    (A4) edge[blue] (D4)

    (5) edge[blue] (A5)
    (A5) edge[blue] (B5)
    (A5) edge[blue] (C5)
    (A5) edge[blue] (D5)

    (6) edge[blue] (A6)
    (A6) edge[blue] (B6)
    (A6) edge[blue] (C6)
    (A6) edge[blue] (D6)

    (7) edge[blue] (A7)
    (A7) edge[blue] (B7)
    (A7) edge[blue] (C7)
    (A7) edge[blue] (D7)       
    ;
\end{scope}
	\end{tikzpicture}
}
  \end{center}
  \caption {Visual representation of 2 multi-horizon subgroups obtained {by dissecting} the strategic {scenarios} into 2 subgroups, each of cardinality 4, with 2 strategic scenarios fixed, consisting of the whole operational sample space at each strategic node. } 
\label{MHESG_2_groups_4_scen_2_fixed}
  \end{figure}

%% file: MHEG_Strategic_Operational_dissection.tex
\begin{figure}[ht] 
	 \begin{center}
  \resizebox{0.9\textwidth}{0.85\textwidth}{%
	\begin{tikzpicture}[-, >=stealth', auto, thick, node distance = 1.2 cm]
    \begin{scope}[shift={(0, 1)}]
	\node[state, red, very thick] (1) {$1$};
 \node at (0,1) [rectangle,draw] (MHSG) {$\mathfrak{T}(3)$};
 \node[state, below left= of 1, red, very thick ] (2)    {$2$};
\node[state, below right=of 1, red, very thick] (3)               {$3$};
\node[below left = -0.2cm and 0.6cm of 1, red, very thick ] (2_probability)  {$\frac{1}{3}$};
 \node[below right = -0.2cm and 0.6cm of 1, red, very thick] (3_probability)  {$\frac{2}{3}$};

\node[state, blue,  below=0.2cm of 1, scale = 0.4] (A1) {A};
\node[state, blue,  below left =0.2cm and 0.45cm of A1, scale = 0.4] (B1) {B};
\node[state, blue,  below left =0.2cm and 0.01cm of A1, scale = 0.4] (C1) {C};
\node[state, blue,   below right =0.2cm and 0.01cm of A1, scale = 0.4] (D1) {D};
\node[state, blue,   below right =0.2cm and 0.45cm of A1, scale = 0.4] (E1) {E};
\node[below =0.005cm of B1 ] (B1_probability)  {$\frac{1}{5}$};
\node[below =0.005cm of C1 ] (C1_probability)  {$\frac{1}{5}$};	
\node[below =0.005cm of D1 ] (D1_probability)  {$\frac{2}{5}$};	
\node[below =0.005cm of E1 ] (E1_probability)  {$\frac{1}{5}$};

\node[state, blue,  below=0.2cm of 2, scale = 0.4] (A2) {A};
\node[state, blue,  below left =0.2cm and 0.45cm of A2, scale = 0.4] (B2) {B};
\node[state, blue,  below left =0.2cm and 0.01cm of A2, scale = 0.4] (C2) {C};
\node[state, blue,   below right =0.2cm and 0.01cm of A2, scale = 0.4] (D2) {D};
\node[state, blue,   below right =0.2cm and 0.45cm of A2, scale = 0.4] (E2) {E};
\node[below =0.005cm of B2 ] (B2_probability)  {$\frac{1}{6}$};
\node[below =0.005cm of C2 ] (C2_probability)  {$\frac{1}{6}$};	
\node[below =0.005cm of D2 ] (D2_probability)  {$\frac{2}{6}$};	
\node[below =0.005cm of E2 ] (E2_probability)  {$\frac{2}{6}$};

\node[state, blue,  below=0.2cm of 3, scale = 0.4] (A3) {A};
\node[state, blue,  below left =0.2cm and 0.45cm of A3, scale = 0.4] (B3) {B};
\node[state, blue,  below left =0.2cm and 0.01cm of A3, scale = 0.4] (C3) {C};
\node[state, blue,   below right =0.2cm and 0.01cm of A3, scale = 0.4] (D3) {D};
\node[state, blue,   below right =0.2cm and 0.45cm of A3, scale = 0.4] (E3) {E};
\node[below =0.005cm of B3 ] (B3_probability)  {$\frac{1}{7}$};
\node[below =0.005cm of C3 ] (C3_probability)  {$\frac{3}{7}$};	
\node[below =0.005cm of D3 ] (D3_probability)  {$\frac{2}{7}$};	
\node[below =0.005cm of E3 ] (E3_probability)  {$\frac{1}{7}$};	

	\path
    (1) edge[red, very thick] (2)
    (1) edge[red, very thick] (3)

    (1)  edge[blue] (A1)
    (A1) edge[blue] (B1)
    (A1) edge[blue] (C1)
    (A1) edge[blue] (D1)
    (A1) edge[blue] (E1)
    
    (2) edge[blue] (A2)
    (A2) edge[blue] (B2)
    (A2) edge[blue] (C2)
    (A2) edge[blue] (D2)
    (A2) edge[blue] (E2)
    
    (3) edge[blue] (A3)
    (A3) edge[blue] (B3)
    (A3) edge[blue] (C3)
    (A3) edge[blue] (D3)
    (A3) edge[blue] (E3)

    ;
 
\draw[->](-5.4,-4) -- (-5.4,-4.5); 
\draw[->](5.4,-4) -- (5.4,-4.5); 
\draw[-](-5.4,-4) -- (5.4, -4);
\node at (5,-3.6)[rectangle,draw] {Subgroups of cardinality $\hat{g}_2=1$};
\end{scope}

\begin{scope}[shift={(-5.8, -4.9)}]
\node[state, red, very thick, scale=0.8] (1) {$1$};
 \node at (0,1) [rectangle,draw] (MHSG) {$(\mathcal{N}_1^{(1)},\Omega_{\textnormal{OG},1}^{(1,3)}):  \phi_{G,1,1}^{(1,3)} =\frac{1}{3}$};
 \node[state, below left= of 1, red, very thick, scale=0.8 ] (2)    {$2$};
\node[state, below right=of 1, white, very thick, scale=0.8] (3)               {$3$};
\node[below left = -0.2cm and 0.6cm of 1, red, very thick] (2_probability)  {$1$};

\node[state, blue,  below=0.2cm of 1, scale = 0.4] (A1) {A};
\node[state, blue,  below left =0.2cm and 0.45cm of A1, scale = 0.4] (B1) {B};
\node[state, blue,  below left =0.2cm and 0.01cm of A1, scale = 0.4] (C1) {C};
\node[state, blue,   below right =0.2cm and 0.01cm of A1, scale = 0.4] (D1) {D};
\node[state, blue,   below right =0.2cm and 0.45cm of A1, scale = 0.4] (E1) {E};
\node[below =0.005cm of B1 ] (B1_probability)  {$\frac{1}{5}$};
\node[below =0.005cm of C1 ] (C1_probability)  {$\frac{1}{5}$};	
\node[below =0.005cm of D1 ] (D1_probability)  {$\frac{2}{5}$};	
\node[below =0.005cm of E1 ] (E1_probability)  {$\frac{1}{5}$};

\node[state, blue,  below=0.2cm of 2, scale = 0.4] (A2) {A};
\node[state, blue,  below left =0.2cm and 0.45cm of A2, scale = 0.4] (B2) {B};
\node[state, blue,  below left =0.2cm and 0.01cm of A2, scale = 0.4] (C2) {C};
\node[state, blue,   below right =0.2cm and 0.01cm of A2, scale = 0.4] (D2) {D};
\node[state, blue,   below right =0.2cm and 0.45cm of A2, scale = 0.4] (E2) {E};
\node[below =0.005cm of B2 ] (B2_probability)  {$\frac{1}{6}$};
\node[below =0.005cm of C2 ] (C2_probability)  {$\frac{1}{6}$};	
\node[below =0.005cm of D2 ] (D2_probability)  {$\frac{2}{6}$};	
\node[below =0.005cm of E2 ] (E2_probability)  {$\frac{2}{6}$};	
	\path
    (1) edge[red, very thick] (2)
    (1) edge[white, very thick] (3)

    (1)  edge[blue] (A1)
    (A1) edge[blue] (B1)
    (A1) edge[blue] (C1)
    (A1) edge[blue] (D1)
    (A1) edge[blue] (E1)
    
    (2) edge[blue] (A2)
    (A2) edge[blue] (B2)
    (A2) edge[blue] (C2)
    (A2) edge[blue] (D2)
    (A2) edge[blue] (E2);
     \draw[->](-3.6,-4) -- (-3.6,-4.5); 
\draw[->](2.8,-4) -- (2.8,-4.5); 
\draw[-](-3.6,-4) -- (2.8, -4);
\end{scope}

\begin{scope}[shift={(5.8, -4.9)}]
 \node[state, red, very thick, scale=0.8] (1) {$1$};
 \node at (0,1) [rectangle,draw] (MHSG) {$(\mathcal{N}_2^{(1)},\Omega_{\textnormal{OG},1}^{(1,3)}):  \phi_{G,2,1}^{(1,3)}  =\frac{2}{3}$};
 \node[state, below left= of 1, white, very thick, scale=0.8 ] (2)    {$2$};
\node[state, below right=of 1, red, very thick, scale=0.8] (3)               {$3$};
 \node[below right = -0.2cm and 0.6cm of 1, red, very thick] (3_probability)  {$1$};

\node[state, blue,  below=0.2cm of 1, scale = 0.4] (A1) {A};
\node[state, blue,  below left =0.2cm and 0.45cm of A1, scale = 0.4] (B1) {B};
\node[state, blue,  below left =0.2cm and 0.01cm of A1, scale = 0.4] (C1) {C};
\node[state, blue,   below right =0.2cm and 0.01cm of A1, scale = 0.4] (D1) {D};
\node[state, blue,   below right =0.2cm and 0.45cm of A1, scale = 0.4] (E1) {E};
\node[below =0.005cm of B1 ] (B1_probability)  {$\frac{1}{5}$};
\node[below =0.005cm of C1 ] (C1_probability)  {$\frac{1}{5}$};	
\node[below =0.005cm of D1 ] (D1_probability)  {$\frac{2}{5}$};	
\node[below =0.005cm of E1 ] (E1_probability)  {$\frac{1}{5}$};

\node[state, white,  below=0.2cm of 2, scale = 0.4] (A2) {A};
\node[state, blue,  below=0.2cm of 3, scale = 0.4] (A3) {A};
\node[state, blue,  below left =0.2cm and 0.45cm of A3, scale = 0.4] (B3) {B};
\node[state, blue,  below left =0.2cm and 0.01cm of A3, scale = 0.4] (C3) {C};
\node[state, blue,   below right =0.2cm and 0.01cm of A3, scale = 0.4] (D3) {D};
\node[state, blue,   below right =0.2cm and 0.45cm of A3, scale = 0.4] (E3) {E};
\node[below =0.005cm of B3 ] (B3_probability)  {$\frac{1}{7}$};
\node[below =0.005cm of C3 ] (C3_probability)  {$\frac{3}{7}$};	
\node[below =0.005cm of D3 ] (D3_probability)  {$\frac{2}{7}$};	
\node[below =0.005cm of E3 ] (E3_probability)  {$\frac{1}{7}$};	

	\path
    (1) edge[red, very thick] (3)

    (1)  edge[blue] (A1)
    (A1) edge[blue] (B1)
    (A1) edge[blue] (C1)
    (A1) edge[blue] (D1)
    (A1) edge[blue] (E1)

    (3) edge[blue] (A3)
    (A3) edge[blue] (B3)
    (A3) edge[blue] (C3)
    (A3) edge[blue] (D3)
    (A3) edge[blue] (E3)

    ;
    \draw[->](-2.8,-4) -- (-2.8,-4.5); 
\draw[->](3.6,-4) -- (3.6,-4.5); 
\draw[-](-2.8,-4) -- (3.6, -4);
\node at (5.6,-3.6)[rectangle,draw] {Subgroups of cardinality $g_2=2$ and $\hat{g}_2=1$};
\end{scope}

\begin{scope}[shift={(-9.4, -10.8)}]
	\node[state, red, very thick, scale=0.8] (1) {$1$};
 \node at (0,1) [rectangle,draw] (MHSG) {$ (\mathcal{N}_1^{(1)},\Omega_{\textnormal{OG},1}^{(1,2)}):\phi_{G,1,1}^{(1,2)}=\frac{1}{3}\left(\frac{2}{5} \times \frac{2}{6}\right)$};
 \node[state, below left= of 1, red, very thick , scale=0.8] (2)    {$2$};
\node[state, below right=of 1, white, very thick, scale=0.8] (3)               {$3$};
\node[below left = -0.2cm and 0.6cm of 1, red, very thick] (2_probability)  {$1$};

\node[state, blue,  below=0.2cm of 1, scale = 0.4] (A1) {A};
\node[state, white,  below left =0.2cm and 0.45cm of A1, scale = 0.4] (B1) {B};
\node[state, blue,  below left =0.2cm and 0.01cm of A1, scale = 0.4] (C1) {B};
\node[state, blue,   below right =0.2cm and 0.01cm of A1, scale = 0.4] (D1) {C};
\node[state, white,   below right =0.2cm and 0.45cm of A1, scale = 0.4] (E1) {E};

\node[below =0.005cm of C1 ] (C1_probability)  {$\frac{1}{2}$};	
\node[below =0.005cm of D1 ] (D1_probability)  {$\frac{1}{2}$};

\node[state, blue,  below=0.2cm of 2, scale = 0.4] (A2) {A};
\node[state, white,  below left =0.2cm and 0.45cm of A2, scale = 0.4] (B2) {B};
\node[state, blue,  below left =0.2cm and 0.01cm of A2, scale = 0.4] (C2) {B};
\node[state, blue,   below right =0.2cm and 0.01cm of A2, scale = 0.4] (D2) {C};
\node[state, white,   below right =0.2cm and 0.45cm of A2, scale = 0.4] (E2) {E};
\node[below =0.005cm of C2 ] (C2_probability)  {$\frac{1}{2}$};	
\node[below =0.005cm of D2 ] (D2_probability)  {$\frac{1}{2}$};		
	\path
    (1) edge[red, very thick] (2)
    (1) edge[white, very thick] (3)

    (1)  edge[blue] (A1)
    (A1) edge[blue] (C1)
    (A1) edge[blue] (D1)
    
    (2) edge[blue] (A2)
    (A2) edge[blue] (C2)
    (A2) edge[blue] (D2);
\end{scope}

\begin{scope}[shift={(-3.1, -10.8)}]
	\node[state, red, very thick, scale=0.8] (1) {$1$};
 \node at (0,1) [rectangle,draw] (MHSG) {$(\mathcal{N}_1^{(1)},\Omega_{\textnormal{OG},2}^{(1,2)}): \phi_{G,1,2}^{(1,2)} =\frac{1}{3}\left(\frac{2}{5} \times \frac{4}{6}\right)$};
 \node[state, below left= of 1, red, very thick, scale=0.8 ] (2)    {$2$};
\node[state, below right=of 1, white, very thick, scale=0.8] (3)               {$3$};
\node[below left = -0.2cm and 0.6cm of 1, red, very thick ] (2_probability)  {$1$};

\node[state, blue,  below=0.2cm of 1, scale = 0.4] (A1) {A};
\node[state, white,  below left =0.2cm and 0.45cm of A1, scale = 0.4] (B1) {B};
\node[state, blue,  below left =0.2cm and 0.01cm of A1, scale = 0.4] (C1) {B};
\node[state, blue,   below right =0.2cm and 0.01cm of A1, scale = 0.4] (D1) {C};
\node[state, white,   below right =0.2cm and 0.45cm of A1, scale = 0.4] (E1) {E};
\node[below =0.005cm of C1 ] (C1_probability)  {$\frac{1}{2}$};	
\node[below =0.005cm of D1 ] (D1_probability)  {$\frac{1}{2}$};

\node[state, blue,  below=0.2cm of 2, scale = 0.4] (A2) {A};
\node[state, white,  below left =0.2cm and 0.45cm of A2, scale = 0.4] (B2) {B};
\node[state, blue,  below left =0.2cm and 0.01cm of A2, scale = 0.4] (C2) {D};
\node[state, blue,   below right =0.2cm and 0.01cm of A2, scale = 0.4] (D2) {E};
\node[state, white,   below right =0.2cm and 0.45cm of A2, scale = 0.4] (E2) {E};
\node[below =0.005cm of C2 ] (C2_probability)  {$\frac{2}{4}$};	
\node[below =0.005cm of D2 ] (D2_probability)  {$\frac{2}{4}$};		
	\path
    (1) edge[red, very thick] (2)
    (1) edge[white, very thick] (3)

    (1)  edge[blue] (A1)
    (A1) edge[blue] (C1)
    (A1) edge[blue] (D1)
    
    (2) edge[blue] (A2)
    (A2) edge[blue] (C2)
    (A2) edge[blue] (D2);
\end{scope}

\begin{scope}[shift={(-9.4, -16)}]
	\node[state, red, very thick, scale=0.8] (1) {$1$};
 \node at (0,1) [rectangle,draw] (MHSG) {$ (\mathcal{N}_1^{(1)},\Omega_{\textnormal{OG},3}^{(1,2)}):\phi_{G,1,3}^{(1,2)} =\frac{1}{3}\left(\frac{3}{5} \times \frac{2}{6}\right)$};
 \node[state, below left= of 1, red, very thick, scale=0.8 ] (2)    {$2$};
\node[state, below right=of 1, white, very thick, scale=0.8] (3)               {$3$};
\node[below left = -0.2cm and 0.6cm of 1, red, very thick] (2_probability)  {$1$};

\node[state, blue,  below=0.2cm of 1, scale = 0.4] (A1) {A};
\node[state, white,  below left =0.2cm and 0.45cm of A1, scale = 0.4] (B1) {B};
\node[state, blue,  below left =0.2cm and 0.01cm of A1, scale = 0.4] (C1) {D};
\node[state, blue,   below right =0.2cm and 0.01cm of A1, scale = 0.4] (D1) {E};
\node[state, white,   below right =0.2cm and 0.45cm of A1, scale = 0.4] (E1) {E};
\node[below =0.005cm of C1 ] (C1_probability)  {$\frac{2}{3}$};	
\node[below =0.005cm of D1 ] (D1_probability)  {$\frac{1}{3}$};

\node[state, blue,  below=0.2cm of 2, scale = 0.4] (A2) {A};
\node[state, white,  below left =0.2cm and 0.45cm of A2, scale = 0.4] (B2) {B};
\node[state, blue,  below left =0.2cm and 0.01cm of A2, scale = 0.4] (C2) {B};
\node[state, blue,   below right =0.2cm and 0.01cm of A2, scale = 0.4] (D2) {C};
\node[state, white,   below right =0.2cm and 0.45cm of A2, scale = 0.4] (E2) {E};
\node[below =0.005cm of C2 ] (C2_probability)  {$\frac{1}{2}$};	
\node[below =0.005cm of D2 ] (D2_probability)  {$\frac{1}{2}$};	
	\path
    (1) edge[red, very thick] (2)
    (1) edge[white, very thick] (3)

    (1)  edge[blue] (A1)
    (A1) edge[blue] (C1)
    (A1) edge[blue] (D1)
    
    (2) edge[blue] (A2)
    (A2) edge[blue] (C2)
    (A2) edge[blue] (D2);
\end{scope}

\begin{scope}[shift={(-3.1, -16)}]
	\node[state, red, very thick, scale=0.8] (1) {$1$};
 \node at (0,1) [rectangle,draw] (MHSG) {$ (\mathcal{N}_1^{(1)},\Omega_{\textnormal{OG},4}^{(1,2)}):\phi_{G,1,4}^{(1,2)}=\frac{1}{3}\left(\frac{3}{5} \times \frac{4}{6}\right)$};
 \node[state, below left= of 1, red, very thick, scale=0.8 ] (2)    {$2$};
\node[state, below right=of 1, white, very thick, scale=0.8] (3)               {$3$};
\node[below left = -0.2cm and 0.6cm of 1, red, very thick] (2_probability)  {$1$};

\node[state, blue,  below=0.2cm of 1, scale = 0.4] (A1) {A};
\node[state, white,  below left =0.2cm and 0.45cm of A1, scale = 0.4] (B1) {B};
\node[state, blue,  below left =0.2cm and 0.01cm of A1, scale = 0.4] (C1) {D};
\node[state, blue,   below right =0.2cm and 0.01cm of A1, scale = 0.4] (D1) {E};
\node[state, white,   below right =0.2cm and 0.45cm of A1, scale = 0.4] (E1) {E};
\node[below =0.005cm of C1 ] (C1_probability)  {$\frac{2}{3}$};	
\node[below =0.005cm of D1 ] (D1_probability)  {$\frac{1}{3}$};

\node[state, blue,  below=0.2cm of 2, scale = 0.4] (A2) {A};
\node[state, white,  below left =0.2cm and 0.45cm of A2, scale = 0.4] (B2) {B};
\node[state, blue,  below left =0.2cm and 0.01cm of A2, scale = 0.4] (C2) {D};
\node[state, blue,   below right =0.2cm and 0.01cm of A2, scale = 0.4] (D2) {E};
\node[state, white,   below right =0.2cm and 0.45cm of A2, scale = 0.4] (E2) {E};
\node[below =0.005cm of C2 ] (C2_probability)  {$\frac{2}{4}$};	
\node[below =0.005cm of D2 ] (D2_probability)  {$\frac{2}{4}$};	
	\path
    (1) edge[red, very thick] (2)
    (1) edge[white, very thick] (3)

    (1)  edge[blue] (A1)
    (A1) edge[blue] (C1)
    (A1) edge[blue] (D1)
    
    (2) edge[blue] (A2)
    (A2) edge[blue] (C2)
    (A2) edge[blue] (D2);
\end{scope}

\begin{scope}[shift={(3.1, -10.8)}]
 \node[state, red, very thick, scale=0.8] (1) {$1$};
 \node at (0,1) [rectangle,draw] (MHSG) {$(\mathcal{N}_2^{(1)},\Omega_{\textnormal{OG},1}^{(1,2)}):\phi_{G,2,1}^{(1,2)} =\frac{2}{3}\left(\frac{2}{5} \times \frac{4}{7}\right)$};
 \node[state, below left= of 1, white, very thick, scale=0.8 ] (2)    {$2$};
\node[state, below right=of 1, red, very thick, scale=0.8] (3)               {$3$};
 \node[below right = -0.2cm and 0.6cm of 1, red, very thick] (3_probability)  {$1$};

\node[state, blue,  below=0.2cm of 1, scale = 0.4] (A1) {A};
\node[state, white,  below left =0.2cm and 0.45cm of A1, scale = 0.4] (B1) {B};
\node[state, blue,  below left =0.2cm and 0.01cm of A1, scale = 0.4] (C1) {B};
\node[state, blue,   below right =0.2cm and 0.01cm of A1, scale = 0.4] (D1) {C};
\node[state, white,   below right =0.2cm and 0.45cm of A1, scale = 0.4] (E1) {E};
\node[below =0.005cm of C1 ] (C1_probability)  {$\frac{1}{2}$};	
\node[below =0.005cm of D1 ] (D1_probability)  {$\frac{1}{2}$};	

\node[state, white,  below=0.2cm of 2, scale = 0.4] (A2) {A};
\node[state, blue,  below=0.2cm of 3, scale = 0.4] (A3) {A};
\node[state, white,  below left =0.2cm and 0.45cm of A3, scale = 0.4] (B3) {B};
\node[state, blue,  below left =0.2cm and 0.01cm of A3, scale = 0.4] (C3) {B};
\node[state, blue,   below right =0.2cm and 0.01cm of A3, scale = 0.4] (D3) {C};
\node[state, white,   below right =0.2cm and 0.45cm of A3, scale = 0.4] (E3) {E};
\node[below =0.005cm of C3 ] (C3_probability)  {$\frac{1}{4}$};	
\node[below =0.005cm of D3 ] (D3_probability)  {$\frac{3}{4}$};	

	\path
    (1) edge[red, very thick] (3)

    (1)  edge[blue] (A1)
    (A1) edge[blue] (C1)
    (A1) edge[blue] (D1)
   
    (3) edge[blue] (A3)
    (A3) edge[blue] (C3)
    (A3) edge[blue] (D3);
   
\end{scope}

\begin{scope}[shift={(9.4, -10.8)}]
 \node[state, red, very thick, scale=0.8] (1) {$1$};
 \node at (0,1) [rectangle,draw] (MHSG) {$(\mathcal{N}_2^{(1)},\Omega_{\textnormal{OG},2}^{(1,2)}):\phi_{G,2,2}^{(1,2)} =\frac{2}{3}\left(\frac{2}{5} \times \frac{3}{7}\right)$};
 \node[state, below left= of 1, white, very thick, scale=0.8 ] (2)    {$2$};
\node[state, below right=of 1, red, very thick, scale=0.8] (3)               {$3$};
 \node[below right = -0.2cm and 0.6cm of 1, red, very thick] (3_probability)  {$1$};

\node[state, blue,  below=0.2cm of 1, scale = 0.4] (A1) {A};
\node[state, white,  below left =0.2cm and 0.45cm of A1, scale = 0.4] (B1) {B};
\node[state, blue,  below left =0.2cm and 0.01cm of A1, scale = 0.4] (C1) {B};
\node[state, blue,   below right =0.2cm and 0.01cm of A1, scale = 0.4] (D1) {C};
\node[state, white,   below right =0.2cm and 0.45cm of A1, scale = 0.4] (E1) {E};
\node[below =0.005cm of C1 ] (C1_probability)  {$\frac{1}{2}$};	
\node[below =0.005cm of D1 ] (D1_probability)  {$\frac{1}{2}$};	

\node[state, white,  below=0.2cm of 2, scale = 0.4] (A2) {A};
\node[state, blue,  below=0.2cm of 3, scale = 0.4] (A3) {A};
\node[state, white,  below left =0.2cm and 0.45cm of A3, scale = 0.4] (B3) {D};
\node[state, blue,  below left =0.2cm and 0.01cm of A3, scale = 0.4] (C3) {D};
\node[state, blue,   below right =0.2cm and 0.01cm of A3, scale = 0.4] (D3) {E};
\node[state, white,   below right =0.2cm and 0.45cm of A3, scale = 0.4] (E3) {E};
\node[below =0.005cm of C3 ] (C3_probability)  {$\frac{2}{3}$};	
\node[below =0.005cm of D3 ] (D3_probability)  {$\frac{1}{3}$};	

	\path
    (1) edge[red, very thick] (3)

    (1)  edge[blue] (A1)
    (A1) edge[blue] (C1)
    (A1) edge[blue] (D1)
    
    (3) edge[blue] (A3)
    (A3) edge[blue] (C3)
    (A3) edge[blue] (D3);
   
\end{scope}

\begin{scope}[shift={(3.1, -16)}]
 \node[state, red, very thick, scale=0.8] (1) {$1$};
 \node at (0,1) [rectangle,draw] (MHSG) {$(\mathcal{N}_2^{(1)},\Omega_{\textnormal{OG},3}^{(1,2)}):\phi_{G,2,3}^{(1,2)} =\frac{2}{3}\left(\frac{3}{5} \times \frac{4}{7}\right)$};
 \node[state, below left= of 1, white, very thick, scale=0.8 ] (2)    {$2$};
\node[state, below right=of 1, red, very thick, scale=0.8] (3)               {$3$};
 \node[below right = -0.2cm and 0.6cm of 1, red, very thick] (3_probability)  {$1$};

\node[state, blue,  below=0.2cm of 1, scale = 0.4] (A1) {A};
\node[state, white,  below left =0.2cm and 0.45cm of A1, scale = 0.4] (B1) {B};
\node[state, blue,  below left =0.2cm and 0.01cm of A1, scale = 0.4] (C1) {D};
\node[state, blue,   below right =0.2cm and 0.01cm of A1, scale = 0.4] (D1) {E};
\node[state, white,   below right =0.2cm and 0.45cm of A1, scale = 0.4] (E1) {E};
\node[below =0.005cm of C1 ] (C1_probability)  {$\frac{2}{3}$};	
\node[below =0.005cm of D1 ] (D1_probability)  {$\frac{1}{3}$};	

\node[state, white,  below=0.2cm of 2, scale = 0.4] (A2) {A};
\node[state, blue,  below=0.2cm of 3, scale = 0.4] (A3) {A};
\node[state, white,  below left =0.2cm and 0.45cm of A3, scale = 0.4] (B3) {B};
\node[state, blue,  below left =0.2cm and 0.01cm of A3, scale = 0.4] (C3) {B};
\node[state, blue,   below right =0.2cm and 0.01cm of A3, scale = 0.4] (D3) {C};
\node[state, white,   below right =0.2cm and 0.45cm of A3, scale = 0.4] (E3) {E};
\node[below =0.005cm of C3 ] (C3_probability)  {$\frac{1}{4}$};	
\node[below =0.005cm of D3 ] (D3_probability)  {$\frac{3}{4}$};	

	\path
    (1) edge[red, very thick] (3)

    (1)  edge[blue] (A1)
    (A1) edge[blue] (C1)
    (A1) edge[blue] (D1)
    
    (3) edge[blue] (A3)
    (A3) edge[blue] (C3)
    (A3) edge[blue] (D3);
   
\end{scope}

\begin{scope}[shift={(9.4, -16)}]
 \node[state, red, very thick, scale=0.8] (1) {$1$};
 \node at (0,1) [rectangle,draw] (MHSG) {$(\mathcal{N}_2^{(1)},\Omega_{\textnormal{OG},4}^{(1,2)}):\phi_{G,2,4}^{(1,2)} =\frac{2}{3}\left(\frac{3}{5} \times \frac{3}{7}\right)$};
 \node[state, below left= of 1, white, very thick, scale=0.8 ] (2)    {$2$};
\node[state, below right=of 1, red, very thick, scale=0.8] (3)               {$3$};
 \node[below right = -0.2cm and 0.6cm of 1, red, very thick] (3_probability)  {$1$};

\node[state, blue,  below=0.2cm of 1, scale = 0.4] (A1) {A};
\node[state, white,  below left =0.2cm and 0.45cm of A1, scale = 0.4] (B1) {B};
\node[state, blue,  below left =0.2cm and 0.01cm of A1, scale = 0.4] (C1) {D};
\node[state, blue,   below right =0.2cm and 0.01cm of A1, scale = 0.4] (D1) {E};
\node[state, white,   below right =0.2cm and 0.45cm of A1, scale = 0.4] (E1) {E};
\node[below =0.005cm of C1 ] (C1_probability)  {$\frac{2}{3}$};	
\node[below =0.005cm of D1 ] (D1_probability)  {$\frac{1}{3}$};	

\node[state, white,  below=0.2cm of 2, scale = 0.4] (A2) {A};
\node[state, blue,  below=0.2cm of 3, scale = 0.4] (A3) {A};
\node[state, white,  below left =0.2cm and 0.45cm of A3, scale = 0.4] (B3) {B};
\node[state, blue,  below left =0.2cm and 0.01cm of A3, scale = 0.4] (C3) {D};
\node[state, blue,   below right =0.2cm and 0.01cm of A3, scale = 0.4] (D3) {E};
\node[state, white,   below right =0.2cm and 0.45cm of A3, scale = 0.4] (E3) {E};
\node[below =0.005cm of C3 ] (C3_probability)  {$\frac{2}{3}$};	
\node[below =0.005cm of D3 ] (D3_probability)  {$\frac{1}{3}$};	

	\path
    (1) edge[red, very thick] (3)

    (1)  edge[blue] (A1)
    (A1) edge[blue] (C1)
    (A1) edge[blue] (D1)
    
    (3) edge[blue] (A3)
    (A3) edge[blue] (C3)
    (A3) edge[blue] (D3);
   
\end{scope}

	\end{tikzpicture}
}
  \end{center}
  
  \caption {Visual representation of 8 multi-horizon subgroups obtained by dissecting both the strategic and operational scenarios of $\mathfrak{T}(3)$ into 2 disjoint subgroups, each of cardinality 2, for the expectation-based bound.   }
\label{Figure_dissect_strategic_then_operational}
  \end{figure}